\documentclass[10pt]{amsart}
\usepackage{amssymb,latexsym,amsfonts,multicol}
\usepackage{amsthm}
\usepackage{amscd}
\usepackage{amsmath,mathrsfs}
\usepackage{fontenc}
\usepackage{amsmath}
\usepackage{geometry}
\usepackage{graphicx}
\usepackage[all]{xy}
\usepackage[T1]{fontenc}
\usepackage{textcomp}
\DeclareMathOperator{\Pro}{Pro}
\DeclareMathOperator{\Cat}{Cat}
\newcommand{\om}{\omega}
\newcommand{\on}{\operatorname}
\DeclareMathOperator{\pr}{pr}

\DeclareMathOperator{\colim}{colim}

\DeclareMathOperator{\fSets}{fSets}
\DeclareMathOperator{\PShv}{PShv}
\DeclareMathOperator{\Shv}{Shv}
\DeclareMathOperator{\Ob}{Ob}
\DeclareMathOperator{\Mor}{Mor}

\DeclareMathOperator{\map}{map}
\DeclareMathOperator{\Mod}{Mod}
\DeclareMathOperator{\Fl}{\mathfrak{Fl}}
\DeclareMathOperator{\Prest}{PreSt}
\DeclareMathOperator{\Coker}{Coker}

\DeclareMathOperator{\St}{St}
\DeclareMathOperator{\ren}{ren}
\DeclareMathOperator{\flim}{flim}
\DeclareMathOperator{\cont}{cont}
\DeclareMathOperator{\op}{op}

\DeclareMathOperator{\qcqs}{qcqs}
\DeclareMathOperator{\cL}{\mathcal{L}}

\DeclareMathOperator{\Spec}{Spec}
\DeclareMathOperator{\Sets}{Sets}

\DeclareMathOperator{\tn}{tn}

\DeclareMathOperator{\Id}{Id}

\DeclareMathOperator{\PrCat}{PrCat_{st,\ell}}
\DeclareMathOperator{\Catst}{Cat_{st,\ell}}
\DeclareMathOperator{\Sch}{Sch}
\DeclareMathOperator{\Alg}{AlgSp}
\DeclareMathOperator{\Aff}{Aff}
\DeclareMathOperator{\Hom}{Hom}

\DeclareMathOperator{\End}{End}

\DeclareMathOperator{\Ima}{Im}
\DeclareMathOperator{\rs}{rs}
\DeclareMathOperator{\codim}{codim}

\DeclareMathOperator{\Fun}{Funct}

\DeclareMathOperator{\EExt}{\mathscr{E}\text{\kern -3pt {\calligra\large xt}}\,}
\DeclareMathOperator{\Ind}{Ind}
\DeclareMathOperator{\Ad}{Ad}

\DeclareMathOperator{\Aut}{Aut}

\DeclareMathOperator{\ad}{ad}

\newcommand{\pg}[1]{\pageref{#1}}
\newcommand{\cal}[1]{\mathcal{#1}}
\newcommand{\C}[1]{\cal#1}

\newcommand{\B}[1]{\mathbb#1}
\DeclareMathOperator{\Perv}{Perv}

\DeclareMathOperator{\Ker}{Ker}
\DeclareMathOperator{\can}{can}

\DeclareMathOperator{\perf}{perf}

\DeclareMathOperator{\Funct}{Funct}

\DeclareMathOperator{\X}{X_{\text{proיt}}}

\DeclareMathOperator{\Gal}{Gal}

\DeclareMathOperator{\id}{id}

\DeclareMathOperator{\Cov}{Cov}

\DeclareMathOperator{\cT}{\mathcal{T}}

\DeclareMathOperator{\pt}{pt}

\DeclareMathOperator{\SL}{SL}
\DeclareMathOperator{\Lie}{Lie}

\DeclareMathOperator{\cl}{cl}

\DeclareMathOperator{\res}{res}
\DeclareMathOperator{\red}{red}

\DeclareMathOperator{\ev}{ev}

\DeclareMathOperator{\ins}{ins}

\newtheorem{theorem}{Theorem}[subsection]

\newtheorem{Thm}[theorem]{Theorem}
\newtheorem{Prop}[theorem]{Proposition}
\newtheorem{Lem}[theorem]{Lemma}
\newtheorem{Cor}[theorem]{Corollary}

\newtheorem{Cl}[theorem]{Claim}

\theoremstyle{definition}
\newtheorem{Def}[theorem]{Definition}

\newtheorem{Emp}[theorem]{}

\numberwithin{equation}{section}

\newcommand{\form}[1]{(\ref{Eq:#1})}
\newcommand{\rsec}[1]{Section \ref{S:#1}}
\newcommand{\rl}[1]{Lemma \ref{L:#1}}

\newcommand{\rcl}[1]{Claim \ref{C:#1}}
\newcommand{\rp}[1]{Proposition \ref{P:#1}}

\newcommand{\re}[1]{\ref{E:#1}}
\newcommand{\rco}[1]{Corollary \ref{C:#1}}

\newcommand{\rt}[1] {Theorem \ref{T:#1}}

\newcommand{\rd}[1]{Definition \ref{D:#1}}

\newcommand{\Dt}{\Delta}
\newcommand{\Si}{\Sigma}
\newcommand{\dt}{\delta}
\newcommand{\un}{\underline}

\newcommand{\clp}{\mathcal{L}^{+}}

\newcommand{\cM}{\mathcal{M}}
\newcommand{\cQ}{\mathcal{Q}}

\newcommand{\cA}{\mathcal{A}}
\newcommand{\bP}{\mathbb{P}}

\newcommand{\bF}{\mathbb{F}}

\newcommand{\co}{\mathcal{O}}
\newcommand{\cO}{\mathcal{O}}

\newcommand{\cS}{\mathcal{S}}

\newcommand{\cB}{\mathcal{B}}

\newcommand{\f}{\phi}
\newcommand{\g}{\gamma}
\newcommand{\lan}{\langle}
\newcommand{\ran}{\rangle}

\newcommand{\si}{\sigma}
\newcommand{\ft}{\frak{t}}
\newcommand{\fC}{\frak{C}}

\newcommand{\fg}{\frak{g}}
\newcommand{\fc}{\frak{c}}

\newcommand{\isom}{\overset {\thicksim}{\to}}

\newcommand{\cY}{\mathcal{Y}}
\newcommand{\la}{\lambda}

\newcommand{\gm}{\mathbb{G}_{m}}

\newcommand{\kt}{\mathfrak{t}}
\newcommand{\kb}{\mathfrak{b}}

\newcommand{\al}{\alpha}
\newcommand{\br}{\textbf{r}}
\newcommand{\hra}{\hookrightarrow}
\newcommand{\sm}{\smallsetminus}
\newcommand{\cC}{\mathcal{C}}

\newcommand{\ov}{\overline}
\newcommand{\bql}{\overline{\mathbb{Q}}_{\ell}}

\newcommand{\kD}{\mathfrak{D}}
\newcommand{\kc}{\mathfrak{c}}
\newcommand{\kg}{\mathfrak{g}}

\newcommand{\bu}{_{\bullet}}

\newcommand{\overbar}[1]{\mkern 1.5mu\overline{\mkern-1.5mu#1\mkern-1.5mu}\mkern 1.5mu}

\newcommand{\NN}{\mathbb{N}}
\newcommand{\bG}{\mathbb{G}}

\newcommand{\Gm}{\Gamma}

\newcommand{\qlbar}{\ov{\mathbb{Q}}_l}

\newcommand{\wt}{\widetilde}

\newcommand{\cX}{\mathcal{X}}
\newcommand{\cD}{\mathcal{D}}
\newcommand{\cZ}{\mathcal{Z}}
\newcommand{\cI}{\mathcal{I}}
\newcommand{\La}{\Lambda}
\newcommand{\cU}{\mathcal{U}}
\newcommand{\lra}{\longrightarrow}
\newcommand{\lar}{\leftarrow}
\newcommand{\lla}{\longleftarrow}
\newcommand{\wh}{\widehat}
\newcommand{\Affft}{\Aff^{\on{ft}}}

\newcommand{\Algft}{\Alg^{\on{ft}}}

\newcommand{\e}{\par \noindent}

\def\dar[#1]{\ar@<2pt>[#1]\ar@<-2pt>[#1]}
\def\tar[#1]{\ar@<3pt>[#1]\ar@<-3pt>[#1]\ar@<0pt>[#1]}
\title[Perverse sheaves and affine Springer theory]
{Perverse sheaves on infinite-dimensional stacks, and affine Springer theory}
\author{Alexis Bouthier, David Kazhdan, Yakov Varshavsky}

\begin{document}
\date\today
\maketitle
\newcommand{\Addresses}{{
  \bigskip
  \footnotesize

  (A.\ Bouthier) \textsc{Institut de Mathematiques de Jussieu,
4 place Jussieu, 75005 Paris, France}\par\nopagebreak
  \textit{E-mail address}: \texttt{alexis.bouthier@imj-prg.fr}
	
	\medskip
	
	(D.\ Kazhdan) \textsc{Einstein Institute of Mathematics, Hebrew University, Givat Ram, Jerusalem,	91904, Israel}\par\nopagebreak
  \textit{E-mail address}: \texttt{kazhdan@math.huji.ac.il}
	\medskip
	
	(Y.\ Varshavsky) \textsc{Einstein Institute of Mathematics, Hebrew University, Givat Ram, Jerusalem,	91904, Israel}\par\nopagebreak
  \textit{E-mail address}: \texttt{vyakov@math.huji.ac.il}
}}

\begin{abstract}
The goal of this work is to construct a perverse $t$-structure on the $\infty$-category of $\ell$-adic $\cL G$-equivariant sheaves on the loop Lie algebra $\cL\fg$ and to show that the affine Grothendieck--Springer sheaf $\C{S}$ is perverse. Moreover, $\cS$ is an intermediate extension of its restriction to the locus of ``compact'' elements with regular semi-simple reduction. Note that classical methods do not apply in our situation  because $\cL G$ and $\cL\fg$ are infinite-dimensional ind-schemes.
\end{abstract}
\tableofcontents
\section*{Introduction}

\subsection{Motivation and brief outline}
\begin{Emp}
{\bf The finite-dimensional case.}
Let $k$ be an algebraically closed field, $G$ a connected reductive group over $k$, $\kg$ the Lie algebra of $G$, $B$ a Borel subgroup of $G$, $\kb$ the Lie algebra of $B$, $W$ the Weyl group of $G$, and 
$\cB:=G/B$ the flag variety of $G$. Consider the variety
\begin{center}
$\tilde{\kg}:=\{(gB,\g)\in \cB\times\kg~ \vert \Ad_{g^{-1}}(\g)\in\kb\}.$
\end{center}
The projection $\frak{p}^{\on{fin}}:\tilde{\kg}\rightarrow\kg$ is known as the Grothendieck--Springer resolution, and its fibers $\C{B}_{\g}$ are known as Springer fibers.

Lusztig observed (see \cite[$\S$3]{Lus1}) that $\frak{p}^{\on{fin}}$ is an $\Ad G$-equivariant small projective morphism, whose source is smooth and whose restriction to the regular semisimple locus is a Galois cover with Galois group $W$. Therefore the derived pushforward $\C{S}^{\on{fin}}:=\frak{p}^{\on{fin}}_{*}\bql[\dim\tilde{\kg}]$ is an $\Ad G$-equivariant semisimple perverse sheaf on $\kg$.\footnote{$\ell$ is a prime, invertible in $k$.} Moreover, $\C{S}^{\on{fin}}$ is equal to the intermediate extension of its restriction to the regular semisimple locus and it is  equipped with an action of $W$. In particular, the action of $W$ on  $\C{S}^{\on{fin}}$ induces an action of $W$ on the cohomology of each $\C{B}_{\g}$.

For each irreducible representation $V$ of $W$, we denote by $\C{S}^{\on{fin}}_{V}$ the $V$-isotypical component of $\C{S}^{\on{fin}}$. Each $\C{S}^{\on{fin}}_{V}$ is
an $\Ad G$-equivariant irreducible perverse sheaf on $\kg$, and these sheaves are (Lie algebra analogs of) special cases of Lusztig's character sheaves \cite{Lus3}. Character sheaves play a central role in the Lusztig's classification of irreducible characters of $G(\bF_{q})$ (see \cite{Lu5, Lu6}).
\end{Emp}

The goal of this paper is to develop an affine analogue of this theory. Lusztig's works \cite{Lu3, Lu4} suggest that there exist affine analogs of character sheaves, and that these objects are closely related to characters of representations of $p$-adic groups.

\begin{Emp} {\bf The affine case.} The Grothendieck--Springer fibration has a natural affine analog, which we are going to describe now. Let $\clp (G)$ and $\cL G$ be the
arc group and the loop group of $G$, respectively. Namely, $\clp (G)$ is a group scheme over $k$, whose group of $k$-points is $G(k[[t]])$, and
$\cL G$ is a group ind-scheme, whose group of $k$-points is $G(k((t)))$. We denote by
$\ev_{G}:\clp(G)\rightarrow G$ the projection, corresponding to the projection $G(k[[t]])\to G(k):t\mapsto 0$,
and let $I:=\ev_{G}^{-1}(B)\subset \clp(G)$ be the Iwahori subgroup scheme.

Let $\Fl:=\cL G/I$ be the affine flag variety, and let $\fC\subset \cL\kg$ be the locus of ``compact elements'' $\g\in\cL \kg$, that is, those $\g$, whose ``characteristic polynomial'' has integral coefficients. More precisely, we define $\fC\subset \cL\kg$ to be the preimage $\fC:=(\cL\chi)^{-1}(\clp(\fc))$, where $\fc:=\kg//\Ad G$ is the Chevalley space of $\kg$, and $\cL\chi:\cL\kg\to\cL\fc$ is the morphism, induced by the projection
$\chi:\kg\to\fc$.

Consider the ind-scheme
\begin{center}
$\wt{\fC}:=\{(gI,\g)\in\Fl \times \fC\,\vert\,\Ad_{g^{-1}}(\g)\in\Lie(I)\}$,
\end{center}
which is the affine analog of $\wt{\kg}$. Then the projection $\frak{p}:\wt{\fC}\to\fC$ is an affine analog of the Grothendieck--Springer fibration, while fibers $\Fl_{\g}$ of $\frak{p}$ are the affine Springer fibers (see \cite{KL}). Lusztig \cite{Lus2} constructed an action of the extended affine Weyl group $\widetilde{W}$ of the cohomology of the $\Fl_{\g}$'s, and a natural question is whether other aspects of classical Springer  theory can be extended to this setting as well.

Note that it is impossible to study the fibration $\frak{p}$ using classical algebro-geometric tools,
because the source and the target are infinite-dimensional ind-schemes.

\end{Emp}




\begin{Emp}
{\bf Letter of MacPherson.}
This project has begun with a letter from the second author to MacPherson in Summer 2009, in which he asks if, by considering the affine Grothendieck--Springer fibration, an appropriate counting of dimensions will tell us that this map is small. MacPherson formulated the notion of smallness which is applicable in our case, and provided the necessary computation which implies that $\frak{p}$ is small (compare \rp{codim}). 
Nevertheless, he concluded his letter by the following sentence:
\begin{center}
``We don't have a theory of intersection homology that works in this context, so the general idea that the map is small doesn't help in constructing a Weyl group action, or reproducing the rest of Springer theory''.
\end{center}
\end{Emp}

The main goal of this work is to establish such a theory.\footnote{In August 2022 (after the paper was already published online) we learned from Mark Goresky that the idea that there should be a version of Springer theory and intersection homology in this infinite dimensional setting is not new. Namely, in June 2009 he gave a talk on a conference in honor of De Concini’s 60th birthday, outlining ideas on why such a theory should exist, and formulated a claim that the affine Grothendieck--Springer resolution is a small map.  This had been the result of his discussions with 
Kottwitz and MacPherson.} In addition, we generalize Lusztig's observation \cite[$\S$3]{Lus1} and provide a supporting evidence of Lusztig's conjectures \cite{Lu3, Lu4}.

\begin{Emp}
{\bf The affine Grothendieck--Springer sheaf.} By a {\em stack (over $k$}), we mean a stack in groupoids in the \'etale topology.

To every stack $\cX$, we associate a presentable stable $\infty$-category $\C{D}(\C{X})$ of $\ell$-adic sheaves on $\C{X}$, and to every morphism $f:\cX\to\cY$ of stacks, we associate a pullback functor $f^!:\cD(\cY)\to \cD(\cX)$ (see \re{intrelladic}). In particular, for every stack $\cX$ we have a dualizing sheaf $\om_{\cX}\in\cD(\cX)$, defined to be the !-pullback of $\qlbar\in\cD(\pt)$.

Since the projection $\frak{p}$ is $\cL G$-equivariant, it induces a morphism $\ov{\frak{p}}:[\wt{\fC}/\cL G]\to [\fC/\cL G]$ of quotient stacks.
The projection $\ov{\frak{p}}$ is ind-fp-proper (see \re{intrlocind}), therefore the pullback $\ov{\frak{p}}^!$ has
a left adjoint $\ov{\frak{p}}_!$ (see \rp{locindpr}). We set
\[
\C{S}:=\ov{\frak{p}}_!(\om_{[\wt{\fC}/\cL G]})\in \cD([\fC/\cL G]),
\]
and call it the {\em affine Grothendieck--Springer sheaf}.
\end{Emp}

\begin{Emp}
{\bf The generically regular semisimple locus.} Let $\fC_{\bullet}\subset \fC$ be the open $\cL G$-invariant substack
of generically regular semisimple elements. We denote by $\ov{\frak{p}}_{\bullet}:[\wt{\fC}_{\bullet}/\cL G]\to [\fC_{\bullet}/\fC G]$ the restriction of $\ov{\frak{p}}$ to $[\fC_{\bullet}/\fC G]$, and let $\cS_{\bullet}\in \cD([\fC_{\bullet}/\cL G])$ be the $!$-pullback of $\cS$.

We let $\ov{\frak{p}}_{\tn,\bullet}:[\wt{\fC}_{\tn,\bullet}/\cL G]\to [{\fC}_{\tn,\bullet}/\cL G]$ be the restriction of $\ov{\frak{p}}_{\bullet}$ to the topologically nilpotent locus (see \re{affsprfib}),  denote by
$\cS_{\tn,\bullet}\in \cD([\fC_{\tn,\bullet}/\cL G])$ the $!$-pullback of $\cS_{\bullet}$ and call it the {\em affine Springer sheaf}.
\end{Emp}

\begin{Emp} \label{E:mainresults}
{\bf Main results.} Assume that the order of $W$ is prime to the characteristic of $k$. The main goals of this work are

$\bullet$ to develop a dimension theory in the infinite-dimensional setting;

$\bullet$ to define perverse $t$-structures on a certain class of stacks\footnote{Actually, $\infty$-stacks.}, which includes the quotient stack $[\fC_{\bullet}/\cL G]$;

$\bullet$ to show that the affine Grothendieck--Springer sheaf $\C{S}_{\bullet}$ is a perverse sheaf, which is
an intermediate extension of its restriction to a locus with regular semisimple reduction;

$\bullet$ to show that the algebra of endomorphisms $\End(\C{S}_{\bullet})$ is the group algebra $\qlbar[\wt{W}]$ of the extended affine group $\wt{W}$ of $G$;

$\bullet$ to show that the affine Springer sheaf $\cS_{\tn,\bullet}$ is perverse.

\vskip 5truept
\noindent In order to establish the properties of $\C{S}_{\bullet}$ and $\cS_{\tn,\bullet}$ stated above, we

$\bullet$ introduce a class of (semi)-small morphisms;

$\bullet$ show that the fibration $\ov{\frak{p}}_{\bullet}:[\wt{\fC}_{\bullet}/\cL G]\to [\fC_{\bullet}/\cL G]$ is small;

$\bullet$ show that  $\ov{\frak{p}}_{\tn,\bullet}$ is semi-small.
\end{Emp}

\begin{Emp}
{\bf Remarks.}
(a) Contrary to the finite-dimensional case, it is crucial for our approach that we divide by the action $\cL G$. Namely, we don't know a framework in which the non-equivariant Grothendieck--Springer fibration $\frak{p}_{\bullet}$ is small, and our perverse $t$-structure on $\cD([\fC_{\bullet}/\cL G])$ does not come from a $t$-structure on a non-equivariant category $\cD(\fC_{\bullet})$.

(b) Also it is essential that we restrict ourselves to the generically regular semisimple locus. Namely, we don't know a framework in which the full Grothendieck--Springer fibration $\ov{\frak{p}}$ is small, and we don't know whether there exists a $t$-structure on $\cD([\fC/\cL G])$ such that the Grothendieck--Springer sheaf $\cS$ lies in its heart.
\end{Emp}

In the next three subsections we describe our constructions and results in detail\footnote{To simplify the exposition, in the introduction we are going to work in a much smaller generality than in the main part of the paper, which nevertheless suffices for the Affine Springer theory.} and outline proofs of results described in \re{mainresults}.

\subsection{Placid stacks, dimension theory, and (semi)-small morphisms.} First we are going to introduce a class of infinite-dimensional geometric objects, called {\em placid stacks}, which behave in many respects like schemes of finite type over $k$.\footnote{A much more general class of placid {\em $\infty$-stacks} satisfying this property is introduced and studied in the main part of the paper.} This will allow us to define a notion of (weakly)-equidimensional morphisms in this context and introduce a class of semi-small morphisms, which includes the affine
Grothendieck--Springer fibration $\ov{\frak{p}}_{\bullet}$.




\begin{Emp} \label{E:intrsimpson}
{\bf Placid stacks.}
(a) We say that an (affine) scheme $X$ {\em admits a placid presentation}, if it has a presentation $X\simeq\lim_{\al}X_{\al}$ as a filtered limit of (affine) schemes of finite type over $k$ with smooth affine transition maps. Moreover, we say that $X$ is {\em strongly pro-smooth}, if in addition each $X_{\al}$ is smooth over $k$.


(b) For the purpose of the introduction, by a {\em placid stack}, we mean a quotient stack $\cX$ of the form $[X/G]$, where $X$ is an affine scheme  admitting a placid presentation, and $G$ is a group scheme, whose connected component $G^0$ is strongly pro-smooth in the sense of (a). Furthermore, we say that a placid stack $\cX$ is {\em smooth}, if it has a presentation $[X/G]$, where $X$ is strongly pro-smooth.
\end{Emp}
\begin{Emp} \label{E:basic}
{\bf Ind-placid ind-schemes and quotient stacks.}
(a) For the purpose of the introduction, we say that $X$ is an {\em ind-placid ind-scheme}, if $X$ can be represented as a filtered colimit $X\simeq\colim_i X_i$, where each $X_i$ is a scheme, admitting a placid presentation, and all transition maps are fp-closed embeddings, where {\em fp} stands for {\em finitely presented}.

(b) All stacks appearing in the introduction are of the form $[X/H]$, where $X$ is open subfunctor of an ind-placid ind-scheme, and $H$ is an {\em ind-placid group} (that is, a group object in the category of ind-placid ind-schemes).
\end{Emp}

\begin{Emp} \label{E:redintr}
{\bf Reduced stacks.}
To every stack $\cX$, one can associate its reduction $\cX_{\red}$ (see Section \ref{S:red}). Though a priori $\cX_{\red}$ is an $\infty$-stack, it turns out  to be a stack in all cases we are interesting in. For example, if $X$ is an (affine) scheme, then $X_{\red}$ is the reduced (affine) scheme associated to $X$ in the classical sense. In particular, $X_{\red}$ admits a placid presentation, if $X$ admits one.
 Moreover, if $\cX\simeq[X/G]$ is a placid stack, then  $\cX_{\red}\simeq[X_{\red}/G]$ is a placid stack as well. Also, if $X$ is an ind-placid ind-scheme with presentation  $X\simeq\colim_i X_{i}$, then $X_{\red}$ is also an ind-placid ind-scheme with presentation $X_{\red}\simeq\colim_i X_{i,\red}$. Furthermore, if $\cX$ is a quotient stack $[X/H]$ as in  \re{basic}(b), then $\cX_{\red}\simeq [X_{\red}/H_{\red}]$ is also of this form.
\end{Emp}

\begin{Emp} \label{E:dimension}
{\bf Dimension function and (weakly) equidimensional morphisms.}

(a) To every morphism $f:X\to Y$ of schemes of finite type over $k$ we associate a dimension function $\un{\dim}_f:X\to\B{Z}$ defined by $\un{\dim}_f(x):=\dim_x(X)-\dim_{f(x)}(Y)$.

(b) It is not difficult to see (see \rco{pullback}) that for every Cartesian diagram
\begin{equation*}
\xymatrix{X'\ar[d]_{g}\ar[r]^{\psi}& Y' \ar[d]^{f}\\X\ar[r]^{\phi}&Y}
\end{equation*}
such that either $f$ or $\phi$ is universally open, we have an equality $\un{\dim}_{\psi}=g^*(\un{\dim}_{\phi})$.

(c) We call a morphism $f$ {\em weakly equidimensional}, if $\un{\dim}_f$ is locally constant, that is, constant on each connected component. Moreover, we call $f$ {\em equidimensional}, if in addition we have an equality $\un{\dim}_f(x)=\dim_x f^{-1}(f(x))$. Notice that every open weakly equidimensional morphism is automatically equidimensional (see \rco{equidim}). Furthermore, we call a (weakly) equidimensional morphism $f$ to be of {\em relative dimension $d$}, if $\un{\dim}_f$ is a constant function with value $d$.

(d) By the property (b), all classes in (c) are stable under all pullbacks with respect to universally open morphisms. Moreover, the class
of universally open   equidimensional morphisms is stable under all pullbacks.
\end{Emp}





\begin{Emp} \label{E:classes}
{\bf Classes of morphisms.}
We say that a morphism $f:\cX\to \cY$ of stacks is {\em (locally) fp-representable} (resp. {\em fp-open/closed/locally closed embedding}), if for the every morphism $Y\to\cY$ from an affine scheme $Y$, the pullback $f\times_{\cY}Y:\cX\times_{\cY}Y\to Y$ is a  (locally) finitely presented morphism of algebraic spaces (resp. fp-open/closed/locally closed embedding of schemes).
\end{Emp}


\begin{Emp}
{\bf (Weakly) equidimensional morphisms of placid stacks.}
(a) Note that if  $f:X\to Y$ is an fp-morphism of affine schemes, where  $Y$ has a placid presentation $Y\simeq\lim_{\al}Y_{\al}$, then there exists an index $\al$ and a morphism $f_{\al}:X_{\al}\to Y_{\al}$ of affine schemes of finite type over $k$ such that $f\simeq f_{\al}\times_{Y_{\al}}Y$.

(b) We say that a morphism $f$ from (a) is {\em (weakly) equidimensional of relative dimension $d$}, if for some
placid presentation $Y\simeq\lim_{\al}Y_{\al}$, there exists $\al$ and $f_{\al}$ as in (a), such that $f_{\al}$ is (weakly) equidimensional of relative dimension $d$. Notice that it then follows from standard limit theorems and \re{dimension}(b) that the same would happen for every placid presentation of $Y$.

(c) Let $f:\cX\to \cY$ be a locally fp-representable morphism of stacks, where $\cY$ is a placid stack. Then for every presentation
$\cY\simeq[Y/H]$ we have a presentation $\cX\simeq[X/H]$, where $X:=\cX\times_{\cY}Y$ is an algebraic space locally finitely presented over $Y$.
In particular, for every \'etale covering $\{X_{\al}\to X\}_{\al}$ by affine schemes, each composition $f_{\al}:X_{\al}\to X\to Y$ is
finitely presented.

(d) We say that morphism   $f$ from (c) is {\em (weakly) equidimensional of relative dimension $d$}, if there exists a presentation
$\cY\simeq[Y/H]$ of $\cY$ and an \'etale covering of $X$ as in (c), each $f_{\al}$ is (weakly) equidimensional of relative dimension $d$. Again, in this case
the same would happen for every presentation of $\cY$ and every \'etale covering of $X$.

(e) We say that a locally closed substack $Y\subset X$ is of {\em pure codimension $d$}, if the inclusion map $Y\hra X$ is weakly equidimensional of relative dimension $-d$.
\end{Emp}

\begin{Emp} \label{E:plstrst}
{\bf Placidly stratified stacks.}
For the purpose of the introduction, by a {\em placidly stratified $\infty$-stack}, we mean an $\infty$-stack $\cY$ of the form
$[Y/H]$ as in \re{basic}(b), equipped with a {\em constructible stratification} $\{\cY_{\al}\}_{\al\in\cI}$ by placid stacks (see \re{adapted} and \re{consstr} for the definition of a constructible stratification).
In particular, we require that each $\cY_{\al}$ is a reduction (see \re{redintr}) of an fp-locally closed substack of $\cY$.
\end{Emp}

Now we are ready to define a class of (semi)-small morphisms of stacks. Following a suggestion of MacPherson, we do it using codimensions in the source rather than in the target.\footnote{See \re{remsemismall} for comparison with the classical notion.}

\begin{Emp} \label{E:inrtsemismall}
{\bf (Semi)-small morphisms.} (a) Let $f:\cX\to\cY$ be a morphism of stacks such that $\cX$ is placid and $(\cY,\{\cY_{\al}\}_{\al\in\cI})$ a placidly stratified. For every $\al\in \cI$, we set $\cX_{\al}:=(f^{-1}(\cY_{\al}))_{\red}$. Then $\cX_{\al}\subset\cX$ is an fp-locally closed substack. Assume that

$\quad\bullet$ every $\cX_{\al}\subset\cX$ is of pure codimension $b_{\al}$;

$\quad\bullet$ every $f_{\al}:\cX_{\al}\to\cY_{\al}$ is locally fp-representable, equidimensional of relative dimension $\dt_{\al}$.

(b) We say that $f$ is {\em semi-small}, if for every $\al\in \cI$ we have an inequality $\dt_{\al}\leq b_{\al}$.

(c) Moreover, let $\cU\subset\cY$ be an fp-open substack, which is a union of strata $\{\cY_{\al}\}_{\al}$. We say that a semi-small $f$ is {\em $\cU$-small}, if for every $\al\in\cI$ such that $\cX_{\al}\subset\cX\sm\cU$, we have a strict inequality $\dt_{\al}< b_{\al}$.
\end{Emp}

\subsection{$\ell$-adic sheaves on stacks and perverse $t$-structures}

\begin{Emp} \label{E:intrelladic}
{\bf $\ell$-adic sheaves on stacks.} (a) To every stack $\C{X}$, we associate a presentable stable $\infty$-category $\C{D}(\C{X})$ of $\ell$-adic sheaves on $\C{X}$, and to every morphism $h:\cX\to\cY$ of stacks, we associate a pullback functor $f^!:\cD(\cY)\to \cD(\cX)$ as follows (compare \cite{Ra, RS}):

$\bullet$ When $X$ is an affine scheme of finite type over $k$, we denote by $\C{D}_c(X)$ the bounded derived $\infty$-category
$\cD^b_c(X,\qlbar)$ of constructible $\ell$-adic  sheaves on $X$, and by $\C{D}(X)$ the ind-category $\Ind\cD_c(X)$.

$\bullet$ When $X$ is an arbitrary affine scheme over $k$, we write $X$ as a filtered limit $X\simeq\lim_{\al}X_{\al}$ of affine schemes of finite type and denote by $\C{D}(X)$ the colimit $\colim_{\al}\C{D}(X_{\al})$, taken with respect to $!$-pullbacks. It is easy to see that the resulting
$\infty$-category is independent of the presentation.

$\bullet$ For an arbitrary stack $\C{X}$, we denote by  $\cD(\C{X})$ the limit category $\lim_{X\to\cX}\C{D}(X)$, taken over all morphisms $X\to\cX$, where $X$ is an affine scheme.


\end{Emp}

\begin{Emp}
{\bf Perverse $t$-structures on placid stacks.}
For every placid stack $\cX$, we equip the $\infty$-category $\cD(\cX)$ with a perverse $t$-structure. We carry out the construction in five steps:

(a) For an equidimensional affine scheme $X$ of finite type over $k$, we equip $\cD(\cX)$ with the perverse $t$-structure $({}^p\cD_c^{\leq 0}(X),{}^p\cD_c^{\geq 0}(X))$ obtained from the classical (middle dimensional) perverse $t$-structure by homological shift by $\dim X$ to the left. In other words, an object $K\in\cD(X)$ is perverse in our $t$-structure
if and only if $K[-\dim X]$ is perverse in the classical $t$-structure.

(b) Next, an arbitrary affine scheme $X$ of finite type over $k$ has a constructible stratification $\{X_i\}_i$ by locally closed equidimensional subschemes, where $X_i$ is the set of all $x\in X$ such that $\dim_x(X)=i$. We denote by $\eta_i:X_i\hra X$ the inclusion,
and let ${}^p\cD_c^{\leq 0}(X)\subset\cD_c(X)$ (resp. ${}^p\cD_c^{\geq 0}(X)\subset\cD_c(X)$) be the full subcategory of all $K\in\cD_c(\cX)$ such that
$\eta_i^*(K)\in {}^p\cD_c^{\leq 0}(X_i)$ (resp. $\eta_i^!(K)\in {}^p\cD^{\geq 0}_c(X_i)$). Now the fact that $({}^p\cD_c^{\leq 0}(X),{}^p\cD_c^{\geq 0}(X))$ is indeed a $t$-structure follows from the gluing lemma of \cite{BBD}.

(c) For an affine scheme  $X$ of finite type over $k$, we equip $\cD(\cX)$ with the unique $t$-structure such that ${}^p\cD^{\leq 0}(X)=\Ind {}^p\cD_c^{\leq 0}(X)$ and similarly for ${}^p\cD^{\geq 0}(X)$. The main property of the $t$-structure we just constructed is that for every smooth morphism $f:X\to Y$ of affine schemes of finite type, the
pullback $f^!:\cD(Y)\to\cD(X)$ is $t$-exact.

(d) We show that for every affine scheme $X$ with a placid presentation $X\simeq\lim_{\al}X_{\al}$, there exists a unique $t$-structure  on $\cD(X)$ such that every pullback $\pi_{\al}^!:\cD(X)\to\cD(X_{\al})$ is $t$-exact. Moreover, this $t$-structure is independent of the presentation.

(e) Finally, we show that for every placid $\infty$-stack $\cX$, there exists a unique $t$-structure  on $\cD(\cX)$ such that for every
presentation $\cX\simeq[X/H]$ as in \re{intrsimpson}(b), the pullback $\pi^!:\cD(\cX)\to\cD(X)$, corresponding to the projection
$\pi:X\to \cX$ is $t$-exact.
\end{Emp}

\begin{Emp} \label{E:intrpervstr}
{\bf Perverse $t$-structures on placidly stratified stacks.} Let $(\cY,\{\cY_{\al}\}_{\al\in\cI})$ be a placidly stratified stack.

(a) For every embedding $\eta_{\al}:\cY_{\al}\hra\cY$ we have two pullback functors
$\eta_{\al}^*,\eta_{\al}^!:\cD(\cY)\to\cD(\cY_{\al})$.

(b) By a {\em perversity} of $\cY=\{\cY_{\al}\}_{\al\in\cI}$ we mean a function $p_{\nu}:\cI\to\B{Z}$, or, what is the same, a collection of integers $\{\nu_{\al}\}_{\al\in\cI}$.

(c) Using gluing lemma, for every perversity $p_{\nu}$, there exists a unique $t$-structure on $\cD(\cY)$ such that ${}^p\cD^{\geq 0}(\cY)$ is the collection of all $K\in \cD(\cY)$ such that $\eta_{\al}^!K\in {}^p\cD^{\geq -\nu_{\al}}(\cY_{\al})$ for all $\al$. Moreover, if the stratification is {\em bounded} (see \re{consstr} what bounded means), then ${}^p\cD^{\leq 0}(\cY)$ is the collection of all $K\in \cD(\cY)$ such that $\eta_{\al}^*K\in {}^p\cD^{\leq -\nu_{\al}}(\cY_{\al})$ for all $\al$.

(d) Let $\cU\subset\cY$ be an fp-open $\infty$-substack, and let $j:\cU\hra\cY$ be the inclusion map.
Then $\cU=\{\cU\cap \cY_{\al}\}_{\al\in\cI}$ is also a placidly stratified stack, and
the pullback $j^!:\cD(\cX)\to\cD(\cU)$ is $t$-exact. Then by usual procedure one can define the intermediate extension functor
$j_{!*}:\Perv(\cU)\to\Perv(\cX)$.
\end{Emp}

\begin{Emp} \label{E:intrlocind}
{\bf Ind-fp-proper morphisms.} (a) Let $Y$ be an affine scheme. We say that a morphism $f:\cX\to Y$ of stacks is {\em ind-fp-proper}, if $\cX$ has a presentation as a filtered colimit $\cX\simeq\colim_{\al}\cX_{\al}$, where each $\cX_{\al}$ is an algebraic space, fp-proper over $Y$ (see \re{classes}), and all transition maps are fp-closed embeddings.

(b) More generally, we say that a morphism $f:\cX\to\cY$ is {\em ind-fp-proper}, if for every morphism
$Y\to\cY$ from an affine scheme $Y$, the pullback $f\times_{\cY} Y:\cX\times_{\cY}Y\to Y$ of $f$ is ind-fp-proper.
\end{Emp}

\begin{Emp} \label{E:fmainthm}
{\bf First Main Theorem.}
To every semi-small morphism $f:\cX\to\cY$ (see \re{inrtsemismall}), one associates a perversity $p_{f}:=\{\nu_{\al}\}_{\al\in\cI}$ on $\cY$, defined by $\nu_{\al}:=b_{\al}+\dt_{\al}$.

We prove (see \rt{small}) that if $f:\cX\to\cY$ is an ind-fp-proper semi-small morphism of stacks and $\cX$ is smooth, then the pushforward $K:=f_!(\om_{\cX})$ is $p_f$-perverse. Moreover, if $f$ is $\cU$-small, and $j:\cU\hra\cY$ is an open embedding, then we have an isomorphism $K\simeq j_{!*}j^!(K)$.
\end{Emp}

\subsection{Affine Springer theory}

\begin{Emp} \label{E:gkm}
{\bf The GKM stratification.} (a) Note that the discriminant function $\frak{D}\in k[\fc]$ induced a morphism $\clp(\fc)\to\clp(\B{A}^1)$ of arc spaces, which we denote again by $\frak{D}$. Consider the constructible stratification $\{\clp(\fc)_d\}_{d\geq 0}$ of the regular part $\clp(\fc)_{\bullet}:=\frak{D}^{-1}(\clp(\B{A}^1)\sm\{0\})\subset \clp(\fc)$ of $\clp(\fc)$, where $\clp(\fc)_{d}(k)$ consists of points $x\in \clp(\fc)(k)=\fc(k[[t]])$ such that the valuation of the discriminant $\frak{D}(x)\in k[[t]]$ equals $d$.

(b)  Following \cite{GKM}, every stratum $\clp(\fc)_{d}$ decomposes as a disjoint union $\clp(\fc)_d=\sqcup_{(w,\br)}\fc_{w,\br}$ of connected components, parameterized by $W$-orbits of pairs $(w,\br)$, where $w$ is an element of $W$, $R$ is the set of roots of $G$, and $\br$ is a function $R\to\B{Q}_{\geq 0}$ such that $\sum_{\al\in R}\br(\al)=d$.

(c) Every stratum $\fc_{w,\br}$ a strongly pro-smooth connected affine scheme. Moreover, $\fc_{w,\br}\subset\clp(\fc)$ is a locally
fp-closed subscheme of pure codimension $b_{w,\br}$, whose explicit formula was obtained in \cite{GKM} (see \rp{codim}).

(d) For a GKM stratum $(w,\br)$, we let  ${\fC}_{w,\br}\subset{\fC}, \wt{\fC}_{w,\br}\subset\wt{\fC}$ and $\Lie(I)_{w,\br}\subset\Lie(I)$ be the preimages of $\fc_{w,\br}\subset\clp(\fc)$. Both ${\fC}_{w,\br}$ and $\wt{\fC}_{w,\br}$ are $\cL G$-invariant, and we let  $\ov{\frak{p}}_{w,\br}:[\wt{\fC}_{w,\br}/\cL G]\to[\fC_{w,\br}/\cL G]$ be the restriction of $\ov{\frak{p}}$.

(e) Let ${\fC}_{0}\subset{\fC}$ and $\wt{\fC}_{0}\subset\wt{\fC}$ be the preimages of $\clp(\fc)_0\subset\clp(\fc)$,
and let $\ov{\frak{p}}_{0}:[\wt{\fC}_{0}/\cL G]\to[\fC_{0}/\cL G]$ be the restriction of $\ov{\frak{p}}$.
Notice that $\fC_0\subset\fC$ is the locus of points with regular semisimple reduction.
\end{Emp}

\begin{Emp} \label{E:geometryasf}
{\bf Geometry of the affine Springer fibration.} We show that:

(a) The projection  $v:\Lie(I)\to\clp(\fc)$ is flat (see \rco{drinfeld}).

(b) The fibration $\frak{p}:\wt{\fC}\to\fC$ is ind-fp-proper (see \rl{affspr}), therefore the induced map
$\ov{\frak{p}}:[\wt{\fC}/\cL G]\to[\fC/\cL G]$ is ind-fp-proper (see \re{pspr}). Moreover, the stack
$[\wt{\fC}/\cL G]\simeq [\Lie(I)/I]$ is a smooth placid stack.

(c) For every GKM stratum $(w,\br)$,

\quad (i) the reduction $[\fC_{w,\br}/\cL G]_{\red}$ is a placid stack (see \rco{stratum2});

\quad (ii) the restriction
\[
\ov{\frak{p}}_{w,\br,\red}:[\wt{\fC}_{w,\br}/\cL G]_{\red}\to[\fC_{w,\br}/\cL G]_{\red}
\]
is a locally fp-representable morphism of placid stacks, which is universally open equidimensional of explicit relative
dimension $\dt_{w,\br}$ (see \rco{topproper}(b));

\quad (iii) let $W_{w,\br}\subset W$ be the stabilizer of $(w,\br)$, let
$\La:=X_*(T)$ be the group of cocharacters of the Cartan group $T$ of $G$, and let $\La_w:=X_*(T)^w$ be the group of $w$-fixed points.
Then for a certain explicit $W_{w,\br}$-torsor $\kt_{w,\br}\to\fc_{w,\br}$,

$\bullet$ there is a natural action of $\La_w$ on $\wt{\fC}_{\kt,w,\br}:=\wt{\fC}_{w,\br}\times_{\fc_{w,\br}}\kt_{w,\br}$ over
${\fC}_{\kt,w,\br}:={\fC}_{w,\br}\times_{\fc_{w,\br}}\kt_{w,\br}$, which commutes with the action of $\cL G$, and

$\bullet$ the induced morphism
\[
[\wt{\fC}_{\kt,w,\br}/\cL G\times\La_w]_{\red}\to[\fC_{\kt, w,\br}/\cL G]_{\red}
\]
is fp-proper (see \rco{topproper}(d)).

(d) The projection
\[
\ov{\frak{p}}_{0,\red}:[\wt{\fC}_{0}/\cL G]_{\red}\to[\fC_{0}/\cL G]_{\red}
\]
is a $\widetilde{W}$-torsor (see \rco{openstr}).
\end{Emp}



Now we are ready to outline proofs of results described in \re{mainresults}.

\begin{Emp} \label{E:strategy}
{\bf Outline of proofs.}

(a) Since $\fC$ is an ind-placid ind-scheme, while $\cL G$ is an ind-placid group, we conclude from \re{geometryasf}(c)(i) that $[\fC_{\bullet}/\cL G]$ is a placidly stratified $\infty$-stack (see \rl{sprsmall}(a)).

(b) Combining \re{gkm}(c), \re{geometryasf}(a) and isomorphism $[\wt{\fC}_{w,\br}/\cL G]\simeq [\Lie(I)_{w,\br}/I]$, we deduce that
each $[\wt{\fC}_{w,\br}/\cL G]\subset [\wt{\fC}/\cL G]$ is of pure codimension $b_{w,\br}$ (see \rl{sprsmall}(b)).

(c) Combining (a),(b) with \re{geometryasf}(b),(c)(ii) and the explicit formula for $b_{w,\br}$, we conclude
that the projection $\ov{\frak{p}}_{\bullet}$ is $[\fC_0/\cL G]_{\red}$-small (see \rl{sprsmall}(c)).

(d) Combining (c) with \re{geometryasf}(b) and \re{fmainthm}, we conclude that
the affine Grothendieck--Springer sheaf $\cS_{\bullet}$ is $p_{\ov{\frak{p}}_{\bullet}}$-perverse. Moreover, it is isomorphic to the intermediate extension of its restriction $\C{S}_0$ to $[\fC_0/\cL G]$ (see \rt{main}(a)).

(e) Using \re{geometryasf}(d), we see that $\C{S}_0$ is equipped with a $\wt{W}$-action. Thus, by (d), the $\wt{W}$-action on $\C{S}_0$ uniquely extends to an action on $\C{S}_{\bullet}$.
Furthermore, we have natural algebra isomorphisms $\End(\cS_{\bullet})\simeq \End(\cS_0)\simeq \qlbar[\wt{W}]$ (see \rt{main}(b)).

(f) Finally, arguing as in (c) and (d), we conclude that   $\ov{\frak{p}}_{\tn,\bullet}$ is semi-small and therefore the affine Springer sheaf $\cS_{\tn,\bullet}$ is perverse (see \rl{sprsm}(c) and \rt{perv3}).
\end{Emp}

\subsection{Possible extensions, generalizations, analogs and applications}
\begin{Emp}{\bf The derived coinvariants.}
For every representation $V$ of $\wt{W}$, we can consider the derived $V$-isotypical component $\C{S}_V\in\cD([\fC_{\bullet}/\cL G])$.

(a) We expect that every $\C{S}_V$ is perverse. Moreover, we can show this result assuming purity of the homology of affine Springer fibers and a slight strengthening of a theorem of Yun \cite{YunII} about the compatibility of the $\widetilde{W}$-action on the affine Springer fibers and the action group of connected components of the centralizer (compare also \cite[Th\`eor\'eme 5.3.1]{Bo}). On the other hand, the $\C{S}_V$'s are not intermediate extensions of their restrictions to $[\fC_0/\cL G]$ in general.

(b) If $V$ is finite-dimensional, then we can show that the corresponding $\C{S}_V$ is ``constructible'', by which we mean in particular that all of its $!$-stalks are constructible (compare \cite[Th\`eor\'eme 4.3.15]{Bo}).
\end{Emp}

\begin{Emp}{\bf Distributions.}
In this work we only construct $t$-structure on the category $\cD([\fC_{\bullet}/\cL G])$, while the affine Grothendieck--Springer sheaf $\C{S}$ naturally lives in a larger category $\cD([\fC/\cL G])$. A natural problem would be to try to construct a t-structure on the whole of $\cD([\fC/\cL G])$ and to show that $\C{S}$ is an intermediate extension of its restriction to $[\fC_{\bullet}/\cL G]$ . This would be a categorical analog
of the well-known fact that many important invariant distributions on a $p$-adic group $G(F)$ are locally $L^1$, and therefore can be reconstructed from their restrictions to $G(F)^{\on{rss}}$.
\end{Emp}

\begin{Emp}{\bf Mixed characteristic case.}
We expect that our results and techniques can be extended to the mixed characteristic case. In order to do this, one needs to use the Witt vector  version of the Affine Grassmannian, introduced by Zhu \cite{Zhu} and studied further by Bhatt--Scholze (\cite{BS}).
\end{Emp}

\begin{Emp} {\bf Application to affine $S$-cells.}
As it is shown in \cite{FKV}, applying results and techniques of this work, one can establish some of Lusztig conjectures \cite{Lus2} on $S$-cells in affine Weyl groups.
\end{Emp}

\subsection{Plan of the paper}
This work consists of four main parts and two appendices.

In the first part we introduce our main players: placid $\infty$-stacks and (weakly) equidimensional morphisms. Namely, in Section 1 we introduce placid algebraic $\infty$-stacks and study their properties. Then, in Section 2 we develop dimension theory, which is interesting for its own and is crucially used for the definition of the perverse $t$-structures. In particular, we introduce placidly stratified $\infty$-stacks and define (semi-)small morphisms in this setting.

In the second part we study the geometry of the affine Grothendieck--Springer fibration.
Namely, in Section 3, we study the Goresky--Kottwitz--MacPherson stratification: first on the arc space of the Chevalley space, following very closely the results of \cite{GKM}, and then on $\Lie(I)$. Next, in Section 4, we use the GKM stratification to study the geometry of the affine Grothendieck--Springer fibration. For example, we show that each reduced GKM stratum $[\fC_{w,\br}/\cL G]_{\red}$ is topologically placid, and study the structure of the fibration over each GKM stratum. We also show that the quotient stack $[\fC_{\bullet}/\cL G]$ is placidly stratified, the affine Grothendieck--Springer fibration is small, while the affine Springer fibration is semi-small.

In the third part we study $\infty$-categories of $\ell$-adic sheaves on $\infty$-stacks (in Section 5), and introduce perverse $t$-structures on placidly stratified  $\infty$-stacks (in Section 6). Finally, in the fourth part we apply constructions and results, obtained in the first three parts, to define the perverse $t$-structures on $[\fC_{\bullet}/\cL G]$ and $[\fC_{\tn,\bullet}/\cL G]$ and to show the perversity of $\cS_{\bullet}$ and $\cS_{\tn,\bullet}$.

We finish this work by two appendices. In Appendix A we discuss a generalization of Simpson construction of $n$-geometric stacks, which provides a categorical framework for our construction of placid $\infty$-stacks. Finally, in Appendix B we provide proofs to some of results from the second part.

\subsection{Acknowledgments}
We express our warm thanks to G. Laumon, with whom discussions are always enlightening and fruitful and to N. Rozenblyum who taught us a lot about $\infty$-categories over the years.
We also thank K. Cesnavicius, B. Hennion, S. Raskin and anonymous referee for their remarks and suggestions, and P. Scholze for stimulating conversations.

The research of Y.V. was partially supported by the ISF grant 822/17. The project has received funding from ERC under grant agreement No. 669655.

\part{Placid $\infty$-stacks and dimension theory}

\section{Placid $\infty$-stacks} In this section we are going to introduce and study our basic geometric objects, namely
$\infty$-stacks, and their important subclass of placid $\infty$-stacks.

\subsection{Algebraic spaces admitting placid presentations} Let $k$ be an algebraically closed field.\label{N:k}

\begin{Emp} \label{E:plpres}
{\bf Notation.}
(a) We say that a morphism of algebraic spaces $f:X\to Y$ over $k$ is {\em strongly pro-smooth}, \label{I:strongly pro-smooth}
if  $X$ has a presentation as a filtered limit
$X\simeq\lim_{\al} X_{\al}$ over $Y$ such that all projections $X_{\al}\to Y$  are fp-smooth, where {\em fp} \label{I:fp} means {\em finitely presented}, while all the transition maps $X_{\al}\to X_{\beta}$ are affine fp-smooth.

(b) We say that an algebraic space $X$ over $k$ {\em admits a placid presentation}, \label{I:placid presentation}
if $X$ can be written as a filtered limit
$X\simeq\lim_{\al} X_{\al}$ of algebraic spaces of finite type over $k$, where all the transition maps $X_{\al}\to X_{\beta}$ are smooth affine. Such a presentation is called {\em placid}.

(c) We say that a scheme (resp. an affine scheme) $X$ {\em admits a placid presentation}, if it admits a placid presentation as an algebraic space.

(d) We say that an algebraic space $X$ is {\em strongly pro-smooth}, if it has a placid presentation $X\simeq\lim_{\al} X_{\al}$ such that all the
$X_{\al}$'s are smooth (over $k$).
\end{Emp}

\begin{Emp} \label{E:remplpres}
{\bf Remarks.} (a) Assume that a scheme (resp. an affine scheme)  $X$ has a presentation $X\simeq\lim_{\al}X_{\al}$ as a filtered limit of algebraic spaces of finite type with affine transition maps. Then $X_{\al}$ is an scheme (resp. affine scheme) for all sufficiently large $\al$ (see, for example, \cite[Proposition 6.2 and Corollary 6.3]{Ry2}). In particular, when we talk about placid (resp. strongly pro-smooth) presentation $X\simeq\lim_{\al}X_{\al}$ of $X$, we can always assume that each $X_{\al}$ is a scheme (resp. affine scheme).

(b) By definition, $X$ admits a placid presentation if and only if it admits a strongly pro-smooth morphism $X\to X'$, where $X'$ is an algebraic space of finite type over $k$.
\end{Emp}

\begin{Emp} \label{E:placidprop}
{\bf Properties.} Notice that the construction \re{plpres} is a particular case of construction \re{general} (see \re{catplinfst}).
Therefore the class of strongly pro-smooth morphisms is closed under pullbacks and compositions (see \rl{pro}) and contains all isomorphisms. In particular, if $f:X\to Y$ is a strongly pro-smooth morphism of algebraic spaces such that $Y$ admits a placid presentation, then $X$ admits a
placid presentation as well.
\end{Emp}

The following lemma will be needed to show that various constructions are independent of a presentation.

\begin{Lem} \label{L:flat}
(a) Let $g:X\to Y$ be a flat map between algebraic spaces with placid presentations $X\simeq \lim_{\al}X_{\al}$ and $Y\simeq \lim_{\beta}Y_{\beta}$. Then for every $\beta$ and every sufficiently large $\al$ the composition $X\overset{g}{\lra}Y \overset{\pr_{\beta}}{\lra} Y_{\beta}$ factors as $X\overset{\pr_{\al}}{\lra} X_{\al}\overset{g_{\al,\beta}}{\lra}Y_{\beta}$ with $g_{\al,\beta}$  flat.

(b) Furthermore, if $g$ is strongly pro-smooth, then for every sufficiently large $\al$, the morphism $g_{\al,\beta}$ is smooth.
\end{Lem}

\begin{proof}
Since  $X\simeq \lim_{\al}X_{\al}$ and $Y_{\beta}$ is of finite type over $k$, there exists $\al$ such that $\pr_{\beta}\circ g:X\to Y_{\beta}$  factors as $X\overset{\pr_{\al}}{\lra} X_{\al}\overset{g_{\al,\beta}}{\lra}Y_{\beta}$. Thus $\pr_{\beta}\circ g=g_{\al',\beta}\circ\pr_{\al'}$ for every $\al'\geq\al$, and it suffices to show that there exists $\al'\geq \al$ such that $g_{\al',\beta}$ is flat (resp. smooth).

Let $X'_{\al}\subset X_{\al}$ be the largest open subscheme such that $g'_{\al,\beta}:=g_{\al,\beta}|_{X'_{\al}}$ is flat (resp. smooth). It suffices to show that the image of $\pr_{\al}$ is contained in $X'_{\al}$. Indeed, in this case, projection $\pr_{\al}:X\to X_{\al}$ induces a morphism $\pr'_{\al}:X\to X'_{\al}$. Since $X\simeq\lim_{\al'\geq\al} X_{\al'}$, there exists $\al'\geq\al$ such that $\pr'_{\al}:X\to X'_{\al}\subset X_{\al}$ factors as $X\overset{\pr_{\al'}}{\lra} X_{\al'}\overset{\pr'_{\al',\al}}{\lra}X'_{\al}\subset X_{\al}$, and $\pr'_{\al',\al}$ is smooth. Therefore $g_{\al',\beta}=g'_{\al,\beta}\circ \pr'_{\al',\al}$ is flat (resp. smooth), as claimed.

Fix a point $x\in X$, and set $x_{\al}:=\pr_{\al}(x)\in X_{\al}$. We want to show that $g_{\al,\beta}$ is flat (resp. smooth) at $x_{\al}$.

(a) Set $y:=g(x)$ and $y_{\beta}:=\pr_{\beta}(y)\in Y_{\beta}$. Notice that both $\pr_{\al}:X\to X_{\al}$ and $Y\to Y_{\beta}$ are strongly pro-smooth, thus flat. Therefore the composition $X\to Y\to Y_{\beta}$ is flat, thus $\C{O}_{X,x}$ is faithfully flat both as an $\C{O}_{X_{\al},x_{\al}}$-algebra and an $\C{O}_{Y_{\beta},y_{\beta}}$-algebra. Therefore $\C{O}_{X_{\al},x_{\al}}$ is a flat $\C{O}_{Y_{\beta},y_{\beta}}$-algebra, thus
$g_{\al,\beta}$ is flat at $x_{\al}$ (compare \cite[Tag 02JZ]{Sta}).

(b) First we claim that since $X\to X_{\al}$ is strongly pro-smooth, its cotangent complex $L_{X/X_{\al}}$ (\cite[Tag 08UQ]{Sta}) is a flat module concentrated in degree zero. Indeed, since cotangent complexes commute with filtered colimits (see \cite[Tag 08S9]{Sta}), we get $L_{X/X_{\al}}\simeq\colim_{\al'}\pr^*_{\al'}L_{X_{\al'}/X_{\al}}$. Since each $\pr_{\al'}$ is flat, and a filtered colimit of flat modules is flat, it suffices to show that each $L_{X_{\al'}/X_{\al}}$ is a flat module concentrated in degree zero. But this follows from the fact that
$X_{\al'}\to X_{\al}$ is smooth (see \cite[Tag 08R4]{Sta}).

Next, since $g$ and $\pr_{\beta}:Y\to Y_{\beta}$ are strongly pro-smooth, the composition $\pr_{\beta}\circ g:X\to Y_{\beta}$ is
strongly pro-smooth as well (see \re{placidprop}). Thus, by the proven above, the cotangent complex $L_{X/Y_{\beta}}$ is also a flat module concentrated in degree zero. Using a distinguished triangle
\[
\pr^*_{\al}L_{X_{\al}/Y_{\beta}}\to L_{X/Y_{\beta}}\to  L_{X/X_{\al}}
\]
(see \cite[Tag 08QR]{Sta}), we conclude that $\pr^*_{\al}L_{X_{\al}/Y_{\beta}}$
is a flat module concentrated in degree zero as well. Since $\pr_{\al}$ is flat, we thus conclude that the stalk of $L_{X_{\al}/Y_{\beta}}$ at
$x_{\al}=\pr_{\al}(x)$ is a flat module concentrated in degree zero. Therefore $g_{\al,\beta}:X_{\al}\to Y_{\beta}$ smooth at $x_{\al}$
(use, for example, \cite[Tag 08RB and 01V4]{Sta}).
\end{proof}

Applying \rl{flat} to the identity map, we get the following consequence:

\begin{Cor} \label{C:indep}
Let $X$ be an algebraic space with two placid presentations $X\simeq \lim_{\al}X_{\al}$ and $X\simeq \lim_{\al}X'_{\beta}$.
Then for every $\beta$ and every sufficiently large $\al$ the projection $\pr_{\beta}:X\to X'_{\beta}$ factors as $X\overset{\pr_{\al}}{\lra} X_{\al}\overset{g_{\al,\beta}}{\lra}X'_{\beta}$ with $g_{\al,\beta}$ smooth.
\end{Cor}

\begin{Emp} \label{E:canpres}
{\bf The canonical placid presentation.} \rco{indep} implies that every algebraic space $X$ admitting a placid presentation, in fact admits a  canonical presentation $X\simeq\lim_{X\to Y}Y$, \label{I:can placid presentation}
whether the limit runs over all strongly pro-smooth morphisms $X\to Y$, where $Y$ is of finite type over $k$, and all transition maps are smooth affine (see \re{proQrem}(b)).
\end{Emp}

\begin{Cor} \label{C:prosmooth}
Let $f:X\to Y$ be an fp-smooth covering of algebraic spaces over $k$ admitting placid presentations such that $X$ is strongly pro-smooth. Then $Y$ is strongly pro-smooth.
\end{Cor}

\begin{proof}
Choose a placid presentation $Y\simeq\lim_{\al}Y_{\al}$ of $Y$ and a strongly pro-smooth presentation $X\simeq\lim_{\beta}X_{\beta}$ of $X$. We want to show that some $Y_{\al}$ is smooth. Since $f$ is finitely presented, there exists an index $\al$ and a smooth covering $f_{\al}:Z_{\al}\to Y_{\al}$ such that $f\simeq f_{\al}\times_{Y_{\al}}Y$. Then $X\simeq\lim_{\al'>\al}(Z_{\al}\times_{Y_{\al}}Y_{\al'})$ is another placid presentation of $X$, so it follows from \rco{indep} that for every $\beta$ there exists $\al'$ such that the projection $X\to X_{\beta}$ factors through a smooth map  $Z_{\al}\times_{Y_{\al}}Y_{\al'}\to X_{\beta}$. Since $X_{\beta}$ is smooth, we deduce that $Z_{\al}\times_{Y_{\al}}Y_{\al'}$ is smooth. Since $Z_{\al}\to Y_{\al}$ and hence also $Z_{\al}\times_{Y_{\al}}Y_{\al'}\to Y_{\al'}$ is a smooth covering, we conclude that $Y_{\al'}$ is smooth as well.
\end{proof}

\subsection{Infinity-stacks}

\begin{Emp} \label{E:infst}
{\bf Notation.}
(a) Let $\Aff_k$ \label{N:affk}
be the category of affine schemes over $k$, equipped with the \'etale topology. In this case, the category of presheaves
$\PShv(\Aff_k)$ (compare \re{shv}) is usually called the category of {\em $\infty$-prestacks} \label{I:infty prestacks}
(over $k$) and denoted by $\Prest_k$ \label{N:prestk}.

(b) We call the category  $\Shv(\Aff_k)$ the category of {\em $\infty$-stacks} \label{I:infty stacks} over $k$ and denote it by $\St_k$.
\label{N:stk} Notice that every $\cX\in\St_k$ can be written as a colimit of affine schemes. More precisely, the canonical morphism $(\colim_{X\to\cX}X)\to\cX$, where the colimit is taken over all morphisms $X\to\cX$ with $X\in\Aff_k$, is an isomorphism (see, for example, \cite[Lemma 5.1.5.3]{Lu}).

(c) We call an epimorphism $\cX\to\cY$ of $\infty$-stacks a {\em covering} \label{I:covering} (compare \re{cons2}(a)).
\end{Emp}

\begin{Emp}
{\bf Remarks.} (a) While all $\infty$-stacks, appearing in our applications, are actually usual $1$-stacks, the introduction of $\infty$-stacks is necessary, because $1$-stacks are not closed under homotopy colimits. For example, $\infty$-stacks are the natural setting to define reduced stacks (see \re{red} below).

(b) In principle, one might consider stacks for the fppf topology instead of \'etale. On the other hand, to work with
\'etale topology is easier.

\end{Emp}

\begin{Emp} \label{E:rem infst}
{\bf Remark.} Let $\Sch_k$ \label{N:schk} (resp. $\Alg_k$ \label{N:algk}) be the categories of schemes (resp. algebraic spaces) over $k$, respectively, equipped with \'etale topology. Then we have natural embeddings $\iota:\Aff_k\to \Sch_k$ (resp. $\iota:\Aff_k\to \Alg_k$), and
the resulting pullback functors $\iota^*:\Shv(\Sch_k)\to \St_k$ and $\iota^*:\Shv(\Alg_k)\to \St_k$ are equivalences of $\infty$-categories. In particular, in order to construct the $\infty$-category $\St_k$ we could also use category  $\Sch_k$ or  $\Alg_k$ instead of $\Aff_k$ in \re{infst}. The same applies to categories $\Sch_k^{\qcqs}$ \label{N:qcqssch} of qcqs schemes or $\Alg_k^{\qcqs}$ \label{N:qcqsalgsp} of qcqs algebraic spaces.
\end{Emp}

\begin{Emp} \label{E:clmor}
{\bf Classes of morphisms.}
Let $(P)$ be a class of morphisms $f:\cX\to Y$ from an $\infty$-stack $\cX$ to an affine scheme $Y$, which is closed under pullbacks.

(a) We say that a morphism $f:\cX\to\cY$ of $\infty$-stacks is {\em $(P)$-representable}, \label{I:prepresentable} if for every morphism $Y\to \cY$, where $Y\in\Aff_k$, the pullback $\cX\times_{\cY}Y\to Y$ belongs to $(P)$. 

(b) We call a morphism $f:\cX\to \cY$  of $\infty$-stacks {\em representable/schematic/affine}, \label{I:representable morpism}
\label{I:schematic morpism}\label{I:affine morpism}
(resp. {\em (locally) fp-representable/(locally) fp-schematic/(locally) fp-affine}, \label{I:fp representable morpism}
\label{I:fp schematic morpism}\label{I:fp affine morpism} \label{I:loc fp representable morpism}
\label{I:loc fp schematic morpism}\label{I:loc fp affine morpism}
resp. {\em an (fp) open/closed/locally closed embedding}), \label{I:open embedding} \label{I:closed embedding} \label{I:locally closed embedding}
\label{I:fp-open embedding} \label{I:fp-closed embedding} \label{I:fp-locally closed embedding}
if it is $(P)$-representable, where $(P)$ is the class of all morphisms $\cX\to Y$, where $\cX$ is an algebraic space/scheme/affine scheme,
(resp. a subclass of (locally) fp-morphisms $\cX\to Y$ where $\cX$ is as above, resp. a class of (fp) open/closed/locally closed embeddings of schemes).
\end{Emp}

\begin{Emp} \label{E:indsch}
{\bf Spaces over $k$ and ind-algebraic spaces/ind-schemes.}
(a) We say that an $\infty$-stack $\cX$ is a {\em space over $k$}, \label{I:infty space} if for every $U\in\Aff_k$, the space $\cX(U)$ is isomorphic to a set, that is, each connected component of $\cX(U)$ is contractible.

(b) We call a space $X$ over $k$ an {\em ind-algebraic space/ind-scheme}, \label{I:ind-algebraic space} \label{I:ind-scheme}
if it has a presentation as a filtered colimit $X\simeq\colim_{\al} X_{\al}$ of qcqs algebraic spaces/schemes, where all of the transition maps are fp-closed embeddings.
\end{Emp}

\begin{Emp} \label{E:tors}
{\bf $H$-torsors.} (a) Let $H$ be a group space over $k$ acting on an $\infty$-stack $\cX$. In this case, we can form the quotient
$[\cX/H]\in\St_k$.

(b) In the situation of (a), we say that a morphism $f:\cX\to \cY$ of $\infty$-stacks is an {\em $H$-torsor}, \label{I:torsor} if $f$ is epimorphism in the \'etale topology, and the natural map $a:H\times \cX\to \cX\times_{\cY} \cX:(h,x)\mapsto (h(x),x)$ is an isomorphism.

(c) As in the classical case, if $f:\cX\to \cY$ is an $H$-torsor, then the morphism $\ov{f}:[\cX/H]\to \cY$, induced by $f$, is an isomorphism.
Indeed, the isomorphism  $a:H\times \cX\to \cX\times_{\cY} \cX$ is $(H\times H)$-equivariant, and taking quotient by
$H\times H$, we get an isomorphism  $[\cX/H]\to [\cX/H]\times_{\cY} [\cX/H]$. This implies that $\ov{f}$ is a monomorphism, that is,
$\ov{f}(U):[\cX/H](U)\to \cY(U)$ is a monomorphism of spaces for each $U\in\Aff_k$. On the other hand, since $f$ is an epimorphism,
we conclude that $\ov{f}$ is a monomorphism and an epimorphism, thus an isomorphism.

(d) As in the classical case, the quotient map $\cX\to[\cX/H]$ is an $H$-torsor, and has the property that for every $U\in\Aff_k$, the $\infty$-stack $[\cX/H](U)$ classifies pairs $(\wt{U}, \wt{\phi})$, where $\wt{U}\to U$ is an $H$-torsor, and $\wt{\phi}:\wt{U}\to \cX$ is an $H$-equivariant map.

Indeed, consider $\cY\in\Prest_k$ such that $\cY(U)$ classifies pairs $(\wt{U},\wt{\phi})$ as above. Then we have natural map $f:\cX\to\cY$, which sends $h:U\to \cX$ to a pair $(\wt{U},\wt{\phi})$, where $\wt{U}=H\times U$ is a trivial $H$-torsor, and  $\wt{\phi}(h,u)=h(u)$. One checks that $\cY$ is actually an $\infty$-stack, and $f$ is an $H$-torsor. Then, by (c), the induced map $[\cX/H]\to\cY$ is an isomorphism, and the assertion is proven.
\end{Emp}


\begin{Lem} \label{L:etalelocal}
Let $(P)$ be a class of morphisms as in \re{clmor}, which is \'etale local on the base.

(a) Supposed we are given a morphism $f:\cX\to\cY$ of $\infty$-stacks and a covering $\cY'\to\cY$ of $\infty$-stacks such that the pullback
$f\times_{\cY}\cY'$ is $(P)$-representable.  Then $f$ is $(P)$-representable.

(b) The class of $(P)$-representable morphisms is {\em stable under quotients}, that is, if $f:\cY\to \cX$ is a $(P)$-representable morphism of $\infty$-stacks, equivariant with respect to an action of a group space $H$ over $k$, then the induced map  $[f]:[\cY/H]\to [\cX/H]$ is $(P)$-representable.
\end{Lem}

\begin{proof}
(a) We have to show that for every morphism  $Y\to\cY$ with affine $Y$, the pullback $f\times_{\cY}Y$ is in $(P)$. By definition, there exists an affine \'etale covering $Y'\to Y$ such that the composition $Y'\to Y\to\cY$ lifts to a morphism $Y'\to\cY'$. Then $f\times_{\cY}Y'$ is in $(P)$, because $f\times_{\cY}\cY'$ is $(P)$-representable, thus $f\times_{\cY}Y$ is in $(P)$, because $(P)$ is \'etale local on the base.

(b) As in the classical case, we have an isomorphism $\ov{f}\times_{[\cY/H]}\cY\simeq f$ (use \re{tors}). Since $\cY\to [\cY/H]$ is a covering,
the assertion follows from (a).
\end{proof}

\begin{Def} \label{D:pspr}
(a) We say that a morphism $f:\cX\rightarrow \cY$ from ind-algebraic space $\cX$ to an affine scheme $Y$ is {\em ind-fp-proper}\label{I:ind-fp-proper},
if $\cX$ has a presentation a filtered colimit $\cX\simeq\colim_{\al}X_{\al}$ such that
each $X_{\al}$ is fp-proper over $Y$.


(b) Notice that class of morphisms in (a) is stable under all pullbacks, therefore the construction of \re{clmor}(a) applies.
In particular, we can talk about ind-fp-proper morphisms of $\infty$-stacks.
\end{Def}

\begin{Emp} \label{E:pspr}
{\bf Remarks.} (a) Arguing as in \cite[Lemma 3.12]{HR}, we see that the class of morphisms in \rd{pspr}(a) is \'etale local on the base.
Namely, assume that $f:\cX\rightarrow Y$ is a morphism from an $\infty$-stack $\cX$ to an affine scheme $Y$ such that there exists an affine \'etale covering $Y'\to Y$ such that the pullback $f':\cX':=\cX\times_{Y}Y'\to Y'$ satisfies property of \rd{pspr}(a). We want to show that $f$ satisfies property \rd{pspr}(a) as well.

For completeness, we sketch the argument. First we observe that $\cX'$ is a space over $k$, therefore $\cX$ is a space over $k$ as well. Next, choose a presentation $\cX'\simeq\colim_{\al}X'_{\al}$ of $\cX'$, where each $X'_{\al}$ is an algebraic space, fp-proper over $Y'$. We claim that for every $\al$ there exists a closed subspace $X_{\al}\subset \cX$ such that $X'_{\al}\subset\cX'$ is the preimage of $X_{\al}\subset \cX$. Namely, arguing as in  \cite[Lemma 3.12]{HR} we take $X_{\al}$ be the schematic image of the composition $X'_{\al}\hra \cX'\to\cX$. Finally, by  standard descent arguments we see that $\cX\simeq\colim_{\al}X_{\al}$, and each $X_{\al}$ is an algebraic space, fp-proper over $Y$.

(b) Let $f:X\to Y$ be an ind-fp-proper morphism between spaces over $k$, which is equivariant with respect to an action of
a group space $H$ over $k$. Then the induced morphism $\ov{f}:[X/H]\to[Y/H]$ is ind-fp-proper. Indeed, by (a),
the assertion follows from \rl{etalelocal}(b).
\end{Emp}

\subsection{Placid $\infty$-stacks}

\begin{Emp} \label{E:plinfst}
{\bf Construction.} (a) We say that

$\bullet$  an $\infty$-stack $\cX$ is {\em $0$-placid}\label{I:0placid}, if it has a decomposition $\cX\simeq \sqcup_{\al}X_{\al}$, where each $X_{\al}$ is an affine scheme admitting a placid presentation;

$\bullet$ a morphism $f:\cX\to \cY$ of $\infty$-stacks is {\em $0$-smooth}\label{I:0smooth}, if for every morphism $Y\to \cY$, where $Y$ is an affine scheme admitting a placid presentation,  there exists a decomposition $\cX\times_{\cY} Y\simeq\sqcup_{\al}X_{\al}$ such that each $X_{\al}$ is an affine scheme, and the composition $X_{\al}\hra \cX\times_{\cY} Y\to Y$ is strongly pro-smooth. In particular, each $X_{\al}$ admits a placid presentation (see \re{placidprop}).

(b) Next, let $n\in\B{N}$ and assume that classes of $n$-placid $\infty$-stacks \label{I:nplacid} and $n$-smooth \label{I:nsmooth} morphisms of $\infty$-stacks are constructed. Then we say that

$\bullet$ an $\infty$-stack $\cX$ is {\em $(n+1)$-placid}, if there exists an $n$-smooth covering  $g:X\to \cX$ such that $X$ is $0$-placid;

$\bullet$ a morphism  $f:\cX\to \cY$ of $\infty$-stacks is {\em $(n+1)$-smooth}, if for every morphism $Y\to \cY$ of $\infty$-stacks with $0$-placid  $Y$, the pullback
$\cX\times_{\cY} Y$ is $(n+1)$-placid, and there exists an $n$-smooth covering $X\to \cX\times_{\cY} Y$ such that $X$ is $0$-placid and the composition $X\to \cX\times_{\cY} Y\to Y$ is $0$-smooth.

(c) Finally, we say that an $\infty$-stack is {\em placid}\label{I:placid infty stack}, if it is $n$-placid for some $n$, and a morphism of $\infty$-stacks is {\em smooth}\label{I:smooth morphism}, if it is $n$-smooth for some $n$.
\end{Emp}

\begin{Emp}
{\bf Convention.} Notice that in our terminology \re{plinfst}(c) smooth morphisms are not assumed to be locally fp. On the other hand, we always assume that \'etale morphisms are  locally fp.
\end{Emp}

\begin{Emp} \label{E:propplst}
{\bf Properties.}
Construction \re{plinfst} is a particular case of construction \re{cons2} (which is applicable because of \re{catplinfst}(b)) in
the case when $\cA=\Aff_k$, $\Ob_0(\cA)\subset \cA$ is the class of affine schemes admitting placid presentations, and $\Mor_0(\cA)\subset\Mor(\cC)$ is the class of strongly pro-smooth morphisms (see \re{plpres}). In particular,

(a) To a placid $\infty$-stack $\cX$ one associates the $\infty$-category $J$, whose objects are smooth morphisms $X\to \cX$, where $X$ is an affine scheme admitting a placid presentation, and morphisms are strongly pro-smooth morphisms $X\to X'$ over $\cX$. Then the canonical morphism $(\colim_{X\to\cX\in J}X)\to\cX$ is an isomorphism (see \rt{colim}).

(b) A morphism $\cX\to\cY$ between $\infty$-stacks is ($n$-)smooth, if there exists a covering $\cZ\to\cY$ such that the pullback
$\cX\times_{\cY}\cZ\to \cZ$ is ($n$-)smooth (see \rl{covering}).

(c) Smooth morphisms are closed under compositions and pullbacks (see \re{rempb} and \rl{simpson}(b))

(d) If $f:\cX\to\cY$ is a smooth morphism of $\infty$-stacks, and $\cY$ is placid, then $\cX$ is placid (see \rl{simpson}(a)).
If $f:\cX\to\cY$ is a smooth covering of $\infty$-stacks, and $\cX$ is placid, then $\cY$ is placid (see \rl{simpson}(c)).

(e) If $f:X\to Y$ is a smooth morphism between affine schemes admitting placid presentations, then there exists a strongly pro-smooth covering
$Z\to X$ such that the composition $Z\to X\to Y$ is strongly pro-smooth (see \rco{mor0}(a)). In particular, every smooth morphism $f$ is flat.

(f) If $X$ is an affine scheme admitting a placid presentation, and $X\simeq\sqcup_{\al}{X_{\al}}$ is a decomposition of $X$ in $\St_k$, then each
$X_{\al}$ is an affine scheme admitting a placid presentation (see \re{catplinfst}(b)).
\end{Emp}

\begin{Emp} \label{E:explinfst}
{\bf Examples.} (a) Note that every Artin stack of finite type  $\cX$ over $k$ is a placid $\infty$-stack. Moreover, every smooth morphism $f:\cX\to\cY$ between Artin stacks of finite type over $k$ in the classical sense is also smooth in the sense of \re{plinfst}(c).

Indeed, denote by $\Affft_k$\label{N:affft} the category of affine schemes of finite type over $k$. Since $\cX$ has a (classically) smooth covering $X\to \cX$ from an $X\in\Affft_k$, and $X$ is $0$-placid, the first assertion follows from the second one. Note that a smooth morphism in $\Affft_k$ is $0$-smooth by definition, hence every affine smooth morphism $f:\cX\to\cY$ between Artin stacks of finite type over $k$ is $0$-smooth. Next, by a standard argument we see that every quasi-affine smooth morphism (between Artin stacks of finite type over $k$) is $1$-smooth, then every schematic smooth morphism is $2$-smooth, hence every representable smooth morphism is $3$-smooth, and finally every smooth morphism is $4$-smooth.

(b) More generally, any locally fp-morphism $f:\cX\to\cY$ of algebraic stacks, which is smooth in the classical sense is also smooth in the sense of \re{plinfst}(c). Indeed, as in (a), one reduces to the case of
a smooth fp-morphism $f:X\to Y$ of affine schemes. In this case, $f$ is a pullback of a smooth morphism in
$\Affft_k$, and the assertion is clear.
\end{Emp}

\begin{Def}  \label{D:placid}
(a) We call an affine scheme/scheme/algebraic space {\em ($n$-)placid}, if it is ($n$-)placid as an $\infty$-stack.\label{I:placid scheme}
\label{I:placid algebraic space}

(b) We say that an ind-algebraic space/ind-scheme is {\em ind-placid}, \label{I:ind-placid}
 if it admits a presentation $X\simeq\colim_{\al}X_{\al}$ as in \re{indsch}(b) such that every algebraic space/scheme $X_{\al}$ is placid.
\end{Def}

\begin{Lem} \label{L:glplacid}
(a) An algebraic space admitting a placid presentation is placid.

(b) Let $f:X\to Y$ be an fp-morphism between algebraic spaces such that $Y$ admits a placid presentation. Then $X$ admits a placid presentation.

(c) Let $f:\cX\to\cY$ be a locally fp-representable morphism of $\infty$-stacks such that $\cY$ is placid. Then $\cX$ is placid.
\end{Lem}

\begin{proof}
(a) Let $X$ be an algebraic space with a placid presentation $X\simeq\lim_{\al}X_{\al}$, and let
$X'_{\al_0}\to X_{\al_0}$ be an \'etale covering with an affine scheme $X'_{\al_0}$. Then $X':=X\times_{X_{\al_0}}X'_{\al_0}$ is an affine scheme with a placid presentation
$X'\simeq \lim_{\al>\al_0}(X_{\al}\times_{X_{\al_0}}X'_{\al_0})$. Thus $X'$ is a placid affine scheme, and $X'\to X$ is an fp-\'etale covering.
Thus $X$ is placid (use \re{explinfst}(b)), as claimed.

(b) Choose a placid presentation $Y\simeq\lim_{\al}Y_{\al}$. Since $X\to Y$ is finitely presented, it is a pullback of a morphism
$X_{\al_0}\to Y_{\al_0}$ of algebraic spaces of finite type over $k$. Then $X\simeq \lim_{\al>\al_0}(Y_{\al}\times_{Y_{\al_0}}X_{\al_0})$ is a placid
presentation of $X$.

(c) Choose a smooth covering $Y\to \cY$ from a $0$-placid $Y$. Then the pullback $\cX\times_{\cY}Y\to\cX$ is a smooth covering, hence
it suffices to show that $\cX\times_{\cY}Y$ is placid (see \re{propplst}(d)). Thus we can assume that $\cY$ is $0$-placid.
Then we have a decomposition $\cY\simeq\sqcup_{\al}Y_{\al}$, where each $Y_{\al}$ is an affine scheme admitting a placid presentation, which induces a decomposition $\cX\simeq\sqcup_{\al}(\cX\times_{\cY}Y_{\al})$. Therefore we can assume that $\cY$ is an affine scheme admitting a placid presentation. In this case, $\cX$ is an algebraic space, locally fp over $\cY$, so the assertion follows from a combination of (b) and (a).
\end{proof}

\begin{Emp} \label{E:placidrem}
{\bf Remarks.}
(a) Notice that if $X'$ and $X''$ are affine schemes with placid presentations $X'\simeq\lim_{\al}X'_{\al}$ and $X'\simeq\lim_{\beta}X'_{\beta}$, then their disjoint union $X:=X'\sqcup X''$ is an affine scheme  with a placid presentation $X\simeq \lim_{\al,\beta}(X'_{\al}\sqcup X''_{\beta})$.

(b) By definition, every affine scheme admitting a placid presentation is $0$-placid. Conversely, every $0$-placid affine scheme admits a placid presentation. Indeed, by definition, every $0$-placid affine $X$ is a disjoint union $\sqcup_{\al}X_{\al}$ of affine schemes admitting placid presentations. Moreover, this disjoint union is finite, because $X$ is quasi-compact. Hence $X$ admits a placid presentation by (a).
On the other hand, we do not expect that every placid affine scheme admits a placid presentation.

(c) By \re{explinfst}(b), if a scheme/algebraic space $X$ has a Zariski/\'etale covering by $0$-placid affine schemes, then $X$ is placid.
Again, we do not expect that the converse is true.

(d) In the construction \re{plinfst}, one can replace affine schemes by either schemes or algebraic spaces. Then, using a variant of
\rl{glplacid} one can show that though the classes of $n$-placid $\infty$-stacks and $n$-smooth morphisms would be slightly different,
the resulting classes of placid $\infty$-stacks and smooth morphisms will not change.

(e) It follows from \rl{glplacid}(c) that for every ind-placid ind-algebraic space/ind-scheme $X$ and {\em every} presentation $X\simeq\colim_{\al}X_{\al}$ as in  \re{indsch}(b), each  $X_{\al}$ is placid.
\end{Emp}

\begin{Emp}
{\bf Convention}. Using \re{placidrem}(b), from now on we will often be using a shorter {\em $0$-placid affine scheme} term instead of a longer {\em affine scheme admitting a placid presentation}.
\end{Emp}

\begin{Emp} \label{E:quot}
{\bf Example.} Let $H$ be a group-scheme acting on a $0$-placid affine scheme $X$. Assume that $H$ is $0$-smooth, that is,
the projection $H\to\pt$ is $0$-smooth. Then the quotient stack $\cX:=[X/H]$ is a $1$-placid $\infty$-stack, and the projection $\pi:X\to\cX$ is $0$-smooth.

Indeed, since $\pi$ is a covering, it remains to show that it is $0$-smooth. By \re{propplst}(b), it suffices to show that the projection $X\times_{\cX}X\to X$ is $0$-smooth. Since $X\times_{\cX}X\simeq H\times X$, and $H\to\pt$ is $0$-smooth, the assertion follows.
\end{Emp}

\begin{Emp} \label{E:decomp}
{\bf Locally closed embeddings}. Recall that for every locally closed embedding $\eta:Y\hra X$ of quasi-compact schemes there exists a unique decomposition
$Y\overset{j}{\hra}Z\overset{i}{\hra}X$, where $j$ is a schematically dense open embedding, and $i$ is a closed embedding.
\end{Emp}


\begin{Lem} \label{L:fplocl}
(a) For every fp-locally closed embedding $\eta:\cY\hra\cX$ of placid $\infty$-stacks, there exists a unique decomposition
$\cY\overset{j}{\hra}\cZ\overset{i}{\hra}\cX$, where $j$ (resp. $i$) is an fp-open (resp. an fp-closed) embedding, and
for every smooth morphism $X\to \cX$, where $X$ is a $0$-placid scheme, the pullback $\cY\times_{\cX}X\hra\cZ\times_{\cX}X\hra X$ is the
canonical decomposition from \re{decomp}.

(b) Moreover, the decomposition of (a) is compatible with smooth pullbacks, that is, for every smooth morphism $\cX'\to\cX$, the pullback
$\cY\times_{\cX}\cX'\hra\cZ\times_{\cX}\cX'\hra\cX'$ is the decomposition of the pullback $\cY\times_{\cX}\cX'\hra\cX'$.
\end{Lem}

\begin{proof}
Notice first that if $\cX'\to\cX$ and $X\to \cX'$ are smooth, then the composition $X\to\cX$ is smooth.
Thus (b) follows immediately from (a).

First we claim that the decomposition in (a) is unique, if exists. Indeed, by \re{propplst}(a) we have a canonical isomorphism $(\colim_{X\to\cX}X)\to\cX$, where the colimit runs over smooth morphisms $X\to\cX$ with $0$-placid affine schemes $X$, and transition maps are strongly pro-smooth.

Therefore for every morphism $\cZ\to\cX$ the canonical morphism $\colim_{X\to\cX}(\cZ\times_{\cX}X)\to\cZ$ is an isomorphism,
from which the uniqueness follows.

Next, for every  $X\to\cX$ as above, consider the canonical decomposition $\cY\times_{\cX}X\hra Z_X\hra X$ from \re{decomp}. Since this
decomposition is functorial in $X\to\cX$, we can form a colimit
$\cY\overset{j}{\to}\cZ\overset{i}{\to}\cX$ with $\cZ:=\colim_{X\to\cX}Z_{X}$.

Since the classes of fp-open and fp-closed embeddings are \'etale local on the base (see \cite[Tags 041V, 041X, 0420]{Sta}), it suffices to show that for every smooth morphism $X'\to\cX$ as above, the pullback $\cY\times_{\cX}X'\to  \cZ\times_{\cX}X'\to X'$ is isomorphic to $\cY\times_{\cX}X'\to Z_{X'}\to X'$ (use \rl{etalelocal}(a)).

We claim that for every pair of smooth morphisms $X_1\to\cX$ and $X_2\to\cX$ as above, we have a canonical isomorphism
$Z_{X_1}\times_{\cX}X_2\simeq X_1\times_{\cX}Z_{X_2}$ over $\wt{\cX}:=X_1\times_{\cX}X_2$. Indeed, by the uniqueness of the decomposition mentioned above, is suffices to show
that for every smooth morphism $f:X\to\wt{\cX}$ with $0$-placid affine $X$, the pullback of $(Z_{X_1}\times_{\cX}X_2)\times_{X_2}X$
is isomorphic to  $Z_X$ (and similarly for $X_1\times_{\cX}Z_{X_2}$). To see this, notice that this pullback is the pullback
$Z_{X_1}\times_{X_1}X$ with respect to the smooth morphism $X\to \wt{\cX}\to X_1$. Since the decomposition of \re{decomp} is compatible with flat pullbacks, while smooth morphisms are flat (see \re{propplst}(e)), the assertion follows.

Now the assertion follows from the canonical isomorphism
\[
\cZ\times_{\cX}X'\simeq(\colim_{X\to\cX}Z_X)\times_{\cX}X'\simeq\colim_{X\to\cX}(Z_X\times_{\cX}X')\simeq \]
\[
\simeq\colim_{X\to\cX}(X\times_{\cX}Z_{X'})\simeq (\colim_{X\to\cX}X)\times_{\cX}Z_{X'}\simeq Z_{X'},
\]
where the middle isomorphism $Z_X\times_{\cX}X'\simeq X\times_{\cX}Z_{X'}$ was constructed above.
\end{proof}

\subsection{Reduced $\infty$-stacks} \label{S:red}

\begin{Emp} \label{E:red}
{\bf The reduced $\infty$-substack.}
(a) Let $\Aff_{\red,k}\subset\Aff_k$ be the category of reduced affine schemes over $k$. Then the inclusion $\iota:\Aff_{\red,k}\hra\Aff_k$ has a right adjoint $X\mapsto X_{\red}$.

(b) Recall that if $f:X\to Y$ is an \'etale morphism of affine schemes, and $Y$ is reduced, then $X$ is reduced as well (see \cite[Tag 03PC(8)]{Sta}).
Therefore the \'etale topology on $\Aff_k$ restricts to the \'etale topology on $\Aff_{\red,k}$, thus the assumption \re{subcategory} is satisfied.
In particular, we can consider the $\infty$-category $\St_{\red,k}:=\Shv(\Aff_{\red,k})$, have the restriction map $\iota^*:\St_k\to \St_{\red,k}$ with a fully faithful left adjoint $\iota_!:\St_{\red,k}\to\St_k$ (see \rl{ff}).

(c) By (a), for every $X\in\Aff_k\subset\St_k$, the pullback $\iota^*X\in\St_{\red,k}$ is
$X_{\red}\in\Aff_{\red,k}\subset \St_{\red,k}$,  thus $\iota_!\iota^*X\in\St_k$ is $X_{\red}\in \Aff_{\red,k}\subset\Aff_k\subset\St_k$.

(d) For every $\cX\in\St_k$, we set  $\cX_{\red}:=\iota_!\iota^*\cX$\label{N:xred} and call it the {\em reduced} $\infty$-stack
of $\cX$ (see remark \re{remred} below). By adjointness, we have a natural counit map $\pi=\pi_{\cX}:\cX_{\red}\to\cX$.

(e) Since $\iota_!$ is fully faithful (see \rl{ff}), while $\iota^*$ commutes with limits, we have a canonical isomorphism  $\iota^*\cZ_{\red}\simeq\iota^*\cZ$, hence isomorphisms
$(\cX\times_{\cY}\cZ_{\red})_{\red}\simeq (\cX\times_{\cY}\cZ)_{\red}$ and $(\cX_{\red}\times_{\cY_{\red}}\cZ_{\red})_{\red}\simeq (\cX\times_{\cY}\cZ)_{\red}$.

(f) We call an $\infty$-stack $\cX\in\St_k$ {\em reduced}\label{I:reduced infty stack}, if the counit map $\cX_{\red}\to\cX$ is an isomorphism, and let $(\St_k)_{\red}\subset\St_k$ be the full subcategory of reduced $\infty$-stacks.
\end{Emp}

\begin{Emp} \label{E:remred}
{\bf Remark.} By \re{red}(c), for every affine scheme $X$, the reduced $\infty$-stack $X_{\red}$ in the sense of \re{red}(d)
is the classical reduced scheme, corresponding to $X$. The same is also true for schemes, algebraic spaces and Artin stacks
(see \rl{redclass}(b) below).
\end{Emp}

\begin{Lem} \label{L:redclass}
(a) For every fp-smooth representable morphism $f:\cX\to\cY$ of $\infty$-stacks, the induced
morphism $\cX_{\red}\to \cX\times_{\cY}\cY_{\red}$ is an isomorphism.

(b) If $\cX$ is a scheme (resp. algebraic space/Artin stack), then $\cX_{\red}$ is the classical reduced scheme (resp. algebraic space/Artin stack) corresponding to $\cX$.
\end{Lem}

\begin{proof}
(a) Since $\cY$ is a colimit of affine schemes, while both the pullback and the reduction functor $(\cdot)_{\red}=\iota_!\iota^*$  commute with colimits, we immediately reduce to the case when $\cY$ is an affine scheme, and $\cX$ is an algebraic space.

Next, choose an \'etale (schematic) covering $f:X\to \cX$ by a scheme $X$. Then $\cX\simeq \colim_{[m]}X^{[m]}$ is a colimit of its \v{C}ech nerve (see \re{chech}), and each $X^{[n]}$ is a scheme. Thus, we reduce to the case when $\cX$ is a scheme. Applying the same argument twice, we reduce to the case when $\cX$ is an affine scheme. In this case, $\cX_{\red}$ and $\cY_{\red}$ are the classical reduced affine schemes (see remark \re{remred}), so the assertion is standard.

(b) Assume that $\cX$ is an Artin stack, and let $f:X\to \cX$ be an fp-smooth representable covering by a scheme. Then the induced morphism
$X_{\red}\to \cX_{\red}$ is an fp-smooth representable covering (by (a)), hence the assertion for $\cX_{\red}$ follows from that for $X_{\red}$, thus we can assume that $\cX$ is a scheme. Next, choose an open covering of $\cX$ by affine schemes $X_{\al}$. Then affine schemes $X_{\al,\red}$
form an open covering of $\cX_{\red}$ (by (a)), and the assertion follows.
\end{proof}

\begin{Lem} \label{L:reduction}
Let $f:\cX\to\cY$ be a smooth morphism of placid $\infty$-stacks (see \re{plinfst}(c)). Then for every morphism of $\infty$-stacks
$g:\cY'\to\cY$, the induced morphism $(\cX\times_{\cY}\cY')_{\red}\to \cX\times_{\cY}\cY'_{\red}$ is an isomorphism. In particular,
the morphism $\cX_{\red}\to \cX\times_{\cY}\cY_{\red}$ is an isomorphism.
\end{Lem}
\begin{proof}
We are going to use the assertion that reduction functors $\cX\mapsto \cX_{\red}$ and pullbacks commute with colimits.

First we show the assertion when $\cX$ and $\cY$ are affine schemes admitting placid presentations, and $f$ is strongly pro-smooth.
Since $\cY'$ is a colimit of affine schemes, we can assume that $\cY'$ is an affine scheme, thus
$\cX':=\cX\times_{\cY}\cY'\to \cY'$ is strongly pro-smooth. In this case we want to show that the induced map
$\cX'_{\red}\to \cX'\times_{\cY'}\cY'_{\red}$ is an isomorphism. Taking pullback to $\cY'_{\red}$, we can assume
that $\cY'$ is reduced (use \re{red}(e)). Thus, we have to show that if $\cX'\to \cY'$ is strongly pro-smooth and $\cY'$ is reduced, then
$\cX'$ is reduced. But this is standard.

Assume now that $\cX$ and $\cY$ are affine schemes admitting placid presentations, but $f$ is only smooth. By \re{propplst}(f) there is a strongly pro-smooth covering $p:X\to\cX$ such that the composition $X\to\cX\to\cY$ is strongly pro-smooth. By the observation \re{chech}, $p$ gives rise to a presentation of $\cX$ as a colimit of the \v{C}ech complex $\cX\simeq\colim_{[m]}X^{[m]}$, where all transition maps are strongly pro-smooth.
Then each composition $f^{[m]}:X^{[m]}\to\cX\to\cY$ is strongly pro-smooth as well. Since both functors $\cdot_{\red}$ and pullbacks commute with colimits, the assertion for $f$ follows from that for $f^{[m]}$, shown before.

Assume now that only $Y$ is an affine scheme admitting a placid presentation. Next, recall (see \re{propplst}(a)) that we have a canonical isomorphism $(\colim_{X\to\cX}X)\to\cX$, where the colimit runs over smooth morphisms $X\to\cX$, where $X$ is an affine scheme admitting a placid presentation. Since all functors commute with colimits, we can replace $f$ by $X\to\cX\to\cY$, thus reducing to the case, shown above.

In the general case, recall that $\cY$ has a canonical presentation (see \re{propplst}(a)) $\cY\simeq \colim_{Y\to\cY}Y$, where each
$Y$ is an affine scheme admitting a placid presentation.
This presentation induces a presentation $\cX\simeq \colim_{Y\to\cY}(\cX\times_{\cY} Y)$ and similarly for $\cY'$. Hence we can replace
$f$ and $g$ by their pullbacks with respect to $Y\to\cY$, thus reducing to the case when $\cY$ is an affine scheme admitting a placid presentation,
shown above.
\end{proof}


\begin{Cor} \label{C:red}
(a) If $f:\cX\to\cY$ is a smooth morphism (resp. covering) of placid $\infty$-stacks, then the induced
morphism $f_{\red}:\cX_{\red}\to\cY_{\red}$ is a smooth morphism (resp. covering) as well.

(b) If $\cX$ is a placid $\infty$-stack, then $\cX_{\red}$ is a placid $\infty$-stack as well, and the morphism $\cX_{\red}\to\cX$ is an fp-closed embedding.

(c) If $X$ is an ind-placid ind-algebraic space/ind-scheme, then so is $X_{\red}$.
\end{Cor}

\begin{proof}
(a) Since smooth morphisms/covering are closed under pullbacks, the assertion follows from \rl{reduction}.

(b) By \rl{glplacid}(c), it suffices to show that  $\cX_{\red}\to\cX$ is an fp-closed embedding.
If $\cX\in\Affft_k$, the assertion is clear. Next, assume that $\cX$ is a $0$-placid affine scheme. Then $\cX$ admits a strongly pro-smooth morphism $\cX\to X$ with $X\in\Affft_k$. Hence  $\cX_{\red}\simeq\cX\times_{X}X_{\red}$ by \rl{reduction}, thus the assertion for
$\cX_{\red}\to\cX$ follows from that for $X_{\red}\to X$.

In the general case, choose a smooth covering $X\to\cX$ with $0$-placid $X$.
Since $X$ is a coproduct of $0$-placid affine schemes, we conclude that the map $\cX_{\red}\times_{\cX}X\simeq X_{\red}\to X$ is an fp-closed embedding (use \rl{reduction}). Since the class of fp-closed embeddings is \'etale local on the base (see \cite[Tags 041V, 0420]{Sta}), the assertion now follows from \rl{etalelocal}(a).

(c) By definition, $X$ has a presentation of the form $X\simeq\colim_{\al}X_{\al}$ such that every $X_{\al}$ is a placid algebraic space/scheme,  and all the transition maps are fp-closed embeddings. Then $X_{\red}$ has a presentation $X_{\red}\simeq\colim_{\al}X_{\al,\red}$, every $X_{\al,\red}$ is a placid algebraic space/scheme (by (b) and \rl{glplacid}(c)), and all the transition maps are fp-closed embeddings by (b).
%
\end{proof}


\subsection{Topological equivalences, and related notions}

\begin{Emp} \label{E:topeq}
{\bf Topological equivalences.} (a) We say that a morphism $\f:\cX\to\cY$ of $\infty$-stacks is a {\em topological equivalence}\label{I:topological equivalence},  if the morphism
$f_{\red}:\cX_{\red}\to\cY_{\red}$ is an isomorphism.

(b) Since $\cX\mapsto\cX_{\red}$ commutes with colimits, the class of (a) is closed under colimits. In particular, it is \'etale local on the base
and closed under quotients.

(c) Using   \re{red}(e), the class of (a) is closed under pullbacks.
\end{Emp}

\begin{Emp} \label{E:redP}
{\bf Classes of morphisms.} Let $(P)$ be a class of morphisms $f:\cX\to Y$ from an $\infty$-stack $\cX$ to an affine scheme $Y$, closed under pullbacks (see \re{clmor}).

(a) We denote by $(P_{\red})$ \label{N:pred} the smallest class of morphisms $f:\cX\to Y$ from an $\infty$-stack $\cX$ to an affine scheme $Y$, which is \'etale local on the base, contains $(P)$ and such that a morphism $f:\cX\to Y$ is in $(P_{\red})$ if and only if the induced map $\cX_{\red}\to Y$ is in $(P_{\red})$. Explicitly, it is the class of  morphisms $f:\cX\to Y$ such that there exists a morphism $\wt{f}:\wt{\cX}'\to Y$ from $(P)$ and an isomorphism $\wt{\cX}'_{\red}\simeq\cX_{\red}$ over $Y$.


(b) Using \re{red}(e), we conclude that the class $(P_{\red})$ of (a) is closed under pullbacks, so construction \re{clmor} applies. Thus we can talk about $(P_{\red})$-representable \label{I:pred representable} morphisms.

\end{Emp}

\begin{Emp} \label{E:exnilp}
{\bf Examples.} (a) In the situation of \re{redP}, assume that the morphism $\cX_{\red}\to \cY$ is $(P)$-representable (see \re{clmor}).
Then $\cX\to\cY$ is $(P_{\red})$-representable. Indeed, we have to show that every pullback $\cX\times_{\cY}Y\to Y$, where $Y$ is an affine scheme, is in $(P_{\red})$. But this follows from isomorphism  $(\cX_{\red}\times_{\cY}Y)_{\red}\simeq (\cX\times_{\cY}Y)_{\red}$
(see \re{red}(e)).

(b) Using properties \re{topeq} one can show that a morphism $f:\cX\to \cY$ of $\infty$-stacks is a topological equivalence if and only if it is $(P_{\red})$-representable, where
$(P)$ is the class of isomorphisms.
\end{Emp}

The following result asserts that in some cases the converse of \re{exnilp} also holds.

\begin{Lem} \label{L:red-fp}
Let $(P)$ be a class of morphisms as in \re{redP} such that

(i) every $f:X\to Y$ from $(P)$ is a locally fp-morphism of algebraic spaces;

(ii) for every morphism $f:X\to Y$ in $(P)$ and an fp-closed embedding $X'\hra X$ such that $X'_{\red}\isom X_{\red}$,
the composition $X'\to X\to Y$ is in $(P)$;

(iii) $(P)$ is \'etale local on the base.

Let $f:\C{X}\to\C{Y}$ be a $(P_{\red})$-representable morphism of $\infty$-stacks such that $\C{Y}$ is placid. Then  the induced map $\C{X}_{\red}\to\C{Y}$ is $(P)$-representable. Furthermore, if $(P)$ only satisfies (i) and (ii), then the assertion holds when
$\C{Y}$ is a placid affine scheme.
\end{Lem}

\begin{proof}
Assume first that  $\C{Y}=Y$ is a placid affine scheme. Since $f$ is $(P_{\red})$-representable, there exists a locally fp morphism $X'\to Y$ of algebraic spaces, which belongs to $(P)$, and an isomorphism $\C{X}_{\red}\simeq X'_{\red}$ over $Y$ (by (i)). Then $X'$ is a placid algebraic space (by \rl{glplacid}(c)), thus $X'_{\red}\to X'$ is an fp-closed embedding (by \rco{red}). Hence, by our assumption (ii), the composition $\C{X}_{\red}\simeq X'_{\red}\hra X'\to Y$ is in $(P)$. 

Assume next that $\cY$ is $0$-placid, that is, $\cY$ decomposes as a coproduct  $\cY\simeq\sqcup_{\al}Y_{\al}$ of $0$-placid affine schemes. Then, any morphism $Y\to\cY$, where $Y$ is an affine scheme, factors through a finite coproduct $\sqcup_{\al_i}Y_{\al_i}$, which is a placid affine scheme. Thus, by the placid affine case, proven above, the map $\C{X}_{\red}\to\C{Y}$ is $(P)$-representable.

In the general case, choose a smooth covering $Y\to\C{Y}$ with $0$-placid  $Y$. Then we have an isomorphism
$\C{X}_{\red}\times_{\C{Y}}Y\simeq (\C{X}\times_{\C{Y}}Y)_{\red}$ (by \rl{reduction}). Therefore, by the proven above, the morphism
$\C{X}_{\red}\times_{\C{Y}}Y\to Y$ is $(P)$-representable. Hence the morphism $\C{X}_{\red}\to\C{Y}$ is $(P)$-representable
by (iii) and \rl{etalelocal}(a).
\end{proof}

\begin{Cor} \label{C:red-fp}
Let $(P)$ be either the class of locally fp morphisms of algebraic spaces, or of fp-proper morphisms, or
of fp-locally closed embeddings. Then the conclusion of \rl{red-fp} holds.
\end{Cor}

\begin{proof}
All these classes satisfy (iii) by \cite[Tags 041T, 041V, 0422, 0CFX]{Sta}, while assumptions (i) and (ii) are clear.
\end{proof}

\begin{Emp} \label{E:stopeq}
{\bf Notation.} (a) We call a morphism $f:\C{X}\to\C{Y}$ of $\infty$-stacks {\em topologically locally fp-representable/schematic} (resp. {\em topologically fp-proper representable/schematic})\label{I:topologically locally fp}\label{I:topologically fp-proper},
if $f$ is $(P_{\red})$-repre-sentable (\re{redP}(b)), where $(P)$ is the class of
locally fp-morphisms (resp. fp-proper morphisms) of algebraic spaces/schemes.


(b) By a straightforward argument, all classes of (a) are closed under compositions.
\end{Emp}

\begin{Emp} \label{E:remsttopeq}
{\bf Remarks.} (a) The notions in \re{topeq} and \re{stopeq} can be generalized. Namely, instead of considering of a category of reduced affine schemes $\Aff_{\red,k}$ one could consider a subcategory of {\em perfect} (or {perfectly reduced}) affine schemes $\Aff_{\perf,k}\subset\Aff_{\red,k}$, that is, affine schemes $X$ such that for every reduced scheme $X'$, a universal homeomorphism
$X'\to X$  is an isomorphism (compare \cite{BGH}). Then in definitions \re{red}, \re{topeq}, \re{redP} and \re{stopeq} one can replace $(\cdot)_{\red}$ by $(\cdot)_{\perf}$ in all places.

(b) Though the notions outlined in (a) are more general (and probably more standard) than those defined in \re{topeq} and \re{stopeq}, the perfectization is much less transparent operation than reduction from geometric point of view. For example, perfectization of a scheme of finite type is almost never of finite type, and  perfectization $\cX_{\perf}$ of a placid $\infty$-stack $\cX$ is almost never placid.
Moreover, with our notion we get much stronger geometric results in part two, and avoid introducing a new (more complicated) notion of perfectly placid $\infty$-stacks.

(c) A construction in the spirit of (a) is necessary, if one would like to extend the results of this work to the mixed characteristic. We are planning to do it in a future.
\end{Emp}

\section{Dimension function and equidimensional morphisms}

Our next goal is to introduce dimension function and an important class of equidimensional morphisms and its variants first for algebraic spaces of finite type, and then for placid $\infty$-stacks.

\subsection{The case of algebraic spaces of finite type}

\begin{Emp} \label{E:leq}
{\bf Equidimensional algebraic spaces and the canonical filtration.} Let $X$ be an algebraic space of finite type over $k$.

(a) To every $x\in X$, one associates the dimension $X$ at $x$, defined as $\dim_x(X):=\min_{x\in U}\dim U$, where $U$ runs over all open neighborhoods of $x$. Alternatively, $\dim_x(X)$ is the maximal of dimensions of irreducible components, containing $x$. We denote by $\un{\dim}_X:X\to\B{Z}$\label{N:dimx} the function $x\mapsto\dim_x(X)$.

(b) Recall that $X$ is called {\em equidimensional}, if all irreducible components of $X$ are of the same dimension. Equivalently, this happens if and only if the dimension function $\un{\dim}_X$ is constant.

(c) For every $i\in\B{N}$, we consider $X_{\geq i}:=\un{\dim}_X^{-1}(\{n\,|\,n\geq i\})$, $X_{\leq i}:=\un{\dim}_X^{-1}(\{n\,|\,n\leq i\})$ and $X_{i}:=\un{\dim}_X^{-1}(\{i\})$. By definition, each $X_{\leq i}\subset X$ is open, $X_{\geq i}=X\sm X_{\leq i-1}$ is closed, and $X_i=X_{\geq i}\cap X_{\leq i}$ is locally closed. Explicitly, each $X_{\geq i}$ is the union of all irreducible components of $X$ of dimensions $\geq i$, and $X_{i}=X_{\geq i}\sm X_{\geq i+1}$. In particular, $X_i$ is equidimensional of dimension $i$. Let $\eta_i:X_i\hra X$ be the embedding.

(d) We say that $X$ is {\em locally equidimensional}, if the dimension function $\un{\dim}_X$ is locally constant. This happens
if and only if each connected component of $X$ is equidimensional, or equivalently, if and only if each
$X_i\subset X$ from (d) is a union of (some of) connected components of $X$.
\end{Emp}

\begin{Emp} \label{E:locdim}
{\bf Dimension function and equidimensional morphisms.}

(a) To every morphism $f:X\to Y$ of algebraic spaces of finite type over $k$, we associate the dimension function
$\un{\dim}_{f}:=\un{\dim}_{X}-f^*(\un{\dim}_{Y}):X\to\B{Z}$.\label{N:dimf} Explicitly, for every $x\in X$ we have $\un{\dim}_f(x)=\dim_x(X)-\dim_{f(x)}(Y)$.

(b) We call $f$ {\em weakly equidimensional},\label{I:weakly equidimensional morphism} if the dimension function $\un{\dim}_f$ is locally constant.

(c) We call $f$ {\em equidimensional},\label{I:equidimensional morphism} if $f$ is weakly equidimensional, and  $\un{\dim}_f(x)=\dim_x f^{-1}(f(x))$ for all $x\in X$.

(d) We say that a locally closed algebraic subspace $X\subset Y$ is of {\em pure codimension $d$},\label{I:pure codimension} and write $\codim_X(Y)=d$, if the embedding
$X\hra Y$ is weakly equidimensional of constant dimension $-d$. For example, each stratum $Y_i\subset Y$ from \re{leq}(c) is of pure codimension $0$,
and $X\subset Y$ is of pure codimension $\dim Y-\dim X$, if both $Y$ and $X$ are equidimensional.

(e) For shortness, we will often call universally open equidimensional morphisms simply {\em uo-equidimensional}\label{I:uo-equidimensional}.
\end{Emp}

\begin{Emp} \label{E:loceq}
{\bf Remarks.}
(a) Our notion of an equidimensional morphism is slightly stronger than that of \cite{EGA}. For example, an embedding of an irreducible component
$i:X'\hra X$ is always equidimensional in the sense of \cite{EGA}, but it is not weakly equidimensional in our sense, if $\dim X'<\dim X$.
On the other hand, both notions coincide, if $f$ is dominant or open.

(b) Notice that $f$ is automatically weakly equidimensional, if $X$ and $Y$ are locally equidimensional. Also every morphism $\iota:\pt\to X$ is weakly equidimensional.

(c) Explicitly, $f$ is weakly equidimensional of dimension $d$ if and only if for every $i\in\B{N}$, we have $f(X_i)\subset Y_{i-d}$.

(d) Notice that an algebraic space $X$ is locally equidimensional if and only if the structure morphism $X\to\pt$ is equidimensional.
\end{Emp}

\begin{Lem} \label{L:open}
For a morphism $f:X\to Y$ of algebraic spaces of finite type over $k$, we have an inequality $\un{\dim}_f(x)\leq\dim_x f^{-1}(f(x))$.
Moreover, this inequality is an equality, if $f$ is an open map.
\end{Lem}

\begin{proof}
The assertion is well-known (see, for example, \cite[14.2.1]{EGA} or \cite[Tag 0B2L]{Sta}).
\end{proof}




\rl{open} immediately implies the following corollary.

\begin{Cor} \label{C:equidim}
If $f$ is open and weakly equidimensional, then it is equidimensional.
\end{Cor}



\begin{Lem} \label{L:2outof3}
Let $X\overset{f}{\to}Y\overset{g}{\to}Z$ be morphisms of algebraic spaces of finite type over $k$.

(a) If $f$ is surjective, and $g\circ f$ is (universally) open, then $g$ is (universally) open.

(b) We have an equality $\un{\dim}_{g\circ f}=\un{\dim}_f+f^*(\un{\dim}_g)$.

(c) Assume that $g$ is weakly equidimensional. Then $f$ is weakly equidimensional if and only if $g\circ f$ is.

(d) Assume that $f$ is open surjective. If  $f$ and $g\circ f$ are weakly equidimensional, then so is $g$.

(e) Assume that $f$ and $g$ are open, and $f$ is surjective. If  $g\circ f$ are weakly equidimensional, then so are $f$ and $g$.
\end{Lem}

\begin{proof}
(a) and (b) are clear, and (c) follows from (b).

(d) By (b), the assumption implies that $f^*(\un{\dim}_g)=\un{\dim}_{g\circ f}-\un{\dim}_f$ is locally constant. Since $f$ is open and surjective, the function $\un{\dim}_g$ is locally constant as well.

(e) By \rl{open}, both functions $\un{\dim}_f$ and $\un{\dim}_g$ are upper semi-continuous, that is, the preimage of $\{n\,|\,n\geq i\}$ is closed for all $i$. Then $f^*(\un{\dim}_g)$ is upper semi-continuous as well. Since the sum $\un{\dim}_{g\circ f}=\un{\dim}_f+f^*(\un{\dim}_g)$ is locally constant, we conclude that both function $\un{\dim}_f$ and $f^*(\un{\dim}_g)$ are lower semi-continuous as well. This implies that both $\un{\dim}_f$ and $f^*(\un{\dim}_g)$ are locally constant, and hence (as in (d)), function $\un{\dim}_g$ is locally constant as well.
\end{proof}


\begin{Cor} \label{C:pullback}
Consider a Cartesian diagram of algebraic spaces of finite type over $k$
\begin{equation} \label{Eq:leq}
\begin{CD}
X' @>g>> Y'\\
@V\psi VV @VV\phi V\\
X @>f>> Y
\end{CD}
\end{equation}
such that either $f$ and $g$ are open or $\phi$ and $\psi$ are open.

(a) Then we have an equality $\un{\dim}_{g}=\psi^{*}(\un{\dim}_{f})$.

(b) If $f$ is weakly equidimensional, then $g$ also is.
\end{Cor}

\begin{proof}
(a) For every $x'\in X'$, we set $x:=\psi (x')\in X$, $y':=g(x')\in Y'$ and $y=f(x)=\phi(y')\in Y$.
We want to show that $\un{\dim}_{g}(x')=\un{\dim}_{f}(x)$.

When $f$ and $g$ are open, then we have to show the equality
$\dim_{x'}g^{-1}(y')=\dim_{x}f^{-1}(y)$ (by \rl{open}). Since diagram \form{leq} is Cartesian, $\psi$ induces an isomorphism $g^{-1}(y')\simeq f^{-1}(y)\times_y y'$, which implies the required equality of dimensions.

Assume now that $\phi$ and $\psi$ are open. Then, by the proven above, we have
$\un{\dim}_{\psi}=g^*(\un{\dim}_{\phi})$. On the other hand, using \rl{2outof3}(b) for $f\circ \psi=\phi\circ g$, we conclude that
\[
\un{\dim}_{\psi} + \psi^{*}(\un{\dim}_{f})=\un{\dim}_{g}+g^*(\un{\dim}_{\phi}),\text{ hence }
\un{\dim}_{\psi}-g^*(\un{\dim}_{\phi})= \un{\dim}_{g}-\psi^{*}(\un{\dim}_{f}).
\]
Since the LHS of the latter expression vanishes by the proven above, the RHS vanishes as well.

(b) The assertion follows immediately from (a).
\end{proof}

\begin{Emp} \label{E:rempullback}
{\bf Notation.} Since diagram \form{leq} is Cartesian, we can denote morphism $g$ by $\phi^*(f)$, and dimension function
$\psi^*(\un{\dim}_f):X'\to \B{Z}$ by $\phi^*(\un{\dim}_f)$. Therefore the equality of \rco{pullback}(a)
can we rewritten as $\phi^*(\un{\dim}_f)=\un{\dim}_{\phi^*(f)}$.
\end{Emp}

\begin{Cor} \label{C:leqop}
The class of universally open equidimensional morphisms is closed under compositions and base change.
\end{Cor}

\begin{proof}
While the first assertion follows from \rl{open} and \rl{2outof3}(c), the second one follows from \rco{pullback}(b).
\end{proof}

\begin{Cor} \label{C:equidim2}
Let $X\overset{f}{\to}Y\overset{g}{\to}Z$ be morphisms of algebraic spaces of finite type over $k$.

(a) If $f$ and $g$ are equidimensional, then so is $g\circ f$.

(b) Assume that $f$ is open surjective. If  $f$ and $g\circ f$ are equidimensional, then so is $g$.
\end{Cor}

\begin{proof}
For $x\in X$, we set $y:=f(x)\in Y$ and $z:=g(y)\in Z$, and let $f_z:(g\circ f)^{-1}(z)\to g^{-1}(z)$ be the restriction of $f$.

(a) Since $f$ and $g$ are equidimensional, we conclude from \rl{open} applied to $f_z$ that for every $x\in X$ we have inequality
\[
\dim_x (g\circ f)^{-1}(z)\leq\dim_x f^{-1}(y)+\dim_y(g^{-1}(z))=\un{\dim}_f(x)+\un{\dim}_g(y)=\un{\dim}_{g\circ f}(x).
\]
Now the assertion follows from \rl{2outof3}(c) and \rl{open}.

(b) Since $f$ is open equidimensional, its restriction $f_z$ is such as well (by \rco{pullback}(a)). Therefore we have equality
$\dim_y g^{-1}(z)=\dim_x (g\circ f)^{-1}(z)-\dim_x f^{-1}(y)$. This together with equidimensionality of $f$ and $g\circ f$ implies that
$\dim_y g^{-1}(z)$ equals $\un{\dim}_{g\circ f}(x)-\un{\dim}_f(x)=\un{\dim}_g(y)$.
Now the assertion follows from \rl{2outof3}(d).
\end{proof}

\begin{Emp} \label{E:smoothflat}
{\bf Examples.} (a) A flat morphism between algebraic spaces of finite type is automatically universally open, but not necessarily weakly  equidimensional. For example, consider the projection $X\to\pt$ from a not locally equidimensional algebraic space.

(b) Any smooth morphism $f:X\to Y$ between algebraic spaces of finite type is uo-equidimensional. Indeed, for every connected component $X'\subset X$, the restriction $f|_{X'}:X'\to Y$ is smooth of some relative dimension $n$. Then the restriction $\un{\dim}_f|_{X'}$ is a constant function with value $n$, thus $f$ is weakly equidimensional.
\end{Emp}





\subsection{Dimension function for fp-morphisms}
In this subsection we are going to define dimension function and introduce equidimensional morphisms and their variants first for  fp-morphisms of $0$-placid affine schemes, and then for locally fp-representable morphisms of placid $\infty$-stacks.

\begin{Emp} \label{E:topsp}
{\bf Underlying topological space.}
(a) Recall that to every $\infty$-stack $\C{X}\in\St_k$, one associates the underlying topological space $[\C{X}]$ \label{N:[x]} such that

$\bullet$ the underlying set is defined to be the set of equivalent classes of pairs $(K,[z])$, where $K/k$ is a field extension,
$[z]\in \pi_0(\C{X}(K))$, and $(K',[z'])\sim (K'',[z''])$, if there exists a larger field
$K\supset K',K''$ such that $[z']$ and $[z'']$ have the same image in $\pi_0(\C{X}(K))$.

$\bullet$ a subset $U\subset [\cX]$ is open, if $U=[\cU]$ for some  open $\infty$-substack $\cU\subset\cX$.

(b) Clearly, every morphism $f:\cX\to \cY$ of $\infty$-stacks induces a continuous map $[f]:[\cX]\to[\cY]$
of topological spaces.

(c) We call a morphism $f:\cX\to\cY$ of $\infty$-stacks {\em open}\label{I:open morphism}, if the induced map $[f]$ is open. We call $f$ {\em universally open}\label{I:universally open morphism},
if every pullback $\cX\times_{\cY}\cZ\to\cZ$ of $f$ is open.

(d) Notice that if $X$ is a scheme or (more generally) an algebraic space, then $[X]$ is nothing else but the underlying topological space.
To simplify the notation, we will often denote the topological space $[\C{X}]$ simply by $\C{X}$ and
the map $[f]$ by $f$.
\end{Emp}

The following lemma summarizes basic properties of the above construction.

\begin{Lem} \label{L:topequiv}
(a) A subset $U\subset[\cX]$ is open if and only if for every morphism $f:X\to \cX$ from an affine scheme $X$, the preimage
$[f]^{-1}(U)\subset [X]$ is open.

(b) An \'etale representable morphism of $\infty$-stacks is open.

(c) For an epimorphism $f:\cX\to\cY$ of $\infty$-stacks, the morphism $[f]:[\cX]\to[\cY]$ is {\em submersive}, that is,
$[f]$ is surjective, and a subset $U\subset[\cY]$ is open if and only if $[f]^{-1}(U)\subset[\cX]$ is open.

(d) For every topological equivalence $f:\cX\to\cY$, the induced map $[f]:[\cX]\to [\cY]$ is a homeomorphism. In particular, for every  $\infty$-stack $\cX$, the map $[\pi_{\cX}]:[\cX_{\red}]\to [\cX]$ is a homeomorphism.

(e) A morphism $f:\cX\to\cY$ of $\infty$-stacks is universally open if and only if the pullback $\cX\times_{\cY}Y\to Y$ is open for every
morphism $Y\to\cY$ from an affine scheme $Y$.
\end{Lem}

\begin{proof}
(a) Clearly, if $\cU\subset\cX$ is an open $\infty$-substack, then for every morphism $f:X\to\cX$ from an affine scheme $X$ the preimage $[f]^{-1}([\cU])=[\cU\times_{\cX}X]\subset [X]$ is open.
Conversely, let $U\subset [\cX]$ be such that every $[f]^{-1}(U)\subset [X]$ is open. Define an
$\infty$-substack $\cU\subset\cX$ by the rule that for every affine scheme
$X$, the subspace $\cU(X)\subset\cX(X)=\Hom(X,\cX)$ consists of all morphisms $f:X\to \cX$ such that $[f]([X])\subset U$.
Unwinding the definitions, one gets that $\cU\subset\cX$ is an open $\infty$-substack, and $[\cU]=U$.

(b) follows from (a) and the corresponding assertion for schemes.

(c) Using (a), one reduces to the case when $\cY$ is an affine scheme $Y$. By assumption, there exists an \'etale covering $\pi:Y'\to Y$
and a morphism $s:Y'\to\cX$ such that $f\circ s=\pi$. Then the assertion follows from the fact that $\pi$ is open.

(d) It suffices to show the second assertion, but this is clear.

(e) follows from (a).
\end{proof}

The following observations will be used several times.

\begin{Emp} \label{E:obs}
{\bf Observations.}
(a) Let $f:X\to Y$ be an fp-morphism of affine schemes. Then there exists a Cartesian diagram $D_{\al}$
\[
\begin{CD}
X @>f>> Y\\
@V\pr_{\al}VV @VV\pr_{\al}V\\
X_{\al} @>f_{\al}>> Y_{\al},
\end{CD}
\]
where $X_{\al}$ and $Y_{\al}$ are affine schemes of finite type (see \cite[Tag 01ZM]{Sta}). Moreover, if $Y$ is $0$-placid, then the
projection $\pr_{\al}$ can be assumed to be strongly pro-smooth.

(b) Let $f:\cX\to\cY$ be a locally fp-representable morphism of placid $\infty$-stacks.  Then

\quad(i) for every smooth morphism $\phi:Y\to\cY$ from a $0$-placid affine scheme $Y$, the pullback $\cX\times_{\cY}Y$ is an algebraic space,
locally fp over $Y$;

\quad(ii) for every \'etale morphism $\psi: X\to \cX\times_{\cY}Y$ from a $0$-placid affine scheme $X$, the composition  $g:X\to \cX\times_{\cY}Y\to Y$ is an fp-morphism of $0$-placid affine schemes.

\quad(iii) The composition $\varphi:X\to \cX\times_{\cY}Y\to \cX$ is smooth (see \re{propplst}(c)).

\end{Emp}


\begin{Lem} \label{L:dim0pl}
(a) For every fp-morphism $f:X\to Y$ of $0$-placid affine schemes, there exists a unique function
$\un{\dim}_f:X\to\B{Z}$\label{N:dimf1} such that for every Cartesian diagram $D_{\al}$ as in \re{obs}(a), we have an equality
$\un{\dim}_f=\pr_{\al}^*(\un{\dim}_{f_{\al}})$.

(b)  The dimension function of (a) is {\em additive}, that is, for every diagram $X\overset{f}{\to} Y\overset{g}{\to} Z$, we have an equality $\un{\dim}_{gf}=\un{\dim}_f+f^*(\un{\dim}_g)$.

(c) If $f:X\to Y$ is an fp-\'etale or an fp-universal homeomorphism,
then $\un{\dim}_f=0$.

(d) For every Cartesian diagram \form{leq} such that $f$ is an fp-morphism of $0$-placid schemes and $\phi$ is strongly pro-smooth,
we have an equality $\un{\dim}_{g}=\psi^*(\un{\dim}_{f})$.
\end{Lem}

\begin{proof}
(a) As the uniqueness is clear, it suffices to show the existence. In other words, we have to show that the pullback
$\pr_{\al}^*(\un{\dim}_{f_{\al}})$ is independent of the Cartesian diagram $D_{\al}$. Notice that using observation \re{canpres} and \cite[Tag 01ZM]{Sta}, any pair of Cartesian diagrams $D_{\al}$ and $D_{\beta}$ are dominated by a third Cartesian diagram $D_{\g}$ such that  the induced projections $Y_{\g}\to Y_{\al}$ and $Y_{\g}\to Y_{\beta}$ are smooth. Now the independence assertion follows from \rco{pullback}.

(b) follows from the corresponding property \rl{2outof3}(b) of schemes of finite type.

(c) follows from the fact that $f$ is a pullback of an \'etale morphism or a universal homeomorphism
of affine schemes of finite type.

(d) follows from the fact that the composition $\pr_{\al}\circ \phi:Y'\to Y\to Y_{\al}$ is strongly pro-smooth (see \re{placidprop}) and the characterization of the dimension function $\un{\dim}_g$, given in (a).
\end{proof}

\begin{Lem} \label{L:dimstacks}
(a) For every locally fp-representable morphism $f:\C{X}\to \C{Y}$ between placid $\infty$-stacks there exists a unique function
$\un{\dim}_f:[\C{X}]\to\B{Z}$\label{N:dimf2} such that in the situation of \re{obs}(b), for every two morphisms $\phi$ and $\psi$,
the pullback  $\varphi^*(\un{\dim}_f)$ equals $\un{\dim}_{g}$, defined in \rl{dim0pl}(a).

(b) The dimension function from (a) is additive.

(c) For every placid $\infty$-stack $\cX$, the canonical map $\pi_{\cX}:\cX_{\red}\to\cX$ is an fp-closed embedding, and
$\un{\dim}_{\pi_{\cX}}=0$.
\end{Lem}




\begin{proof}
(a) Since a placid $\infty$-stack $\cY$ has a smooth covering $\{\phi:Y\to\cY\}_{\phi}$ by $0$-placid affine schemes $Y$, while each algebraic space  $\cX\times_{\cY}Y$ has an \'etale covering
 $\{\psi:X\to\cX\times_{\cY}Y\}_{\psi}$ by $0$-placid affine schemes $X$, we conclude that
morphisms $\{\varphi:X\to \cX\}_{\phi,\psi}$ form a covering of $\cX$. Therefore there exists at most one function $\un{\dim}_f:[\C{X}]\to\B{Z}$ such that for every pair $(\phi,\psi)$ as in \re{obs}(b), we have an equality
$\varphi^*(\un{\dim}_f)=\un{\dim}_{g}$. Explicitly, for every $\un{x}\in[\cX]$ there exists a pair $(\phi,\psi)$ and a point $x\in X$ such that $\varphi(x)=\un{x}$, and we have an equality
$\un{\dim}_f(\un{x})=\un{\dim}_{g}(x)$.

To show the existence, we have to show that the expression $\un{\dim}_{g}(x)$ is independent of all choices. Namely, we have to show that for every two triples $(\phi',\psi',x')$ and $(\phi'',\psi'',x'')$ as above, we have an equality $\un{\dim}_{g'}(x')=\un{\dim}_{g''}(x'')$. First we claim that these two triples are dominated by a third triple $(\phi,\psi,x)$ such that the induced morphisms $Y\to Y'$ and $Y\to Y''$ are strongly pro-smooth.

Indeed, since $\varphi'(x')=\un{x}=\varphi''(x'')$, the pair $(x',x'')$ gives rise to a point $\wt{x}$ of $\wt{X}:=X'\times_{\cX}X''$, whose image in $\wt{Y}:=Y'\times_{\cY}Y''$ we denote by $\wt{y}$. Then $\wt{Y}$ is a placid $\infty$-stack (use \re{propplst}(c),(d)), hence there exists a smooth covering $Y\to\wt{Y}$ of $\wt{y}$, where $Y$ is a $0$-placid affine scheme. Then both projections $p':Y\to Y'$ and $p'':Y\to Y''$ are smooth, thus we can replace $Y$ by its smooth covering such that both $p'$ and $p''$ are strongly pro-smooth (by \re{propplst}(e)). We denote by $\phi$ the induced morphism $Y\to\cY$.

Next, notice that since $\psi'$ and $\psi''$ are \'etale schematic morphisms, their fiber product $\wt{X}\to \cX\times_{\cY}\wt{Y}$ and hence its pullback $\wt{X}\times_{\wt{Y}}Y \to \cX\times_{\cY}Y$ are also such. Since the projection $\cX\times_{\cY}Y\to Y$ is locally fp-proper,
we conclude that $\wt{X}\times_{\wt{Y}}Y$ is a placid algebraic space (by \rl{glplacid}(c)).
Moreover, projection  $\wt{X}\times_{\wt{Y}}Y\to \wt{X}$ is a covering of $\wt{x}$ (because $Y\to \wt{Y}$ is a covering of $\wt{y}$), thus there exists an \'etale morphism $X\to\wt{X}\times_{\wt{Y}}Y$,  where $X$ is a $0$-placid affine scheme, and a point $x\in X$ such that $\wt{x}\in\wt{X}$ is the image of $x$ under the composition $X\to \wt{X}\times_{\wt{Y}}Y\to \wt{X}$. Finally, we set $\psi$ to be the composition $X\to \wt{X}\times_{\wt{Y}}Y \to \cX\times_{\cY}Y$.

It suffices to show the equality $\un{\dim}_{g'}(x')=\un{\dim}_{g}(x)=\un{\dim}_{g''}(x'')$. To show the first equality, notice that we have a commutative diagram
\[
\begin{CD}
X @>g>> Y\\
@Vq'VV @VVp'V \\
X' @>g'>> Y',
\end{CD}
\]
where $p'$ is strongly pro-smooth, $q'(x)=x'$, and the induced map $\wt{\psi}:X\to X'\times_{Y'}Y$ is \'etale.

It suffices to show the equality
$q'^*(\un{\dim}_{g'})=\un{\dim}_{g}$. But this follows from the fact $\un{\dim}_{g'}$
commutes with strongly pro-smooth pullbacks (see \rl{dim0pl}(d)), dimension function is additive (by \rl{dim0pl}(b)) and
$\un{\dim}_{\wt{\psi}}=0$ (by \rl{dim0pl}(c)).

(b) By (a), the assertion reduces to the corresponding assertion for fp-morphisms of $0$-placid affine schemes, so
the assertion follows from \rl{dim0pl}(b).

(c) The first assertion follows from \rco{red}(b). Next, by the characterization of (a) and \rl{reduction}, the assertion reduces to
the corresponding assertion for $0$-placid affine schemes, so the assertion follows from the second assertion in \rl{dim0pl}(c).
\end{proof}

\begin{Emp} \label{E:eqmor}
{\bf Equidimensional morphisms.}
A locally fp-representable morphism $f:\cX\to\cY$
of placid $\infty$-stacks is called

$\bullet$ {\em weakly equidimensional (of relative dimension $d$)}\label{I:we}, if the dimension function
$\un{\dim}_f:[\cX]\to\B{Z}$ from \rl{dimstacks}(a) is locally constant (constant with value $d$);

$\bullet$ {\em equidimensional}\label{I:eq}, if it is weakly equidimensional and satisfies
 $\un{\dim}_{f}(x)=\dim_x f^{-1}(f(x))$  for every $x\in [\cX]$;

$\bullet$ {\em uo-equidimensional}\label{I:uoeq}, if it weakly equidimensional and universally open (see \re{topsp}(c)).
\end{Emp}

\begin{Emp} \label{E:exfintype}
{\bf Examples.} Let $f:X\to Y$ be a morphism of algebraic spaces of finite type. Then $f$ is an fp-representable morphism of
placid $\infty$-stacks (by \re{explinfst}(a)). Moreover, since an \'etale morphism of algebraic spaces of finite type is
equidimensional of relative dimension zero, one deduces that the dimension function $\un{\dim}_f$ of \rl{dimstacks}(a) coincides with the dimension function $\un{\dim}_f$ of \re{locdim}(a). Therefore $f$ is (weakly) equidimensional in the sense of \re{eqmor} if and only if it is (weakly) equidimensional in the sense of \re{locdim}(b),(c).
\end{Emp}

\begin{Cor} \label{C:dimred}
Let $f:\C{X}\to \C{Y}$ be a locally fp-representable morphism of placid $\infty$-stacks. Then the restriction
$f|_{\C{X}_{\red}}:\cX_{\red}\to\cY$ and the reduction $f_{\red}:\C{X}_{\red}\to \C{Y}_{\red}$ are locally fp-representable morphisms of placid $\infty$-stacks, and we have equalities $\un{\dim}_{f|_{\C{X}_{\red}}}= \un{\dim}_{f_{\red}}=\pi^*_{\cX}(\un{\dim}_f)$.
In particular, $f$ is (weakly) equidimensional (of relative dimension $d$) if and only if $f_{\red}$ (resp. $f|_{\cX_{\red}}$) satisfies this property.
\end{Cor}

\begin{proof}
The first assertion follows from \rco{red}(b), the equality of dimension functions follows from \rl{dimstacks}(b),(c). The last assertion now follows from \rl{topequiv}(d).
\end{proof}


\begin{Lem} \label{L:dimweak}
(a)  Let $f:X\to Y$ be an fp-morphism of $0$-placid affine schemes. Then $f$ (weakly) equidimensional (resp. universally open) if and only if there exists a strongly pro-smooth morphism $\pr:Y\to Y_{\al}$ with $Y_{\al}\in\Affft_k$ and a (weakly) equidimensional (resp. universally open) morphism
$f_{\al}:X_{\al}\to Y_{\al}$ in $\Affft_k$ such that $f\simeq Y\times_{Y_{\al}}f_{\al}$ (see \re{obs}(a)).

(b) Let $f:\cX\to\cY$ be a locally fp-presentable morphism of placid $\infty$-stacks. Then $\cY$ is (weakly) equidimensional (resp. universally open) if and only if for every smooth morphism $\phi:Y\to\cY$ and \'etale morphism
$\psi:X\to \cX\times_{\cY}Y$ as in \re{obs}(b), the composition $g:X\to \cX\times_{\cY}Y\to Y$ is (weakly) equidimensional (resp. universally open).

\end{Lem}

\begin{proof}
In all cases, we will only show the assertions concerning the weakly equidimensional and universally open morphisms, while the equidimensionality assertion is immediate.


(a) Clearly, if $\un{\dim}_{f_{\al}}$ is locally constant (resp. $f_{\al}$ is universally open) then $\un{\dim}_f=\pr^*(\un{\dim}_{f_{\al}})$ is locally constant (resp. $f$ is universally open).

To show the converse, choose a placid presentation $Y\simeq\lim_{\al}Y_{\al}$. Since $f$ is finitely presented, there exists a morphism $f_{\al}:X_{\al}\to Y_{\al}$ such that $f\simeq f_{\al}\times_{Y_{\al}}Y$. If $f$ is universally open, then it follows from a standard limit theorem (see \cite[Proposition B.3(xii)]{Ry2}) that there exists $\beta>\al$
such that $f_{\beta}:=f_{\al}\times_{Y_{\al}}Y_{\beta}$ is universally open.

Assume now that $f$ is weakly equidimensional. By assumption, $\un{\dim}_{f}=\pr_{\al}^*(\un{\dim}_{f_{\al}})$ is locally constant. Since $X$ is quasi-compact, there exists a finite open covering $X=\cup_i U_i$ such that each $U_i$ is quasi-compact and $\un{\dim}_{f}$ is constant on each $U_i$. This covering is induced by an open covering $X_{\beta}=\cup_i U_{\beta,i}$ for some $\beta>\al$. Since $\un{\dim}_{f}=\pr_{\beta}^*(\un{\dim}_{f_{\beta}})$, we thus conclude that $\un{\dim}_{f_{\beta}}$ is constant on each
$U_{\beta,i}\cap\pr_{\beta}(X)$, thus $\un{\dim}_{f_{\beta}}$ is locally constant on $\pr_{\beta}(X)$.

It suffices to show that there exists an open neighborhood $U\supset \pr_{\beta}(X)$ such that the restriction
$\un{\dim}_{f_{\beta}}|_U$ is locally constant. Indeed, in this case, we would have an isomorphism $f\simeq Y\times_{Y_{\beta}}(f_{\beta}|_{U})$.

By assumption, for every $x\in \pr_{\beta}(X)$ there exists an open neighborhood $U_x\subset X_{\beta}$ of $x$ such that
$\un{\dim}_{f_{\beta}}$ is constant on $U_x\cap\pr_{\beta}(X)$, and we claim that there exists a smaller open neighborhood $V_x\subset U_x$ of $x$ such that $\un{\dim}_{f_{\beta}}$ is constant on $V_x$.

Consider the set $Z:=\{y\in U_x\,|\,\un{\dim}_{f_{\beta}}(y)\neq \un{\dim}_{f_{\beta}}(x)\}$. By our assumption on $U_{\al}$, we have $Z\cap \pr_{\beta}(X)=\emptyset$.
Since $\pr_{\beta}:X\to X_{\beta}$ is strongly pro-smooth, thus flat, the image $\pr_{\beta}(X)\subset X_{\beta}$ is closed under generalizations. In particular, no generalization of $x$ belongs to $Z$. Since $Z\subset U_x$ is constructible, we conclude that the closure
$\ov{Z}\subset U_x$ of $Z$ does not contain $x$. Hence $V_x:=U_x\sm\ov{Z}$ is an open neighborhood of $x$, and $\un{\dim}_{f_{\beta}}$ is constant on $V_x$.

(b) Assume that $f$ is weakly equidimensional, that is, $\un{\dim}_f$ is locally constant. Then for every $\phi$ and $\psi$, the induced morphism $\un{\dim}_g=\varphi^*(\un{\dim}_f)$ is
locally constant. Conversely, assume that each $\un{\dim}_g=\varphi^*(\un{\dim}_f)$  is locally constant. Since morphisms $\{\varphi:X\to\cX\}_{\phi,\psi}$ form a covering of $\cX$,
it follows from \rl{topequiv}(c) that $\un{\dim}_f$ is locally constant.

Next, if $f$ is universally open, then its pullback $\phi^*(f):\cX\times_{\cY}Y\to Y$ is universally open, thus $g=\phi^*(f)\circ\psi$ is universally open (by \rl{topequiv}(b)). Conversely,
assume that each $g$ is universally open. Since the $\psi$'s form an \'etale covering of $\cX\times_{\cY}Y$ we conclude that each $\phi^*(f)$ is universally open. Since the $\phi$'s form a covering of $\cY$, the assertion that $f$ is universally open follows from \rl{topequiv}(c).
%
\end{proof}

The following simple lemma will be useful later.

\begin{Lem} \label{L:weakequid}
Let $f:Y\to X$ be an fp-morphism between strongly pro-smooth algebraic spaces such that $Y$ is connected. Then $f$ is a weakly equidimensional morphism of constant relative dimension.
\end{Lem}

\begin{proof}
Choose a strongly pro-smooth presentation
$X\simeq\lim_{\al}X_{\al}$. Since the morphism $f:Y\to X$ is finitely presented, it comes from a morphism $f_{\al}:Y_{\al}\to X_{\al}$.
Then $Y$ has a placid presentation $Y\simeq\lim_{\beta>\al}Y_{\beta}$ with $Y_{\beta}=Y_{\al}\times_{X_{\al}}X_{\beta}$. Since $Y$ is strongly pro-smooth,
it follows from \rco{indep} that $Y_{\beta}$ is smooth, if $\beta$ is sufficiently large. Moreover,  since $Y$ is connected, one can assume that
$Y_{\beta}$ is connected. Since $Y_{\beta}$ and   $X_{\beta}$ are smooth, thus locally equidimensional, we conclude that the morphism
$f_{\al}:Y_{\beta}\to X_{\beta}$ is of constant relative dimension (see \re{loceq}(b)). Therefore its pullback $f:Y\to X$ is of constant relative dimension as well.
\end{proof}





\subsection{Equidimensional morphisms in general} \label{S:ext}
In this section we are going to extend construction from the previous section to a much more general class of morphisms.


\begin{Emp} \label{E:clM}
{\bf Notation.} Consider the following classes of morphims:

(a) A class of morphisms $f:X\to Y$ between $0$-placid affine schemes such that
for every strongly pro-smooth morphism $Y\to Y'$ with $Y'\in\Affft_k$, the composition $X\to Y\to Y'$ decomposes as
$X\to X'\to Y'$, where $X'\in\Affft_k$, morphism $X\to X'$ is strongly pro-smooth, and $X'\to Y'$ is weakly equidimensional (resp. equidimensional, resp. universally open).

(b) A class of morphisms $f:\cX\to\cY$ between placid $\infty$-stacks such that for every smooth morphisms $f:Y\to\cY$
and $X\to \cX\times_{\cY} Y$, where $Y$ and $Y$ are $0$-placid  affine schemes, the composition $X\to \cX\times_{\cY} Y\to Y$ belongs to the class of (a).

\end{Emp}

The following lemma lists simple properties and compatibilities of classes \re{clM}(a),(b) as well as their relations with those of \re{eqmor}.

\begin{Lem} \label{L:clM}
(i) Classes \re{clM}(a) and \re{clM}(b) are closed under compositions.

(ii) For a morphism $f:X\to Y$ between $0$-placid affine schemes, the following are equivalent:

\quad (1) $f$ belongs to \re{clM}(a);

\quad (2) for every strongly pro-smooth morphism $Y\to Y'$ with $Y'\in\Affft_k$, every placid presentation $X\simeq\lim_{\al}X_{\al}$ of $X$ and every sufficiently large $\al$, the composition $X\to Y\to Y'$ decomposes as $X\overset{\pr_{\al}}{\lra} X_{\al}\to Y'$, where $X_{\al}\to Y'$ is weakly equidimensional (resp. equidimensional, resp. universally open);

\quad (3) there exists a placid presentation $Y\simeq\lim_{\al}Y_{\al}$ of $Y$ such that for every $\al$, the composition $X\to Y\to Y_{\al}$ decomposes as $X\to X_{\al}\to Y_{\al}$, where $X_{\al}\in\Affft_k$, morphism $X\to X_{\al}$ is strongly pro-smooth, and $X_{\al}\to Y_{\al}$ is weakly equidimensional (resp. equidimensional, resp. universally open).

(iii) A morphism $f:\cX\to\cY$ between placid $\infty$-stacks belongs to \re{clM}(b) if and only if there exist smooth coverings
$\{Y_{\al}\to\cY\}_{\al}$ and $\{X_{\al\beta}\to \cX\times_{\cY}Y_{\al}\}_{\beta}$ such that all  $Y_{\al}$  and $X_{\beta}$ are  $0$-placid  affine schemes and all compositions $X_{\al\beta}\to \cX\times_{\cY} Y_{\al}\to Y_{\al}$ belong to \re{clM}(a).

(iv) A morphism $f:X\to Y$ between $0$-placid affine schemes belongs to
\re{clM}(a) if and only if it belongs to \re{clM}(b).


(v) A locally fp-representable morphism $f:\cX\to\cY$ between placid $\infty$-stacks  belongs to \re{clM}(b) if and only if it is weakly equidimensional (resp. equidimensional, resp. universally open) in the sense of \re{eqmor}.




\end{Lem}

\begin{proof}
(i) The assertion for \re{clM}(a) follows from \rl{2outof3}(c) and \rco{equidim2}(a), while the rest follows from definitions.

(ii) By definition, $f$ belongs to \re{clM}(a) if and only if the condition of (2) holds for some placid presentation
$X\simeq\lim_{\al}X_{\al}$ and some index $\al$.  Since the classes of weakly equidimensional and universally open morphisms in $\Affft_k$ contain smooth morphisms and are closed under compositions, the assertion (2) for every presentation and every sufficiently large $\al$ follows from \rco{indep}. Similarly, $f$ belongs to \re{clM}(a) if and only if the condition of (3) holds for every placid presentation $Y\simeq\lim_{\al}Y_{\al}$ of $Y$. Therefore the equivalence of (1) and (3) follows again from \rco{indep}.

(iii),(iv) follow from \rp{proQ}(a),(b) (applicable because of \re{catequid}).


(v) Combining \re{clM}(b) and \rl{dimweak}(b), it suffices to show the corresponding assertion for fp-morphisms between $0$-placid affine schemes and class \re{clM}(a) instead of \re{clM}(b). But this follows from a combination of \rl{dimweak}(a) and \rp{proQ}(c) (applicable because of \re{catequid}).
\end{proof}




\begin{Cor} \label{C:clMc}
The class \re{clM}(b):

(a) contains smooth morphisms (between placid $\infty$-stacks);

(b) is stable under pullbacks with respect to smooth morphisms;

(c) is local with respect to smooth morphisms, that is, if
$\cZ\overset{g}{\to}\cX\overset{f}{\to}\cY$ are morphisms between placid $\infty$-stacks
such that $f\circ g$ belongs to \re{clM}(b) and $g$ is a smooth covering,
then $f$ belongs to \re{clM}(b).
\end{Cor}

\begin{proof}
Since smooth morphisms are closed under composition, assertion (b) follows immediately from definition of \re{clM}(b). Next, using \rl{clM}(iii), we reduce the assertion to the corresponding assertion for the class \re{clM}(a) and strongly pro-smooth morphisms. Now, assertion (a) follows from definition of \re{clM}(a) and the fact that strongly pro-smooth morphisms are closed under compositions (see \re{placidprop}). Finally, assertion (c) follows immediately from \rl{clM}(iii),(iv).
\end{proof}

\begin{Emp} \label{E:clM2}
{\bf Notation.} We call a morphism $f:\cX\to\cY$ of placid $\infty$-stacks

(a) {\em weakly equidimensional}\label{I:we1}
(resp. {\em equidimensional}\label{I:eq1}, resp. {\em pro-universally open}\label{I:pro-universally open}), if it belongs to the class of \re{clM}(b).

(b)  {\em uo-equidimensional}\label{I:uoeq1}, if it is weakly equidimensional and  pro-universally open.
\end{Emp}

\begin{Emp}
{\bf Remark.} By \rl{clM}(v), for locally fp-representable morphisms, the notions of \re{clM2} coincide with those of \re{eqmor}.
\end{Emp}

\begin{Cor} \label{C:uo-equid}
Let $f:\cX\to \cY$ be a uo-equidimensional morphism between placid $\infty$-stacks.

(a) Then $f$ is equidimensional.

(b) Moreover, for every locally fp-representable morphism $\cY'\to\cY$ of placid $\infty$-stacks, the base change
$\cX\times_{\cY}\cY'\to\cY'$ is uo-equidimensional.
\end{Cor}

\begin{proof}
By \re{clM}(b) and \rl{clM}(ii), it suffices to show the corresponding assertions for morphisms of affine schemes of finite type over $k$. In this case, the assertions follow from Corollaries \ref{C:equidim} and \ref{C:leqop}.
\end{proof}




\begin{Lem} \label{L:pb}
Consider a Cartesian diagram
\begin{equation*} 
\begin{CD}
\cX' @>g>> \cY'\\
@V\psi VV @VV\phi V\\
\cX @>f>> \cY
\end{CD}
\end{equation*}
of placid $\infty$-stacks such that $f$ is locally fp-representable, while $\phi$ is pro-universally open.

(a) Then we have an equality $\un{\dim}_{g}=\psi^*(\un{\dim}_{f})$.

(b) If $f$ is (weakly) equidimensional (of relative dimension $d$), then so is $g$.
\end{Lem}

\begin{proof}
As in \rco{pullback}, assertion (b) is an immediate corollary of (a).
To prove (a), we are going to reduce it to the case of $\Affft_k$, treated in \rco{pullback}.
Namely, using characterization of \re{clM}(b) and \rl{dimstacks}(a), and taking pullback with respect to smooth morphisms, we can assume that $f$ and $g$ are fp-affine morphisms between $0$-placid affine schemes and  $\phi$ is strongly pro-smooth.

In this case, $f$ can be written as a pullback
of a morphism $f_{\al}:X_{\al}\to Y_{\al}$ in $\Affft_k$ with respect to a strongly pro-smooth morphism $Y\to Y_{\al}$. Since $\phi$ is pro-universally open,
the composition $Y'\to Y\to Y_{\al}$ decomposes as $Y'\to Y'_{\al}\to Y_{\al}$, where  $Y'\to Y'_{\al}$ is strongly pro-smooth, while
$Y'_{\al}\to Y_{\al}$ is a universally open morphism in $\Aff_k$. Thus we reduce to the case when $f$ and $g$ are morphisms in $\Affft_k$, and $\phi$ is
universally open. Thus the assertion follows from \rco{pullback}.
\end{proof}

\subsection{Placidly stratified $\infty$-stacks and semi-small morphisms}

\begin{Emp} \label{E:compl}
{\bf Complementary $\infty$-substacks.} (a) Let $\cX$ be an $\infty$-stack, and let $\cY\subset\cX$ be an $\infty$-substack, that is, $\cY$ is an $\infty$-stack, and
$\cY(U)\subset\cX(U)$ is a {\em subspace}, that is, a union of connected components, for every $U\in\Aff_k$.

(b) For every $U\in\Aff_k$, consider the subspace $(\cX\sm\cY)(U)\subset\cX(U)$ consisting of all morphisms  $a:U\to\cX$ such that $U\times_{\cX}\cY=\emptyset$. Then $\cX\sm\cY\subset\cX$ is an $\infty$-substack.\label{N:x-y}

(c) By definition, for every morphism $f:\cX'\to\cX$ of $\infty$-stacks, we have a natural identification
$(\cX\sm\cY)\times_{\cX}\cX'\simeq \cX'\sm(\cY\times_{\cX}\cX')$.

(d) Notice that we always have an inclusion $\cY\subset\cX\sm(\cX\sm\cY)$, which is not an equality in general.
\end{Emp}

\begin{Emp} \label{E:fp-locally closed}
{\bf Notation.} (a) We say that a monomorphism $\iota:\cY\to\cX$ of $\infty$-stacks is a {\em topologically (fp)-closed/locally closed embedding},
\label{I:topologically (fp)-closed embedding}\label{I:topologically (fp)-locally closed embedding}
if it is $(P_{\red})$-representable (see \re{redP}(b)), where $(P)$ is the class of (fp)-closed/locally closed embeddings.

(b) We say that an $\infty$-substack $\cY\subset\cX$ is a {\em topologically (fp)-closed/locally closed}, if the inclusion $\iota:\cY\to\cX$ is a topologically (fp)-closed/locally closed embedding.

(c) As in the case of schemes of finite type (see \re{locdim}(d)), we say that a topologically fp-locally closed substack $\cY\subset\cX$ is of {\em pure codimension $d$}\label{I:pure2}, if the embedding $\iota:\cY\hra\cX$ is weakly equidimensional of relative dimension $-d$ (see \re{eqmor}).

\end{Emp}

\begin{Emp} \label{E:complopcl}
{\bf The case of open and closed embeddings.}
(a) Notice that if $\cX$ is a scheme $X$ and $\cU$ is an open subscheme $U$, then the reduced complement
$(\cX\sm\cU)_{\red}$ (\re{compl}(b)) is the reduced closed subsheme $(X\sm U)_{\red}\subset X$. Therefore it follows
from \re{compl}(c) that if $\cU\subset\cX$ is an (fp)-open $\infty$-substack, then
$\cX\sm\cU\subset\cX$ is a topologically (fp)-closed substack.

(b) Conversely, if  $\cZ\subset\cX$ is a topologically (fp)-closed $\infty$-substack, then
$\cX\sm\cZ\subset\cX$ is an (fp)-open $\infty$-substack. Indeed, using \re{compl}(c), one reduces to the case
when $\cX=X$ is a scheme. Moreover, since $\cX\sm\cZ=\cX\sm\cZ_{\red}$, we can also assume that $\cZ\subset X$ is an (fp)-closed subscheme,
in which case the complement $\cX\sm\cZ$ is an (fp)-open subscheme.

(c) It follows from \re{compl}(c) and the scheme case that we always have an equality
$\cU=\cX\sm(\cX\sm\cU)$ when $\cU$ is open, and $\cZ_{\red}=(\cX\sm(\cX\sm\cZ))_{\red}$ when $\cZ\subset\cX$ is topologically closed.
\end{Emp}

\begin{Emp} \label{E:adapted}
{\bf Set-up.} (a) Let $\cX$ be an $\infty$-stack, and $\{\cX_{\al}\}_{\al\in \C{I}}$ a collection of non-empty topologically fp-locally closed reduced $\infty$-substacks of $\cX$ such that $ \cX_{\al}\cap\cX_{\beta}=\emptyset$ for all $\al\neq\beta$ in $\cI$.

(b) For every $\infty$-substack $\cX'\subset\cX$, we set $\cI_{\cX'}:=\{\al\in\cI\,|\,\cX_{\al}\subseteq\cX'\}$.\label{N:IX'}

(c) We say that $\cX'$ is {\em $\{\cX_{\al}\}_{\al}$-adapted}\label{I:adapted infty substack}, if for every $\al\in\cI\sm\cI_{\cX'}$, we have $\cX_{\al}\cap\cX'=\emptyset$, that is, $\cX_{\al}\subset\cX\sm\cX'$. Equivalently, $\cX'$ is $\{\cX_{\al}\}_{\al}$-adapted if and only if for every $\al\in\cI$ we have either $\cX_{\al}\subseteq\cX'$ or $\cX_{\al}\cap\cX'=\emptyset$.

(d) By definition, the class of $\{\cX_{\al}\}_{\al}$-adapted $\infty$-substacks is closed under arbitrary intersections and complements.
\end{Emp}

\begin{Emp} \label{E:consstr}
{\bf Constructible stratification}. In the situation of \re{adapted}(a), we say that $\{\cX_{\al}\}_{\al\in \C{I}}$ forms a

(a) {\em finite constructible stratification}\label{I:constructible stratification} of $\cX$, if $\cI$ is finite, and there exists
an ordering $\al_1<\ldots<\al_n$ of $\cI$ and an increasing sequence of fp-open substacks $\emptyset=\cX_0\subsetneq\cX_1\subsetneq\ldots\subsetneq\cX_n=\cX$ such that $\cX_{\al_i}\subseteq\cX_i\sm\cX_{i-1}$, and the embedding $\cX_{\al_i}\hra\cX_i\sm\cX_{i-1}$ is a topological equivalence
for all $i=1,\ldots,n$.

(b) {\em bounded constructible stratification} of $\cX$, if there is a presentation $\cX\simeq\colim_{U\in \C{J}}\cX_U$ of $\cX$ as a filtered colimit such that each $\cX_U$ is an fp-open $\{\cX_{\al}\}_{\al}$-adapted substack of $\cX$,
and $\{\cX_{\al}\}_{\al\in \C{I}_{\cX_U}}$ forms a finite constructible stratification of $\cX_U$.

(c) {\em constructible stratification} of $\cX$, if $\cX$ can be written as a filtered colimit $\cX\simeq\colim_{\la\in \La}\cX_{\la}$ such that each $\cX_{\la}$ is a topologically fp-closed $\{\cX_{\al}\}_{\al}$-adapted substack of $\cX$, and $\{\cX_{\al}\}_{\al\in \C{I}_{\cX_{\la}}}$ forms a bounded constructible stratification of $\cX_{\la}$.
\end{Emp}


\begin{Emp} \label{E:remcons}
{\bf Finite case.} Assume that we are in the situation of \re{adapted}(a).

(a) A collection $\{\cX_{\al}\}_{\al\in\cI}$ forms a finite constructible stratification of $\cX$ if and only if there exists $\beta\in\cI$ such that $\cX_{\beta}\subset\cX$ is topologically
fp-closed, and $\{\cX_{\al}\}_{\al\in\cI\sm\beta}$ forms a finite constructible stratification of $\cX\sm\cX_{\beta}$.

Indeed, if such a $\beta$ exists, then the embedding $\cX_{\beta}\hra\cX\sm(\cX\sm\cX_{\beta})$ is a topological equivalence (see \re{complopcl}(c)). Conversely, in the situation of \re{consstr}(a), we have  $\cX_{n-1}=\cX\sm\cX_{\al_n}$, so $\beta=\al_n$ satisfies the required property.

(b) Assume that $\cZ\subset\cX$ is a topologically fp-closed $\{\cX_{\al}\}_{\al}$-adapted substack such that
$\{\cX_{\al}\}_{\al\in \C{I}_{\cZ}}$ (resp. $\{\cX_{\al}\}_{\al\in \C{I}_{\cX\sm\cZ}}$) forms a finite constructible stratification of $\cZ$
(resp. $\cX\sm\cZ$). Then  $\{\cX_{\al}\}_{\al\in \C{I}}$ forms a finite constructible stratification of $\cX$.
Indeed, this follows from (a) by induction on the cardinality of $\cI_{\cZ}$.
\end{Emp}

The following lemma summarizes simple properties of the notions we introduced.

\begin{Lem} \label{L:propcons}

Assume that $\{\cX_{\al}\}_{\al\in \C{I}}$ forms a (finite/bounded) constructible stratification of $\cX$.

(a) For a morphism $f:\cY\to\cX$ of $\infty$-stacks, the collection $\{f^{-1}(\cX_{\al})_{\red}\}_{\al\in\cI,f^{-1}(\cX_{\al})\neq\emptyset}$ forms a (finite/bounded) constructible stratification of $\cY$.

(b) If an $\infty$-substack $\cX'$ of $\cX$ is $\{\cX_{\al}\}_{\al}$-adapted, then $\{\cX_{\al}\}_{\al\in\cI_{\cX'}}$ forms a (finite/bounded) constructible stratification of $\cX'$.

(c) If $\{\cX_{\al,\beta}\}_{\beta\in\C{J}_{\al}}$ forms a finite constructible stratification of $\cX_{\al}$ for all $\al\in\cI$, then the collection
$\{\cX_{\al,\beta}\}_{\al\in\cI,\beta\in\C{J}_{\al}}$ forms a (finite/bounded) constructible stratification of $\cX$.

\end{Lem}

\begin{proof}
All assertions easily reduce to the case of finite stratifications, which we assume from now on. In this case, (a) follows from observation $f^{-1}(\cX_i\sm \cX_{i-1})\simeq f^{-1}(\cX_i)\sm f^{-1}(\cX_{i-1})$ (see \re{compl}(c)), (b) follows immediately from (a), and
(c) follows by induction on $|\cI|$, using \re{remcons}(b).
\end{proof}

\begin{Emp} \label{E:stratified}
{\bf Placidly stratified $\infty$-stacks.}
(a) We say that an $\infty$-stack $\cX$ is {\em $\cI$-stratified} (or simply {\em stratified}),\label{I:stratified infty stack} if it is equipped with a constructible stratification $\{\cX_{\al}\}_{\al\in\cI}$;

(b) We say that a stratified $\infty$-stack $(\cX,\{\cX_{\al}\}_{\al})$ is {\em placidly stratified},\label{I:placidly stratified infty stack} if every
$\cX_{\al}$ is a placid $\infty$-stack.
\end{Emp}

\begin{Emp} \label{E:semismall}
{\bf Semi-small maps.} (a) Let $\cY$ be a placidly stratified $\infty$-stack with a bounded constructible stratification $\{\cY_{\al}\}_{\al}$, and let $f:\cX\to\cY$ be a morphism of $\infty$-stacks. For every $\al\in\cI$, we set $\cX_{\al}:=f^{-1}(\cY_{\al})_{\red}$, and let
$f_{\al}:\cX_{\al}\to\cY_{\al}$ be the restriction $f$. 

(b) Assume that $\cX$ is placid, each $\cX_{\al}\subset\cX$ is of pure codimension $b_{\al}$ (see \re{fp-locally closed}(c)), and each $f_{\al}:\cX_{\al}\to\cY_{\al}$ is locally fp-representable and equidimensional of relative dimension $\dt_{\al}$ (see \re{eqmor}).

(c) We say that $f$ is {\em semi-small},\label{I:semismall morphism} if for every $\al\in \cI$ we have an inequality $\dt_{\al}\leq b_{\al}$.
Moreover, we say that a stratum $\cY_{\al}$ is {\em $f$-relevant}, if the equality  $\dt_{\al}=b_{\al}$ holds.

(d) Let $\cU\subset\cY$ be a $\{\cY_{\al}\}_{\al}$-adapted fp-open substack. We say that a semi-small map $f$ is {\em $\cU$-small}, \label{I:small morphism} if for every $\al\in\cI\sm\cI_{\cU}$, we have a strict inequality $\dt_{\al}< b_{\al}$, that is, all $f$-relevant strata are contained in $\cU$.
\end{Emp}

\begin{Emp} \label{E:remsemismall}
{\bf Relation to the classical notion.} Assume that $f:X\to Y$ is a dominant morphism of irreducible schemes of finite type over $k$, all of whose fibers are equidimensional.

(a) One can show that there exists a constructible stratification $Y_{\al}$ of $Y$ such that each $Y_{\al}$ is irreducible
and each $f_{\al}:X_{\al}\to Y_{\al}$ is equidimensional of constant relative dimension $\dt_{\al}$.
Then each $Y_{\al}\subset Y$ is of pure codimension, each $X_{\al}$ is equidimensional, thus each $X_{\al}\subset X$ is of pure codimension $b_{\al}$ in $X$. In other words, $f$ with stratification $\{Y_{\al}\}_{\al}$ satisfies the assumptions of \re{semismall}.

(b) Recall that classically a morphism $f$ is called {\em semi-small}, if we have
$\codim_Y(Y_{\al})\geq 2\dt_{\al}$ for all $\al$, and we claim that this happens if and only if we have
$\dt_{\al}\leq b_{\al}$ for all $\al$.

Indeed, it suffices to show that
\begin{equation} \label{Eq:codimss}
\codim_Y(Y_{\al})=b_{\al}+\dt_{\al}\text{ for all }\al.
\end{equation}

Let $Y_{\al_0}\subset Y$ be the open stratum. Then each of inequalities $\dt_{\al_0}\leq b_{\al_0}$ and $\codim_Y(Y_{\al_0})\geq 2\dt_{\al_0}$ implies that $\dt_{\al_0}=b_{\al_0}=0$. Thus $f$ is generically finite, hence $\dim X=\dim Y$. So equality \form{codimss} follows from equalities $\dim X_{\al}=\dim Y_{\al}+\dt_{\al}$ and $\dim X=\dim X_{\al}+b_{\al}$.

(c) Equality \form{codimss} also implies that a morphism $f$ is {\em small} in the classical sense if and only if it is $Y_{\al_0}$-small in the sense of \re{semismall}(d).

(d) Note that the assumptions of \re{semismall}(b) are never satisfied when not all fibers of $f:X\to Y$ are equidimensional. On the other hand, it is possible to weaken the assumptions of \re{semismall}(b) in order to include a general case of semi-small maps between irreducible schemes of finite type. 
\end{Emp}

\part{The Affine Springer fibration}

\section{The Goresky--Kottwitz--MacPherson stratification}

\subsection{Arc and loop spaces} We set $\co=k[[t]]$\label{N:O}, $F=k((t))$\label{N:F}.
We recall some basic definitions on arc and loop spaces (compare \cite[$\S$ 2-3]{ME}).

\begin{Emp} \label{E:arcloop}
{\bf Arc spaces.} \label{I:arc space}
(a) For every $n\in\B{N}$, we have a functor $A\mapsto A[t]/(t^{n+1})=A[[t]]/(t^{n+1})$ from $k$-algebras to $\C{O}$-algebras.
Thus, to every $\infty$-stack $\cX$ over $\cO$ and $n\in\B{N}$, we associate an $\infty$-prestack $\cL^+_n(\cX)$\label{N:clpnx} by the formula
$\clp_{n}(\cX)(A):=\cX(A[t]/(t^{n+1}))$ for every $k$-algebra $A$. It is easy to see that $\clp_n(\cX)$ automatically satisfies the sheaf condition, that is, $\cL^+_n(\cX)$ is an $\infty$-stack.

(b) Note that the functors $A\mapsto A[t]/(t^{n+1})$ form a projective system, thus we can form the limit $\clp(\cX):=\lim_n\clp_n(\cX)$,\label{N:clpx} called the {\em arc-$\infty$-stack} of $\cX$. For every $k$-algebra $A$, we have a natural morphism $\cX(A[[t]])\to\clp(\cX)(A)$. Moreover, this morphism is clearly an isomorphism when $\cX$ is an affine scheme.

(c) We have an evaluation map  $\ev_{\cX}:\clp (\cX)\to\clp_0(\cX)=\cX$\label{N:evx}.
\end{Emp}

\begin{Emp}
{\bf Remark.} The map $\cX(A[[t]])\to\clp(\cX)(A)$ from \re{arcloop}(b) is an isomorphism when
$\cX$ is a qcqs algebraic space (see, for example, \cite[Corollary 1.2]{Bh}). On the other hand, this nontrivial
fact will not be used in this work (beyond the affine case).
\end{Emp}

\begin{Lem} \label{L:arcetale}
If $f:X\to Y$ is an \'etale map of schemes of finite type over $\C{O}$, then the commutative diagram
\begin{equation*}
\begin{CD}
\clp (X) @>>> \clp_{n+1}(X) @>>> \clp_{n}(X) @>>> X\\
@VVV @VVV @VVV @VVV\\
\clp (Y) @>>> \clp_{n+1}(Y) @>>> \clp_{n}(Y) @>>> Y
\end{CD}
\end{equation*}
is Cartesian.  In particular, the induced map $\clp (X)\to \clp (Y)$ is \'etale and finitely presented.
\end{Lem}

\begin{proof}
See, for example, \cite[Lemma 2.9 and beginning of $\S$3]{ME}.
\end{proof}

\begin{Emp} \label{E:proparc}
{\bf Simple properties.} (a) Notice that we have a natural isomorphism
\[
\clp_n(\B{A}^1)\isom\B{A}^{n+1}:\sum_{i=0}^na_i t^i\mapsto(a_0,\ldots,a_n),
\]
which identifies $\clp(\B{A}^1)$ with $\B{A}^{\infty}=\Spec k[\{a_i\}_{i\in\B{N}}]\simeq\lim_n\B{A}^n$. In particular, $\clp(\B{A}^1)$ is a strongly pro-smooth affine scheme.

(b) By definition, the arc functor $\cX\mapsto\clp(\cX)$ commutes with all limits. Thus, it follows from (a) that  $\cL^+(\B{A}^n)\simeq\cL^+(\B{A})^n$ is a  strongly pro-smooth affine scheme as well.

(c) Note that if $f:\cX\to\cY$ is a closed embedding of $\infty$-stacks over $\C{O}$, then the induced maps $\clp_n(\cX)\to\clp_n(\cY)$ and $\clp(\cX)\to\clp(\cY)$ are closed embeddings as well. In particular, using (b), we deduce that for every affine scheme $X$ over $\cO$, the arc space $\clp(X)$ is an affine scheme over $k$.

(d) By (c) and \rl{arcetale}, we conclude that if $X$ is a scheme (resp. algebraic space) over $\cO$, then $\clp(X)$ is a scheme (resp. algebraic space) as well. Namely, if $\{U_{\al}\}_{\al}$ forms an open (resp. \'etale) affine covering of $X$, then $\{\clp(U_{\al})\}_{\al}$ forms an open (resp. \'etale) affine covering of $\clp(X)$.

(e) If $X$ is a smooth scheme of finite type over $\cO$, then the schemes $\clp_{n}(X)$ are smooth, and transition maps
$\clp_{n+1}(X)\rightarrow\clp_{n}(X)$ are smooth, affine and surjective. Indeed, $X$ is \'etale locally isomorphic to $\B{A}^n$, so the assertion
follows from the case of $\B{A}^n$ (see (b)) and \rl{arcetale}. Therefore the arc space $\clp(X)$ is strongly pro-smooth.

(f) Let $H$ be a smooth group scheme over $\C{O}$, and let $X\to Y$ be an $H$-torsor over $\cO$. Then the induced morphism
$\clp(X)\to \clp(Y)$ is an $\clp(H)$-torsor (see \re{tors}). Indeed, since $H$ is smooth, there exists an \'etale covering $Y'\to Y$ such that $X\times_Y Y'\to Y'$ is a trivial $H$-torsor. Then  $\clp(X)\times_{\clp(Y)}\clp(Y')\to\clp(Y')$ is a trivial $\clp(H)$-torsor (by (b)), hence $\clp(X)\to \clp(Y)$ is an $\clp(H)$-torsor (by \rl{arcetale}). In particular, the induced map $[\clp(X)/\clp(H)]\to \clp(Y)$ is an isomorphism (see \re{tors}(c)).
\end{Emp}

\begin{Emp} \label{E:loop}
{\bf Loop spaces.}\label{I:loop space}
(a) To every $\infty$-stack $\cX$ over $F$ and $n\in\B{N}$, we associate an $\infty$-prestack $\cL\cX$\label{N:lx} by the formula
$\cL\cX(A):=\cX(A((t)))$ for every $k$-algebra $A$. The loop functor $\cX\mapsto\cL\cX$ commutes with all limits.

(b) Notice that  we have a natural isomorphism of sets $A((t))=\colim (A[[t]]\overset{t}{\to}A[[t]] \overset{t}{\to}\ldots)$, which induces an isomorphism
$\cL(\B{A}^1)\simeq\colim(\clp(\B{A}^1)\overset{t}{\to} \clp(\B{A}^1)\overset{t}{\to}\ldots)$, where all transition maps are fp-closed embeddings.
Thus $\cL(\B{A}^1)$ is represented by an ind-affine scheme. Therefore  $\cL(\B{A}^n)\cong\cL(\B{A}^1)^n$ is represented by an ind-affine scheme as well.

(c) Note that for every closed embedding $\cX\to\cY$ of $\infty$-stacks over $\C{O}$, the induced morphism $\cL\cX\to\cL\cY$ is a closed embedding as well.
Using (b), we conclude that for every affine scheme $X$ of finite type over  $F$, the loop space $\cL X$ is an ind-affine ind-scheme.
In particular, $\cL X$ is an $\infty$-stack, that is, satisfies the sheaf condition.
\end{Emp}

\begin{Emp}
{\bf Remarks.} (a) Though the loop functor $\cL\cX$ can be defined for an arbitrary $\infty$-prestack, it does not seem to be a natural object to consider, when $\cX$ is not an affine scheme. In particular, it fails to be an $\infty$-stack in general.

(b) A theorem of Drinfeld (see \cite[Theorem 6.3]{Dr}) asserts that for every smooth affine scheme $X$ of finite type over  $F$, the ind-scheme $\cL X$ is ind-placid. We are not going to use this fact in this work.
\end{Emp}

\begin{Emp} \label{E:assumption}
{\bf Set up.} Until end of this subsection, we fix $h\in\B{N}$\label{N:m} prime to the characteristic of $k$ and choose a primitive $h$-th root of unity $\xi\in k$. We set $\C{O}':=k[[t^{1/h}]]$\label{N:O'} and $F':=k((t^{1/h}))$.\label{N:F'}
\end{Emp}

\begin{Emp} \label{E:twistedver}
{\bf Twisting}. Assume that we are in the situation of \re{assumption}.

(a) For an affine scheme of finite type $X$ over $\cO$ (resp. $F$), we denote by $X_{\C{O}'}:=X\otimes_{\C{O}}\C{O}'$ (resp. $X_{F'}:=X\otimes_F F'$) the extension of scalars of $X$, and let $X':=R_{\C{O}'/\C{O}}(X_{\C{O}'})$ (resp. $X':=R_{F'/F}(X_{F'})$) be the Weil restriction of $X_{\C{O}'}$ (resp. $X_{F'}$). In other words, $X'$ is a scheme over $\cO$ (resp. $F$) defined by the rule $X'(A)=X(A\otimes_{\cO}\cO')$ (resp. $X'(A)=X(A\otimes_{F}F')$) for every $\cO$-algebra (resp. $F$-algebra) $A$. Arguing as is \re{proparc}(c)  one shows that $X'$ is an affine scheme of finite type over $\cO$ (resp. $F$). Moreover, if $X$ is smooth (thus formally smooth) over $\cO$ (resp. $F$), then so is $X'$ (compare \cite[Section 16]{GKM}).

(b) Let $\si\in \Aut (\C{O}'/\C{O})=\Aut(F'/F)$ be the automorphism $\si(t^{1/h}):=\xi t^{1/h}$. Then $\si$ induces an automorphism of $X'$,
and the scheme of fixed points $(X')^{\si}$ is $X$.

(c) For an automorphism $\phi$ of $X$ over $\cO$ (resp. $F$), we denote by $X_{\phi}:=(X')^{\phi\si}\subset X'$ the scheme of fixed points
of $\phi\si$. In particular, we have $X_{\phi}=X$, if $\phi$ is the identity.
Note that $X_{\phi}$ depends on $h$ and $\xi$, so we will sometimes denote it by
$X_{\phi}(\xi)$ or $X_{\phi}(\xi,h)$. Notice that if $\phi$ is of finite order, prime to the characteristic of $k$, then
the same is true for $\phi\si$. In particular each $X_{\phi}$ is smooth over $\cO$ (resp. $F$), if $X$ satisfies this property
(see \cite[Lemma 15.4.1]{GKM}).

(d) Let $W$ be a finite group acting on $X$ over $\C{O}$ (resp. $F$). Then  (by (c)) for every $w\in W$, we can consider the twist $X_w$ of $X$.
In particular, each $X_w$\label{N:xw} is smooth over $\C{O}$ (resp. $F$), if $X$ is smooth over $\C{O}$ (resp. $F$) and the exponent of $W$ divides $h$.
\end{Emp}

\begin{Emp} \label{E:twistloop}
{\bf Twisted arc/loop spaces}. Assume that we are in the situation of \re{twistedver}.

(a) For an affine scheme of finite type $X$ over $\cO$ (resp. $F$), we denote by  $\cL'^+(X)$\label{N:clp'x} (resp. $\cL'(X)$)\label{N:l'x} the functor
$\cL'^+(X)(A)=X(A[[t^{1/h}]])$ (resp. $\cL'(X)(A)=X(A((t^{1/h})))$). In particular, it satisfies all the properties that $\cL^+(X)$ (resp. $\cL(X)$) does. For example, $\cL'^+(X)$ is an affine scheme (resp. $\cL'(X)$ is an ind-affine scheme).
Notice also that we have a natural identification $\cL'^+(X)\simeq\cL^+(X')$ (resp. $\cL'(X)\simeq\cL(X')$).

(b) The automorphism $\si\in\Aut(\cO'/\cO)=\Aut(F'/F)$ induces an automorphism of   $\cL'^+(X)$ (resp. $\cL'(X)$), and we have
an identification   $\cL'^+(X)^{\si}=\cL^+(X)$ (resp. $\cL'(X)^{\si}=\cL(X)$).

(c) For an automorphism $\phi$ of $X$ over $\cO$ (resp. $F$), we have an identification
$\cL^+(X_{\phi})\simeq \cL'^+(X)^{\phi\si}$ (resp. $\cL(X_{\phi})\simeq \cL'(X)^{\phi\si}$).

(d) A morphism $\pi:X\to Y$ of affine schemes over $\cO$ (resp. $F$) induces a map $\cL'^+(X)\to \cL'^+(Y)$
(resp. $\cL'(X)\to \cL'(Y)$) of the corresponding arc (resp. loop) spaces, and we set
\[
\cL'^+(X)_{\cL^+(Y)}:=\cL'^+(X)\times_{\cL'^+(Y)}\cL^+(Y)\text{ (resp. }\cL'(X)_{\cL Y}:=\cL'(X)\times_{\cL'(Y)}\cL(Y).)
\]
\end{Emp}

\begin{Lem} \label{L:twist}
Let $W$ be a finite group, and let $\pi:X\to Y$ be a $W$-torsor of affine schemes over $F$. Then, in the notation of \re{twistloop}(d) and  \re{twistedver}(d), we have an equality
\[
\cL'(X)_{\cL Y}=\sqcup_{w\in W}\cL(X_w)\subset \cL'(X).
\]
\end{Lem}

\begin{proof}
Since $\cL Y=\cL'(Y)^{\si}$, and $\cL'$ commutes with fiber products, the closed ind-subscheme  $\cL'(X)_{\cL(Y)}\subset \cL'(X)$
is identified with the locus of all $x\in \cL'(X)$ such that $(x,\si(x))$ lies in $\cL'(X)\times_{\cL'(Y)}\cL'(X)\subset \cL'(X)\times \cL'(X)$.

Since  $\pi:X\to Y$ is a $W$-torsor, the action map $W\times X\to X\times_Y X:(w,x)\to (x,wx)$ is an isomorphism.
Since $\cL'$ commutes with limits, and $\cL'(W)=W$, the induced morphism  $W\times \cL'(X)\to \cL'(X)\times_{\cL'(Y)}\cL'(X)$ is an
isomorphism as well.

Therefore the ind-subscheme $\cL'(X)_{\cL Y}\subset \cL'(X)$ decomposes as a disjoint union of closed ind-subschemes
$\{x\in \cL'(X)\,|\,(x,\si(x))=(x,w^{-1}(x))\}$, that is, of $\cL'(X)^{\si w}=\cL(X_w)$.
\end{proof}

The following version of \rl{arcetale} for loop spaces will be proven in \rsec{pfrlBC}.

\begin{Prop} \label{P:BC}
In the situation of \rl{twist}, assume that the exponent of $W$ divides $h$ (see \re{assumption}). Then the induced map $\phi:\cL'(X)_{\cL Y}\to \cL(Y)$ is a $W$-torsor.
\end{Prop}

\begin{Emp}
{\bf Remark.} Notice that it follows from \rp{BC} that if  the exponent of $W$ divides $h$, then the ind-subscheme $\cL'(X)_{\cL(Y)}\subset\cL'(X)$ does not depend on $h$. It can be seen directly. Namely,

(a) In the situation of \re{twistedver}(c), for every $r\in\B{N}$ prime to $h$, the power $\xi^r$ is another primitive $h$-th root of unity, and we have an equality  $X_{\phi}(\xi^r)=X_{\phi^r}(\xi)$. Moreover, for every $n\in\B{N}$ and every primitive $hn$-th root of unity $\xi$, its power $\xi^n$ is the $n$-root of unity, and we have an equality $X_{\phi}(\xi,hn)=X_{\phi}(\xi,h)$.

(b) It follows from (a) that the collection  $\{X_w\}_{w\in W}$ does not depend neither on $h$ nor on $\xi$. Therefore in the situation of \rl{twist}, the ind-subscheme $\cL'(X)_{\cL(Y)}\subset\cL'(X)$ does not depend on $h$, as claimed.
\end{Emp}

\subsection{Stratification by valuation}

\begin{Emp} \label{E:strval}
{\bf The $\B{A}^1$-case.}
Recall that for every $k$-algebra $A$, the set  $\clp_n(\B{A}^1)(A)$ classifies power series $\sum_{i=0}^{\infty}a_it^i\in A[[t]]$.

(a) For every $n\in\B{N}$, let $\clp(\B{A}^1)_{\geq n}\subset\clp(\B{A}^1)$ be the fp-closed subscheme of zeros of $a_0,\ldots, a_{n-1}$. Explicitly, $\clp(\B{A}^1)_{\geq n}(A)$ classifies power series
$\sum_{i=0}^{\infty}a_it^i\in A[[t]]$ with $a_0=\ldots=a_{n-1}=0$.

(b) We set $\clp(\B{A}^1)_{\leq n}:=\clp(\B{A}^1)\sm \clp(\B{A}^1)_{\geq n+1}$. Then $\clp(\B{A}^1)_{\leq n}\subset\clp(\B{A}^1)$ is an fp-open subscheme, and $\{\clp(\B{A}^1)_{\leq n}\}_{n\geq 0}$
gives an fp-open covering of $\clp(\B{A}^1)_{\bullet}:=\clp(\B{A}^1)\sm\{0\}$.\label{N:a1bullet} Explicitly, $\clp(\B{A}^1)_{\leq n}(A)$ classifies power series $\sum_{i=0}^{\infty}a_it^i\in A[[t]]$ such that the ideal $(a_0,\ldots,a_{n})\subset A$ is $A$.

(c) For every $n\in\B{N}$, we set $\clp(\B{A}^1)_n:=\clp(\B{A}^1)_{\geq n}\cap \clp(\B{A}^1)_{\leq n}$. Then $\clp(\B{A}^1)_n\subset \clp(\B{A}^1)_{\geq n}$ is the basic open set $a_n\neq 0$. In particular, it is an affine fp-locally closed subscheme of $\clp(\B{A}^1)$. Explicitly, $\clp(\B{A}^1)_{n}(A)$ classifies power series $\sum_{i=0}^{\infty}a_it^i\in A[[t]]$ with $a_0=\ldots =a_{n-1}=0$ and $a_n\in A^{\times}$.
Moreover, $\{\clp(\B{A}^1)_{n}\}$ forms a bounded constructible stratification (see \re{consstr}) of $\clp(\B{A}^1)_{\bullet}$.

(d) Notice that the open embedding $\B{G}_m\hra\B{A}^1$ induces an isomorphism $\clp(\B{G}_m)\isom \clp(\B{A}^1)_0$ and a monomorphism  $\cL(\B{G}_m)\hra\cL(\B{A}^1)$. Furthermore, the map $f\mapsto (f,f^{-1}):\cL(\B{G}_m)\hra\cL(\B{A}^2)$ identifies $\cL(\B{G}_m)$ with a closed ind-subscheme of $\cL(\B{A}^2)$ classifying
$(f,g)\in A((t))$ such that $fg=1$.

(e) For $n\in\B{N}$, the composition
$\clp(\B{A}^1)_{n}\hra \clp(\B{A}^1)\hra\cL(\B{A}^1)$ factors through $\cL(\B{G}_m)\hra \cL(\B{A}^1)$. Moreover, the induced map $\clp(\B{A}^1)_{n}\hra\cL(\B{G}_m)$ is an fp-closed embedding.
Indeed, its image classifies pairs $f=\sum_{i}a_it^i, g=\sum_{j} b_jt^j\in A((t))$ such that $fg=1$, $a_i=0$ for all $i<n$ and $b_j=0$ for all $j<-n$.
\end{Emp}

\begin{Emp} \label{E:strval2}
{\bf The general case.} Let $X$ be an affine scheme over $\C{O}$,  let $f\in\C{O}[X]$ be a regular function, and $n\in\B{N}$.

(a) Then $f$ can be viewed as a morphism $f:X\to\B{A}^1$ of affine schemes over $\C{O}$, hence induces a morphism $f:\clp(X)\to \clp(\B{A}^1)$, and we denote by $\clp(X)_{(f;\geq n)}$, $\clp(X)_{(f;\leq n)}$ and $\clp(X)_{(f;n)}$\label{N:clpxfn} the reduced schematic preimages $f^{-1}(\clp(\B{A}^1)_{\geq n})_{\red}$, $f^{-1}(\clp(\B{A}^1)_{\leq n})_{\red}$ and $f^{-1}(\clp(\B{A}^1)_{n})_{\red}$, respectively.
Moreover, if $g\in \C{O}[X]$ is another regular function we can form the reduced schematic intersection $\clp(X)_{(f;n),(g;m)}$ of $\clp(X)_{(f;n)}$ and $\clp(X)_{(g;m)}$.

(b) Note that every $\clp(X)_{(f;\geq n)}\subset \clp(X)$ is a reduction of an fp-closed subscheme, and $\clp(X)_{(f;n)}\subset \clp(X)_{(f;\geq n)}$ is a basic open subscheme given by equation $f^*(a_n)\neq 0$. In particular, both $\clp(X)_{(f;\geq n)}$ and $\clp(X)_{(f;n)}$ are affine.

(c) By \re{strval}(d) and \rl{propcons}(a), we conclude that $\{\clp(X)_{(f;n)}\}$ forms a bounded
constructible stratification of $\clp(X)_{f\neq 0}:=f^{-1}(\clp(\B{A}^1)_{\bullet})\subset\clp(X)$.\label{N:clpxfneq0}

(d) Let $X_f\subset X$ be the basic open subset $f\neq 0$. Using \re{strval}(e), we get an isomorphism $\clp(X_f)_{\red}\isom \clp(X)_{(f;0)}$ and a topologically fp-closed embedding $\clp(X)_{(f;n)}\hra \cL(X_f)$.
\end{Emp}

\begin{Lem} \label{L:red}
In the situation of \re{strval2}, assume that $f$ decomposes as a product $f=\prod_{i=1}^k f_i$. Then $\clp(X)_{(f;n)}$ decomposes as a disjoint union $$\sqcup_{m_1,\ldots, m_k, \sum m_i=n}  \clp(X)_{(f_1;m_1),\ldots, (f_k;m_k)}.$$
\end{Lem}

\begin{proof}
By induction, we reduce to the case $k=2$. Moreover, using morphism
$(f_1,f_2):X\to\B{A}^2$, we reduce to the case when $X$ is the affine space $\B{A}^2$ with coordinates $x,y$ and $f=xy$. It suffices to show that the stratum $\clp(\B{A}^2)_{(xy;n)}$ decomposes as a disjoint union
$\sqcup_{m=0}^n  \clp(\B{A}^2)_{(x;m),(y;n-m)}$, which is straightforward.
%
\end{proof}

\begin{Emp}
{\bf Remark.} The assertion of the lemma is false, if our strata are not assumed to be reduced.
\end{Emp}

\begin{Emp} \label{E:ext}
{\bf Extension of scalars.} (a) Applying the construction of \re{strval} to the situation of \re{twistloop}, we get a stratum
$\cL'^+(\B{A}^1)_n\subset \cL'^+(\B{A}^1)$ for each $n\in\frac{1}{h}\B{N}$. Note that the intersection $\cL^+(\B{A}^1)\cap \cL'^+(\B{A}^1)_n$ is empty, if $n\notin\B{N}$, and equals $\cL^+(\B{A}^1)_n$, if $n\in\B{N}$.

(b) Each $f\in\cO[X]$ defines a morphism $f:\cL'^+(X)\to \cL'^+(\B{A}^1)$, thus defines a stratum
$\cL'^+(X)_{(f,n)}:=f^{-1}(\cL'^+(\B{A}^1)_n)$ for every $n\in\frac{1}{h}\B{N}$.
\end{Emp}

The following result is a version of \rl{twist} and \rp{BC} for arc spaces.

\begin{Lem} \label{L:BC2}
In the situation of \re{twistedver}, assume $\pi:X\to Y$ is a finite morphism of affine schemes over $\cO$, and group $W$ acts on $X$ over $Y$. Let
$f\in\cO[Y]$ be such that the induced map $(X_f)_F\to (Y_f)_F$ is a $W$-torsor, where we identify $f\in\cO[Y]$ with $\pi^*(f)\in\cO[X]$.
For every $n\in\B{N}$,

(a) the induced morphism $\phi:(\cL'^+(X)_{\cL^+(Y)})_{(f,n)}\to \cL^+(Y)_{(f,n)}$ is a $W$-torsor;

(b) we have an equality $(\cL'^+(X)_{\cL^+(Y)})_{(f,n)}=\sqcup_{w\in W}\cL^+(X_w)_{(f,n)}\subset\cL'^+(X)$.
\end{Lem}

\begin{proof}
(a) Applying  \rp{BC} to the restriction $(X_f)_F\to (Y_f)_F$, and taking pullback with respect to
a closed embedding $\cL^+(Y)_{(f,n)}\hra \cL(Y_f)$, we conclude that the induced map
\begin{equation} \label{Eq:wtorsor}
\cL'(X_f)\times_{\cL'(Y_f)}\cL^+(Y)_{(f,n)}\to \cL^+(Y)_{(f,n)}
\end{equation}
is a $W$-torsor. In particular, the LHS of \form{wtorsor} is a reduced affine scheme.
Next we observe that we have a natural closed embedding $\cL'^+(X)_{(f,n)}\hra \cL'(X_f)$, which induces a closed embedding
\[
(\cL'^+(X)_{\cL^+(Y)})_{(f,n)}\hra\cL'(X_f)\times_{\cL'(Y_f)}\cL^+(Y)_{(f,n)}
\]
of reduced affine schemes. Therefore in order to show that the latter map is an isomorphism, it suffices to show that it is
bijective on geometric points.

Note that $(\cL'^+(X)_{\cL^+(Y)})_{(f,n)}\subset\cL'^+(X)$ is the reduction of the preimage of $\cL^+(Y)_{(f,n)}\subset \cL'^+(Y)$, under the
map $\pi':\cL'^+(X)\to\cL'^+(Y)$, induced by $\pi$. Since all geometric fibers of \form{wtorsor} are $W$-torsors, while the map $\pi':\cL'^+(X)\to\cL'^+(Y)$ is $W$-equivariant, it suffices to show for every geometric point $y\in \cL^+(Y)_{(f,n)}(K)\subset\cL'^+(Y)(K)$, the fiber $(\pi')^{-1}(y)\subset \cL'^+(X)(K)$ is non-empty.

Since the map \form{wtorsor} is surjective on $K$-points, there exists a point $x\in\cL'(X_f)(K)\subset\cL'(X)(K)$ such $\pi(x)=y$.
Note that $y$ and $x$ can be viewed as points of $Y(K[[t^{1/h}]])$ and  $X(K((t^{1/h})))$, respectively.
Since $\pi:X\to Y$ is finite, thus proper, it follows from the valuative criterion of properness that
$x\in X(K[[t^{1/h}]])=\cL'^+(X)(K)$.

(b) Applying  \rl{twist} to the restriction $(X_f)_F\to (Y_f)_F$, and taking pullback with respect to
a closed embedding $\cL^+(Y)_{(f,n)}\hra \cL(Y_f)$, we get an equality
\begin{equation} \label{Eq:wtorsor2}
\cL'(X_f)\times_{\cL'(Y_f)}\cL^+(Y)_{(f,n)}=\sqcup_{w\in W}\cL((X_f)_w)\times_{\cL(Y_f)}\cL^+(Y)_{(f,n)}.
\end{equation}
Moreover, as it is proven in (a), the LHS of \form{wtorsor2} is canonically identified with $(\cL'^+(X)_{\cL Y})_{(f,n)}$. Therefore it suffices to show that the natural morphism
\[
\cL^+(X_w)_{(f,n)}\to \cL((X_f)_w)\times_{\cL(Y_f)}\cL^+(Y)_{(f,n)}
\]
is an isomorphism.

By the proven above, the RHS of \form{wtorsor2} is reduced, so it remains to show that it is contained in  $\cL^+(X_w)=\cL^+(X')^{w\si}$.
But this follows from the fact that it is contained in $\cL((X_f)_w)=\cL(X'_f)^{w\si}$ (by definition) and in $\cL'^+(X)$ (by \form{wtorsor2} and (a)).
\end{proof}

\begin{Emp} \label{E:strvalsm}
{\bf The smooth case.} In the situation of \re{strval2}, let $X$ be smooth over $\C{O}$.

(a) In this case, the affine scheme $\clp(X)$ is strongly pro-smooth (see \re{proparc}(b)). Therefore
$\clp(X)_{(f;\leq n)}=f^{-1}(\clp(\B{A}^1)_{\leq n})$ is a pro-smooth scheme, while $\clp(X)_{(f,\geq n)}$, and
$\clp(X)_{(f,n)}$ are fp subschemes of $\clp(X)$ (by \rco{red}(b)). Furthermore, $\clp(X)_{f\neq 0}$ has an open covering $\clp(X)_{f\neq 0}=\cup_{n\geq 0}\clp(X)_{(f,\leq n)}$.

(b) By (a) and \rl{glplacid}, $\clp(X)_{(f;\leq n)}$,  $\clp(X)_{(f,\geq n)}$, and
$\clp(X)_{(f,n)}$ are affine schemes with placid presentations. By \re{strval2}(d), we have $\clp(X)_{(f;0)}\simeq\clp(X_f)$.
\end{Emp}


\subsection{Application to the Chevalley space}
In this subsection, we present the results of Goresky--Kottwitz--MacPherson \cite{GKM} in the spirit of previous subsections.


\begin{Emp} \label{E:basicnot}
{\bf Basic notation.}
(a) Let $G$\label{N:G} be a connected reductive group over $k$, $\kg$\label{N:g} the Lie algebra of $G$, and $g\mapsto\Ad_g$\label{N:adg} the adjoint action of $G$ on $\kg$.
Let $B\supset T$\label{N:B}\label{N:T} be a Borel group and a maximal torus of $G$, respectively,  $W$\label{N:W} its Weyl group, $X_{*}(T)$\label{N:X*T} the lattice of cocharacters, $\widetilde{W}:=X_{*}( T)\rtimes W$\label{N:Waff} the extended affine Weyl group, and
$R$\label{N:R} the set of roots of $G$, respectively. We also set $\kt:=\Lie(T)$,\label{N:t} $\kb:=\Lie(B)$,\label{N:b} $r:=\dim(\kt)$,\label{N:r} and assume that the characteristic of $k$ is prime to the order of $W$.


(b) Let $\kc:=\kt//W=\Spec(k[\kg]^G)$\label{N:c} be the Chevalley space of $\kg$. Then we have canonical projections  $\chi:\kg\rightarrow \kc$\label{N:chi} and $\pi:\kt\rightarrow\kc$\label{N:pi} (compare \cite[Theorem 1.1.1]{N}).
Recall that $\pi$ is finite, flat and surjective.

(c) Let $\mathfrak{D}:=\prod_{\alpha\in R}d\alpha$\label{N:mathfrakD} be the discriminant function. Then $\mathfrak{D}\in k[\kc]=k[\kt]^W$, and
the regular semisimple locus $\kc^{\rs}\subset\kc$\label{N:crs} is the complement of the locus of zeros of $\mathfrak{D}$.
We denote by  $\kg^{\rs}:=\chi^{-1}(\kc^{\rs})$\label{N:grs} and $\kt^{\rs}:=\pi^{-1}(\kc^{\rs})$\label{N:trs} the preimages of $\kc^{\rs}$.

(d) Note that the morphism $\chi:\kg\rightarrow \kc$ induces a morphism
$\chi:\clp(\kg)\rightarrow\clp(\kc)$ between arc spaces.

(e) Let $I:=\ev_{G}^{-1}(B)\subset\clp (G)$\label{N:I} be the Iwahori group scheme of $\cL G$, whose Lie algebra
is $\Lie(I)=\ev_{\kg}^{-1}(\kb)\subset\clp(\kg)$.
\end{Emp}

\begin{Emp} \label{E:stratoft}
{\bf Stratification of $\clp(\frak{t})$.} (a) Let $X:=\ft$, and $\kD\in k[\ft]$ be the discriminant function. Then, by \re{strval2}(c), we have a bounded stratification $\{\clp(\ft)_{(\kD;n)}\}_n$ of $\clp(\ft)_\bullet:=\clp(\ft)_{\kD\neq 0}$.\label{N:clpftbullet}

Since $\kD=\prod_{\al\in R}d\al$, it follows from  \rl{red} that each $\clp(\ft)_{(\kD;n)}$ decomposes as a disjoint union
$\clp(\ft)_{(\kD;n)}=\sqcup_\br \ft_\br$\label{N:ftbr}, where $\br$ runs over functions $\br:R\to\B{Z}_{\geq 0}$ such that $d_{\br}:=\sum_{a\in R}\br(\al)$\label{N:dbr} equals $n$.

(b) Explicitly, $\ft_{\br}$ classifies power series $\sum_{i\geq 0} x_i t^i$, where $x_i\in\ft$ for all $i$ such that $\al(x_i)=0$ for $0\leq i<\br(\al)$, and $\al(x_i)\neq 0$ for $i=\br(\al)$. In other words, $\ft_\br\subset\ft$ is given by finitely many equalities of linear functions, and finitely many inequalities. In particular, $\ft_\br\subset\clp(\ft)$ is a connected strongly pro-smooth locally closed affine subscheme.

(c) Note that the natural action of $W$ on $\ft$ induces a $W$-action on $\clp(\ft)_{(\kD,n)}$. Moreover, every $u\in W$
induces an isomorphism $\kt_{\br}\isom\kt_{u(\br)}$, where $u(\br)$ is defined by the rule $u(\br)(\al)=\br(u^{-1}(\al))$ and $u^{-1}(\al)(x)=\al(u(x))$ for all $x\in\ft$.
\end{Emp}

\begin{Emp} \label{E:twisted}
{\bf The twisted version}. Let $h$ be the order of $W$, and we use the notation of \re{twistedver}.

(a) By \re{stratoft}(a), the scheme $\cL'^+(\ft)_{\kD\neq 0}$ has a bounded constructible stratification $\{\cL'^+(\ft)_{(\kD;n)}\}_n$ with $n\in\frac{1}{h}\B{Z}_{\geq 0}$, and every $\cL'^+(\ft)_{(\kD;n)}$ decomposes as a disjoint union $\sqcup_\br \ft'_\br$, where $\br$ runs over functions $\br:R\to\frac{1}{h}\B{Z}_{\geq 0}$ such that $d_{\br}=n$.

(b) For each $w\in W$, define  $\ft_w$\label{N:tw} to be the scheme of fixed points of $w\si$ in $\ft'$ (compare \re{twistedver}(c)).
Then $\clp(\ft_w)$ is the scheme of fixed points of $w\si$ in $\cL'(\ft)$ (see \re{twistloop}). Explicitly, $\clp(\ft_w)$ classifies power series $\sum_{i=0}^{\infty} x_i t^{i/h}$ with $x_i\in\ft$ such that $w^{-1}(x_i)=\xi^i x_i$ for all $i$. Moreover, $\clp(\ft_w)_{(\kD;n)}$ is the scheme of fixed points of $w\si$ in $\cL^+(\ft)_{(\kD;n)}$. The decomposition $\cL'^+(\ft)_{(\kD;n)}=\sqcup_\br \ft'_\br$ from (a) is $\si$-invariant and induces a decomposition $\clp(\ft_w)_{(\kD;n)}=\sqcup_{\br\,|\,w(\br)=\br} \ft_{w,\br}$.\label{N:ftwbr}

(c) Note that $\cL'^+(\ft)\simeq \lim_n\cL'^+_n(\ft)$ is a pro-vector space, and the action of $w\si$ on $\cL'^+(\ft)$ comes from a
compatible system of linear actions on vector spaces $\cL'^+_n(\ft)$. Therefore the scheme of fixed points $\clp(\ft_{w})=\cL'^+(\ft)^{\lan w\si\ran}$
is a pro-vector space, thus it is connected and strongly pro-smooth.

Similarly, $\ft_{w,\br}$ is the scheme of fixed points $(\ft'_{\br})^{\lan w\si\ran}$, and $\ft'_{\br}$ is an fp-open subscheme
of a pro-vector space. Therefore  $\ft_{w,\br}$ is an fp-open subscheme of a pro-vector space as well, thus  it is connected and strongly pro-smooth.


(d) The $W$-action on $\ft$ induces $W$-actions on $\ft'$, $\cL'^+(\ft)_{(\kD;n)}$
and $\sqcup_{w\in W}\cL^+(\ft_w)_{(\kD;n)}\subset \cL'^+(\ft)_{(\kD;n)}$. Moreover, the $W$-action is compatible with the stratification of (b).  Namely, for every $u,w\in W$ and $\br:R\to\frac{1}{h}\B{Z}_{\geq 0}$, the $u$-action  induces an isomorphism $u:\ft_{w,\br}\isom \ft_{uwu^{-1},u(\br)}$ (compare \re{stratoft}(c)).
\end{Emp}

\begin{Emp} \label{E:strofc}
{\bf Stratification of $\clp(\fc)$.}
(a) Recall that $\pi:\ft\to\fc$ is a finite $W$-equivariant morphism, whose restriction $\ft_{\kD}\to\fc_{\kD}$ is a
$W$-torsor. Therefore it follows from a combination of \rl{BC2} (a) and (b) that for every $n\in\B{N}$, the induced morphism
\[
\pi_{\kD,n}:\sqcup_{w\in W}\clp(\ft_w)_{(\kD,n)}\to \clp(\fc)_{(\kD,n)}
\]
is a $W$-torsor. In particular, $\pi_{\kD,n}$ is finite and \'etale.

(b) Recall that for every $w\in W$ we have a decomposition $\clp(\ft_w)_{(\kD;n)}=\sqcup\ft_{w,\br}$, taken over all $(w,\br)$ such that $d_{\br}=n$ (see \re{twisted}(b)).
Therefore for all pairs $(w,\br)$, the induced morphism $\pi_{w,\br}:\ft_{w,\br}\to \clp(\fc)_{(\kD;d_{\br})}$ is finite \'etale.

(c) Since $\ft_{w,\br}$ is connected, $\pi_{w,\br}$ induces a finite \'etale covering $\pi_{w,\br}:\ft_{w,\br}\to\fc_{w,\br}$ for a certain connected component  $\fc_{w,\br}$\label{N:fcwbr} of
$\clp(\fc)_{(\kD;d_{\br})}$.

(d) Let $W_{w,\br}$\label{N:Wwbr} be the stabilizer of $(w,\br)$ in $W$ via the action $u(w,\br):=(uwu^{-1},u(\br))$. Then
$u\in W$ induces an isomorphism $\ft_{w,\br}\isom\ft_{u(w,\br)}$ (see \re{twisted}(d)).
Since $\pi_{\kD,n}$ is a $W$-torsor, we conclude that for every $u\in W$ we have $\fc_{u(w,\br)}=\fc_{w,\br}$, the map
$\pi_{w,\br}:\ft_{w,\br}\to\fc_{w,\br}$ is a $W_{w,\br}$-torsor, and $\clp(\fc)_{(\kD;n)}$ decomposes as a disjoint
union $\clp(\fc)_{(\kD;n)}=\sqcup\fc_{w,\br}$, taken over all representatives of $W$-orbits of pairs $(w,\br)$ such that
$d_{\br}=n$.

(e)  Since  $\fc_{w,\br}$ is a connected component of $\clp(\fc)_{(\kD;n)}$, we conclude that
$\fc_{w,\br}$ is a locally fp-closed affine subscheme of $\clp(\fc)$. In particular,
$\fc_{w,\br}$ is a connected  affine scheme admitting a placid presentation.

(f) Since $\pi_{w,\br}:\ft_{w,\br}\to\fc_{w,\br}$ is finite \'etale covering (by (d)), $\ft_{w,\br}$ is strongly pro-smooth (by \re{twisted}(c)), while $\fc_{w,\br}$ is a connected affine scheme admitting a placid presentation (by (e)), we conclude from \rco{prosmooth} that $\fc_{w,\br}$ is strongly pro-smooth.
\end{Emp}

\begin{Emp}
{\bf Remark.} We don't know whether the closure of a stratum $\kc_{w,\br}$ is a union of strata.
\end{Emp}

\subsection{Codimension of strata}

\begin{Emp} \label{E:codim}
{\bf Notation.}
(a) Recall that $\ft_{w,\br}\subset \clp(\ft_w)$ and $\fc_{w,\br}\subset \clp(\fc)$ are strongly pro-smooth
locally fp-closed subschemes (see \re{twisted}(c) and \re{strofc}(c),(f)). Hence they are of pure codimension (see \rl{weakequid}), so  we can consider codimensions  $a_{w,\br}:=\codim_{\clp(\ft_w)}(\ft_{w,\br})$\label{N:awbr} and $b_{w,\br}:=\codim_{\clp(\fc)}(\fc_{w,\br})$.\label{N:bwbr}

(b) Recall that $r$ is the rank of $G$, and $d_\br=\sum_{\al\in R}\br(\al)$. We also set $c_w:=r-\dim\ft^w$\label{N:cw} and $\delta_{w,\br}:=\frac{d_{\br}-c_w}{2}$.\label{N:dtwbr}
\end{Emp}

The following formula of \cite[Theorem 8.2.2(2)]{GKM} is crucial for this work.

\begin{Prop} \label{P:codim}
For every $(w,\br)$ we have an equality $b_{w,\br}=\dt_{w,\br}+a_{w,\br}+c_w$.
\end{Prop}

\begin{Cor} \label{C:codim1}
For every $(w,\br)$ we have an inequality $b_{w,\br}\geq\dt_{w,\br}$. Moreover, the equality holds if and only if
$w=1$ and $\br=0$.
\end{Cor}

\begin{proof}
The inequality $b_{w,\br}\geq\dt_{w,\br}$ follows from \rp{codim} and observation that $a_{w,\br},c_w\geq 0$. Furthermore,
equality holds if and only if  $c_w=a_{w,\br}=0$. Note that equality $c_w=0$ holds if and only if $w=1$. In this case, equality $a_{w,\br}=0$ holds if and only if the subscheme $\kt_{w,\br}=\kt_{\br}\subset\clp(\kt)$ is open, and this happens if and only if $\br=0$.
\end{proof}

\begin{Emp} \label{E:topnilp}
{\bf The topologically nilpotent locus.} (a) We set  $\cL^+(\kc)_{\tn}:=\ev_{\kc}^{-1}(0)\subset\clp(\kc)$.\label{N:clpkctn}
Then $\cL^+(\kc)_{\tn}\subset\clp(\kc)$ is a strongly pro-smooth connected fp-closed subscheme of codimension $\dim\fc=r$.

(b) For every $w\in W$, we denote by $\cL^+(\kt_{w})_{\tn}\subset\clp(\kt_{w})$\label{N:clpktwtn} the preimage of $\cL^+(\kc)_{\tn}\subset\clp(\kc)$. Recall that
$\clp(\ft_w)$ classifies power series $\sum_{i=0}^{\infty} x_i t^{i/h}$ such that $w^{-1}(x_i)=\xi^i x_i$ for all $i$ (see \re{twisted}(b)).
In particular, we have $x_0\in \ft^w$.

Under this description, $\cL^+(\kt_{w})_{\tn}\subset\clp(\kt_{w})$ classifies power series with $x_0=0$. Therefore $\cL^+(\kt_{w})_{\tn}$ is a connected strongly pro-smooth affine scheme, and $\cL^+(\kt_{w})_{\tn}\subset\clp(\kt_{w})$ is an fp-closed subscheme of codimension $\dim\ft^w=r-c_w$.

(c) Notice that we have inclusions $\ft_{w,\br}\subset\cL^+(\kt_{w})_{\tn}$ and $\fc_{w,\br}\subset\cL^+(\fc)_{\tn}$ if $\br(\al)>0$ for all $\al\in R$,
and $\ft_{w,\br}\subset\clp(\ft_w)\sm \cL^+(\kt_{w})_{\tn}$ and $\fc_{w,\br}\subset\cL^+(\fc)\sm\cL^+(\fc)_{\tn}$, otherwise. In the first case, we say that $(w,\br)>0$.

(d) For every $(w,\br)>0$, we set $b^+_{w,\br}:=b_{w,\br}-r$\label{N:b+wbr} and $a^+_{w,\br}:=a_{w,\br}+c_w-r$\label{N:a+wbr} (see \re{codim}).

(e) Since $\codim_{\clp(\fc)}(\clp(\fc)_{\tn})=r$ (by (a)) and $\codim_{\clp(\ft_w)}(\clp(\ft_w)_{\tn})=r-c_w$ (by (b)), we conclude from (d) that $b^+_{w,\br}=\codim_{\clp(\fc)_{\tn}}(\fc_{w,\br})$ and $a^+_{w,\br}=\codim_{\clp(\ft_w)_{\tn}}(\ft_{w,\br})$.
\end{Emp}


\begin{Cor} \label{C:codim}
For every $(w,\br)>0$, we have $b^+_{w,\br}=\dt_{w,\br}+a^+_{w,\br}$. In particular, we have an inequality $b^+_{w,\br}\geq \dt_{w,\br}$, and an equality holds if and only if $\ft_{w,\br}\subset\cL^+(\ft_w)_{\tn}$ is an open stratum.
\end{Cor}

\begin{proof}
%
%
The equality $b^+_{w,\br}=\dt_{w,\br}+a^+_{w,\br}$ follows immediately from \rp{codim}. Next, since $a^+_{w,\br}=\codim_{\clp(\ft_w)_{\tn}}(\ft_{w,\br})$ (see \re{topnilp}(e)), we conclude that $a^+_{w,\br}\geq 0$. Moreover, the equality
holds if and only if  $\ft_{w,\br}\subset\cL^+(\ft_w)_{\tn}$ is an open stratum.
\end{proof}


\begin{Emp}
{\bf Stratifications on  $\clp(\kg)$ and $\Lie(I)$}. (a) For $n\in\NN$, let $\ev_{n,\kg}:\clp_{n}(\kg)\rightarrow\kg$ be the evaluation map, set $\Lie_n(I):=(\ev_{n,\kg})^{-1}(\kb)$, and let
$v_n: \Lie_n(I)\to \clp_{n}(\kc)$ be the restriction of $\chi_{n}:\clp_{n}(\kg)\rightarrow\clp_{n}(\kc)$.
Note that the isomorphism $\clp(\kg)\isom \lim_n \clp_{n}(\kg)$ induces an isomorphism $\Lie(I)\isom \lim_n \Lie_n(I)$.

(b) For every GKM stratum $\fc_{w,\br}\subset\clp(\fc)$, we denote by $\fg_{w,\br}\subset\clp(\fg)$ and $\Lie(I)_{w,\br}\subset\Lie(I)$\label{N:lieIwbr} its preimages.
\end{Emp}

The proof of the following result will be given in \rsec{pfrtdrinfeld}.

\begin{Thm}\label{T:drinfeld}
For every $n\in\NN$, the morphisms $\chi_{n}:\clp_{n}(\kg)\rightarrow\clp_{n}(\kc)$ and $v_{n}:\Lie_n(I)\rightarrow\clp_{n}(\kc)$ are flat.
\end{Thm}

\begin{Cor} \label{C:drinfeld}
The morphisms $\clp(\chi):\clp(\kg)\rightarrow\clp(\kc)$ and $v:\Lie(I)\rightarrow\clp(\kc)$ are flat.
\end{Cor}

\begin{proof}
Since property of being flat is preserved by base change and passing to a filtered limit, the assertion follows from
\rt{drinfeld}.
\end{proof}

\begin{Cor} \label{C:codim2}
(a) The locally fp-closed subschemes $\Lie(I)_{w,\br}\subset\Lie(I)$ and $\kg_{w,\br}\subset\kg$ are of pure codimension $b_{w,\br}$ (see \re{codim}).

(b) The induced maps $\chi_{w,\br}:\kg_{w,\br}\rightarrow\kc_{w,\br}$ and $v_{w,\br}:\Lie(I)_{w,\br}\rightarrow\kc_{w,\br}$\label{N:vwbr} are flat and
uo-equidimensional (see \re{clM2}).
\end{Cor}

\begin{proof}
 Since $\chi_n$ and $v_n$ are flat morphisms between irreducible varieties (by \rt{drinfeld}), we conclude that they are
uo-equidimensional (see \re{loceq}(b)). Therefore it follows from \rl{clM}(ii),(iv) that morphisms
$\clp(\chi)$ and $v$ are uo-equidimensional in the sense of \re{clM2}. Now both assertions are easy:

(a)  Since $\kc_{w,\br}\subset\clp(\kc)$ is of pure codimension $b_{w,\br}$ (by \rp{codim}), while $\clp(\chi)$ and $v$ are pro-universally open,
the assertion follows from \rl{pb}.

(b) Since $\clp(\chi)$ and $v$ are flat and uo-equidimensional, their pullbacks $\chi_{w,\br}$ and $v_{w,\br}$ are flat and uo-equidimensional as well (by \rco{uo-equid}).
\end{proof}


\section{Geometry of the Affine Grothendieck--Springer fibration}

\subsection{Generalities}

\begin{Emp} \label{E:affspr}
{\bf The Affine Grothendieck--Springer fibration.}

(a) The Chevalley morphism $\chi:\kg\rightarrow\fc$ induces a morphism of ind-schemes $\cL\chi:\cL\kg\rightarrow\cL\kc$, which we denote
by simplicity by $\chi$. We set $\fC:=\chi^{-1}(\clp(\fc))\subset\cL\kg$.\label{N:fC}

(b) Since $\clp(\fc)\subset\cL \fc$ is an fp-closed subscheme, the preimage
$\fC\subset\cL\kg$ is an fp-closed ind-subscheme of $\cL\fg$. Since $\cL\kg$ is an ind-placid ind-scheme
(with a presentation $\cL\kg\simeq\colim_i t^{-i}\clp(\kg)$), we thus conclude that $\fC$ is an ind-placid ind-scheme as well.

(c) Set $\wt{\fC}:=\cL G\times ^I\Lie(I)$,\label{N:wtfC} that is, $\wt{\fC}$ is a quotient of $\cL G\times\Lie(I)$  by $I$ with respect to the action $h(g,\g):=(gh^{-1},\Ad_h(\g))$.
Then $\cL G$ acts on $\wt{\fC}$ by the rule $h([g,\g])=([hg,\g])$, and we have a natural isomorphism of stacks $[\wt{\fC}/\cL G]\simeq [\Lie(I)/I]$.

(d) We have a natural projection map $\frak{p}:\wt{\fC}\to \fC:[g,\g]\mapsto\Ad_g(\g)$,\label{N:frakp} called the {\em affine Grothendieck--Springer fibration}.
The fibers of this map are the affine Springer fibers, first introduced and studied by Kazhdan and Lusztig in \cite{KL}.
\end{Emp}

\begin{Emp}
{\bf Remark.} The notation $\fC$ comes to indicate that it is the locus of ``compact'' elements in $\cL\fg$.
\end{Emp}

\begin{Emp} \label{E:afffl}
{\bf The affine flag variety.}
(a) Let $\Fl:=\cL G/I$\label{N:Fl} be the affine flag variety. Notice that the map $\iota:\wt{\fC}\to \Fl\times\cL\fg:[g,x]\mapsto ([g],\Ad_g(x))$ identifies
$\wt{\fC}$ with the closed ind-subscheme
\[
\{([g],\g)\in \Fl\times \fC\,|\,\Ad_{g^{-1}}(\g)\in\Lie(I)\}\subset\Fl\times \fC.
\]
Under this identification, the fibration $\frak{p}:\wt{\fC}\to \fC$ of \re{affspr} decomposes as
$\wt{\fC}\hra \Fl\times \fC\overset{\pr}{\lra}\fC$.

(b) Note that $\Fl$ has a structure of an ind-projective scheme over $k$, with a canonical presentation
$\Fl\simeq\colim_i Y_i$ as a colimit of its $I$-invariant closed projective subschemes.

(c) The presentation of (b) induces a canonical $(I\times I)$-equivariant presentation $\cL G\simeq\colim_i \wt{Y}_i$ of $\cL G$, where
$\wt{Y}_i\subset\cL G$ is the preimage of $Y_i\subset \Fl$ under the natural projection $\cL G\to\Fl$.
Notice that each $\wt{Y}_i\to Y_i$ is an $I$-torsor, hence it is strongly pro-smooth. Therefore each $\wt{Y}_i$ is a scheme, admitting a placid presentation, thus  $\cL G\simeq\colim_i \wt{Y}_i$ is an ind-placid ind-scheme.

\end{Emp}

\begin{Lem} \label{L:affspr}
The projection $\frak{p}:\wt{\fC}\to \fC$ is ind-fp-proper.
\end{Lem}

\begin{proof}
Recall that $\frak{p}$ factors as a composition of the closed embedding $\iota:\wt{\fC}\hra\Fl\times \fC$ and the projection
$p:\Fl\times \fC\to \fC$. Since $\Fl$ is ind-projective, the projection $p$ is ind-fp-proper. Thus it suffices to show that
$\iota$ is an fp-closed embedding.

Note that the action morphism $a:\cL G\times \fC\to \cL\kg:a(g,x)=\Ad_{g^{-1}}(x)$ gives rise to the morphism $\ov{a}:\Fl\times \fC\to [\cL\kg/I]$,
and by definition $\wt{\fC}$ is the preimage of $[\Lie(I)/I]\subset [\cL\kg/I]$. Since $\Lie(I)\hra\cL\kg$ is an fp-closed embedding and the class of fp-closed embeddings is \'etale local on the base, we conclude that  $[\Lie(I)/I]\hra [\cL\kg/I]$ is an fp-closed embedding, and we are done. \end{proof}


\begin{Emp} \label{E:maxtori}
{\bf Maximal tori.}
Recall that there is a natural bijection $w\mapsto T_w$\label{N:Tw} between conjugacy classes of elements in $W$ and conjugacy classes of maximal tori in $G_F$ (see, for example, \cite[Lemma 2]{KL}, where the assertion is stated in characteristic zero, but the argument works without any changes when $|W|$ is invertible in $k$). Notice that the Lie algebra $\Lie(T_w)$ is canonically isomorphic to the Lie algebra $\frak{t}_w$ described in \re{twisted}(b). In particular, we have an embedding of $\frak{t}_w\hra \frak{g}_F$ over $\fc_F$, unique up to a conjugacy.
\end{Emp}

\begin{Emp} \label{E:str}
{\bf GKM strata.}
(a) Recall that we have defined strata $\frak{t}_{w,\br}$ of $\clp(\kt_w)$ (see \re{twisted}(b)) and the corresponding strata $\fc_{w,\br}$ of $\clp(\fc)$ (see \re{strofc}(c)). Moreover, every projection $\pi:\ft_{w,\br}\to \fc_{w,\br}$ is a $W_{w,\br}$-torsor (see \re{strofc}(d)).

(b) For every GKM stratum of $\fc_{w,\br}$ we set $\fC_{w,\br}:=\chi^{-1}(\fc_{w,\br})\subset \fC$\label{N:fCwbr} and
$\wt{\fC}_{w,\br}:=\frak{p}^{-1}(\fC_{w,\br})\subset \wt{\fC}$.\label{N:wtfCwbr} We have an identification $\wt{\fC}_{w,\br}\simeq\cL G\times ^I\Lie(I)_{w,\br}$, hence
$[\wt{\fC}_{w,\br}/\cL G]\simeq [\Lie(I)_{w,\br}/I]$. The projection $\frak{p}$ induces projections $\frak{p}_{w,\br}:\wt{\fC}_{w,\br}\to\fC_{w,\br}$\label{N:frakpwbr} and $\ov{\frak{p}}_{w,\br}:[\wt{\fC}_{w,\br}/\cL G]\to[\fC_{w,\br}/\cL G]$\label{N:ovfrakpwbr}.

(c) The embedding of $\frak{t}_w\hra\frak{g}_F$ from \re{maxtori} induces an embedding $\cL(\frak{t}_w)\hra\cL\frak{g}$ over $\cL\fc$, unique up to a conjugacy. Hence it induces an embedding $\frak{t}_{w,\br}\hra{\fC}_{w,\br}$ over $\fc_{w,\br}$, unique up to an $\cL G$-conjugacy, thus a canonical morphism $\psi_{w,\br}:\frak{t}_{w,\br}\to [{\fC}_{w,\br}/\cL G]$.
Moreover, the morphism $\psi_{w,\br}$\label{N:psiwbr} is $W_{w,\br}$-equivariant, thus it induces a canonical morphism $\ov{\psi}_{w,\br}:\fc_{w,\br}=[\frak{t}_{w,\br}/W_{w,\br}]\to [{\fC}_{w,\br}/\cL G]$.\label{N:ovpsiwbr}

\end{Emp}

\begin{Emp} \label{E:consstr2}
{\bf Constructible stratification.}
(a) As in \re{strval2}, we set $\clp(\kc)_{\leq m}:=\clp(\kc)_{(\kD;\leq m)}$ for every $m\in\B{N}$, and let $\Lie(I)_{\leq m}\subset\Lie(I)$, \label{N:lieleqm}  $\fC_{\leq m}\subset\fC$\label{N:fCleqm} and $\wt{\fC}_{\leq m}\subset\wt{\fC}$\label{N:wtfCleqm} be the preimages of $\clp(\kc)_{\leq m}$. Notice that $\clp(\kc)_{\leq m}\subset \clp(\kc)$ and $\fC_{\leq m}\subset\fC$ are fp-open subschemes,
thus $\fC_{\leq m}$ is an ind-placid ind-scheme (by \re{affspr}(b)).
Note also that $\clp(\kc)_{\leq 0}=\clp(\kc^{\rs})$ (see \re{strvalsm}).

(b) As in \re{strval2}(c), we consider the regular locus  $\cL^+(\kc)_{\bullet}:=\clp(\kc)_{\kD\neq 0}\subset\clp(\kc)$.\label{N:clpcbullet}
Next, we set  $\fC\bu:=\chi^{-1}(\clp(\fc)_{\bullet})\subset \fC$,\label{N:fCbullet} $\Lie(I)\bu:=v^{-1}(\clp(\fc)_{\bullet})\subset \Lie(I)$\label{N:lieibullet} and $\wt{\fC}\bu:=\frak{p}^{-1}(\fC_{\bullet})\subset \wt{\fC}$.\label{N:wtfCbullet} In particular, we have a
natural identification $\wt{\fC}\bu\simeq\cL G\times ^I\Lie(I)\bu$.

(c) By definition, we have an fp-open covering $\clp(\kc)_{\bullet}=\cup_{m\geq 0}\clp(\kc)_{\leq m}$, which gives rise to an fp-open covering  $\fC_{\bullet}=\cup_{m\geq 0}\fC_{\leq m}$ and a
constructible stratification $\{\clp(\kc)_{(\kD;m)}\}_m$  of $\clp(\kc)_{\bullet}$ (by \re{strvalsm}).

(d) Since $\kc_{w,\br}$ is a connected component of  $\clp(\kc)_{(\kD;d_{\br})}$ for each $(w,\br)$ (see \re{strofc}(b)), we conclude that
$\{\kc_{w,\br}\}_{w,\br}$ forms a bounded constructible stratification of $\clp(\kc)_{\bullet}$ (use \rl{propcons}(c)).

(e) The constructible stratification $\{\kc_{w,\br}\}_{w,\br}$ of $\clp(\kc)_{\bullet}$ induces a constructible stratification
 $\{[\fC_{w,\br}/\cL G]_{\red}\}_{w,\br}$ of $[\fC_{\bullet}/\cL G]$ (see \rl{propcons}(a)).
\end{Emp}

\begin{Lem} \label{L:isom}
For every GKM stratum $(w,\br)$, we have natural isomorphisms
\[
\cL(G_F/T_w)\times \kt_{w,\br}\isom \fC_{w,\br}\times_{\fc_{w,\br}} \ft_{w,\br}:(g,x)\mapsto (\Ad_g(x),x);
\]
\[\cL(G_F/T_w)\times^{W_{w,\br}} \kt_{w,\br}\isom\fC_{w,\br}:(g,x)\mapsto\Ad_g(x).\]
\end{Lem}
\begin{proof}
Let $\ft^{\rs}_w\subset\ft_w$\label{N:twrss} be the preimage of $\fc^{\rs}$. Then the map $(g,x)\mapsto (\Ad_g(x), x)$ induces an isomorphism
\[
(G_F/T_w)\times \ft_w^{\rs}\simeq \fg^{\rs}\times_{\fc^{\rs}} \ft_w^{\rs}\simeq \fg\times_{\fc} \ft_w^{\rs}
\]
over $\ft_w^{\rs}$. Since functor $\cL$ preserves limits, it induces an isomorphism
\[
\cL(G_F/T_w)\times \cL(\ft_w^{\rs})\isom \cL\fg\times_{\cL\fc} \cL(\kt_w^{\rs})
\]
over $\cL(\ft_w^{\rs})$. Restricting it to $\ft_{w,\br}\subset\clp(\ft_{w})_{(\kD,d_{\br})}\subset\cL(\ft_w^{\rs})$ (see \re{strval2}(d)), we get an isomorphism
\[
\cL(G_F/T_w)\times \ft_{w,\br}\isom \cL\fg\times_{\cL\fc} \ft_{w,\br}.
\]
From this the first isomorphism follows using identifications
\[
\cL\fg\times_{\cL\fc} \ft_{w,\br}\simeq (\cL\fg\times_{\cL\fc}\fc_{w,\br})\times_{\fc_{w,\br}} \ft_{w,\br}\simeq \fC_{w,\br}\times_{\fc_{w,\br}} \ft_{w,\br},
\]
while the second isomorphism is obtained from the first one by taking the quotient by $W_{w,\br}$.
\end{proof}

The following result will be proven in the Appendix (see \re{pfrtorbit}).\footnote{Compare \cite[Th\`eor\'eme 2.6.4]{Bo} for a stronger assertion.}

\begin{Thm} \label{T:orbit}
For every (not necessary split) maximal torus $S\subset G_F$, the natural projection $\psi_S: [\cL G/\cL S]\to \cL(G_F/S)$\label{N:psiS} is a topological equivalence.
\end{Thm}

\begin{Cor} \label{C:stratum}
The map $[\psi_{w,\br}]:[\ft_{w,\br}/(W_{w,\br}\rtimes(\cL T_w)_{\red})]\to [\fC_{w,\br}/\cL G]_{\red}$,\label{N:[psiwbr]} induced by the map $\psi_{w,\br}$ from \re{str}(c), is an isomorphism.
\end{Cor}

\begin{proof}
Since the projection $\psi_{T_w}: [\cL G/\cL T_w]\to \cL(G_F/T_w)$ is a topological equivalence by \rt{orbit}, it induces an
$(\cL G\times W_{w,\br})$-equivariant equivalence $[\cL G/\cL T_w]\times\ft_{w,\br}\to \cL(G_F/T_w)\times\ft_{w,\br}$ (see \re{topeq}(c)). Taking the quotient by $W_{w,\br}$, we deduce from \re{topeq}(b) that the map
\begin{equation} \label{Eq:topeq}
[\cL G/\cL T_w]\times^{W_{w,\br}}\ft_{w,\br}\to \cL(G_F/T_w)\times^{W_{w,\br}}\ft_{w,\br}\simeq \fC_{w,\br}:(g,x)\mapsto\Ad_g(x)
\end{equation}
is an $\cL G$-equivariant topological equivalence. Dividing by $\cL G$,
we get that the induced map
$[\ft_{w,\br}/(W_{w,\br}\rtimes\cL T_w)]\to [\fC_{w,r}/\cL G]$ is a topological equivalence.
Finally, since $\ft_{w,\br}$ and $W_{w,\br}$ are reduced, we get the identification
\[
[\kt_{w,\br}/(W_{w,\br}\rtimes\cL T_w)]_{\red}\simeq [\kt_{w,\br}/(W_{w,\br}\rtimes(\cL T_w)_{\red})],
\]
and we conclude that $[\psi_{w,\br}]$ is an isomorphism.
\end{proof}

\begin{Emp} \label{E:ostr}
{\bf Loop groups on tame tori.}
(a) Recall that every torus $S$ over $F$ has a natural structure of a smooth group scheme $S_{\C{O}}$ over $\C{O}$, also known as the N\'eron model. Moreover, when $S$ is tame, that is, split over a tamely ramified extension $F'/F$, it has the following simple description: Let $F'/F$ be the splitting field of $S$ with Galois group $\Gm:=\Gal(F'/F)$. Then the torus $S':=S_{F'}$ is split, thus has a natural structure $S'_{\C{O}}$ over $\C{O}_{F'}$, and we set $S_{\C{O}}:=(S'_{\C{O}})^{\Gm}$. Then we can define the arc group $\clp(S):=\clp(S_{\C{O}})$.\label{N:clpS}

(b) Let $\La_S:=\Hom_{F}(\B{G}_m,S)$\label{N:laS} be the group of cocharacters of $S$ defined over $F$. We claim that we have a natural isomorphism
\[
(\cL S)_{\red}\simeq \clp(S)\times\La_S.
\]
Indeed, when $S$ is split, the assertion reduces to the case of $S=\gm$, which is easy. In the general case, let $S':=S_{F'}$ be as in (a). Then
$(\cL {S'})_{\red}\simeq \clp(S')\times\La_{S'}$, by the split case. Thus, taking $\Gm$-invariants, we get
$((\cL {S'})_{\red})^{\Gm}\simeq \clp(S')^{\Gm}\times(\La_{S'})^{\Gm}$.

Since $\La_S=(\La_{S'})^{\Gm}$ (by definition), while $\clp(S')^{\Gm}\simeq \clp(S)$ and
$\cL(S')^{\Gm}\simeq \cL S$ (because loop and arc functors commute with limits), it suffices to show that
$((\cL {S'})_{\red})^{\Gm}\simeq ((\cL {S'})^{\Gm})_{\red}$.
Since $((\cL {S'})^{\Gm})_{\red}\subset ((\cL {S'})_{\red})^{\Gm}\subset (\cL {S'})^{\Gm}$, it suffices to show that
$((\cL {S'})_{\red})^{\Gm}\simeq \clp(S')^{\Gm}\times(\La_{S'})^{\Gm}$
is reduced, that is, $\clp(S')^{\Gm}$ is reduced. Since $\clp(S')^{\Gm}\simeq\lim_n\clp_n(S')^{\Gm}$, and $|\Gm|$ is prime to the characteristic of $k$, each $\clp_n(S')^{\Gm}$ is smooth (see \cite[15.4.2]{GKM}), thus reduced.

(c) Our assumption that $|W|$ is invertible in $k$ implies that for every $w\in W$ the corresponding maximal torus $T_w\subset G_F$ (see \re{maxtori}) is tame. Moreover, the canonical conjugacy class of embeddings $T_w\hra G_F$ has a representative defined over $\co$.
\end{Emp}

\begin{Cor} \label{C:stratum2}

The $\infty$-stack $[\fC_{w,\br}/\cL G]_{\red}$ is placid, and the map
$\psi_{w,\br}:\kt_{w,\br}\to [{\fC}_{w,\br}/\cL G]_{\red}$ from \re{str}(c) is a smooth covering.
\end{Cor}

\begin{proof}
Recall that $\kt_{w,\br}$ is a placid scheme (see \re{twisted}(c)) and $(\cL T_w)_{\red}$ is a group scheme, whose connected component of the identity is the strongly pro-smooth group $\clp (T_{w})$ (see \re{ostr}(b) and \re{proparc}(e)). Therefore  $W_{w,\br}\rtimes (\cL T_w)_{\red}$ is a $0$-smooth group scheme. Thus, by \re{quot}, the quotient $[\kt_{w,\br}/(W_{w,\br}\rtimes (\cL T_w)_{\red})]$ is a placid stack, and the projection $\pr:\kt_{w,\br}\to [\kt_{w,\br}/(W_{w,\br}\rtimes (\cL T_w)_{\red})]$ is a smooth covering. Since $\psi_{w,\br}=[\psi_{w,\br}]\circ\pr$, the assertion follows from \rco{stratum}.
\end{proof}

\subsection{The fibration over a regular stratum}

\begin{Lem} \label{L:isom2}
The map $(I/\clp(T))\times\clp(\ft^{\rs})\to \Lie(I)_{\leq 0}:(g,x)\mapsto \Ad_g(x)$ is an isomorphism.
\end{Lem}
\begin{proof}
Note that the map $(g,x)\mapsto\Ad_g(x)$ induces an isomorphism
\begin{equation} \label{Eq:e1}
(B/T)\times\ft^{\rs}\isom \kb^{\rs}.
\end{equation}
The assertion now follows formally. Applying $G\times^B-$, we get an isomorphism
\begin{equation} \label{Eq:e2}
(G/T)\times\ft^{\rs}\isom  G\times^{B}\kb^{\rs}\simeq\wt{\fg}^{\rs},
\end{equation}
where $\pi:\wt{\fg}\to\fg$ is the Grothendieck--Springer resolution, and $\wt{\kg}^{\rs}$ is the preimage of $\kg^{\rs}$.
Next, applying $\clp$ and using \re{proparc}(f), we get an isomorphism
\begin{equation} \label{Eq:e3}
(\clp(G)/\clp(T))\times\clp(\ft^{\rs})\simeq \clp(G/T)\times\clp(\ft^{\rs})\isom\clp(\wt{\fg}^{\rs}).
\end{equation}
Taking the fiber product of \form{e1} and \form{e3} over \form{e2}, we get an isomorphism
\begin{equation} \label{Eq:e4}
(I/\clp(T))\times\clp(\ft^{\rs})\isom\clp(\wt{\fg}^{\rs})\times_{\wt{\fg}^{\rs}}\kb^{\rs}.
\end{equation}
Finally, applying \rl{arcetale} for the \'etale morphism  $\pi:\wt{\fg}^{\rs}\to \fg$, the right hand side of \form{e4} is isomorphic to  $\clp(\fg)\times_{\fg}\kb^{\rs}\simeq\Lie(I)_{\leq 0}$.
\end{proof}

\begin{Cor} \label{C:regstr2}
We have a natural isomorphism
$[\wt{\fC}_{\leq 0}/\cL G]\simeq [\clp(\kt^{\rs})/\clp(T)]$.
\end{Cor}

\begin{proof}
Dividing the isomorphism from \rl{isom2} by the action of $I$, we get an isomorphism $[\clp(\kt^{\rs})/\clp(T)]\isom [\Lie(I)_{\leq 0}/I]$.
Since $[\wt{\fC}_{\leq 0}/\cL G]\simeq [\Lie(I)_{\leq 0}/I]$, we are done.
\end{proof}

\begin{Cor} \label{C:regstr}
We have a natural commutative diagram
\begin{equation} \label{Eq:regstratum}
\begin{CD}
(\cL G/\clp(T))\times \clp(\ft^{\rs})@>\sim>> \wt{\fC}_{\leq 0}\\
@V\pr VV @VV\frak{p}_{\leq 0}V \\
\cL(G/T)\times^{W} \clp(\ft^{\rs})@>\sim>> \fC_{\leq 0}.
\end{CD}
\end{equation}
\end{Cor}
\begin{proof}
The top isomorphism  is obtained as a composition
\[
(\cL G/\clp(T))\times \clp(\ft^{\rs})\simeq \cL G\times^{I}[(I/\clp(T))\times\clp(\ft^{\rs})]\isom\cL G\times^{I}\Lie (I)_{\leq 0}\simeq \wt{\fC}_{\leq 0},
\]
induced by  isomorphism of \rl{isom2}, while the bottom isomorphism is the isomorphism of \rl{isom}
applied to the open stratum $(w,\br)=(1,0)$. The fact that the diagram is commutative is straightforward.
\end{proof}


\begin{Emp} \label{E:waffact}
{\bf The $\wt{W}$-action on $\wt{\fC}_{\leq 0}$.} (a) Let $N:=N_G(T)\subset G$ be the normalizer. First we claim that we have a natural
isomorphism $\wt{W}\simeq (\cL N)_{\red}/\clp(T)$ of group spaces. Indeed, by definition, we have an isomorphism of groups $\wt{W}\simeq (\cL N)(k)/\clp(T)(k)$, so it suffices to show that the quotient space $[(\cL N)_{\red}/\clp(T)]$ is discrete. To see this, notice that
the isomorphism $N/T\simeq W$ induces an isomorphism $\cL N/\cL T\simeq W$, hence an isomorphism  $(\cL N)_{\red}/(\cL T)_{\red}\simeq W$.
Since $[(\cL T)_{\red}/\clp(T)]\simeq\La_T$ (by \re{ostr}(b)) is discrete, the discreteness of $[(\cL N)_{\red}/\clp(T)]$ follows.

(b) The identification of (a) gives rise to an action of $\wt{W}$ on $(\cL G/\clp(T))\times \clp(\ft^{\rs})$ over $\cL (G/T)\times^W \clp(\ft^{\rs})$, given by the formula $w(g,x):=(gw^{-1},w(x))$. Moreover, the quotient of $(\cL G/\clp(T))\times\clp(\ft^{\rs})$ by $\wt{W}$ is naturally identified with $(\cL G/(\cL T)_{\red})\times^W \clp(\ft^{\rs})$.

(c) Using the identification \form{regstratum}, we obtain from (b) an action of $\wt{W}$ on $\wt{\fC}_{\leq 0}$ over $\fC_{\leq 0}$, which induces an identification  $[\wt{\fC}_{\leq 0}/\wt{W}]\simeq(\cL G/(\cL T)_{\red})\times^W \clp(\ft^{\rs})$.
\end{Emp}

\begin{Cor} \label{C:openstr}
The projection $\frak{p}_{\leq 0}:[\wt{\fC}_{\leq 0}/\wt{W}]\to \fC_{\leq 0}$ is a topological equivalence.
\end{Cor}

\begin{proof}
The identifications of \form{regstratum} and  \re{waffact} identify $\frak{p}_{\leq 0}$ with the topological equivalence \form{topeq} in the case $w=1$ and $\br=0$.
\end{proof}

\subsection{The fibration over a general stratum}

Recall that in \re{clM2}(b) we defined a class of uo-equidimensional morphisms between placid $\infty$-stacks.

\begin{Prop} \label{P:equid}
The fibration $\ov{\frak{p}}_{w,\br,\red}:[\wt{\fC}_{w,\br}/\cL G]_{\red}\to [\fC_{w,\br}/\cL G]_{\red}$ is uo-equidimensional.
\end{Prop}

\begin{proof}
Since the projection
\[
\Lie (I)_{w,\br,\red}\to [\Lie (I)_{w,\br,\red}/I]\simeq[\Lie (I)_{w,\br}/I]_{\red}\simeq [\wt{\fC}_{w,\br}/\cL G]_{\red}
\]
is a smooth covering (see \re{quot}), it suffices to show that the composition
\[
\wt{\frak{p}}_{w,\br}:\Lie (I)_{w,\br,\red}\to [\wt{\fC}_{w,\br}/\cL G]_{\red}\to [\fC_{w,\br}/\cL G]_{\red}
\]
is uo-equidimensional (see \rco{clMc}(c)). Consider the commutative diagram
\[
\begin{CD}
\wt{X}_{w,\br} @>\wt{\pi}_{w,\br}>> \wt{X}'_{w,\br} @>\wt{\psi}_{w,\br}>> \Lie (I)_{w,\br,\red} @= \Lie (I)_{w,\br,\red}\\
@V\wt{g}_{w,\br}VV @V\wt{g}'_{w,\br}VV @V\wt{\frak{p}}_{w,\br}VV @VVv_{w,\br}V\\
\kt_{w,\br} @>{\pi}_{w,\br}>> \fc_{w,\br} @>\ov{\psi}_{w,\br}>>[\fC_{w,\br}/\cL G]_{\red}@>\pr>>\kc_{w,\br},
\end{CD}
\]
whose left and middle inner squares are Cartesian.

Since $\psi_{w,\br}=\ov{\psi}_{w,\br}\circ \pi_{w,\br}$ is a  smooth covering
(see \rco{stratum2}), the same is true for the composition
$\wt{\psi}_{w,\br}\circ \wt{\pi}_{w,\br}$. In particular, both $\ov{\psi}_{w,\br}\circ \pi_{w,\br}$ and $\wt{\psi}_{w,\br}\circ \wt{\pi}_{w,\br}$ are uo-equidimensional by \rco{clMc}(a). Moreover, since $\pi_{w,\br}$ and hence also $\wt{\pi}_{w,\br}$ is a finite \'etale covering (see \re{strofc}(c)), we conclude from \rco{clMc}(c) that $\ov{\psi}_{w,\br}$ and $\wt{\psi}_{w,\br}$ are uo-equidimensional.

Next, since $\pr\circ\ov{\psi}_{w,\br}=\Id$, we get $\wt{g}'_{w,\br}=v_{w,\br}\circ \wt{\psi}_{w,\br}$. Since $v_{w,\br}$ is uo-equidimensional by \rco{codim2}(b), we conclude from \rl{clM}(i) that $\wt{g}'_{w,\br}$ and hence also
\[
\wt{\frak{p}}_{w,\br}\circ(\wt{\psi}_{w,\br}\circ \wt{\pi}_{w,\br})\simeq \ov{\psi}_{w,\br}\circ \wt{g}'_{w,\br}\circ \wt{\pi}_{w,\br}
\]
are uo-equidimensional as well. Therefore $\wt{\frak{p}}_{w,\br}$ is uo-equidimensional (by \rco{clMc}(c)), and the proof is complete.
\end{proof}

\begin{Emp} \label{E:laaction}
{\bf The $\La_w$-action on $\wt{\fC}_{\ft,w,\br}$}.

We set ${\fC}_{\ft,w,\br}:=\fC_{w,\br}\times _{\fc_{w,\br}}\ft_{w,\br}$,\label{N:fCftwbr}  $\wt{\fC}_{\ft,w,\br}:=\wt{\fC}_{w,\br}\times _{\kc_{w,\br}}\kt_{w,\br}$,\label{N:wtfCftwbr} and let $\La_w:=\La_{T_w}$\label{N:Law} (be as in \re{ostr}(b)).

(a) Recall that the embedding $\iota:\wt{\fC}\hra \Fl\times \fC$ of \re{afffl}, identifies $\wt{\fC}_{w,\br}$ with a closed ind-subscheme
\[
\{([g],x)\in\Fl\times \fC_{w,\br}\,|\,\Ad_{g^{-1}}(x)\in\Lie(I)\}\subset \Fl\times \fC_{w,\br}.
\]

(b) Using isomorphism $\cL(G_F/T_w)\times \kt_{w,\br}\simeq{\fC}_{\ft,w,\br}$ from \rl{isom}, the ind-scheme $\wt{\fC}_{\ft,w,\br}$ can be identified with a closed ind-subscheme
\[
\{([g],h,x)\in\Fl\times \cL(G_F/T_w)\times \kt_{w,\br}\,|\,\Ad_{g^{-1}h}(x)\in\Lie(I)\}\subset \Fl\times \cL(G_F/T_w)\times \kt_{w,\br}.
\]
Note that $g\in\cL G$ is defined up to a right $I$-multiplication, thus $g^{-1}h\in \cL(G_F/T_w)$ is defined up to a left multiplication. Therefore  $\Ad_{g^{-1}h}(x)\in\cL\kg$ is defined up to an $\Ad I$-action, hence the condition $\Ad_{g^{-1}h}(x)\in\Lie(I)$ makes sense.

(c) Consider the action  of $\cL(T_w)$ on  $\Fl\times \cL(G_F/T_w)\times \kt_{w,\br}$ over $\cL(G_F/T_w)\times \kt_{w,\br}$, defined by the formula $t([g],h,x):=([(hth^{-1})g],h,x)$. Using the equality
\[
g^{-1}(hth^{-1})^{-1}h=g^{-1}ht^{-1}h^{-1}h=g^{-1}ht^{-1},
\]
we conclude that the closed ind-subscheme $\wt{\fC}_{\ft,w,\br}\subset\Fl\times \cL(G_F/T_w)\times\kt_{w,\br}$ from (b) is $\cL(T_w)$-invariant. Thus we obtain an action of $\cL(T_w)$ on $\wt{\fC}_{\ft,w,\br}$ over ${\fC}_{\ft,w,\br}$.

(d) Recall that $\La_w=\Hom_F(\B{G}_m,T_w)$ is naturally a subgroup of $\cL(T_w)$ via the embedding $\la\mapsto \la(t)$. Therefore the action of $\cL(T_w)$ from (c)
induces an action of $\La_w$ on $\wt{\fC}_{\ft,w,\br}$ over ${\fC}_{\ft,w,\br}$.
Thus we can form a quotient $\ov{\fC}_{\ft,w,\br}:=[\wt{\fC}_{\ft,w,\br}/\La_w]$.\label{N:ovfCftwbr}
\end{Emp}






Recall that in \re{stopeq} we defined classes of topologically locally fp-schematic and topologically fp-proper representable/schematic morphisms between $\infty$-stacks. The proof of the following important result will be proven in the Appendix (see \re{pfrttopproper}).

\begin{Thm} \label{T:topproper}
(a) The projection $\frak{p}_{\ft,w,\br}:\wt{\fC}_{\ft,w,\br}\to{\fC}_{\ft,w,\br}$ is topologically locally fp-schematic, and the induced morphism $\ov{\fC}_{\ft,w,\br}\to{\fC}_{\ft,w,\br}$ is topologically fp-proper representable.

(b) For every morphism $f:Y\to{\fC}_{\ft,w,\br}$ from an affine scheme $Y$, there exists a subgroup $\La'_w\subset \La_w$ of finite index such that the quotient $[(\wt{\fC}_{\ft,w,\br}\times_{{\fC}_{\ft,w,\br}}Y)/\La'_w]_{\red}$ is actually a scheme.
\end{Thm}

\begin{Cor} \label{C:topproper}
Consider a Cartesian diagram
\begin{equation} \label{Eq:basic1}
\begin{CD}
{X}_{w,\br} @>{\phi}_{w,\br}>> [\wt{\fC}_{w,\br}/\cL G]_{\red} \\
 @V{{g}_{w,\br}}VV @VV{\ov{\frak{p}}_{w,\br,\red}}V\\
\kt_{w,\br} @>\psi_{w,\br}>> [\fC_{w,\br}/\cL G]_{\red}.
\end{CD}
\end{equation}

(a) Then ${X}_{w,\br}$ is a placid reduced scheme locally fp over $\kt_{w,\br}$, equipped with an action of
$\La_w$. Moreover, the quotient $[{X}_{w,\br}/\La_w]$ is an algebraic space, fp-proper over $\kt_{w,\br}$. Furthermore, there exists a subgroup $\La'_w\subset \La_w$ of finite index such that the quotient
$[{X}_{w,\br}/\La'_w]$ is a scheme.

(b) The map $\ov{\frak{p}}_{w,\br,\red}$ is a locally fp-representable morphism between placid $\infty$-stacks. Moreover, it is
uo-equidimensional of relative dimension $\dt_{w,\br}$.

(c) The projections ${g}_{w,\br}$ and $\ov{g}_{w,\br}:[{X}_{w,\br}/\La_w]\to\kt_{w,\br}$ are
uo-equidimensional of dimension $\dt_{w,\br}$.

(d)  The morphism $f:[\wt{\fC}_{\kt,w,\br}/(\cL G\times\La_w)]_{\red}\to[\fC_{\kt, w,\br}/\cL G]_{\red}$, induced by
$\frak{p}_{\ft,w,\br}$, is fp-proper.
\end{Cor}

\begin{proof}

(a) Note that the morphism $\psi_{w,\br}$ has a natural lift to $\kt_{w,\br}\to{\fC}_{\ft,w,\br}$, so we have a natural isomorphism
$X_{w,\br,\red}\simeq (\kt_{w,\br}\times_{\fC_{\ft,w,\br}}\wt{\fC}_{\ft,w,\br})_{\red}$. Since $\psi_{w,\br}$ is a smooth morphism of placid $\infty$-stacks (see \rco{stratum2}), we conclude from \rl{reduction} that $X_{w,\br}$ is reduced.
Since $\kt_{w,\br}$ is a placid affine scheme (by \re{twisted}(c)), all assertions follow immediately from \rt{topproper}
and \rl{red-fp}.

(b) Since the class of locally fp-representable morphisms is \'etale local on the base, while $\psi_{w,\br}$ is a covering, the first assertion follows from (a) and \rco{red-fp}. Next, since $\ov{\frak{p}}_{w,\br,\red}$ is uo-equidimensional by \rp{equid}, it is equidimensional by \rco{uo-equid}(a).
Therefore $\ov{\frak{p}}_{w,\br,\red}$ is equidimensional in the sense of \re{eqmor} (by \rl{clM}(v)), thus
it suffices to show that all fibers of $\ov{\frak{p}}_{w,\br,\red}$ are equidimensional
of dimension $\dt_{w,\br}$. As fibers or $\ov{\frak{p}}_{w,\br,\red}$ are the affine Springer fibers, so the assertion follows from
the Kazhdan--Lusztig--Bezrukavnikov formula \cite{KL}, \cite{Be} (see remark \re{Be}(a) below).

(c) Since $\psi_{w,\br}$ is pro-universally open (by \rco{clMc}(a)), the assertion about $g_{w,\br}$ follows immediately from (b) and \rl{pb}.
The assertion about $g_{w,\br}$ now follows from the fact that the projection ${X}_{w,\br}\to[{X}_{w,\br}/\La_w]$ is \'etale.

(d) Since the class of fp-proper morphisms is \'etale local on the base (see \cite[Tags 02L0, 02L1]{Sta}), and $\psi_{w,\br}$ is a covering
(by \rco{stratum2}), it suffices to show the pullback of $f$ with respect to $\psi_{w,\br}$ is fp-proper (by
\rl{etalelocal}). As this pullback is isomorphic to the projection $[{X}_{w,\br}/\La_w]\to\kt_{w,\br}$, the assertion follows from (a).
\end{proof}

\begin{Emp} \label{E:Be}
{\bf Remarks.}
(a) In their works Kazhdan--Lusztig \cite{KL} and Bezrukavnikov \cite{Be} only treat the case of topologically nilpotent elements.
However the general case can be deduced from it using the topological Jordan decomposition. See, for example, \cite[Appendix B]{BV} where a similar
assertion for groups is shown.

(b) The above argument gives an alternative proof of the equidimensionality of affine
Springer fibers, independent of the argument of Kazhdan--Lusztig.
\end{Emp}

\subsection{The (semi)-smallness}

\begin{Emp} \label{E:not}
{\bf Notation.} Recall that the  fibration
$\frak{p}:\wt{\fC}\rightarrow \fC$ is ind-fp-proper (see \rl{affspr}) and $\cL G$-equivariant, the induced fibrations
$\ov{\frak{p}}:[\wt{\fC}/\cL G]\rightarrow[\fC/\cL G]$\label{N:ovfrakp}  and $\ov{\frak{p}}_{\bullet}:[\wt{\fC}_{\bullet}/\cL G]\rightarrow[\fC_{\bullet}/\cL G]$\label{N:ovfrakpbullet} are ind-fp-proper maps (see \re{pspr}).
\end{Emp}

\begin{Lem} \label{L:sprsmall}
(a) The collection $\{[\fC_{w,\br}/\cL G]_{\red}\}_{w,\br}$ forms a constructible stratification of $[\fC_{\bullet}/\cL G]$, making it a placidly stratified $\infty$-stack.

(b) Every $[\wt{\fC}_{w,\br}/\cL G]\subset[\wt{\fC}_{\bullet}/\cL G]$ is an fp-locally closed subscheme of pure codimension $b_{w,\br}$.

(c) The $\infty$-stack $[\wt{\fC}_{\bullet}/\cL G]$ is smooth, while the projection $\ov{\frak{p}}_{\bullet}$ is small, relative to the open stratum $(w,\br)$ with $w=1$ and $\br=0$.
\end{Lem}

\begin{proof}
(a) Every $[\fC_{w,\br}/\cL G]_{\red}$ is a placid $\infty$-stack by \rco{stratum2}, and $\{[\fC_{w,\br}/\cL G]_{\red}\}_{w,\br}$ is a constructible stratification by \re{consstr2}.

(b) Note that $\Lie(I)_{w,\br}\subset\Lie(I)$ is an fp-locally closed subscheme of pure codimension $b_{w,\br}$ (see \rco{codim2}), hence  $[\Lie(I)_{w,\br}/I]\subset [\Lie(I)/I]$ is an fp-locally closed subscheme of pure codimension $b_{w,\br}$ (by \rl{pb}). Since $[\wt{\fC}_{\bullet}/\cL G]\simeq[\Lie(I)_{\bullet}/I]$ and $[\wt{\fC}_{w,\br}/\cL G]\simeq[\Lie(I)_{w,\br}/I]$, the assertion
follows.

(c) The smoothness assertion follows from the fact that $[\wt{\fC}/\cL G]\simeq[\Lie(I)/I]$ is smooth, while the smallness assertion follows from the combination of (b) and Corollaries \ref{C:dimred}, \ref{C:topproper}(b) and \ref{C:codim1}.
\end{proof}

\begin{Emp} \label{E:affsprfib}
{\bf The affine Springer fibration.}
(a) Recall (see \re{topnilp}(a)) that $\cL^+(\fc)_{\tn}\subset\clp(\fc)$ is an fp-closed subscheme, and let
$\cL^+(\fc)_{\tn,\bullet}\subset\clp(\fc)_{\bullet}$,\label{N:clpfcbullet} $\fC_{\tn,\bullet}\subset \fC_{\bullet}$,\label{N:fCtnbullet}  $\wt{\fC}_{\tn,\bullet}\subset \wt{\fC}_{\bullet}$,\label{N:wtfCtnbullet} $\Lie(I)_{\tn}\subset\Lie(I)$,\label{N:lieitnbullet} etc.  be the preimages
of $\clp(\fc)_{\tn}\subset\clp(\fc)$. In particular, $\iota_{\tn}:[\fC_{\tn,\bullet}/\cL G]\to[\fC_{\bullet}/\cL G]$ is an fp-closed embedding.

(b) Using \re{topnilp}(c) we see that $\cL^+(\fc)_{\tn,\bullet}\subset\clp(\fc)_{\bullet}$ is a union of all strata $\fc_{w,\br}$ such that $(w,\br)>0$. Hence the same also holds for $\fC_{\tn,\bullet}\subset \fC_{\bullet}$ and $\Lie(I)_{\tn,\bullet}\subset\Lie(I)_{\bullet}$,
therefore $\{[\fC_{w,\br}/\cL G]_{\red}\}_{(w,\br)>0}$ is a constructible stratification of $[\fC_{\tn,\bullet}/\cL G]$.

(c) Let $\ov{\frak{p}}_{\tn,\bullet}:[\wt{\fC}_{\tn,\bullet}/\cL G]\to [\fC_{\tn,\bullet}/\cL G]$\label{N:ovfraktnbullet} be the restriction of $\ov{\frak{p}}$.
\end{Emp}

\begin{Lem} \label{L:sprsm}
(a) The collection $\{[\fC_{w,\br}/\cL G]_{\red}\}_{(w,\br)>0}$ forms a constructible stratification of $[\fC_{\tn,\bullet}/\cL G]$, making it a placidly stratified $\infty$-stack.

(b) Every $[\wt{\fC}_{w,\br}/\cL G]\subset[\wt{\fC}_{\tn,\bullet}/\cL G]$ is an fp-locally closed subscheme of pure codimension
$b^+_{w,\br}$.

(c) The $\infty$-stack $[\wt{\fC}_{\tn,\bullet}/\cL G]$ is smooth, the fibration $\ov{\frak{p}}_{\tn,\bullet}:[\wt{\fC}_{\tn,\bullet}/\cL G]\rightarrow [\fC_{\tn,\bullet}/\cL G]$ is
ind-fp-proper and semi-small.

(d) Moreover, a stratum $[\fC_{w,\br}/\cL G]$ is $\ov{\frak{p}}_{\tn,\bullet}$-relevant if and only if $\kt_{w,\br^+_w}\subset\cL^+(\ft_w)_{\tn}$ is an open stratum.
\end{Lem}
\begin{proof}
Assertion (a) follows immediately from \rl{sprsmall}(a). Next, since $\Lie(I)_{\tn}\subset\Lie(I)$ is an fp-closed subscheme of pure codimension $r$, assertion (b) follows from \rl{sprsmall}(b). Then, the smoothness assertion follows from the isomorphism $[\wt{\fC}_{\tn,\bullet}/\cL G]\simeq [\Lie(I)_{\tn,\bullet}/\Lie(I)]$ and the smoothness of $\Lie(I)_{\tn}$. The ind-fp-properness of $\ov{\frak{p}}_{\tn,\bullet}$ follows from that for $\ov{\frak{p}}_{\bullet}$. The remaining assertions now follow from Corollaries \ref{C:dimred}, \ref{C:topproper}(b) and \ref{C:codim}.
\end{proof}


\begin{Emp} \label{E:minsrt}
{\bf Relevant strata.} (a) For every $w\in W$, the affine scheme $\cL^+(\ft_w)_{\tn}$ is irreducible (see \re{topnilp}(b)). Therefore there exists a unique open stratum $\kt_{w,\br^+_w}\subset\cL^+(\ft_w)_{\tn}$.

(b) It follows from \rl{sprsm}(d) that $\{[\fC_{w,\br_{w}^+}/\cL G]_{\red}\}_w$ are all $\ov{\frak{p}}_{\tn,\bullet}$-relevant strata.
\end{Emp}

\begin{Emp} \label{E:conjecture}
{\bf Conjecture.} (a) The union $\cup_{w}\fc_{w,\br_{w}^+}$ is open in $\cL^+(\fc)_{\tn}$. More precisely, $\{\fc_{w,\br_{w}^+}\}_{w\in W}$ is
a constructible stratification of a certain fp-open subscheme $\cL^+(\fc)_{\tn,+}\subset\cL^+(\fc)_{\tn,\bullet}$.

(b)  Assume that Conjecture (a) holds, and let $\fC_{\tn,+}\subset \fC_{\tn,\bullet}$ be the preimage of  $\cL^+(\fc)_{\tn,+}$. Then $[\fC_{\tn,+}/\cL G]\subset [\fC_{\tn,\bullet}/\cL G]$ is an open union of strata.
\end{Emp}

\begin{Emp}
{\bf Example.}
Assume that $G=\SL_{2}$. In this case, $\fc\simeq\B{A}^1$,  $\cL^+(\fc)_{\tn}\subset\clp(\fc)$ is the locus
$\clp(\B{A}^1)_{\geq 1}$, and $\cL^+(\fc)_{\tn,+}\subset \cL^+(\fc)_{\bullet}$ is the locus $\clp(\B{A}^1)_{\geq 1}\cap\clp(\B{A}^1)_{\leq 2}$, which is the union of two strata $\clp(\B{A}^1)_{1}$ and $\clp(\B{A}^1)_{2}$.
In particular, our conjecture holds in this case.
\end{Emp}

\part{Sheaves on $\infty$-stacks and perverse $t$-structures}

\section{Categories of $\ell$-adic sheaves on $\infty$-stacks}

\subsection{Limits and colimits of $\infty$-categories}

\begin{Emp} \label{E:notconv}
{\bf Notation and convention.} Let $k$ be a field, and $\ell$ \label{N:ell} be a prime different from the characteristic of $k$.

(a) All categories appearing in this work are $\infty$-categories, all functors are functors between $\infty$-categories, and all limits and colimits are the homotopical one. In particular, ordinary categories are viewed as $\infty$-categories. We say that a morphism in an $\infty$-category $\cC$ is an {\em isomorphism}, if it is an isomorphism in the homotopy category of $\cC$.

(b) Let $\Catst$ \label{N:catst} be the $\infty$-category, whose objects are stable $\bql$-linear small $\infty$-categories, and morphisms are exact functors, that is, those functors that preserve finite colimits.

(c) Let $\PrCat$ \label{N:prcat} be the $\infty$-category, whose objects are stable $\bql$-linear presentable $\infty$-categories (see \cite[5.5.0.1]{Lu}), and morphisms are continuous functors, that is, functors commuting with all small colimits.

(d) Recall that the $\infty$-categories $\PrCat$ and $\Catst$ have all small limits and filtered colimits (see \cite[4.2.4.8, 5.5.3.13, 5.5.3.18]{Lu}, \cite[1.1.4.4, 1.1.4.6]{HA})
and there is a natural functor
\[
\Ind:\Catst\rightarrow\PrCat:\cC\mapsto \Ind(\cC).\label{N:ind}
\]

(e) Notice that functor $\Ind$ from (d) and functor, which associates to an $\infty$-category $\cC$ its homotopy category, commute with all small filtered colimits (compare \cite[5.3.5.10]{Lu}, \cite[1.9.2]{DG} and \cite{Ro}).
\end{Emp}

\begin{Emp} \label{E:adjthm}
{\bf The limit=colimit theorem.}
(a) Let $\cI$ be a small category and $\Psi:\cI\rightarrow\PrCat$ a functor.
In particular, for every $i\in \cI$, we are given an $\infty$-category $\cC_{i}$ and for every morphism $(i\stackrel{\alpha}{\rightarrow} j)\in \cI$ we are given a continuous functor $\psi_{\al}\in\Funct_{\cont}(\cC_{i},\cC_{j})$.

(b) Suppose that for every morphism $\al:i\rightarrow j$ in $\cI$, the functor $\psi_{\al}$ admits a continuous right adjoint $\phi_{\al}$.
Since adjoints are compatible with compositions, the data $(\cC_i,\phi_{\al})$ extends to a functor
$\Phi:\cI^{\op}\rightarrow\PrCat$ (see \cite[5.5.3.4]{Lu}).
\end{Emp}

The following result allows to rewrite a colimit as a limit and vice versa (see \cite[5.5.3.3]{Lu} or \cite[Sections 1.7-1.9]{DG}).

\begin{Thm} \label{T:lu1}
The colimit \[\cC:=\colim\Psi=\colim_{i\in \cI} \cC_{i}\in\PrCat\] exists and is canonically equivalent to the limit
\[\wh{\cC}:=\lim\Phi=\lim_{i\in \cI^{\op}} \cC_{i}\in\PrCat.\]
Moreover, the equivalence $\cC\isom \wh{\cC}$ is uniquely characterized by the condition that for every $i\in \cI$ the evaluation functor
$\ev_{i}:\wh{\cC}\to\cC_{i}$\label{N:evi} is the right adjoint to the tautological functor
$\ins_{i}:\cC_{i}\to\cC$.\label{N:insi}
\end{Thm}

\begin{Emp} \label{E:filt}
{\bf Filtered case.} Assume  $\cI$ is filtered. Then one shows that for every $i,j\in\cI$ the composition $\ev_j\circ\ins_i:\cC_i\to\cC\isom\wh{\cC}\to\cC_j$ can be written as a colimit
\[
\ev_j\circ\ins_i\simeq\colim_{\al:i\to k,\beta:j\to k}\phi_{\beta}\circ\psi_{\al}
\]
(compare \cite{Ro}). This gives another description of the equivalence $\cC\isom \wh{\cC}$ in this case.
\end{Emp}

\begin{Cor} \label{C:lurie}
For every object  $c\in\cC$, the assignment $i\mapsto \ins_{i}\circ~\ev_{i}(c)\in\cC$ gives rise to the functor
$\cI\to\cC$, and the canonical map
\begin{equation}
\colim_{i\in\cI}\ins_{i}\circ~\ev_{i}(c)\to c
\label{isoga}
\end{equation}
is an isomorphism.
\end{Cor}

\begin{proof}
Though the assertion is standard among specialists (compare \cite[0.8.3]{Gai}), we sketch the argument for the convenience of the reader.
Since $\cC\simeq\lim_{i\in \cI} \cC_{i}$, for every $d\in\cC$, we have a natural isomorphism from the mapping space $\map_\cC(c,d)$ to
\[
\lim_{i\in\cI}\map_{\cC_i}(\ev_i(c),\ev_i(d))\simeq \lim_{i\in\cI}\map_{\cC}(\ins_i\circ\ev_i(c),d)\simeq \map_{\cC}(\colim_{i\in\cI}\ins_{i}\circ~\ev_{i}(c),d),
\]
the first of which follows from the adjunction of $\ev_{i}$ and $\ins_{i}$, and the second one by the definition of the colimit. The assertion now follows from the Yoneda lemma
(see \cite[Proposition 5.1.3.1]{Lu}).
\end{proof}

\begin{Emp} \label{E:compgen}
{\bf Compactly generated case.}
In the situation of \re{adjthm}(a), assume that each $\cC_{i}$ is compactly generated, and denote by $\cC_{i}^{c}\subset\cC_i$ be the sub-category of compacts objects.

(a) Assume in addition that each $\psi_{\al}$ preserves compact objects. Then the functor
$\Psi$ defines a functor $\cI\to\Catst:i\mapsto\cC_{i}^{c}$, and we have a natural equivalence
$\cC\simeq\Ind(\colim_{i\in \cI} \cC_{i}^{c})$ (compare \re{notconv}(e)). 	
In particular, $\cC$ is compactly generated.

(b) Notice that assumption (a) is actually equivalent to the assumption of \re{adjthm}(b).
\end{Emp}


We finish this subsection by recalling a general result about existence of adjoints in a limit and colimit of categories.

\begin{Emp} \label{E:ass}
{\bf Assumptions.} (a) Let $\Cat_{\ell}$ \label{N:catell} be either $\Catst$ or $\PrCat$. Let $\cI$ be a small category, and let $\cD_{\cdot},\cC_{\cdot}$ be two functors
$\cI\to\Cat_{\ell}$. In particular, we are given categories $\cC_i,\cD_i\in \Cat_{\ell}$ and functors $\cC_{\al}:\cC_i\to\cC_j$ and $\cD_{\al}:\cD_i\to\cD_j$
for every morphism $\al:i\to j$ in $\cI$.

(b) Let $\Phi_{\cdot}:\cC_{\cdot}\to\cD_{\cdot}$ be a morphism in $\Fun(\cI,\Cat_{\ell})$. Then  $\Phi_{\cdot}$ gives rise to
\begin{itemize}
\item a functor  $\Phi_i:\cC_i\to\cD_i$ for every $i\in\cI$ and
\item an equivalence $\Phi_{\al}:\cD_{\al}\circ\Phi_i\simeq \Phi_j\circ\cC_{\al}$ for every morphism
$\al:i\to j$ in $\cI$.
\end{itemize}

(c) Assume that
\begin{itemize}
\item For every $i\in\cI$ the morphism $\Phi_i:\cC_i\to\cD_i$ has a left adjoint $\Psi_i$.

\item (Beck--Chevalley condition) For every morphism $\al:i\to j$ in $\cI$ the base change morphism
$BC_{\al}:\Psi_j\circ \cD_{\al}\to\cC_{\al}\circ \Psi_i$, obtained from the counit map
$\cD_{\al}\to \cD_{\al}\circ\Phi_i\circ\Psi_i\simeq  \Phi_j\circ\cC_{\al}\circ \Psi_i$ by adjointness, is an equivalence.
\end{itemize}

\end{Emp}

The following standard assertion will be central for what follows.

\begin{Prop}\label{P:adjlu}
Assume that we are in the situation of \re{ass}.

(a) The collection of $\Psi_i$ and $BC_{\al}$ can be upgraded to a morphism of functors $\Psi_{\cdot}:\cD_{\cdot}\to \cC_{\cdot}$.

(b) The limit functor $\wh{\Phi}=\lim_{i\in\cI}\Phi_i:\lim_{i\in\cI}\cC_i\to \lim_{i\in\cI}\cD_i$ has a left adjoint $\wh{\Psi}$, and the natural base change morphism
\begin{equation} \label{Eq:adjlim}
\Psi_i\circ\ev_i^{\cD}\rightarrow\ev_i^{\cC}\circ\wh{\Psi}
\end{equation}
is an equivalence for every $i\in\cI$.

(c)  Assume that $\cI$ is filtered. Then the colimit functor $\Phi:\colim_{i\in\cI}\cC_i\to \colim_{i\in\cI}\cD_i$ has a left adjoint $\wh{\Psi}$,  and the natural base change morphism
\begin{equation} \label{Eq:adjcolim}
\Psi\circ\ins_i^{\cD}\rightarrow\ins_i^{\cC}\circ~\Psi_{i}
\end{equation}
is an equivalence for every $i\in\cI$.
\end{Prop}

\begin{Emp} \label{E:inf2}
{\bf Remarks.} (a) One does not need the assumption that $\cI$ is filtered in \rp{adjlu}(c). However, in this case the assertion reduces to a corresponding assertion about homotopy categories (see \re{notconv}(e)), which is an easy exercise.

(b) The notion of adjoint functors can be generalized to morphisms in an arbitrary $(\infty,2)$-category (\cite{GR}).
One can show that in the situation of \re{ass} morphism $\Phi_{\cdot}$ in the
$(\infty,2)$-category $\Fun(\cI,\Cat_{\ell})$ has a left adjoint $\Psi_{\cdot}:\cD_{\cdot}\to\cC_{\cdot}$ such that the induced morphism
$\Psi_i\circ\ev_i^{\cD}\rightarrow\ev_i^{\cC}\circ\Psi_{\cdot}$ of functors $\cD_{\cdot}\to \cC_i$ is an equivalence for every $i\in\cI$.
Having this, to get \rp{adjlu} one has to observe that the functors $\lim:\Fun(\cI,\Cat_{\ell})\to\Cat_{\ell}$ and $\colim:\Fun(\cI,\Cat_{\ell})\to\Cat_{\ell}$ are functors of
$(\infty,2)$-categories.
\end{Emp}

%

\subsection{Sheaves on qcqs algebraic spaces}


\begin{Emp} \label{E:basicfunc}
{\bf Sheaves on algebraic spaces of finite type.}
Let $\Algft_k$\label{N:algft} be the category of algebraic spaces of finite type over $k$.

(a) To every $X\in\Algft_k$ one associates a stable $\infty$-category $\cD_{c}(X):=\cD^{b}_{c}(X,\bql)$\label{N:DcX}
whose homotopy category $D_c(X)$ is $D^{b}_{c}(X,\bql)$ (compare \cite{LZ1}, \cite{LZ2} or \cite{GL}).

(b) For every morphism $f:X\to Y$ in $\Algft_k$ we have two pairs $(f^*,f_*)$ and $(f_!,f^!)$ of adjoint functors. Moreover, we have natural  isomorphisms $f_!\simeq f_*$, when $f$ is proper, and $f^*\simeq f^!$, when $f$ is \'etale. In particular, $f^*$ is a left adjoint of $f_!$, when $f$  is proper,
and $f_*$ is a right adjoint of $f^!$, when $f$ is \'etale.
\end{Emp}


%
%


\begin{Emp} \label{E:catqcqs}
{\bf Construction.}
(a) Note that the correspondence $X\mapsto \cD_{c}(X),f\mapsto f^!$ from \re{basicfunc} naturally upgrades to a functor of $\infty$-categories $\cD_{c}=\cD_c^!:(\Algft_k)^{\op}\rightarrow\Catst$ \label{N:Dc} (compare \cite{LZ1}, \cite{LZ2} or \cite{GL}). We also consider functor $\cD:=\Ind\circ\cD_{c}:(\Algft_k)^{\op}\to\PrCat$\label{N:D}.

(b) Let $\Alg^{\qcqs}_k$ be the category of quasi-compact and quasi-separated algebraic spaces over $k$.
%
We denote by $\cD_{c}:(\Alg^{\qcqs}_k)^{\op}\rightarrow\Catst$ and $\cD:(\Alg^{\qcqs}_k)^{\op}\rightarrow\PrCat$ the functors, obtained
by applying the left Kan extension to the functors $\cD_c$ and $\cD$ from (a). Then for every morphism
$f:X\to Y$ in $\Alg^{\qcqs}_k$ we have functors $f^!:\cD_{c}(Y)\rightarrow\cD_{c}(X)$ and $f^!:\cD(Y)\rightarrow\cD(X)$.
\end{Emp}

\begin{Emp}
{\bf Categorical input.} (a) Notice that the functor $\cD_c^!:(\Algft_k)^{\op}\rightarrow\Catst$ \re{catqcqs}(a) is the only categorical input we are using to develop our theory. Moreover, it suffices to use the restriction of $\cD_c^!$ to $(\Affft_k)^{\op}$, from which the full
$\cD_c^!$ can be reconstructed as the right Kan extension.

(b) Specifically, we are not using the extension of this functor to the category of correspondences, developed in \cite{GR}, which would simplify and strengthen some of our results and constructions but would require a much heavier categorical machinery of $(\infty,2)$-categories.
\end{Emp}

\begin{Emp} \label{E:remqcqs}
{\bf Remarks.}
(a)  Recall that every $X\in\Alg^{\qcqs}_k$ has a presentation as a filtered limit $X\simeq\lim_{\al} X_{\al}$, where each $X_{\al}\in\Algft_k$, and all transition maps are affine (see, for example, \cite[Theorem D]{Ry2}). Moreover, since $\Algft_k$ has finite limits, one conclude that the category $(X/\Algft_k)^{\op}$, whose objects are morphisms $X\to Y$ with $Y\in\Algft_k$, is filtered, and we have a canonical presentation
$X\simeq\colim_{X\to Y} Y$.

(b) By the explicit description of the left Kan extension, for every $X\in\Alg^{\qcqs}_k$ we have a natural equivalence
$\cD_c(X)\simeq\colim_{X\to Y}\cD_c(Y)$ and similarly, for $\cD(X)$.
More generally, for every presentation $X\simeq\lim_{\al} X_{\al}$ as in (a) the canonical morphism $\colim^!_{\al}\cD_c(X_{\al})\to \cD_c(X_{\al})$ is an equivalence, and similarly for $\cD(X)$.

(c) Since functor $\Ind$ from \re{notconv}(e) commutes with filtered colimits, we have a natural equivalence $\cD(X)\simeq\Ind(\cD_{c}(X))$, thus $\cD(X)$ is compactly generated. Since passage to homotopy categories commutes with filtered colimits, for every presentation $X\simeq\lim_{\al}X_{\al}$ we have an equivalence $D_c(X)\simeq\colim^!_{\al}D_c(X_{\al})$, and similarly for $\cD$.

(d) Recall that if $f:X'\to X$ is an fp-morphism of qcqs algebraic spaces, then for every presentation $X\simeq\lim_{\al} X_{\al}$ as in (c),
there exists an index $\al$, a finitely presented map $f_{\al}:X'_{\al}\to X_{\al}$ and an isomorphism $X'\isom X\times_{X_{\al}}X'_{\al}$ (see \cite[Proposition B.2(ii)]{Ry2}). Then $X'$ can be written as a filtered limit $X'\simeq\lim_{\beta\geq \al} X'_{\beta}$ with $X'_{\beta}:=X'_{\al}\times_{X_{\al}}X_{\beta}$, thus we have a canonical equivalence  $\cD_c(X')\simeq\colim^!_{\beta\geq \al}\cD_c(X'_{\beta})$.

(e) By definition, for every morphism of qcqs algebraic spaces $f:X'\to X$ we have a $!$-pullback functor $f^!:\cD_c(X)\to\cD_c(X')$, but the other three functors $f_!,f_*,f^*$ are not defined in general.
\end{Emp}

\begin{Prop}\label{P:adj}
Let $f:X'\rightarrow X$ be an fp-morphism of qcqs algebraic spaces.

(a) Assume that either $f$ is fp-proper or $X$ admits a placid presentation (see \re{plpres}). Then $f^!:\cD_c(X)\to\cD_c(X')$ has a left adjoint $f_!$.

(b) Assume that $f$ is \'etale. Then  $f^!$ has a right adjoint $f_*$.

(c) Assume that  $f$ is fp-proper and $X$ admits a placid presentation. Then the functor $f_!$ has a left adjoint $f^*$.
\end{Prop}

\begin{proof}
As in \re{remqcqs}(d), we can choose presentations $X\simeq\lim_{\al\in \cI} X_{\al}$ and $X'\simeq\lim_{\al\in \cI} X'_{\al}$.
Moreover, by the standard limit arguments (see \cite[Proposition B.3]{Ry2} and references within), we can assume that each projection $f_{\al}:X'_{\al}\to X_{\al}$ is proper (resp. \'etale), if $f$ is such,
and the transition maps $\pi_{\beta,\al}:X_{\beta}\to X_{\al}$ are smooth, if $X$ admits a placid presentation. Since  $\cD_c(X)\simeq\colim^!_{\al}\cD_c(X_{\al})$
and  $\cD_c(X')\simeq\colim^!_{\al}\cD_c(X'_{\al})$, all assertions will be deduced from \rp{adjlu}(c). Since adjoints are known to exist for algebraic spaces of finite type over $k$, we will only have to show that the Beck--Chevalley condition in \re{ass}(c) is satisfied. Consider the Cartesian diagram
\begin{equation}
\begin{CD}
X'_{\beta} @>f_{\beta}>> X_{\beta}\\
@V\pr'_{\beta,\al}VV @VV\pr_{\beta,\al}V\\
X'_{\al} @>f_{\al}>> X_{\al}
\end{CD}
\end{equation}

(a) We want to apply \rp{adjlu}(c) to the morphism $\Phi_{\cdot}=f_{\cdot}^!$ of functors
\[
\cI^{\op}\to\Catst:\al\mapsto \cD_c(X_{\al}),(\al\to\beta)\mapsto \pr_{\beta,\al}^!.
\]
We have to show that the base change morphism $(f_{\beta})_!\pr'^!_{\beta,\al}\to \pr'^!_{\beta,\al}(f_{\al})_!$ is an isomorphism, when $f_{\al}$ is proper (resp. $\pr_{\beta,\al}$ is smooth).
But this follows from proper (resp. smooth) base change.

(b) Now we want to apply \rp{adjlu}(c) to the morphism $\Phi_{\cdot}=f_{\cdot}^!$ of functors
\[
\cI^{\op}\to\Catst:\al\mapsto \cD_c(X_{\al})^{\op}, (\al\to \beta)\mapsto \pr_{\beta,\al}^!.
\] The assumptions of \re{ass}(c) are satisfied since the base change $\pr^!_{\beta,\al}(f_{\al})_*\to (f_\beta)_*\pr'^!_{\beta,\al}$ is an isomorphism.

(c) We want to apply \rp{adjlu}(c) to the morphism $\Phi_{\cdot}=(f_{\cdot})_!$ of functors
\[
\cI^{\op}\to\Catst:i\mapsto \cD_c(X_{\al}),(\al\to\beta)\mapsto \pr_{\beta,\al}^!.
\]
This follows from the fact that the base change map $f_{\beta}^*\pr^!_{\beta,\al}\to \pr'^!_{\beta,\al}f_{\al}^*$ is an isomorphism,
when $\pr_{\beta,\al}$ is smooth (by smooth base change).
\end{proof}

%

\begin{Emp}
{\bf Remark.} Functors $f_*$ and $f^*$ can be defined in more cases. For example, $f_*$ can be defined when
$f$ is finitely presented. Moreover, $f_*$ has a left adjoint, if $f$ is finitely presented, while $X$ and $Y$ admit placid presentations.
On the other hand, these facts are deeper and will not be used in this work.
\end{Emp}

The adjoint maps from \rp{adj} satisfy the following base change formulas.

\begin{Prop}\label{P:base}
Consider Cartesian diagram
\begin{equation} \label{Eq:basicCD}
\begin{CD}
\wt{X} @>\wt{f}>>\wt{Y}\\
@V a VV @VV b V\\
X @>f>> Y
\end{CD}
\end{equation}
of qcqs algebraic spaces such that $b$ is finitely presented.

(a) If $b$ is \'etale, then the base change morphism $f^! b_*\to a_*\wt{f}^!$ is an isomorphism.

(b) If $b$ is proper, then the base change morphism $a_!\wt{f}^!\to f^! b_!$ is an isomorphism.

(c) If $Y$ admits a placid presentation and $f$ is strongly pro-smooth,
then the base change morphism $a_!\wt{f}^!\to f^! b_!$ is an isomorphism.

(d) If $b$ is fp-proper, $Y$ admits a placid presentation and $f$ is  strongly pro-smooth,
then the base change morphism $\wt{f}^!b^*\to a^*f^!$, induced from the isomorphism of (b), is an isomorphism.
\end{Prop}

\begin{proof}
(a) We want to show that the map $f^! b_*(K)\to a_*\wt{f}^!(K)$ is an isomorphism for every object $K\in\cD_c(\wt{Y})$.
Assume first that $Y$ and $\wt{Y}$ are of finite type. Then we can assume that $Y=X_{\al}$ for some presentation $X\simeq \lim_{\al} X_{\al}$ of $X$. Then $\wt{X}$ has a presentation $\wt{X}\simeq\lim_{\al'\geq\al}(X_{\al'}\times_Y\wt{Y})$, so our assertion follows from \rp{adjlu}(c), because our base change is simply the map \form{adjcolim}.

In the general case, choose presentations $Y\simeq\lim_{\al} Y_{\al}$ and
$\wt{Y}\simeq\lim_{\al} \wt{Y}_{\al}$ is \re{remqcqs}(d) and choose $\al$ such that $K$ is a pullback of some $K_{\al}\in \cD_c(\wt{Y}_{\al})$.
Then the assertion for $K$ follows from the assertion for $K_{\al}$ applied to the right and the exterior square of the Cartesian diagram
\[
\xymatrix{\wt{X}\ar[r]^{\wt{f}}\ar[d]_{a}&\wt{Y}\ar[d]^{b}\ar[r]^{\wt{p}_{\al}}&\wt{Y}_{\al}\ar[d]^{b_{\al}}\\
X\ar[r]^{f}&Y\ar[r]^{p_{\al}}&Y_{\al}}.
\]

Namely, we have to show that the morphism $f^! b_*\wt{p}_{\al}^!(K_{\al})\to a_*\wt{f}^!\wt{p}_{\al}^!(K_{\al})$  is an isomorphism.
But for this it suffices to show that in the composition
\[
f^!{p}_{\al}^!(b_{\al})_*(K_{\al})\to f^! b_*\wt{p}_{\al}^!(K_{\al})\to a_*\wt{f}^!\wt{p}_{\al}^!(K_{\al})
\]
the first map and the composition are isomorphisms. In other words, we have to show the assertion for $p_{\al}:Y\to Y_{\al}$ and
$p_{\al}\circ f:X\to Y_{\al}$ instead of $f$.

(b)-(d) The arguments are essentially identical to that of (a), except that in the case when $Y$ admits a placid presentation
we only consider presentations $Y\simeq\lim_{\al} Y_{\al}$, where all transition maps $Y_{\beta}\to Y_{\al}$ are smooth.
%
%
%
\end{proof}


\begin{Emp} \label{E:sheafcond}
{\bf Sheaf property.} Note that $\cD_c:(\Alg^{\qcqs}_k)^{\op}\to\Catst$ and $\cD:(\Alg^{\qcqs}_k)^{\op}\to\PrCat$ are ``sheaves'' in the \'etale topology (and even for $h$-topology see, for example, \cite{Ra} or \cite{RS}). In other words, for an \'etale covering $f:X\to Y$ in $\Alg^{\qcqs}_k$, the natural functor $\cD_c(Y)\to\lim_{[m]}\cD_c(X^{[m]})$ is an equivalence, and similarly for $\cD$.
%
%
\end{Emp}

\subsection{Sheaves on $\infty$-stacks}

\begin{Emp} \label{E:shpr}
{\bf Construction and basic properties.}
(a) Applying the right Kan extension to the functors $\cD_c$ and $\cD$ from \re{catqcqs},  we get functors
\[
\cD_{c}:\PShv(\Alg_k)^{\op}\rightarrow\Catst\text{ and }\cD:\PShv(\Alg^{\qcqs}_k)^{\op}\rightarrow\PrCat.\label{N:D2}
\]

(b) By definition, for every morphism $f:\cX\to\cY$ of $\infty$-prestacks, we have pullbacks functors $f^!:\cD_{\cdot}(\cY)\to \cD_{\cdot}(\cX)$. In particular, for every $\infty$-prestack $\cX$ we have a dualizing sheaf $\om_{\cX}\in\cD_c(\cX)$,\label{N:omX} defined to be the !-pullback of $\qlbar\in\cD_c(\pt)$.

(c) Using sheaf property \re{sheafcond}, the functors $\cD_{\cdot}$ from (a) factor through the category
$\Shv(\Alg_k)$ of sheaves in the \'etale topology. Since the pullback $\iota^*:\Shv(\Alg^{\qcqs}_k)\to \St_k$,
corresponding to the inclusion $\iota:\Aff_k\hra\Alg^{\qcqs}_k$, is an equivalence (compare \re{rem infst}), functors $\cD_{\cdot}$
can be viewed as functors from $\St_k$.

(d) By basic properties of the right Kan extension, the functors $\cD_{\cdot}$ from (a) commute with all small limits. Using \re{sheafcond} again,
we deduce that functors $\cD_{\cdot}$ from (c) commute with all small limits as well. In particular, the natural functor $\C{D}_{\cdot}(\C{X})\to\lim_{X\to\C{X}}\C{D}_{\cdot}(X)$, taken over all morphisms $X\to\cX$ with $X\in\Aff_k$, is an equivalence. Therefore
$\C{D}_{\cdot}$ is equivalent to the right Kan extension of its restriction to $\Aff_k^{\op}$.

(e) Let $f:\cX\to \cY$ be an epimorphism of $\infty$-stacks. Then $\cY$ is the colimit of the \v{C}ech complex with terms $\cX^{[m]}$ (see \re{chech}(c)). Hence we conclude by (c) that $\cD_{\cdot}(\cY)$ is the limit of the corresponding complex with terms  $\cD_{\cdot}(\cX^{[m]})$. In particular, the pullback $f^!:\cD_{\cdot}(\cY)\to \cD_{\cdot}(\cX)$ is conservative.
\end{Emp}


\begin{Emp} \label{E:remprest}
{\bf Remark.}
The inclusion $\cD_c(\C{X})\hra\cD(\C{X})$ induces a functor $\Ind(\cD_c(\C{X}))\to \cD(\C{X})$, which is an equivalence, when
$\cX\in\Alg^{\qcqs}_k$, but not in general.
\end{Emp}

\begin{Emp} \label{E:indsch2}
{\bf Sheaves on ind-algebraic spaces.}
(a) Let $X$ be an ind-algebraic space with a presentation $X\simeq\colim_{\al} X_{\al}$. By definition, we have a canonical equivalence $\cD(X)\simeq\lim^!_{\al}\cD(X_{\al})$.

(b) Recall that for every fp-closed embedding $i:X_{\al}\to X_{\beta}$ in $\Alg^{\qcqs}_k$ the functor $i^!$ has a left adjoint $i_!$ (see \rp{adj}(a)). Then it follows from \rt{lu1} that we have a natural equivalence $\cD(X)\simeq\colim^!_{\al}\cD(X_{\al})$.

(c) It follows from (b) and \re{compgen} that $\cD(X)$ is compactly generated, and we have a natural equivalence $\cD(X)\simeq\Ind(\colim^!_{\al}\cD_c(X_{\al}))$.
\end{Emp}

\begin{Emp} \label{E:remindsch}
{\bf Remark.} In the situation of \re{indsch2} we have a fully faithful functor $\colim^!_{\al}\cD_c(X_{\al})\hra\cD_c(X)$, which is not an equivalence. In particular, we have natural functors
\[
\Ind(\colim^!_{\al}\cD_c(X_{\al}))\hra\Ind(\cD_c(X))\to\cD(X),
\]
the first of which is fully faithful, the second one is essentially surjective, and the composition is an equivalence.
\end{Emp}

Topological equivalences do not change categories of sheaves:

\begin{Lem} \label{L:invtop}
If $f:X\to Y$ is a universal homeomorphism of affine schemes, then the pullback functors
$f^{!}:\cD_c(Y)\to\cD_c(X)$ and $f^{!}:\cD(Y)\to\cD(X)$ are equivalences.
\end{Lem}

\begin{proof}
Since $\cD\simeq \Ind\cD_c$, the assertion for $\C{D}$ follows from that for $\cD_c$. Since $f$ is a universal homeomorphism, $X$ has a presentation as filtered limit
$X\simeq\lim_{X'}X'$, where each $X'\to Y$ is an fp-universal homeomorphism (see \cite[Tag 0EUJ]{Sta}).
Then $\cD_c(X)\simeq\colim_{X'}\cD_c(X')$ is a filtered colimit, so it suffices to show
that each $\cD_c(Y)\to \cD_c(X')$ is an equivalence.

Note that every fp-universal homeomorphism $X'\to Y$ comes from a universal homeomorphism between finite type schemes $X'_0\to Y_0$ by \cite[8.10.5]{EGA}. Writing $Y$ as a limit $Y\simeq\lim Y_{\al}$ over $Y_0$, we get that
$X'\simeq\lim_{\al} X'_{\al}$ with $X'_{\al}=Y_{\al}\times_{Y_0}Y'_0$. Thus it suffices to show that each functor $f_{\al}^!:\cD_c(Y_{\al})\to\cD_c(X'_{\al})$ is an equivalence. Since $f_{\al}:X'_{\al}\to X_{\al}$ is  a universal homeomorphism between finite type affine schemes, the assertion follows from the fact that $f_{\al}$ induces an equivalence between \'etale sites on $X'_{\al}$ and $X_{\al}$.
\end{proof}

\begin{Cor} \label{C:invtop}
(a) For an $\infty$-stack $\cX$, the canonical functors $\pi^!:\cD_c(\cX)\to \cD_c(\cX_{\red})$ and $\pi^!:\cD(\cX)\to \cD(\cX_{\red})$, induced by the projection $\pi:\cX_{\red}\to\cX$, are equivalences.

(b) For a topological equivalence $f:\cX\rightarrow\cY$ between $\infty$-stacks, the pullbacks $f^{!}:\cD_c(\cY)\to\cD_c(\cX)$ and
$f^{!}:\cD(\cY)\to\cD(\cX)$ are equivalences.
\end{Cor}

\begin{proof}
We will write $\cD_{\cdot}$ to treat both $\cD$ and $\cD_c$.

(a) Since $\cX$ as a colimit of affine schemes $\cX\simeq\colim_{X\to\cX} X$, we have $\C{D}_{\cdot}(\cX)\simeq\lim_{X\to\cX}\C{D}_{\cdot}(X)$.
Since functor $\iota_!\iota^*:\cX\mapsto\cX_{\red}$ preserves colimits, we get an isomorphism $\cX_{\red}\simeq\colim_{X\to\cX} X_{\red}$, hence an equivalence $\C{D}_{\cdot}(\cX_{\red})\simeq\lim_{X\to\cX}\C{D}_{\cdot}(X_{\red})$. Therefore it suffices to show the induced map $\cD_{\cdot}(X)\to \cD_{\cdot}(X_{\red})$ is an equivalence. Since $X_{\red}\to X$ is a universal homeomorphism of affine schemes, the assertion follows from \rl{invtop}.

(b) Follows immediately from (a).
%
\end{proof}





\begin{Prop}\label{P:locindpr}
Let $f:\cX\rightarrow \cY$ be an ind-fp-proper morphism of $\infty$-stacks (see \rd{pspr}). Then the pullback $f^{!}$ has a left adjoint $f_{!}$, satisfying the base change. In other words, for every Cartesian diagram of $\infty$-stacks
\begin{equation*}
\xymatrix{\wt{\cX}\ar[d]_{\wt{f}}\ar[r]^{g}&\cX\ar[d]_{f}\\\wt{\cY}\ar[r]^{h}&\cY,}
\end{equation*}
the base change map
\begin{equation} \label{Eq:bcind}
\wt{f}_{!}g^{!}\to h^{!}f_{!}
\end{equation}
is an isomorphism.
\end{Prop}

\begin{proof}
Our argument is almost identical to the one outlined in \cite[Proposition 1.5.2]{Gai}.

\vskip 4truept
{\bf Step 1.} It is enough to show the assertion when $\cY$ and $\wt{\cY}$ are affine schemes.
\begin{proof}
A presentation $\cY\simeq\colim_{U\to\cY} U$ of $\cY$ as a colimit of affine schemes induces a presentation $\cX\simeq\colim_{U\to\cY}(\cX\times_\cY U)$, and every $f_U: \cX\times_\cY U\to U$ is ind-fp-proper.
If every $f_U^!$ has a left adjoint $(f_U)_!$, satisfying the base change (for morphisms between affine schemes), then
\rp{adjlu}(b) implies that the left adjoint of $f^!$ exists and satisfies the base change for morphisms $U\to\cY$ with $U$ affine.

Next, for an arbitrary morphism $h:\wt{\cY}\to\cY$, notice that $\cD(\wt{\cY})\simeq\lim_{U\to\wt{\cY}}\cD(U)$ taken over all morphisms $U\to\wt{\cY}$ with affine $U$. Therefore in order to show the base change for $h$ it suffices to show that for every morphism $\al:U\to\wt{\cY}$ the base change morphism
$\al^!\tilde{f}_{!}g^{!}\to \al^!h^{!}f_{!}$ is an isomorphism. Arguing as in \rp{base}(a), it thus suffices to show the base change for the morphisms
$\al:U\to\wt{ \cY}$ and $h\circ\al:U\to\cY$, which was shown above.
\end{proof}

{\bf Step 2.} The assertion holds, if $f$ is fp-proper.

\begin{proof}
Arguing as in Step 1, one reduces the assertion to the case when $\cY$ and $\wt{\cY}$ are affine. In this case, the existence of $f_!$
was shown in \rp{adj}(a), and the base change property was shown in \rp{base}(b).
\end{proof}

{\bf Step 3.} The assertion holds, if $\cY\simeq\colim_{\al}Y_{\al}$ is an ind-algebraic space, and $f$ is the inclusion $f=i_{\al}:Y_{\al}\to \cY$.

\begin{proof}
Since $i_{\al}$ is fp-proper, the assertion follows from Step 2.
\end{proof}

{\bf Step 4.} Completion of the proof.
\vskip 4truept
By Step 1, we can assume that $\cY$ and $\wt{\cY}$ are affine. Choose a presentation $\cX\simeq\colim X_{\al}$ of $\cX$ over  $\cY$, let $i_{\al}:X_{\al}\hra X$ be the embedding, and set $f_{\al}:=f\circ i_{\al}:X_{\al}\to \cY$. By Step 3, the adjoint $(i_{\al})_!$ exists and satisfies the base change.

By the adjoint functor theorem \cite[Corollary 5.5.2.9]{Lu}, to show the existence of $f_!$ it suffices to show that $f^!$ preserves all small limits.
Since $\cD(\cX)\simeq \lim_{\al}\cD(X_{\al})$ and $i_{\al}^!$ preserves all small limits by Step 3, it suffices to show that the composition $f_{\al}^!=i^!_{\al}\circ f^!: \cD(\cY)\to \cD(X_{\al})$
preserves all small limits. Since $f_{\al}$ is fp-proper, the pullback $f_{\al}^!$ has a left adjoint $(f_{\al})_!$ by \rp{adj}(a). Therefore $f_{\al}^!$ preserves all small limits, and the existence of $f_!$ follows.

Recall (see \rco{lurie}) that for every object $K\in\cD(X)$ we have a canonical isomorphism $\colim_{\al}(i_{\al})_!i_{\al}^!K\to K$.
Since all functors in \form{bcind} preserve small colimits, in order to show that \form{bcind}  is an isomorphism, it suffices to check
that the induced map $\tilde{f}_{!}g^{!}(i_{\al})_!\to h^{!}f_{!}(i_{\al})_!$ is an isomorphism. As in the proof of
\rp{base}(a), it suffices to show that $(i_{\al})_!$ and $(f_{\al})_!$ satisfy base change. Since  $f_{\al}$ are fp-proper, the assertion follows from Steps 2 and 3.
\end{proof}



\begin{Emp} \label{E:rempspr2}
{\bf Remark.} Actually, as in \cite{Gai} one can consider a more general notion of {\em pseudo-proper} morphisms,  in which we do not require in \rd{pspr}(a) that the colimit  $\colim_{\al}X_{\al}$ is filtered and impose no restriction on the transition maps.
The assertion \rp{locindpr} also holds for pseudo-proper morphisms as well. Namely, all steps in the argument except Step 3 work word-by-word. Though an analog of Step 3 is not difficult as well, one seems to need a more general categorical framework of $(\infty,2)$-categories to give an honest proof of it.
\end{Emp}


\begin{Prop} \label{P:glue1}
(a) Let $\C{Y}$ be a placid $\infty$-stack, and let $f:\cX\to\cY$ be an fp-representable morphism. Then there exists a left adjoint $f_!:\C{D}(\C{X})\to \C{D}(\C{Y})$ of $f^!:\C{D}(\C{Y})\to \C{D}(\C{X})$. Moreover, if in addition $f$ is proper,
then there exists a left adjoint $f^*:\C{D}(\C{Y})\to \C{D}(\C{X})$ of $f_!$.

(b) Let $h:\wt{\cY}\to\cY$ be a smooth morphism of placid $\infty$-stacks, let $g:\wt{\cX}\to\cX$ and $\wt{f}:\wt{\cX}\to\wt{\cY}$ be the pullbacks of $h$ and $f$, respectively. Then the base change morphism $\wt{f}_!g^!\to h^! f_!$ is an isomorphism. Moreover, if $f$ is fp-proper, then the induced base change morphism
$\wt{f}^*h^!\to g^!f^*$ is an isomorphism as well.
\end{Prop}

\begin{proof}
(a) Recall (see \re{propplst}(a)) that we have a canonical isomorphism $(\colim_{Y\to\cY}Y)\to\cY$, where the colimit runs over smooth morphisms
$Y\to\cY$ from affine schemes $Y$ admitting placid presentations, where the transition maps are strongly pro-smooth. This isomorphism induces an isomorphism $(\colim_{Y\to\cY}X_Y)\to\cX$ with $X_Y:=\cX\times_{\cY}Y$.

Moreover, $f$ induces an fp-representable morphism $X_Y\to Y$ for all $Y$, which is proper if $f$ is such. In particular, $X_Y$ is an algebraic space admitting a placid presentation (by \rl{glplacid}). Furthermore, by \rp{adj}(a),(c) and \rp{base}(c),(d), the assumptions of \rp{adjlu}(b) are satisfied. Therefore there exists an adjoint $f_!$ of $f^!$ (and also $f^*$ of $f_!$, if $f$ is proper), which satisfies the  base change with respect to each $h:Y\to\cY$ as above.

(b) Choose a smooth covering $p=\sqcup_{\al}p_{\al}:\sqcup_{\al}Y_{\al}\to\wt{\cY}$, where each $Y_{\al}$ is an affine scheme admitting a placid presentation, and let $p'_{\al}:X_{\al}\to\wt{\cX}$ be the pullback of $p_{\al}$.
Since  $p^!$ is conservative, to show that $\wt{f}_!g^!\to h^! f_!$ is an isomorphism, it suffices to show that the pullback
$p_{\al}^!\wt{f}_!g^!\to p_{\al}^!h^! f_!$ is an isomorphism for all $\al$.

Arguing as in \rp{base}(a), it suffices to show the base change for morphisms $p_{\al}:Y_{\al}\to\wt{\cY}$ and  $h\circ p_{\al}:Y_{\al}\to\cY$.
Since $Y_{\al}$ is an affine scheme admitting a placid presentation, while $p_{\al}$ and $h\circ p_{\al}$  are smooth, the assertion follows
from the observation at the end of (a). The proof of the second assertion is similar.
\end{proof}


\subsection{Fp-locally closed embeddings}




\begin{Lem} \label{L:glue1}
Let $j:\C{U}\hra\cX$ be an fp-open embedding with a complementary topologically fp-closed embedding $i:\cZ\hra\cX$ (see \re{complopcl}(a)). Then

(a) there exists a right adjoint $j_*$ of $j^!:\C{D}(\C{X})\to \C{D}(\C{U})$, which is continuous, preserves  $\C{D}_c$ and satisfies the base change;

(b) there exists a left adjoint $i_!$ of $i^!:\C{D}(\C{X})\to \C{D}(\C{Z})$, which preserves  $\C{D}_c$ and satisfies the base change;

(c) functors $i_!$ and $j_*$ are fully faithful, and $j^{!}i_{!}\simeq 0$;

(d) for every $K\in\cD(\C{X})$, the unit and counit maps extend to a fiber sequence \[i_{!}i^{!}K\rightarrow K\rightarrow j_{*}j^{!}K.\]
 \end{Lem}

\begin{proof}

(a) A presentation $\cX\simeq\colim_{X\to\cX} X$ of $\cX$ as a colimit of affine schemes, induces a presentation $\C{U}\simeq\colim_{X\to\cX} X_{\C{U}}$, where $X_\C{U}:=X\times_{\cX}\cU$ is an fp-open subscheme of $X$. In particular, $j^!:\C{D}(\C{X})\to \C{D}(\C{U})$ is a limit
$\lim_{X\to\cX} j_X^!: \lim_{X\to\cX} \C{D}(X)\to \lim_{X\to\cX} \C{D}(X_\C{U})$ and similarly for $\C{D}_c$.
Since the pullback $j_X^!:\cD_c(X)\to \cD_c(X_\C{U})$ has a right adjoint (see \rp{adj}(b)), which satisfies the base change (see \rp{base}(a)),
the existence of $j_*$ follows from \rp{adjlu}(b), applies to $\cD_{\cdot}^{\op}$ (as in the proof of \rp{adj}(b)). To show the assertion about the base change, we argue as in \rp{locindpr}. It remains to show that $j_*$ is continuous. Since  $\cD(\cX)\simeq\lim_{X\to\cX} \C{D}(X)$, and each projection $\cD(\cX)\to\cD(X)$ is continuous, it suffices to show that each composition
$\cD(\cU)\overset{j_*}{\lra} \cD(\cX)\to \cD(X)$, or what is the same,  $\cD(\cU)\to\cD(X_{\cU})\overset{(j_X)_*}{\lra}\cD(X)$ is continuous.
Since each $(j_X)_*:\cD(X_{\cU})\to\cD(X)$ is continuous, the assertion follows.

(b) The argument is similar, except we use \rco{invtop}(a) and Propositions \ref{P:adj}(a) and \ref{P:base}(b) instead. Notice that all assertions except the one about $\cD_c$ can be easily deduced from \rco{invtop}(a) and \rp{locindpr}.

(c) We have to show that the counit morphisms $\Id\to i^!i_!$ and $\Id\to j_*j^!$ are isomorphisms, and $j^{!}i_{!}\simeq 0$.
Since all functors are defined as limits of the corresponding functors for qcqs schemes, we immediately reduce to the case of qcqs schemes. In this case, $\C{D}\simeq\Ind\C{D}_c$, so we reduce to the case of $\C{D}_c$. Next, using \rco{invtop}(a), we can assume that
$i:\cZ\to\cX$ is fp-closed, thus is a pullback of a closed embedding of schemes of finite type. Therefore all functors are colimits of
the corresponding functors between schemes of finite type, hence we reduce to this case. In this case, all assertions are standard.

(d) Since $j^{!}i_{!}\simeq 0$, the composition  $i_{!}i^{!}K\rightarrow K\rightarrow j_{*}j^{!}K$ is naturally isomorphic to zero. Let $K'$ be the fiber of the unit map $K\to j_*j^! K$. Then the counit map $i_{!}i^{!}K\rightarrow K$ factors canonically as  $i_{!}i^{!}K\to K'\to K$, and it remains to show that $i_{!}i^{!}K\to K'$ is an isomorphism. Arguing as in (c), we reduce to the case of schemes of finite type, in  which
case the assertion is standard.
\end{proof}


\begin{Emp} \label{E:support}
{\bf Sheaves with support.}\label{I:support}
(a) Let $\cX$ be an $\infty$-stack, let $\cY\subset\cX$ be an $\infty$-substack, and let $\iota:\cX\sm\cY\hra\cX$ be the inclusion (see \re{compl}).
Let $\cD_{\cY}(\cX)\subset \cD(\cX)$ \label{N:DYX} be the full $\infty$-subcategory consisting of $K\in\cD(\cX)$ such that $\iota^!K\simeq 0$, and we say that objects $K\in \cD_{\cY}(\cX)$ are {\em supported on $\cY$}. Since $\iota^!$ is continuous, we deduce that $\cD_{\cY}(\cX)\subset \cD(\cX)$ is closed under all colimits.

(b) For every morphism $f:\cX'\to\cX$ we have an inclusion $f^!(\cD_{\cY}(\cX))\subset \cD_{\cY\times_{\cX}\cX'}(\cX')$. Indeed, this follows from the commutative diagram
\[
\begin{CD}
\cX'\sm(\cY\times_{\cX}\cX') @>>> \cX'\\
@VVV @VVV\\
\cX\sm\cY @>>>\cX.
\end{CD}
\]

(c) A canonical isomorphism $\cD(\cX)\simeq \lim_{X\to\cX}\cD(X)$ induces an isomorphism
\[
\cD_{\cY}(\cX)\simeq \lim_{X\to\cX}\cD_{X\times_{\cX}\cY}(X).
\]
\begin{proof}
We have to show that if $K\in \cD(\cX)$ corresponds to a compatible system $\{K_X\in \cD(X)\}_{X\to\cX}$, then
$K\in \cD_{\cY}(\cX)$ if and only if $K_X\in \cD_{X\times_{\cX}\cY}(X)$ for every $X\to\cX$. The ``only if'' assertion follows from
(b). Conversely, assume that $K_X\in \cD_{X\times_{\cX}\cY}(X)$ for every $X\to\cX$, and we want to show that
$\iota^!K\simeq 0$, that is, for every $a:X\to \cX\sm\cY\subset \cX$ we have $K_X:=a^!K\simeq 0$.
By assumption, $X\times_{\cX}\cY=\emptyset$, thus $K_X\in \cD_{\emptyset}(X)=\{0\}$.
\end{proof}
\end{Emp}

\begin{Lem} \label{L:locclemb}
Let $\eta:\cY\hra\cX$ be a topologically fp-locally closed embedding. Then $\eta^!$ induces an equivalence of categories
$\eta^!:\cD_{\cY}(\cX)\to\cD(\cY)$.
\end{Lem}

\begin{proof}
Since $\cX\simeq\colim_{X\to\cX}X$, we get an isomorphism $\cY\simeq\colim_{X\to\cX}(\cY\times_{\cX}X)$, and hence equivalences
$\cD(\cY)\simeq \lim_{X\to\cX}\cD(X\times_{\cX}\cY)$ and
$\cD_{\cY}(\cX)\simeq \lim_{X\to\cX}\cD_{X\times_{\cX}\cY}(X)$ (see \re{support}(c)). Thus it suffices to show that $\eta$ induces an equivalence
$\eta_X^!:\cD_{X\times_{\cX}\cY}(X)\to\cD(X\times_{\cX}\cY)$ for every morphism $X\to\cX$. In other words, we reduce the assertion to the case when $\cX$ is an affine scheme $X$.

Using \rco{invtop}(a), we can assume that $\eta$ is an fp-locally closed embedding, that is, $\eta$ decomposes as
$Y\overset{j}{\lra}Z\overset{i}{\lra}X$, where $i$ (resp. $j$) is an fp-closed (resp. fp-open) embedding. Next we observe that it is enough
to show the assertion separately for $\eta=i$ and $\eta=j$.

We claim that both assertions easily follow from \rl{glue1}. Assume first that $\eta=i$. Since the left adjoint $i_!$ is fully faithful,
the unit map $\Id\to i^!i_!$ is an isomorphism. So it suffices to show that $i_!$ induces an equivalence
$\cD(Z)\isom\cD_{Z}(X)$. Since $j^!i_!\simeq 0$, the image of $i_!$ lies inside $\cD_{Z}(X)$. Conversely, if  $K\in\cD_{Z}(X)$ we have $j^!K\simeq 0$, then the map $i_!i^!K\to K$ is an isomorphism (by \rl{glue1}(d)), thus $K$ lies in the essential image of $i_!$. 

In the case  $\eta=j:U\hra X$, the argument is similar. Namely, since the right adjoint $j_*$ is fully faithful,
the counit map $j^!j_*\to\Id$ is an isomorphism, so it suffices to show that $j_*$ induces an equivalence
$\cD(U)\isom\cD_{U}(X)$. We complete as before.
\end{proof}

\begin{Emp} \label{E:loccl}
{\bf Functor $\eta_*$.} \label{N:eta*}
In the situation of \rl{locclemb}, we denote by
$\eta_*:\cD(\cY)\isom\cD_{\cY}(\cX)\subset\cD(\cX)$ the inverse of $\eta^!:\cD_{\cY}(\cX)\isom\cD(\cY)$. Since  $\cD_{\cY}(\cX)\subset \cD(\cX)$ is closed under all colimits (see \re{support}), the functor $\eta_*$ is continuous.
\end{Emp}

\begin{Lem} \label{L:loccl2}
Let $\eta:\cY\to\cX$ be a topologically fp-locally closed embedding. Then for every Cartesian diagram of $\infty$-stacks
\[
\begin{CD}
\wt{\cY} @>\wt{\eta}>>\wt{\cX}\\
@VgVV @VVfV\\
\cY @>\eta>>\cX,
\end{CD}
\]
we have a canonical isomorphism $f^!\eta_*\simeq\wt{\eta}_* g^!$.
\end{Lem}

\begin{proof}
Notice that for every $K\in\cD(\cY)$, we have $\wt{\eta}_* g^!(K)\in \cD_{\wt{\cY}}(\wt{\cX})$, and also $\eta_*(K)\in \cD_{\cY}(\cX)$, thus  $f^!\eta_*(K)\in \cD_{\wt{\cY}}(\wt{\cX})$ (by \re{support}(b)). Therefore by \rl{locclemb}, it suffices to construct an isomorphism
$\wt{\eta}^!f^!\eta_*\simeq\wt{\eta}^!\wt{\eta}_* g^!$. Since ${\eta}^!{\eta}_*\simeq\Id$ and $\wt{\eta}^!\wt{\eta}_*\simeq\Id$, the composition
$\wt{\eta}^!f^!\eta_*\simeq g^!{\eta}^!\eta_*\simeq g^!\simeq\wt{\eta}^!\wt{\eta}_* g^!$ does the job.
\end{proof}

\begin{Cor} \label{C:comp}
 Let $\eta:\cY\overset{\eta'}{\lra}\cZ\overset{\eta''}{\lra}\cX$ be a composition of topologically fp-locally closed embeddings. Then the functor $\eta_*$ coincides with the composition $\eta''_*\circ \eta'_*$.
\end{Cor}

\begin{proof}
Since $\eta^!\simeq \eta'^!\circ\eta''^!$, it remains to show that  $\eta''_*\circ\eta'_* (\cD(\cY))\subset\cD_{\cY}(\cX)$ or equivalently that    $\eta''_*(\cD_{\cY}(\cZ))\subset\cD_{\cY}(\cX)$. But this follows from
\rl{loccl2} applied to $\eta:=\eta''$ and $f:\cX\sm\cY\hra \cX$.
\end{proof}

\begin{Emp} \label{E:exloccl}
{\bf Examples.} (a) If $\eta$ is an fp-open (resp. topologically fp-closed) embedding
$j:\cU\hra\cX$ (resp. $i:\cZ\hra\cX$), then $\eta_*$ coincides with $j_*$ (resp. $i_!$). Indeed, since $i^!j_*\simeq 0$ (resp. $j^!i_!\simeq 0$), functor  $j_*$ (resp. $i_!$) induces a functor $j_*:\cD(\cU)\to\cD_{\cU}(\cX)$ (resp. $i_!:\cD(\cZ)\to\cD_{\cZ}(\cX)$). Since $j^!j_*\simeq \Id$
(resp. $i^!i_!\simeq \Id$) by \rl{glue1}(c), we are done.

(b) Assume that an fp-locally closed embedding $\eta:\cX\to\cY$ decomposes as a composition
$\cY\overset{j}{\lra}\cZ\overset{i}{\lra}\cX$, where $i$ (resp. $j$) is a topologically fp-closed (resp. fp-open) embedding. Then it follows from (a) and \rco{comp} that
the functor $\eta_*$ coincides with the composition $i_!\circ j_*$.
\end{Emp}

\begin{Emp} \label{E:remdecomp}
{\bf Remarks.}
(a) Since every fp-locally closed embedding of schemes $\eta$ has a decomposition as in \re{exloccl}(b),
we can define $\eta_*$ by the formula $\eta_*:=i_!\circ j_*$. Moreover, it is not difficult to see that this composition is independent
of the decomposition, thus  $\eta_*$ is well defined.

(b) Moreover, since $i_!$ and $j_*$ commute with all $!$-pullbacks, one can show that functors $\eta_*$ from (a) commute with $!$-pullbacks,
and using this to define functors $\eta_*$ in general.

(c) Though the definition of $\eta_*$ using (a) and (b) is the standard way of doing it, our method is much easier and more intrinsic,
because it does not use any choices.

(d) All results of this subsection also hold for $\cD_c$.
\end{Emp}


\subsection{Infinity-stacks admitting gluing of sheaves}


\begin{Def} \label{D:glue}
We say that an $\infty$-stack $\C{X}$ {\em admits gluing of sheaves},\label{I:gluing} if for every topologically fp-locally closed embedding $\eta:\cY\hra\cX$ the pushforward
$\eta_*:\cD(\cY)\to\cD(\cX)$ (see \re{loccl}) admits a left adjoint $\eta^*:\cD(\cX)\to\cD(\cY)$.
\end{Def}

\begin{Emp}
{\bf Remarks.}
(a) We will see later that $\infty$-stacks admitting gluing of sheaves in the sense of \rd{glue}, admit the gluing of sheaves in the sense of \cite{BBD}.

(b) In this subsection it is essential that we work with $\cD$ rather than $\cD_c$. In particular, analogs of \rl{openglue} and \rp{glue2} for
$\cD_c$ would be false.
\end{Emp}

\begin{Lem} \label{L:glue}
Let $\C{X}$ be an $\infty$-stack, admitting gluing of sheaves, let $j:\C{U}\hra\cX$ be an fp-open embedding with a complementary topologically fp-closed embedding $i:\cZ\rightarrow\cX$. Then

(a) there exists a left adjoint $j_!$ of $j^!:\C{D}(\C{X})\to \C{D}(\C{U})$;

(b) the functor $j_!$ is fully faithful, and $i^*j_!\simeq 0$;

(c) for every $K\in\cD(\C{X})$, the unit and counit maps extend to a fiber sequence \[j_{!}j^{!}K\rightarrow K\rightarrow i_!i^*K.\]
\end{Lem}

\begin{proof}
All assertions are rather straightforward applications of \rl{glue1}.

(a) By the adjoint functor theorem \cite[Corollary 5.5.2.9]{Lu}, it suffices to show that functor $j^!$ preserves limits. Since functor $j_*$ is fully faithful (by \rl{glue1}(c)) and preserves limits (because it is a right adjoint), it suffices to show that the composition $j_*j^!$ preserves limits. Since $j_*j^![-1]$ is the fiber of the counit map $i_!i^!\to\Id$ (by  \rl{glue1}(d)), it suffices to show that the composition $i_!i^!$ preserves limits.
Since $i^!$ is the right adjoint of $i_!$, the assertion follows from our assumption that $i_!=i_*$ (see \re{remdecomp}(a)) admits a left adjoint $i^*$.


%
%

(b) Since $j^!j_!$ is the left adjoint of $j^! j_*$ and   $j^!j_*\simeq\Id$ (by \rl{glue1}(c)), we conclude that $j^!j_!\simeq\Id$, thus $j_!$ is fully faithful.
Similarly, since  $i^*j_!$ is the left adjoint of $j^!i_!\simeq 0$ (by \rl{glue1}(c)), we conclude that $i^*j_!\simeq 0$.

(c) follows by adjointness from the fiber sequence of \rl{glue1}(d).
\end{proof}

\begin{Cor} \label{C:glue}
Suppose that we are given a Cartesian diagram
\[
\begin{CD}
\wt{\cY} @>\wt{\eta}>>\wt{\cX}\\
@VgVV @VVfV\\
\cY @>\eta>>\cX,
\end{CD}
\]
where $\C{X}$ and $\wt{\cX}$ are $\infty$-stacks admitting gluing of sheaves, $\eta:\cY\to \cX$ is a topologically fp-locally closed embedding, while
functors $f^!$ and $g^!$ admit left adjoints $f_!$ and $g_!$, respectively. Then we have a canonical isomorphism
$\eta^*f_!\simeq g_!\wt{\eta}^*$.
\end{Cor}
\begin{proof}
Since $\eta^*f_!$ and $g_!\wt{\eta}^*$ are left adjoints of functors $f^!\eta_*$ and $\wt{\eta}_*g^!$, respectively, the assertion follows from
\rl{loccl2}.
\end{proof}

\begin{Lem} \label{L:openglue}
(a) Assume that $\cX$ admits gluing of sheaves, and let $\eta:\cY\hra\cX$ be a topologically fp-locally closed embedding.
Then $\cY$ admits gluing of sheaves as well.

(b) Assume that $\cX$ has a presentation as a filtered colimit $\cX\simeq\colim_{\al}\cX_{\al}$ such that each $\cX_{\al}$ admits gluing of sheaves and each transition map is an fp-open (resp. fp-closed) embedding. Then $\cX$ admits gluing of sheaves as well.
\end{Lem}

\begin{proof}
(a) Let $\nu:\cZ\hra\cY$ be a topologically fp-locally closed embedding. Then $\eta\circ\nu:\cZ\to\cX$ is a topologically fp-locally closed embedding as well, therefore the pushforward $(\eta\circ\nu)_*\simeq \eta_*\nu_*$ has a left adjoint $(\eta\circ\nu)^*$. Since $\eta_*$ is fully faithful, the composition $(\eta\circ\nu)^*\eta_*: \cD(\cY)\to\cD(\cX)\to\cD(\cZ)$ is the left adjoint of $\nu_*$. Indeed, for every $K\in \cD(\cY)$ and $L\in \cD(\cZ)$, we have a
natural isomorphism
\[
\Hom((\eta\circ\nu)^*\eta_*(K),L)\simeq \Hom(\eta_*(K),\eta_*\nu_*(L))\simeq  \Hom(K,\nu_*(L)).
\]

(b) Let $\eta:\cY\hra\cX$ be a topologically fp-locally closed embedding. By the adjoint functor theorem \cite[Corollary 5.5.2.9]{Lu}, it suffices to show that $\eta_*$ preserves limits. Notice that presentation  $\cX\simeq\colim_{\al}\cX_{\al}$ induces the presentation $\cY\simeq\colim_{\al}\cY_{\al}$, and the induced maps $\eta_{\al}:\cY_{\al}\hra\cX_{\al}$ are  topologically  fp-locally closed embeddings.

Since each transition map $\nu_{\al,\beta}:\cX_{\al}\hra\cX_{\beta}$ is an fp-open (resp. fp-closed) embedding and each
$X_{\beta}$ admits gluing of sheaves, we conclude from \rl{glue}(a) (resp. \rl{glue1}(b)) that  $\nu^!_{\al,\beta}$ admits a left adjoint.  Therefore it follows from \rt{lu1} that the pullback $\nu^!_{\al}:\cD(\cX)\to\cD(\cX_{\al})$ admits a left adjoint, thus preserves limits. Moreover, since each $Y_{\beta}$ admits gluing of sheaves (by (a)), we show by the same argument that each $\nu^!_{\al}:\cD(\cY)\to\cD(\cY_{\al})$ preserves limits as well.

Now the assertion is easy. Indeed, since $\cD(\cX)\simeq\lim_{\al}\cD(\cX_{\al})$, and each $\nu^!_{\al}:\cD(\cX)\to\cD(\cX_{\al})$ preserves limits,
it suffices to show that each composition $\nu_{\al}^!\eta_*: \cD(\cY)\to \cD(\cX)\to\cD(\cX_{\al})$ preserves limits.
By \rl{loccl2}, this composition can be rewritten as  $(\eta_{\al})_*\nu_{\al}^!: \cD(\cY)\to \cD(\cY_{\al})\to\cD(\cX_{\al})$, so it suffices to show that both functors $(\eta_{\al})_*$ and $\nu_{\al}^!$ preserve limits. The assertion for the pullback $\nu_{\al}^!: \cD(\cY)\to \cD(\cY_{\al})$ was mentioned above, while the assertion for $(\eta_{\al})_*$ follows from the fact that $(\eta_{\al})_*$ has a left adjoint, because $X_{\al}$ admits gluing of sheaves.
\end{proof}



Now we are going to provide two classes of $\infty$-stacks, admitting gluing of sheaves.

\begin{Lem} \label{L:glue3}
Every placid $\infty$-stack $\C{X}$ admits gluing of sheaves.
\end{Lem}
\begin{proof}
We have to show that for every topologically fp-locally closed embedding $\eta:\cY\hra\cX$, the pushforward $\eta_*:\cD(\cY)\to\cD(\cX)$ admits a left adjoint. Replacing $\eta$ by $\eta_{\red}:\cY_{\red}\hra\cX_{\red}$ and using \rco{invtop}, we reduce to the case when $\eta$ is an fp-locally closed embedding of placid $\infty$-stacks (by \rco{red-fp}). In this case, $\eta$ decomposes as a composition $\cY\overset{j}{\hra}\cZ\overset{i}{\hra}\cX$ of an fp-open and fp-closed embeddings (see \rl{fplocl}), thus it suffices to consider these two cases $\eta=j$ and $\eta=i$ separately. Since $j_*$ and $i_*=i_!$ admit left adjoints (by \rl{glue1}(a) and \rp{glue1}(a)), we are done.
\end{proof}


\begin{Prop} \label{P:glue2}
Let $H$ be an ind-placid group, that is, a group object in ind-placid ind-algebraic spaces, acting on an  ind-placid ind-algebraic space $X$. Then the quotient stack $\cX=[X/H]$ admits gluing of sheaves.
\end{Prop}

\begin{proof}
Assume that $H=1$, that is, $\C{X}$ is an ind-placid ind-algebraic space. In this case, the assertion follows from a combination of
\rl{glue3} and \rl{openglue}(b).

Now let $H$ be general, and let $\eta:\cY\hra\cX$ be a topologically fp-locally closed embedding. We set $Y:=\C{Y}\times_{\cX} X$. Then $Y\hra X$ is a topologically  fp-locally closed embedding, and we have a natural isomorphism $\C{Y}\simeq [Y/H]$.

Using the \v{C}ech complexes, corresponding to the projections $X\to\C{X}$ and $Y\to\C{Y}$, we get equivalences $\C{D}(\C{X})\simeq \lim_{[m]} \cD(H^m\times X)$ (see \re{shpr}(e)), and $\C{D}(\C{Y})\simeq \lim_{[m]} \cD(H^m\times Y)$, respectively.

By the case of  ind-placid ind-algebraic spaces, shown above, there exists a left adjoint  $\eta^*$ of $\eta_*:\C{D}(H^m\times Y)\to \C{D}(H^m\times X)$. Thus, in order to apply \rp{adjlu}(b) and to finish the proof, we have to show that the pullbacks $\eta^*$ satisfy the base change with respect to pullbacks $\nu^!:\C{D}(H^n\times X)\to \C{D}(H^m\times X)$.

Notice that every morphism $H^m\times X\to H^n\times X$ decomposes as a composition of the action morphisms $H\times X\to X:(h,x)\mapsto h(x)$, multiplications morphisms $H\times H\to H$ and projections. Since the action morphism $H\times X\to X$ decomposes as a composition of the isomorphism
\[
H\times X\isom H\times X:(h,x)\mapsto (h,h(x))
\]
and the projection, it suffices to show that the $\eta^*$'s satisfy the base change with respect to pullbacks, corresponding to projections. Thus the assertion follows from \rl{bcproj} below.
\end{proof}

\begin{Lem} \label{L:bcproj}
Consider the Cartesian diagram
\[
\begin{CD}
Y\times Z @>\eta\times\Id>> X\times Z \\
@V\pr_Z VV   @VV\pr_Z V \\
Y @>\eta>> X
\end{CD}
\]
where $X$ and $Z$ are ind-placid ind-algebraic spaces, and $\eta:Y\hra X$ is a topologically  fp-locally closed embedding. Then the base change morphism $(\eta\times\Id)^*\pr_Z^!\to\pr_Z^!\eta^*$ is an isomorphism.
\end{Lem}

\begin{proof}
Assume first that $X$ and $Z$ are placid algebraic spaces. Arguing as in \rl{glue3}, we can assume that $\eta$ is either an fp-open or an fp-closed
embedding of placid algebraic spaces. Hence there exist smooth coverings $X'\to X$ and $Z'\to Z$ by disjoint union of affine schemes admitting placid presentations. Since $\eta^*$ commutes with smooth $!$-pullbacks (use \rp{glue1}(b) when $\eta$ is an fp-closed embedding), we reduce to the case when $X$ and $Z$ are affine schemes admitting placid presentations. In this case, the assertion for $\cD=\Ind\cD_c$ follows from that for $\cD_c$.

When $X$ and $Z$ are of finite type over $k$, the assertion for $\cD_c$ is well known. Namely, it follows from the K\"unneth formula $\pr_Z^!(K)\simeq K\boxtimes \om_Z$ (see \cite[Expos\'e III, Proposition 1.7.4]{SGA5}).
In the general case, choose a placid presentation $X\simeq\lim_{\al}X_{\al}$. Then that every object $K\in\cD_c(X)$ comes from some
$K_{\al}\in\cD_c(X_{\al})$, so the assertion follows from the finite type case.

Next we assume that $X$ is a placid algebraic space, but $Z$ is an ind-placid ind-algebraic space. Choose presentation $Z\simeq\colim_{\al}Z_{\al}$, and let $i_{\al}:Z_{\al}\to Z$ be an inclusion. Then we have a natural isomorphism
$\pr_Z^!\simeq \colim_{\al}i_{\al,!}i_{\al}^!\pr_Z^!\simeq\colim_{\al}i_{\al,!}\pr_{Z_{\al}}^!$ (by \rco{lurie}).
Since $\eta^*$ commutes with colimits and the $i_{\al,!}$'s (by \rco{glue}), the assertion for $X$ and $Z$ follows from the corresponding assertion for $X$ and $Z_{\al}$, shown earlier. Finally, the extension to the case when $X$ is an ind-placid ind-algebraic space is similar.
\end{proof}

\subsection{Endomorphisms of $\om_{\cX}$}

\begin{Emp} \label{E:prosets}
{\bf Pro-sets.} (a) The natural fully faithful embedding $\Sets\hra\frak{S}$ of the category of sets into an $\infty$-category of spaces preserves all small (homotopy) limits and filtered colimits. In particular, for every (ordinary) category $\cA$ it induces a
fully faithful embedding $\Fun(\cA,\Sets)\to\Fun(\cA,\frak{S})$, also preserving  small (homotopy) limits and filtered colimits.

(b) Notice that the Yoneda embedding $X\mapsto\Hom(X,-)$ defines a fully faithful embedding
$j:\Pro(\Sets)\to \Fun(\Sets,\Sets)^{\op}$ preserving all small colimits and filtered limits. Moreover, its essential
image is the subcategory $\Fun_{\flim}(\Sets,\Sets)^{\op}$ of functors preserving finite limits (compare \cite[Corollary 5.3.5.4]{Lu}).

(c) For every $X\in\Pro(\Sets)$ and $A\in\Sets$, we set $A^X:=j(X)(A)\in\Sets$. Explicitly,
$A^{X}$ is the set of continuous maps $\Hom_{\cont}(X,A)$, where $A$ is equipped with discrete topology.

(d) For every $X\in\Pro(\Sets)$, and a commutative ring $A$, the set
$A^X\simeq\Hom_{\cont}(X,A)$ has a natural structure of an $A$-algebra. Moreover, for every $A$-module $M$, we have a natural morphism of $A^X$-modules
\begin{equation} \label{Eq:axmodules}
A^X\otimes_A M\to M^X.
\end{equation}

(e) The morphism \form{axmodules} is an isomorphism, if $X$ is pro-finite. Indeed, this is clear, when $X$ is finite, so the assertion follows from the fact that $M^{\lim_{\al}X_{\al}}\simeq\colim_{\al}M^{X_{\al}}$ and similarly for $A$.
\end{Emp}

\begin{Emp} \label{E:pi0}
{\bf Functor of connected components.}
(a) Consider functor $\iota:\Sets\to\St_k$, which sends $A$ to a coproduct $\colim_A\pt$ of $A$ copies of $\pt$.
Since $\St_k$ has all (small) limits, functor $\iota$ uniquely extends to a functor $\wh{\iota}:\Pro(\Sets)\to \St_k$, which preserves
all filtered limits. Moreover, since $\iota$ preserves finite limits, $\wh{\iota}$ preserves all limits.

(b) By (a), functor $\wh{\iota}$ has a left adjoint $\un{\pi}_0:\St_k\to\Pro(\Sets)$.\label{N:pi0} Explicitly, $\un{\pi}_0$ is the functor
\[
\St_k\to\Fun_{\flim}(\Sets,\Sets)^{\op}\simeq\Pro(\Sets)
\]
(see \re{prosets}(b)), defined by the rule $\un{\pi}_0(\cX)(A)=\Hom_{\St_k}(\cX,\iota(A))$ for every $A\in\Sets$.

(c) Notice that there is a natural functor $\pi_0:\Aff_k\to\Pro(\fSets)\subset\Pro(\Sets)$, which associates to $X$ the pro-finite set of its connected components. Namely, $\pi_0$ is the right Kan extension of its restriction to $\Affft_k$. It is easy to see that the restriction of
$\un{\pi}_0$ to $\Aff_k$ is naturally isomorphic to $\pi_0$.

(d) By definition and the Yoneda lemma, functor $\un{\pi}_0$ is equivalent to the left Kan extension of its restriction to $\Aff_k$. Thus, by (c), $\un{\pi}_0$ can be defined as the left Kan extension of the functor $\pi_0:\Aff_k\to\Pro(\Sets)$ of connected components.

(e) Since $\un{\pi}_0$ is the left adjoint, for every epimorphism $f:\cX\to\cY$ of $\infty$-stacks, the induced map $\un{\pi}_0(f):\un{\pi}_0(\cX)\to \un{\pi}_0(\cY)$ is an epimorphism of pro-sets.

(f) We say that $\cX$ is {\em connected},\label{I:connected} if $\un{\pi}_0(\cX)=\pt\in\Sets\subset\Pro(\Sets)$. It thus follows from (f) that if $f:\cX\to\cY$ is an epimorphism of $\infty$-stacks and $\cX$ is connected, then $\cY$ is connected.
\end{Emp}

\begin{Lem} \label{L:endomor}
(a) For every $\cX\in\St_k$, the endomorphism algebra $\End_{\cD(\cX)}(\om_{\cX})$ is a discrete $\qlbar$-algebra, canonically isomorphic to $\qlbar^{{\un{\pi}}_0(\cX)}$.

(b) Moreover, if $\un{\pi}_0(\cX)\in\Pro(\Sets)$ is pro-finite, then for every $\qlbar$-vector space $V$, we have a canonical isomorphism of  $\End_{\cD(\cX)}(\om_{\cX})$-modules
\begin{equation} \label{Eq:endomor}
\End_{\cD(\cX)}(\om_{\cX})\otimes_{\qlbar}V\simeq \Hom_{\cD(\cX)}(\om_{\cX},V\otimes_{\qlbar}\om_{\cX}).
\end{equation}
\end{Lem}

\begin{proof}
(a) Note first that if $X\in\Affft_k$, then we have a canonical isomorphism
\[
\Hom_{\cD(X)}(\om_{X},\om_{X})\simeq \Hom_{\cD(X)}(\qlbar,\qlbar)\simeq\qlbar^{\pi_0(X)},
\]
where the first isomorphism follows from the Verdier duality, and the second one from the fact that the constant sheaf $\qlbar$ has no negative self-exts.

Next, for $X\in\Aff_k$, we choose a presentation $X\simeq\lim_{\al}X_{\al}$ as a filtered limit, where $X_{\al}\in\Affft_k$ for all $\al$.
Then $\cD_c(X)\simeq\colim_{\al}\cD_c(X_{\al})$, hence $\End_{\cD(X)}(\om_X)\simeq\colim_{\al}\End_{\cD(X_{\al})}(\om_{X_{\al}})$ by \cite{Ro}. Thus $\End_{\cD(\cX)}(\om_{\cX})$ is a discrete $\qlbar$-algebra, being a filtered colimit of discrete spaces, which by the proven above is isomorphic to $\colim_{\al}\qlbar^{\pi_0(X_{\al})}=\qlbar^{\pi_0(X)}$.

For an arbitrary $\infty$-stack $\cX$, the identification $\cX\simeq\colim_{X\to\cX}X$, gives us an identification
$\cD(\cX)\simeq\lim_{X\to\cX}\cD(X)$, under which $\om_{\cX}$ corresponds to the compatible system of the $\om_X$'s. Thus we have an identification
$\End_{\cD(\cX)}(\om_{\cX})\simeq\lim_{X\to\cX}\End_{\cD(X)}(\om_X)$, hence $\End_{\cD(\cX)}(\om_{\cX})$ is a discrete algebra isomorphic to
$\lim_{X\to\cX}\qlbar^{\pi_0(X)}\simeq\qlbar^{\un{\pi}_0(\cX)}$.

(b) Notice first that the isomorphism \form{endomor} for $X\in\Aff_k$ follows from the fact that $\om_{X}\in\cD_c(X)$ is compact in $\cD(X)$.
Next, it follows from a combination of (a) and \re{prosets}(e) that the LHS of \form{endomor} is isomorphic to
$V^{\un{\pi}_0(\cX)}$. Finally, combining last observation and isomorphism \form{endomor} for $X\in\Aff_k$, we see that the RHS of \form{endomor} is isomorphic to
\[
\lim_{X\to\cX}\Hom_{\cD(X)}(\om_{X},V\otimes_{\qlbar}\om_{X})\simeq\lim_{X\to\cX}\End_{\cD(X)}(\om_{X})\otimes_{\qlbar}V
\simeq\lim_{X\to\cX}V^{\un{\pi}_0(X)}\simeq V^{\un{\pi}_0(\cX)}
\]
as well.
\end{proof}



\begin{Emp} \label{E:quotdiscr}
{\bf Quotient by a discrete group.} (a) Let $\Gm$ be a discrete group acting on an $\infty$-stack $\cX$, let $\cY:=[\cX/\Gm]$ be the quotient $\infty$-stack, and let $f:\cX\to\cY$ be the projection.

(b) Since the trivial $\Gm$-torsor $\Gm\times \cX\to \cX$ is clearly ind-fp-proper, we conclude from \re{pspr} that
$f$ is ind-fp-proper.


(c) By remark (b) and \rp{locindpr}, the pullback $f^!:\cD(\cY)\to\cD(\cX)$ admits a left adjoint $f_!:\cD(\cX)\to\cD(\cY)$, satisfying the base change.
\end{Emp}

\begin{Lem}\label{L:end}
In the situation of \re{quotdiscr}, assume that $\un{\pi}_0(\cX)$ is pro-finite. Then we have a natural isomorphism of $\qlbar$-algebras
$\End(f_{!}(\omega_{\cX}))\simeq\bql[\Gm]\otimes_{\qlbar}\qlbar^{\un{\pi}_0(\cX)}$.

\end{Lem}

\begin{proof}
The group action of $\Gm$ on $\cX$ over $\cY$ induces a group homomorphism $\Gm\to\Aut(f_!(\om_{\cX}))$, commuting with the action of
$\End(\om_{\cX})$. Hence it induces a homomorphism of $\qlbar$-algebras
\begin{equation} \label{Eq:canhom}
\qlbar[\Gm]\otimes_{\qlbar} \End(\om_{\cX})\to\End(f_!(\om_{\cX})).
\end{equation}
Since $\End(\om_{\cX})\simeq\qlbar^{\un{\pi}_0(\cX)}$ (by \rl{endomor}), it now suffices to show that
\form{canhom} is an isomorphism of $\qlbar$-vector spaces.
Since $f$ is a $\Gm$-torsor (by \re{tors}(c)), we have a Cartesian diagram
\[
\begin{CD}
\Gm\times \cX @>a>> \cX \\
@V\pr VV   @VVfV \\
\cX @>f>> \cY.
\end{CD}
\]
Since $f_!$ satisfies the base change, we get a natural isomorphism
\[
f^!f_!(\om_{\cX})\simeq \pr_! a^!(\om_{\cX})\simeq\pr_!(\om_{\Gm\times \cX})\simeq \qlbar[\Gm]\otimes_{\qlbar}\om_{\cX}.
\]
Therefore by adjunction we have an isomorphism of $\qlbar$-vector spaces
\[
\End(f_{!}(\omega_{\cX}))\simeq\Hom(\omega_{\cX},f^{!}f_{!}(\omega_{\cX}))\simeq\Hom(\om_{\cX}, \qlbar[\Gm]\otimes_{\qlbar} \om_{\cX})\simeq\qlbar[\Gm]\otimes_{\qlbar}\End(\om_{\cX}),
\]
where the last isomorphism follows from \form{endomor}. Unwinding the definitions, one sees that this isomorphism coincides with the
canonical homomorphism \form{canhom} we started from.
\end{proof}

\begin{Cor}\label{C:end}
$f:\cX\to\cY$ be a morphism of $\infty$-stacks, and let $\Gm$ be a discrete group acting on $\cX$ over $\cY$ such that $\un{\pi}_0(\cX)$ is pro-finite, and the induced map $[f]:[\cX/\Gm]\to\cY$ is a topological equivalence. Then $f^!:\cD(\cY)\to\cD(\cX)$ has a left adjoint $f_!:\cD(\cX)\to\cD(\cY)$, and we have a natural isomorphism of $\qlbar$-algebras
$\End(f_{!}(\omega_{\cX}))\simeq\bql[\Gm]\otimes_{\qlbar}\qlbar^{\un{\pi}_0(\cX)}$.
\end{Cor}

\begin{proof}
Set $\cY':=[\cX/\Gm]$, and let $\pi:\cX\to\cY'$ be the projection. Since $[f]:\cY'\to\cY$ is a topological equivalence,
the pullback $[f]^!:\cD(\cY)\to\cD(\cY')$ is an equivalence (by \rco{invtop}(b)), hence has a left adjoint $[f]_!$. Therefore
$f^!\simeq \pi^!\circ[f]^!$ has a left adjoint $f_!:=[f]_!\circ \pi_!$. Now the assertion follows from \rl{end} and the observation
that  $[f]_!$ is an equivalence.
\end{proof}

\section{Perverse $t$-structures}

\subsection{Generalities}

\begin{Emp} \label{E:gentstr}
{\bf Recollections.} Let $\cD$ be a stable $\infty$-category.

 (a) Recall (see Lurie \cite[1.2.1]{HA}) that a {\em $t$-structure} on $\cD$ is a pair $(\cD^{\leq 0},\cD^{\geq 0})$ of full subcategories of $\cD$ satisfying certain
properties. In particular, the embedding $\cD^{\geq 0}\to \cD$ (resp. $\cD^{\leq 0}\to \cD$) has a left (resp. right) adjoint
\begin{center}
$\tau_{\geq 0}: \cD\to \cD^{\geq 0}$
(resp. $\tau_{\leq 0}: \cD\to \cD^{\leq 0}$).
\end{center}
Similarly, we define truncation functors $\tau_{\geq 1}$ and $\tau_{\leq -1}$. Notice that
\begin{equation}
x\in \cD^{\leq 0} ~~(\text{resp.}~x\in\cD^{\geq 0}) \text{ if and only if } \tau_{\geq 1}(x)=0 ~~(\text{resp.}~  \tau_{\leq -1}(x)=0).
\label{etstr}
\end{equation}

(b) Let $F:\cD_1\to \cD_2$ be an exact functor between stable $\infty$-categories equipped with $t$-structures. Recall that $F$ is called {\em right (resp. left) $t$-exact}, if $F$ satisfies $F(\cD_1^{\leq 0})\subset \cD_2^{\leq 0}$ (resp. $F(\cD_1^{\geq 0})\subset \cD_2^{\geq 0}$), and it is called {\em $t$-exact}, if it is both left and right $t$-exact.

(c) Every $t$-exact $F$ commutes with truncation functors. Indeed, for each $x\in\cD_1$, functor $F$ maps the fiber sequence
$\tau_{\leq 0}(x)\to x\to\tau_{\geq 1}(x)$ to the fiber sequence $F(\tau_{\leq 0}(x))\to F(x)\to F(\tau_{\geq 1}(x))$. Since
$F(\tau_{\leq 0}(x))\in\cD_2^{\leq 0}$ and $F(\tau_{\geq 1}(x))\in \cD_2^{\geq 1}$ by assumption, we get that
$F(\tau_{\leq 0}(x))\simeq\tau_{\leq 0}(F(x))$ and $F(\tau_{\geq 1}(x))\simeq\tau_{\geq 1}(F(x))$, as claimed.

(d) Recall that a functor $F$ is called {\em faithful}, if $F(x)\not\simeq 0$ when $x\not\simeq 0$.
\end{Emp}

\begin{Lem} \label{L:remtstr}
(a) Every $t$-structure on $\cD$ has a unique extension to a $t$-structure on $\Ind(\cD)$ such that
$\Ind(\cD)^{\geq 0}$ is closed under filtered colimits. Explicitly, we have $\Ind(\cD)^{\leq 0}=\Ind(\cD^{\leq 0})$ and
$\Ind(\cD)^{\geq 0}=\Ind(\cD^{\geq 0})$.


(b) Let $\cD$ be a stable $\infty$-category with a $t$-structure.
Then $\cD^{\leq 0}$ is closed under all colimits that exist in $\cD$ and $\cD^{\geq 0}$ is closed under all limits that exist in $\cD$.

(c) Assume that $F:\cD_1\to\cD_2$ be a $t$-exact and faithful functor between stable $\infty$-categories. Then for every object $x\in\cD_1$ we have
\begin{center}
$x\in \cD_1^{\leq 0} \text{ if and only if } F(x)\in \cD_2^{\leq 0}$
\end{center}
and similarly for $\cD_i^{\geq 0}$.

(d) Let $F:\cD_1\to\cD_2$ and $G:\cD_2\to\cD_3$ be exact functors between stable $\infty$-categories, equipped with $t$-structures, such that
$G$ is $t$-exact and faithful. Then $F$ is $t$-exact if and only if $G\circ F$ is.

(e) The $t$-structure $(\cD^{\leq 0},\cD^{\geq 0})$ on $\cD$ is uniquely determined by $\cD^{\geq 0}$. Namely, an object $x\in\cD$ belongs to $\cD^{\leq 0}$ if and only if
$\Hom_{\cD}(x,y)\simeq 0$ for every $y\in\cD^{\geq 1}$.
\end{Lem}

\begin{proof}
(a) and (b) follow from \cite[4.1.2.4]{GR} and \cite[Corollary 1.2.1.6]{HA}, respectively.

(c) By \eqref{etstr}, we have $x\in \cD_1^{\leq 0}$ if and only if $\tau_{\geq 1}(x)\simeq 0$, while  $F(x)\in D_2^{\leq 0}$ if and only if $\tau_{\geq 1}(F(x))\simeq 0$. Since $\tau_{\geq 1}(F(x))\simeq F(\tau_{\geq 1}(x))$ (because $F$ is $t$-exact), we have to show that $\tau_{\geq 1}(x)\simeq 0$ if and only if $F(\tau_{\geq 1}(x))\simeq 0$. Since $F$ is faithful, we are done.

(d) The ``if'' assertion follows from (c), while the converse assertion is clear.

(e) is standard.
\end{proof}

\begin{Lem}\label{L:tstr}
Let $\cI$ be a category and $\Psi:\cI\to\Catst$ a functor. Assume that for every object $a\in \cI$ the category $\cD_{a}$ is equipped
with a $t$-structure, and for every morphism $\al:a\to b$ in $\cI$ the induced functor $\psi_{\al}:\cD_a\to \cD_b$ is $t$-exact.

(a) Assume that $\cI$ is filtered. Then there exists a unique $t$-structure on $\cD:=\colim_{a\in \cI} \cD_a$ such that every
functor $\ins_a:\cD_a\to \cD$ is $t$-exact. Explicitly,
 $\cD^{\leq 0}:=\colim_{a\in \cI} \cD^{\leq 0}_a$ and similarly for $\cD^{\geq 0}$.

(b) There exists a unique $t$-structure on $\cD:=\lim_{a\in \cI^{\op}} \cD_a$ such that every functor $\ev_a:\cD\to \cD_a$ is $t$-exact. Explicitly,  $\cD^{\leq 0}=\lim_{a\in \cI} \cD^{\leq 0}_a$ and similarly for $\cD^{\geq 0}$.
\end{Lem}

\begin{proof}
(a)	Let us prove that full subcategories $(\cD^{\leq 0}, \cD^{\geq 0})$, defined as $\cD^{\leq 0}:=\colim_{a\in \cI} \cD^{\leq 0}_a$ and $\cD^{\geq 0}:=\colim_{a\in \cI} \cD^{\geq 0}_a$,
equip $\cD$ with a $t$-structure. Recall that every object $x\in \cD$ is of the form $x=\ins_a(x_a)$ for some $x_{a}\in \cD_a$. By assumption, for  every $a\in\cI$ there exists a fiber sequence
$\tau_{\leq 0}(x_{a})\rightarrow x_{a}\rightarrow \tau_{\geq 1}(x_{a})$ in $\cD_{a}$ with $\tau_{\leq 0} (x_{a})\in \cD^{\leq 0}_a$ and
$\tau_{\geq 1}(x_{a})\in \cD^{\geq 1}_a$. Applying $\ins_a$, we get the corresponding fiber sequence for $x$.

It remains to show that for $x\in \cD^{\leq 0}$ and $y\in \cD^{\geq 1}$, we have
$\Hom_{\cD}(x,y)\simeq 0$. Since $\cI$ is filtered, $x$ and $y$ come from $x_{a}\in \cD_a^{\leq 0}$ and $y_{a}\in \cD_{a}^{\geq 1}$. Moreover, as the colimit is filtered, it follows from \cite[Lemma 0.2.1]{Ro} that every  $\phi\in \Hom_{\cD}(x,y)$  comes from a morphism
$\phi_b\in \Hom_{\cD_{b}}(\psi_{\al}(x_{a}),\psi_{\al}(y_{a}))$ for some $\al:a\to b$ in $\cI$.
As $\psi_{\al}$ are $t$-exact, we conclude that $\Hom_{\cD_{b}}(\psi_{\al}(x_{a}),\psi_{\al}(y_{a}))\simeq 0$. Thus $\phi_b\simeq 0$, hence $\phi\simeq 0$.

(b) We want to show that full subcategories $(\cD^{\leq 0}, \cD^{\geq 0})$, defined as $\cD^{\leq 0}:=\lim_{a\in \cI} \cD^{\leq 0}_a$ and $\cD^{\geq 0}:=\lim_{a\in \cI} \cD^{\geq 0}_a$,
equip $\cD$ with a $t$-structure. First we claim that for every $x\in \cD^{\leq 0}$ and $y\in \cD^{\geq 1}$ we have $\Hom_{\cD}(x,y)\simeq 0$.
Indeed, using for example \cite[1.6.2]{DG} and \cite[3.3.3.2]{Lu} one has
\begin{equation}
\Hom_{\cD}(x,y)\simeq\lim_{a\in\cI}\Hom_{\cD_a}(\ev_{a}(x),\ev_{a}(y)).
\label{projd}
\end{equation}
Since $\ev_a(x)\in \cD_a^{\leq 0}$ and $\ev_a(y)\in \cD_a^{\geq 1}$, all spaces on the right hand side are contractible. So the assertion follows from the standard fact that a homotopy limit of contractible spaces is contractible.

Next we claim that the inclusion functor $\cD^{\geq 0}\hra \cD$ has a left adjoint $\tau_{\geq 0}$. Indeed, since every $\cD_a$ is equipped with a $t$-structure, the inclusion $\cD_a^{\geq 0}\hra \cD_a$ has a left adjoint. Moreover, since every $\psi_{\al}$ is t-exact, these left adjoints satisfy the Beck--Chevalley condition (use \re{gentstr}(c)). Therefore the existence of $\tau_{\geq 0}:\cD\to\cD^{\geq 0}$ follows from \rp{adjlu}(b).
Now for every $x\in\cD$, let $\tau_{\geq 1}(x)$ be the cofiber of the counit map $\tau_{\leq 0}(x)\to x$. It suffices to show that
$\tau_{\geq 1}(x)\in\cD^{\geq 1}$. But this follows from the fact that cofiber in the limit category is a compatible system of cofibers,
that the cofiber of each $\tau_{\leq 0}(\ev_a(x))\to\ev_a(x)$ lies in $\cD_a^{\geq 1}$, and $\cD^{\geq 1}=\lim_{a\in \cI} \cD^{\geq 1}_a$.
\end{proof}

The following assertion is not needed for the perversity of the affine Springer sheaf.

\begin{Prop} \label{P:limtstr}
Let $\Psi:\cI\to\PrCat$ be a functor $a\mapsto\cD_a$. Assume that $\cI$ is filtered, that for every object $a\in \cI$ the category $\cD_{a}$ is equipped with a $t$-structure such that $\cD_a^{\geq 0}$ is closed under filtered colimits, and that for every morphism $\al:a\to b$ in $\cI$ the induced functor $\psi_{\al}:\cD_a\to \cD_b$ is $t$-exact and has a continuous right adjoint $\phi_{\al}$.

Then there exists a unique $t$-structure $(\cD^{\leq 0},\cD^{\geq 0})$ on $\cD:=\colim_{a\in \cI} \cD_a$ such that $\cD^{\geq 0}$ is closed under filtered colimits and
every functor $\ins_a:\cD_a\to \cD$ is $t$-exact.
\end{Prop}

\begin{Emp} \label{E:remlimtstr}
{\bf Remarks.} (a) For every $a\in\cI$, denote by $\ev_a:\cD\to\cD_a$ the right adjoint of $\ins_a:\cD_a\to\cD$
(which is automatically continuous by \rt{lu1}). It follows from the proof below that
\begin{equation} \label{Eq:geq0'}
\cD^{\geq 0}=\{x\in \cD\,|\,\ev_a(x)\in\cD_a^{\geq 0}\text{ for all }a\in\cI\}.
\end{equation}
Furthermore, this is the only $t$-structure on $\cD$ satisfying this property (see \rl{remtstr}(e)).

(b) For applications we currently have in mind, all categories $\cD_a$ are compactly generated. In this case, \rp{limtstr}
can be deduced from a combination of \rl{remtstr}(a) and \rl{tstr}(a).

Indeed, $\cD_a\simeq\Ind\cD_a^c$, while the assumption that the right adjoints $\phi_{\al}$ are continuous implies that $\Psi$ induces a functor $\cI\to\Catst:a\mapsto \cD^c_a$. Hence we have a natural equivalence $\cD\simeq\Ind\cD^c$ with $\cD^c:=\colim_{a\in\cI}\cD_a^c$.

Next, the assumption that each $\cD_a^{\geq 0}$ is closed under filtered colimits implies that the $t$-structure on
$\cD_a$ induces a $t$-structure on $\cD_a^c$. Hence \rl{tstr}(a) provides us with a $t$-structure on $\cD^c$, while
\rl{remtstr}(a) provides us with a $t$-structure on $\cD$ such that $\cD^{\geq 0}$ is closed under filtered colimits.
\end{Emp}

\begin{Emp}
\begin{proof}[Proof of \rp{limtstr}]
Let  $\cD^{\leq 0}\subset\cD$ be the smallest full subcategory, containing $\ins_a(x_a)$ with $x_a\in\cD_a^{\leq 0}$ and closed under all colimits, and let $\cD^{\geq 0}\subset\cD$ be the full subcategory, defined by \form{geq0'}. We claim that pair $(\cD^{\leq 0},\cD^{\geq 0})$ defines a $t$-structure on $\cD$.

First of all, we have to check that for every  $x\in\cD^{\leq 0}$ and  $y\in\cD^{\geq 1}$ we have $\Hom(x,y)\simeq 0$.
By the definition of $\cD^{\leq 0}$, we can assume that $x=\ins_a(x_a)$ with $x_a\in\cD^{\leq 0}_a$. In this case, we have
\[
\Hom_{\cD}(x,y)= \Hom_{\cD}(\ins_a(x_a),y)\simeq \Hom_{\cD_a}(x_a,\ev_a(y))\simeq 0,
\]
because  $x_a\in\cD^{\leq 0}_a$ (by assumption), and  $\ev_a(y) \in\cD_a^{\geq 1}$ (by \form{geq0'}).

Next, we are going to show that for every $x\in\cD$ there exists a fiber sequence $x_{\leq 0}\to x\to x_{\geq 1}$ with  $x_{\leq 0}\in\cD^{\leq 0}$ and
$x_{\geq 1}\in\cD^{\geq 1}$. By \rco{lurie}, for every $x\in \cD$, we have a natural functor $\cI\to\cD:a\mapsto\ins_a(x_a)$ with $x_a:=\ev_a(x)\in\cD_a$, and that the natural map $\colim_a \ins_a(x_a)\to x$ is an isomorphism.

Using the $t$-structure on $\cD_a$, we get a fiber sequence
\[
S_a:\tau_{\leq 0}(x_a)\to x_a\to \tau_{\geq 1}(x_a)
\]
with $\tau_{\leq 0}(x_a)\in\cD_a^{\leq 0}$ and $\tau_{\geq 1}(x_a)\in\cD_a^{\geq 1}$.

We claim that the functor $a\mapsto\ins_a(x_a)$ extends to the functor $a\mapsto\ins_a(S_a)$.
It suffices to show that a collection of morphisms $x_a\to \tau_{\geq 1}(x_a)$ gives rise to a morphism $\ins_{a}(x_{a})\to\ins_{a}(\tau_{\geq 1}(x_{a}))$ of functors $\cI\to\cD$.

The main point is to show that the assignment $a\mapsto \ins_{a}(\tau_{\geq 1}(x_{a}))$ is functorial in $a\in\cI$.
In other words, we want to show that every morphism $\al:a\to b$ in $\cI$
induces a canonical morphism $\ins_a(\tau_{\geq 1}(x_a))\to\ins_b(\tau_{\geq 1}(x_b))$.

Since $\ins_a\simeq\ins_b\circ\psi_{\al}$, it suffices to construct a morphism $\psi_{\al}(\tau_{\geq 1}(x_a))\to\tau_{\geq 1}(x_b)$, or, by adjointness, a morphism
$\iota_{\al}:\tau_{\geq 1}(x_a)\to\phi_{\al}(\tau_{\geq 1}(x_b))$. Since $\psi_{\al}$ is $t$-exact, we conclude that $\phi_{\al}$ is left $t$-exact.
Thus $\phi_{\al}(\tau_{\geq 1}(x_b))\in\cD_a^{\geq 1}$, so the natural morphism
\[
\Hom_{\cD_a}(\tau_{\geq 1}(x_a),\phi_{\al}(\tau_{\geq 1}(x_b)))\to \Hom_{\cD_a}(x_a,\phi_{\al}(\tau_{\geq 1}(x_b)))\simeq \Hom_{\cD_b}(\psi_{\al}(x_a),\tau_{\geq 1}(x_b)),
\]
induced by the morphism  $\pr_{\geq 1}: x_a\to \tau_{\geq 1}(x_a)$, is an isomorphism. Let $\iota_{\al}:\tau_{\geq 1}(x_a)\to\phi_{\al}(\tau_{\geq 1}(x_b))$ be the morphism corresponding to the composition
\[
\psi_{\al}(x_a)\simeq\psi_{\al}\circ\phi_{\al}(x_b)\overset{\on{counit}}{\lra} x_b\to \tau_{\geq 1}(x_b).
\]

Taking the colimit $\colim_a\ins_a(S_a)$, we get a fibred sequence
\[
x_{\leq 0}:=\colim_a\ins_a(\tau_{\leq 0}(x_a))\to x\to x_{\geq 1}:=\colim_a\ins_a(\tau_{\geq 1}(x_a)).
\]
Since $\tau_{\leq 0}(x_a)\in\cD_a^{\leq 0}$, the definition of $\cD^{\leq 0}$ implies that $x_{\leq 0}\in \cD^{\leq 0}$.

Next we show that $x_{\geq 1}\in\cD^{\geq 1}$, that is, $\ev_b(x_{\geq 1})\in\cD_b^{\geq 1}$ for all $b$.
Since $\ev_b$ commutes with all (small) colimits, and $\cD_b^{\geq 0}$ is closed under filtered colimits, we conclude
that $\cD^{\geq 1}$ is closed under all filtered colimits. Thus it suffices to show that for every $y_a\in \cD_a^{\geq 1}$,
we have $\ins_a(y_a)\in\cD^{\geq 1}$, that is, we have $\ev_b\circ \ins_a(y_a)\in\cD_b^{\geq 1}$ for all $b\in\cI$.

Since $\ev_b\circ\ins_a$ is a filtered colimit $\colim_{\al:a\to c,\beta:b\to c}\phi_{\beta}\circ \psi_{\al}$ (see \re{filt}), and
$\cD_b^{\geq 1}$ is closed under all filtered colimits, it suffices to show that
$\phi_{\beta}\circ \psi_{\al}(y_a)\in\cD_b^{\geq 1}$. But this follows from the fact $y_a\in\cD_a^{\geq 1}$, while both $\phi_{\beta}$ and $\psi_{\al}$ are left $t$-exact.

This completes the proof that $(\cD^{\leq 0},\cD^{\geq 0})$ is a $t$-structure. Moreover, the formula for $\cD^{\leq 0}$ implies that $\ins_a$ is right $t$-exact. Furthermore, since each $\cD_a^{\geq 0}$ is closed under filtered colimits by assumption and each $\ev_a$ is continuous, formula \form{geq0'} implies that $\cD^{\geq 0}$ is closed under filtered colimits.

Assume now that $(\cD'^{\leq 0},\cD'^{\geq 0})$ is another $t$-structure on $\cD$ such that
$\cD'^{\geq 0}$ is closed under filtered colimits and every functor $\ins_a:\cD_a\to \cD$ is $t$-exact.
We are going to show that in this case we have inclusions $\cD^{\leq 0}\subseteq\cD'^{\leq 0}$ and $\cD^{\geq 0}\subseteq\cD'^{\geq 0}$, therefore both inclusions have to be equalities
(say, by \rl{remtstr}(c)).

First of all, for every $x_a\in\cD_a^{\leq 0}$ we have $\ins_a(x_a)\in \cD'^{\leq 0}$, because $\ins_a$ is $t$-exact.
Since $\cD'^{\leq 0}$ is closed under all colimits, the first inclusion follows from the definition of $\cD^{\leq 0}$.
Next, for every $x\in\cD^{\geq 0}$ we have $\ev_a(x)\in\cD_a^{\geq 0}$ by \form{geq0'}, thus
$\ins_a(\ev_a(x))\in\cD'^{\geq 0}$, because $\ins_a$ is $t$-exact. Hence $x\simeq\colim_a\ins_a(\ev_a(x))\in\cD'^{\geq 0}$,
because $\cD'^{\geq 0}$ is closed under filtered colimits.
\end{proof}
\end{Emp}

\subsection{The case of algebraic spaces of finite type over $k$}

\begin{Emp} \label{E:cltstr}
{\bf Classical (middle-dimensional) perverse $t$-structures.}

(a) For an algebraic space $X$ of finite type over $k$, we denote by $({}^{p_{\cl}}\cD_c^{\leq 0}(X),{}^{p_{\cl}}\cD_c^{\geq 0}(X))$ the classical, that is, middle dimensional perverse $t$-structure on $\cD_c(X)$.

(b) Let $f:X\to Y$ be a morphism of algebraic spaces of finite type over $k$ such that all non-empty fibers of $f$ are of dimension $\leq d$. Then functors $f^*[d]$ and $f_![d]$ are right $t$-exact, while $f^![-d]$ and $f_*[-d]$  are left $t$-exact
(see \cite[4.2.4]{BBD}).
\end{Emp}

\begin{Emp} \label{E:gluetstr}
{\bf Gluing of $t$-structures.} Let $X$ be an algebraic space of finite type over $k$, assume that we are given a constructible stratification $\{X_{\al}\}_{\al}$ of $X$, and denote by $\eta_{\al}:X_{\al}\to X$ the embeddings. Suppose that we are given
a $t$-structure $(\cD_c^{\leq 0}(X_{\al}),\cD_c^{\geq 0}(X_{\al}))$ on each $\cD_c(X_{\al})$.

Then, by
the gluing lemma \cite[Theorem 1.4.10]{BBD} and induction on the number of strata, there exists a unique t-structure
$(\cD_c^{\leq 0}(X),\cD_c^{\geq 0}(X))$ on $\cD_c(X)$ such that all functors $\eta_{\al}^*$ are right $t$-exact, and all
functors $\eta^!_{\al}$ are left $t$-exact. Explicitly, for $K\in\cD_c(X)$, we have $K\in\cD_c^{\leq 0}(X)$
(resp. $K\in\cD_c^{\geq 0}(X))$ if and only if
$\eta_{\al}^*K\in\cD_c^{\leq 0}(X_{\al})$ (resp. $\eta_{\al}^!K\in\cD_c^{\geq 0}(X_{\al})$) for all $\al$.
\end{Emp}
\begin{Emp} \label{E:tstr}
{\bf $!$-Adapted perverse $t$-structure} (see Remark \re{remshifted} below for the explanation of the term).
Let $X$ be an algebraic space of finite type over $k$.

(a) Assume that $X$ is equidimensional of dimension $d$. We define ${}^p\cD_{c}^{\leq 0}(X)$\label{N:pDc} (resp. ${}^p\cD_{c}^{\geq 0}(X)$) to be the full subcategory of all $K\in\cD_{c}(X)$ such that $K[-d]$ belongs to  ${}^{p_{\cl}}\cD_{c}^{\leq 0}(X)$ (resp. ${}^{p_{\cl}}\cD_{c}^{\geq 0}(X)$). In other words,
$({}^p\cD_{c}^{\leq 0}(X),{}^p\cD_{c}^{\geq 0}(X))$ is $({}^{p_{\cl}}\cD_{c}^{\leq -d}(X),{}^{p_{\cl}}\cD_{c}^{\geq -d}(X))$, that is, the classical perverse $t$-structure, shifted by $\dim X$ to the left.

(b) Let now $X$ be arbitrary, and let $X_i$ be the canonical equidimensional stratification from \re{leq}(c).
We define ${}^p\cD_{c}^{\leq 0}(X)$ (resp. ${}^p\cD_{c}^{\geq 0}(X)$) to be the full subcategory of all $K\in\cD_{c}(X)$ such that $\eta_i^* K\in {}^p \cD_c^{\leq 0}(X_i)$ (resp. $\eta_i^! K\in {}^p\cD_c^{\geq 0}(X_i)$) for all $i$. Then $({}^p\cD_{c}^{\leq 0}(X),{}^p\cD_{c}^{\geq 0}(X))$ is a $t$-structure by the gluing lemma (see \re{gluetstr}).
\end{Emp}

\begin{Emp} \label{E:renpb}
{\bf Renormalized $*$-pullback.} Let $X\in\Algft_k$, and $K\in \cD_c(X)$.

(a) For every $d\in\B{Z}$ we set $K\lan d\ran:=K[2d](d)\in\cD(X)$. \label{N:lanran} More generally, to every locally constant function
$\un{d}:X\to\B{Z}$, we associate an object $K\lan \un{d}\ran\in\cD(X)$ \label{N:lanran2} such that for every connected component
$X^0\subset X$, we have $K\lan \un{d}\ran|_{X^0}:=K|_{X^0}\lan \un{d}(X_0)\ran$.

(b) For every weakly equidimensional morphism $f:X\to Y$ in $\Algft_k$, we define functor $f^{*,\ren}:\cD_c(Y)\to \cD_c(X)$ \label{N:fren} by
$f^{*,\ren}(K):=f^*(K)\lan\un{\dim}_f\ran$.
 \end{Emp}

\begin{Lem} \label{L:exact}
(a) Let $f:X\to Y$ be an equidimensional morphism in $\Algft_k$. Then $f^{*,\ren}$ is right $t$-exact, while $f^!$ is left $t$-exact.

(b) If $f:X\hra Y$ is a weakly equidimensional locally closed embedding of dimension $-d$ (see \re{locdim}(d)), then the pullback $f^*[-d]:\cD_{c}(Y)\to \cD_c(X)$ (resp.  $f^![-d]:\cD_{c}(Y)\to \cD_c(X)$)
is right (rest. left) $t$-exact.

(c) If $f:X\to Y$ is smooth or a universal homeomorphism, then the pullback $f^!$ is $t$-exact.
\end{Lem}

\begin{proof}
(a) Replacing $X$ by its connected component, we can assume that there exists $d\in\B{N}$ such that $\un{\dim}_f(x)=d$ for all $x\in X$. Then $f^{*,\ren}=f^*\lan d\ran$, all non-empty fibers of $f$ are equidimensional of dimension $d$, and  morphism $f$ induces a morphism $f_i:X_i\to Y_{i-d}$ for all $i$. We want to show that for every $K\in {}^p\cD_c^{\leq 0}(Y)$  we have $f^{*,\ren}(K)\in {}^p\cD_c^{\leq 0}(X)$.

Assume first that $Y$ is equidimensional, and hence $X$ is equidimensional as well.
Then our assumption $K\in {}^p\cD_c^{\leq 0}(Y)={}^{p_{\cl}}\cD_c^{\leq -\dim Y}(Y)$ implies (by \re{cltstr}(b)) that
\[
f^*(K)\in {}^{p_{\cl}}\cD_c^{\leq d-\dim Y}(X)={}^{p}\cD_c^{\leq d+\dim X-\dim Y}(X).
\]
Since $\dim X-\dim Y=d$, this implies that $f^*(K)\in {}^p\cD_c^{\leq 2d}(X)$, thus $f^{*,\ren}(K)\in {}^p\cD_c^{\leq 0}(X)$, as claimed.
In particular, the assertion holds for each morphism $f_i:X_i\to Y_{i-d}$.

In the general case, our assumption $K\in {}^p\cD_c^{\leq 0}(Y)$ implies that
$\eta_{i-d}^* K\in {}^p\cD_c^{\leq 0}(Y_{i-d})$ for all $i$. Therefore, by the assertion for $f_i$, we conclude
that
\[
\eta^*_i(f^{*,\ren}(K))\simeq \eta^*_i(f^*(K))\lan d\ran\simeq f_i^*(\eta^*_{i-d}K)\lan d\ran\simeq f_i^{*,\ren}(\eta^*_{i-d}K)\in {}^{p}\cD_c^{\leq 0}(X_{i})
\]
for all $i$, thus $f^{*,\ren}(K)\in {}^p\cD_c^{\leq 0}(X)$. The proof of the assertion for $f^!$ is similar.

(b) The argument is similar to (a) but simpler. Namely, as in (a), one reduces to the case when $Y$ is equidimensional. In this case, the assertion follows from the fact that $f^*$ (resp. $f^!$) is right (resp. left) $t$-exact for the classical perverse $t$-structure.

(c) If $f$ is smooth or a universal homeomorphism, then $f$ is equidimensional (see \re{smoothflat}(b)), and we have a canonical isomorphism $f^{*,\ren}\isom f^!$. Thus the assertion follows from (a).
\end{proof}

\begin{Emp} \label{E:remshifted}
{\bf Remark.}  The reason why we consider the $!$-adapted $t$-structure rather than the standard one is to guarantee that
for smooth morphisms the $!$-pullback is $t$-exact. This will enable us to define perverse $t$-structures on
placid $\infty$-stacks later.

\end{Emp}

\subsection{The case of placid $\infty$-stacks}
In this subsection we extend the $t$-structures defined in the previous section to placid $\infty$-stacks and study their properties.

\begin{Prop} \label{P:tplaff}
For every $Y\in \Affft_k$, we equip the category $\cD_c(Y)$ with the $!$-adapted $t$-structure, defined in \re{tstr}.
Then

(a) For every $0$-placid affine scheme $X$  there exists a unique $t$-structure on $\cD_c(X)$ such that for every strongly
pro-smooth morphism $f:X\to Y$ with $Y\in\Affft_k$, the pullback $f^!:\cD_c(Y)\to\cD_c(X)$ is $t$-exact.

(b) The $t$-structures from (a) satisfy the property that for every strongly pro-smooth morphism $f:X\to Y$ between $0$-placid affine schemes,
the pullback $f^!:\cD_c(Y)\to \cD_c(X)$ is $t$-exact.
\end{Prop}

\begin{proof}
(a) Recall (see \re{canpres}) that every $0$-placid affine scheme $X$ has a canonical placid presentation $X\simeq\lim_{X\to Y}Y$,
where the limit runs over all strongly pro-smooth morphisms $\pi:X\to Y$ with $Y\in\Affft_k$ and all the transition maps are smooth. This presentation induces a presentation of $\cD_c(X)$ as a filtered colimit $\cD_c(X)\simeq\colim_{X\to Y}\cD_c(Y)$, and all the transition maps are $t$-exact by \rl{exact}(c). Therefore it follows from \rl{tstr}(a) that there exists a unique $t$-structure on $\cD_c(X)$ such that for every $\pi$ as above the pullback $\pi^!:\cD_c(Y)\to\cD_c(X)$ is $t$-exact.

(b) Choose a placid presentation $Y\simeq\lim_{\al}Y_{\al}$ of $Y$. Then we have ${}^p\cD^{\leq 0}_c(Y)\simeq\colim_{\al}{}^p\cD^{\leq 0}_c(Y_{\al})$ and ${}^p\cD^{\geq 0}_c(Y)\simeq\colim_{\al}{}^p\cD^{\geq 0}_c(Y_{\al})$ (see \rl{tstr}(a)). Thus it suffices to show that
the composition $f^!\pr_{\al}^!:\cD_c(Y_{\al})\to \cD_c(Y)\to\cD_c(X)$ is $t$-exact. Since $f$ and $\pr_{\al}$ are strongly pro-smooth by assumption, the composition $\pr_{\al}\circ f$ is strongly pro-smooth (see \re{placidprop}). Therefore it follows from the characterization of the $t$-structure on $\cD_c(X)$ that $f^!\pr_{\al}^!\simeq (\pr_{\al}\circ f)^!$ is $t$-exact.
\end{proof}

\begin{Emp} \label{E:tstrind}
{\bf Perverse $t$-structures on $\cD(X)$.}
(a) Let $X$ be a $0$-placid affine scheme. Then the $\infty$-category $\cD(X)$ is the ind-category $\Ind\cD_c(X)$ (see \re{basicfunc}). Therefore it follows from \rl{remtstr}(a) that the perverse $t$-structure $({}^p\cD_c^{\leq 0}(X),{}^p\cD_c^{\geq 0}(X))$ on $\cD_c(X)$ defined in \rp{tplaff}(a) extends uniquely to a perverse $t$-structure $({}^p\cD^{\leq 0}(X),{}^p\cD^{\geq 0}(X))$ on $\cD(X)$ such that
the subcategory ${}^p\cD^{\geq 0}(X)\subset\cD(X)$ is closed under filtered colimits. Explicitly, we have ${}^p\cD^{\geq 0}(X)=\Ind({}^p\cD_c^{\geq 0}(X))$ and similarly for ${}^p\cD^{\leq 0}(X)$.

(b) For every strongly pro-smooth morphism $f:X\to Y$ of $0$-placid affine schemes, the pullback $f^!:\cD(Y)\to \cD(X)$ is $t$-exact.
Indeed, since the pullback $f^!:\cD_c(Y)\to \cD_c(X)$ is $t$-exact by \rp{tplaff}(b), the assertion follows from the continuity of $f^!$ and formulas for ${}^p\cD^{\leq 0}(X)$ and ${}^p\cD^{\geq 0}(X)$.
\end{Emp}

From now on we will write $\cD_{\cdot}(\C{X})$\label{N:Ddot} to refer both to $\cD_{c}(\C{X})$ and $\cD(\C{X})$.

\begin{Prop} \label{P:tstrpl}
For every $0$-placid affine scheme $X$, we equip $\cD_{\cdot}(X)$ with $t$-structure, constructed in \rp{tplaff} and construction \re{tstrind}.

(a) For every placid $\infty$-stack $\cX$,  there exists a unique $t$-structure $({}^p\cD^{\leq 0}_{\cdot}(\cX),\cD^{\geq 0}_{\cdot}(\cX))$ on $\cD_{\cdot}(\cX)$ such that for every smooth morphism $f:X\to\cX$ from a $0$-placid affine scheme $X$, the pullback $f^!:\cD_{\cdot}(\cX)\to \cD_{\cdot}(Y)$ is $t$-exact.

(b) The subcategory $\mathstrut^{p}\cD^{\geq 0}(\cX)\subset \cD(\cX)$ is closed under filtered colimits.

(c) If $f:\cX\to\cY$ is either a smooth morphism between placid $\infty$-stacks or a topological equivalence, then the pullback $f^!$ is $t$-exact.
\end{Prop}

\begin{proof}
(a) Recall (see \re{propplst}(a)) that we have a canonical isomorphism $(\colim_{X\to\cX}X)\to\cX$, where the colimit runs over smooth morphisms
$f:X\to\cX$ from $0$-placid affine schemes $X$, and transition maps are strongly pro-smooth.
This isomorphism induces an equivalence  $\cD_{\cdot}(\cX)\simeq\lim_{X\to \cX}\cD_{\cdot}(X)$, where all transition maps are
$t$-exact by \rp{tplaff} and construction \re{tstrind}. Therefore it follows from \rl{tstr}(b) that there exists a unique $t$-structure on $\cD_{\cdot}(\cX)$ such that all pullbacks $f^!:\cD_{\cdot}(\cX)\to\cD_{\cdot}(X)$ are $t$-exact. Explicitly, ${}^p\cD^{\geq 0}_{\cdot}(\cX)=\lim_{X\to\cX}{}^p\cD^{\geq 0}_{\cdot}(X)$ and similarly for ${}^p\cD^{\leq 0}_{\cdot}(\cX)$.

(b) follows from the corresponding assertion for $0$-placid affine schemes (see \re{tstrind}) and continuity of $f^!$.

(c) Since smooth morphisms are closed under composition (see \re{propplst}(b)), the assertion for smooth $f$ follows from
the explicit description of $({}^p\cD^{\leq 0}_{\cdot}(\cX),{}^p\cD^{\geq 0}_{\cdot}(\cX))$, given above.

Next, for the assertion for topological equivalences, it suffices to show the particular case when $f$ is the canonical morphism
$\cY_{\red}\to\cY$. Next, using definitions of $t$-structures and \rl{reduction}, we reduce to the case when
$\cY$ is an affine scheme of finite type over $k$. In this case, the $t$-exactness of $f^!$ was shown in \rl{exact}(c).
\end{proof}

\begin{Emp} \label{E:extstr}
{\bf Example.} Let $X$ be an algebraic space of finite type over $k$. Then $X$ is a placid $\infty$-stack (see \re{explinfst}(a)), and
the perverse $t$-structure on $\cD_c(X)$ defined in \re{tstr} coincides with the $t$-structure defined in
\rp{tstrpl}(a).

Indeed, let $f:Y\to X$ be an \'etale covering by an affine scheme $Y$. Then  $\cD_c(Y)$ is equipped with a perverse $t$-structure
(see \re{tstr}), pullback $f^!$ is faithful (see \re{shpr}(e)), hence there exists at most one $t$-structure  on $\cD_c(X)$ such that $f^!$ is $t$-exact
(by \rl{remtstr}(d)). On the other hand, it follows from \rl{exact}(c) and \rp{tstrpl}(a) that both above $t$-structures on $\cD_c(X)$ satisfy this property.
\end{Emp}

%
%




\begin{Lem} \label{L:exact2}
(a) Let $\cX$ be a placid $\infty$-stack. Then $\om_{\cX}\in {}^p\cD^{\geq 0}(\cX)$.

(b) Let $f:\cX\to\cY$ be an equidimensional morphism (see \re{clM2}(a)) of placid $\infty$-stacks. Then the functor
$f^!$ is left $t$-exact.


(c) Let $f:\cX\to\cY$ be an fp-locally closed embedding of placid $\infty$-stacks of relative dimension $-d$. Then the pullback $f^*[-d]$ (resp. $f^![-d]$) is right (resp. left) $t$-exact.

(d) In the situation of (c), assume that $\cY$ is smooth. Then $f^{*}(\om_{\cY})\in{}^p\cD^{\leq -2d}(\cX)$.
\end{Lem}


\begin{proof}

(a) Assume first that $X\in\Affft_k$, and let $\pi:X\to\pt$ be the projection. If $X$ is locally equidimensional, then $\om_X=\pi^!(\om_{\pt})\in
{}^p\cD^{\geq 0}(X)$ by \rl{exact}(a). In the general case, let $\{X_i\}_i$ be the equidimensional stratification from \re{leq}(c).
Since the assertion holds for each $X_i$, we have $\eta_i^!(\om_X)\simeq\om_{X_i}\in {}^p\cD^{\geq 0}(X_i)$ for all $i$, thus
$\om_X\in {}^p\cD^{\geq 0}(X)$.

Next, let $X\in\Aff_k$ be an affine scheme with a placid presentation $X\simeq\lim_{\al}X_{\al}$.  Then we have $\om_X\simeq\pr_{\al}^!(\om_{X_{\al}})\in {}^p\cD^{\geq 0}(X)$, because $\om_{X_{\al}}\in {}^p\cD^{\geq 0}(X_{\al})$, and $\pr_{\al}^!$ is
$t$-exact.

Then, for a placid $\infty$-stack $\cX$, choose a smooth covering $f=\sqcup_{\al}f_{\al}:\sqcup_{\al}X_{\al}\to\cX$, where
each $X_{\al}$ is an affine scheme with a placid presentation. Then, by the proven above, $f_{\al}^!(\om_{\cX})\simeq\om_{X_{\al}}\in {}^p\cD^{\geq 0}(X_{\al})$ for all $\al$, therefore $\om_{\cX}\in {}^p\cD^{\geq 0}(\cX)$.


(b) 
Choose a smooth covering $Y\to \cY$, where $Y\simeq\sqcup_{\al}Y_{\al}$, and each $Y_{\al}$ is a $0$-placid affine scheme. Since it suffices to show a result after a base change to each $Y_{\al}$, we can assume that $\cY$ is a $0$-placid affine scheme $Y$. Then, choose a smooth covering $X\to \cX$, where $X\simeq\sqcup_{\al}X_{\al}$, and each $X_{\al}$ is a $0$-placid affine scheme. Since it suffices to show the assertion for each $X_{\al}\to X\to\cX\to Y$, we can assume that $\cX$ is a $0$-placid affine scheme $X$.

In this case, it suffices to show the assertion for $\cD_c$. Choose placid presentations $X\simeq\lim_{\al}X_{\al}$ and $Y\simeq\lim_{\beta}Y_{\beta}$. Then $\cD_c(\cY)\simeq \colim_{\beta}\cD_c(Y_{\beta})$, so it suffices to show the left $t$-exactness of each $f^{!}\circ \pr_{\beta}^!\simeq(\pr_{\beta}\circ f)^!$.

By \rl{clM}(iv), composition $\pr_{\beta}\circ f$ decomposes as $X\overset{\pr_{\al}}{\lra}X_{\al}\overset{f_{\al,\beta}}{\lra}Y_{\beta}$, where $f_{\al,\beta}$ is  equidimensional. Then $\pr_{\al}^!$ is $t$-exact, because $\pr_{\al}$ is strongly pro-smooth, and  while $f^!_{\al,\beta}$ is left $t$-exact by \rl{exact}(a). Hence
$(\pr_{\beta}\circ f)^!\simeq  \pr_{\al}^!\circ f^!_{\al,\beta}$ is left $t$-exact, as claimed.

(c) 
By \rl{fplocl}, $f$ decomposes as a composition $f=i\circ j$ of an fp-open and an fp-closed embedding. Hence,
by \rp{glue1}(b), the pullback $f^*$ satisfies the base change with respect to smooth $!$-pullbacks. Thus (as in (b)), we reduce to the case when $\cY$ is a $0$-placid affine scheme, and it suffices to show the assertion for $\cD_c$ instead of $\cD$.

Next, as in (b), we choose a placid presentation  $\cY\simeq\lim_{\al}Y_{\al}$. Then the fp-locally closed embedding $f:\cX\to \cY$ is a pullback of a locally closed embedding $f_{\al}:X_{\al}\to Y_{\al}$ of relative dimension $-d$, and it suffices to show the corresponding assertions for
$f_{\beta}:=f_{\al}\times_{Y_{\al}}Y_{\beta}$ for all sufficiently large $\beta>\al$. In this case, the assertion follows from
\rl{exact}(b).

(d) Arguing as in (c), we reduce to the case when $\cY$ is a smooth affine scheme $Y$ of finite type over $k$, and $f:X\hra Y$ is a locally closed embedding. Then $\pr_{Y}:Y\to\pt$ is equidimensional, and we have a canonical isomorphism $\om_Y\simeq\pr_Y^!(\om_{\pt})\simeq \pr_Y^{*,\ren}(\om_{\pt})$. Therefore we have
\[
f^{*}(\om_Y)\lan d\ran\simeq f^{*,\ren}(\om_Y)\simeq f^{*,\ren}(\pr_Y^{*,\ren}(\om_{\pt}))\simeq \pr_X^{*,\ren}(\om_{\pt})\in {}^p\cD^{\leq 0}(X)
\]
by \rl{exact}(a), thus $f^{*}(\om_Y)\in {}^p\cD^{\leq -2d}(X)$.
%
\end{proof}

%

\begin{Lem} \label{L:exact3}
Let $f:\cX\to\cY$ be a morphism between placid $\infty$-stacks, which is ind-fp-proper, locally fp-representable, and equidimensional of relative dimension $d$. Then the functor $f_![-2d]$ is left $t$-exact.
\end{Lem}

\begin{proof}
Replacing $f$ by a pullback with respect to a smooth morphism $Y\to\cY$, we can assume that $\cY$ is a $0$-placid affine scheme $Y$, and $f:\cX\to Y$ is an ind-fp-proper morphism. Choose a  presentation $\cX\simeq\colim_{\al} X_{\al}$, where each $f_{\al}:X_{\al}\to Y$ is fp-proper and all transition maps are fp-closed embeddings. Replacing $\cX$ by  $\cX_{\red}$ (and $X_{\al}$ by $X_{\al,\red}$), we can assume that $\cX$  is an algebraic space $X$, which is locally fp and  equidimensional of relative dimension $d$ over $Y$ (use \rco{red-fp}).

Denote by  $i_{\al}:X_{\al}\to X$ the inclusion. By \rco{lurie}, for every $K\in\cD(X)$, we have a natural isomorphism $K\simeq\colim_{\al}(i_{\al})_!i_{\al}^!K$, which induces an isomorphism
$f_!(K)\simeq \colim_{\al}(f_{\al})_!i_{\al}^!(K)$. Since ${}^p\cD^{\geq 0}(X)\subset \cD(X)$ is closed under filtered colimits, it suffices to show that each composition $(f_{\al})_!i_{\al}^![-2d]$ is left $t$-exact.

Next, since $f_{\al}$ is an fp-proper morphism between algebraic spaces admitting placid presentations, $(f_{\al})_!$ has a left adjoint $f_{\al}^*$ (by  \rp{adj}(c)). Therefore passing to left adjoints, it suffices to show that
each composition $(i_{\al})_!f_{\al}^*[2d]:\cD(Y)\to \cD(X)$ is right $t$-exact. Thus, it suffices to show that for every \'etale morphism $\eta:U\to X$, where $U$ is an affine scheme, fp over $Y$, the composition $\eta^!(i_{\al})_!f_{\al}^*[2d]$ is right $t$-exact.

By definition, $\eta:U\to X$ factors through some $X_{\beta}$. Let $i_{\al,\beta}$ be the embedding $X_{\al}\to X_{\beta}$, and let
$i:U_{\al}\to U$ be the pullback of $i_{\al}$. Since $f_{\al}^*\simeq i_{\al,\beta}^*f_{\beta}^*$, we have a natural (base change) isomorphism
\[
\eta^!(i_{\al})_!f_{\al}^*[2d]\simeq \eta^!(i_{\al})_!i_{\al,\beta}^*f_{\beta}^*[2d]\simeq i_!i^*(\eta^!f_{\beta}^*[2d]).
\]
Since composition $\eta^! f^*_{\beta}[2d]$ is right $t$-exact (see \rcl{rtexact} below), and
composition $i_!i^*$ is right $t$-exact (see \rco{closedemb} below), the assertion follows.
\end{proof}

\begin{Cl} \label{C:rtexact}
Let $U\overset{\eta}{\to}X\overset{f}{\to}Y$ be morphisms of algebraic spaces admitting placid presentations
such that $f$ is fp-proper, $\eta$ is \'etale, and $f\circ \eta$ is equidimensional of relative dimension $d$. Then the composition
$\eta^!f^*[2d]$ is right $t$-exact.
\end{Cl}
\begin{proof}
Arguing as in \rl{exact2}, we reduce to the case when $U, X$ and $Y$ are algebraic spaces of finite type over $k$. Since $\eta$ is \'etale, while
$f\circ \eta$ is equidimensional of relative dimension $d$, the composition $\eta^!f^*[2d]$ is isomorphic to
\[
\eta^!f^*[2d]\simeq \eta^*f^*[2d]=(f\circ\eta)^*[2d]\simeq(f\circ\eta)^{*,\ren}.
\]
Therefore it is right $t$-exact by \rl{exact}(a).
\end{proof}






\subsection{The case of placidly stratified $\infty$-stacks}
In this section we will define a larger class of $\infty$-stacks, which admit perverse $t$-structures.

\begin{Emp} \label{E:pervfunc}
{\bf Perversity function.}
By a {\em perversity}\label{I:perversity} on an $\cI$-stratified $\infty$-stack $\cX$ (see \re{stratified}), we mean a function $p_{\nu}:\C{I}\to\B{Z}:\al\mapsto \nu_{\al}$, \label{N:pnu}\label{N:nual} or, what is the same, a collection $p_{\nu}=\{\nu_{\al}\}_{\al\in\cI}$ of integers.
\end{Emp}

Recall that if $(\cX, \{\cX_{\al}\}_{\al})$ is a placidly stratified $\infty$-stack, admitting gluing of sheaves (see \rd{glue}), then every $\C{D}(\cX_{\al})$ is equipped with a perverse $t$-structure ${}^p\C{D}(\cX_{\al})$ (see \rp{tstrpl}), and  we have two pullback functors $\eta_{\al}^!,\eta_{\al}^*:\C{D}(\cX)\to \C{D}(\cX_{\al})$.

\begin{Prop} \label{P:tstrglue}
Let $(\cX,\{\cX_{\al}\}_{\al\in\cI})$ be a placidly stratified $\infty$-stack, admitting gluing of sheaves, and equipped with a perversity
$p_{\nu}=\{\nu_{\al}\}$.

(a) If the stratification is bounded, then there exists a unique $t$-structure
$({}^{p_{\nu}}\cD^{\leq 0}(\cX),{}^{p_{\nu}}\cD^{\geq 0}(\cX))$ on $\cD(\cX)$ such that
\begin{equation} \label{Eq:geq0}
\mathstrut^{p_{\nu}}\cD^{\geq 0}(\cX)=\{K\in \cD(\cX)~\vert~\eta^{!}_{\alpha}K\in \mathstrut^{p}\cD^{\geq -\nu_{\al}}(\cX_{\alpha})\text{ for all }\alpha\in\C{I}\},
\end{equation}
\begin{equation} \label{Eq:leq0}
\mathstrut^{p_{\nu}}\cD^{\leq 0}(\cX)=\{K\in \cD(\cX)~\vert~\eta^{*}_{\alpha}K\in \mathstrut^{p}\cD^{\leq -\nu_{\al}}(\cX_{\alpha})\text{ for all }\alpha\in\C{I}\}.
\end{equation}
Moreover, the subcategory $\mathstrut^{p_{\nu}}\cD^{\geq 0}(\cX)\subset \cD(\cX)$ is closed under filtered colimits.

(b) In the general case, there exists a unique  $t$-structure
$({}^{p_{\nu}}\cD^{\leq 0}(\cX),{}^{p_{\nu}}\cD^{\geq 0}(\cX))$ on $\cD(\cX)$ satisfying \form{geq0}.
\end{Prop}
	
\begin{proof}

(a) Assume first that $\cI$ is finite. In this case, the assertion follows from the gluing theorem \cite[Theorem 1.4.10]{BBD} by induction on $|\C{I}|$:

Since for $|\C{I}|=1$ the assertion is clear, we may assume that $|\C{I}|>1$.
By \re{remcons}(a), there exists $\beta\in\cI$ such that $Z:=\cX_{\beta}\subset\cX$ is topologically fp-closed, and $\{\cX_{\al}\}_{\al\in \C{I}\sm\beta}$ forms a constructible stratification of $\C{U}:=\cX\sm\cZ$. Then $(\cU, \{\cX_{\al}\}_{\al\in \C{I}\sm\beta})$ is a placidly stratified $\infty$-stack, admitting  gluing of sheaves (by \rl{openglue}(a)).
Therefore, by the induction hypothesis, there exists a unique $t$-structure $(\mathstrut^{p_{\nu}}\cD^{\leq 0}(\cU),\mathstrut^{p_{\nu}}\cD^{\geq 0}(\cU))$ on $\cU$ satisfying \form{geq0} and \form{leq0} for $\al\in\cI\sm\beta$.

Now let $i:\cZ\hra\cX$ and $j:\cU\hra \cX$ be the corresponding topologically fp-closed and fp-open embeddings.
Since $\C{X}$ admits gluing of sheaves, we conclude from Lemmas \ref{L:glue1} and \ref{L:glue} that
all the assumptions of \cite[1.4.3]{BBD} are satisfied. Therefore by \cite[Theorem 1.4.10]{BBD} there exists a unique
$t$-structure $({}^{p_{\nu}}\cD^{\leq 0}(\cX), {}^{p_{\nu}}\cD^{\geq 0}(\cX))$ on $\cD(\cX)$ such that $K\in\cD(\cX)$ belongs to
${}^{p_{\nu}}\cD^{\leq 0}(\cX)$ (resp. ${}^{p_{\nu}}\cD^{\geq 0}(\cX)$) if and only if we have
$j^*K\in {}^{p_{\nu}}\cD^{\leq 0}(\cU)$ and $i^*K\in {}^p\cD^{\leq -\nu_{\al}}(\cZ)$ (resp.
$j^!K\in {}^{p_{\nu}}\cD^{\geq 0}(\cU)$ and $i^!K\in {}^p\cD^{\geq -\nu_{\al}}(\cZ)$).
This finishes the argument when $\cI$ is finite.

In the general case, $\cX$ can be written as a filtered colimit  $\cX\simeq\colim_{U}\cX_U$, where each
$\cX_U\subset\cX$ is an fp-open substack having a finite constructible stratification $\{\cX_{\al}\}_{\al\in\cI_U}$.
Since each $\cX_U$ admits a gluing of sheaves (see \rl{openglue}(a)), we deduce from the finite case shown above that
each $\cD(\cX_{U})$ is equipped with a unique $t$-structure, satisfying \form{geq0} and \form{leq0}
for $\al\in\cI_{\cX_U}$. Furthermore, equalities \form{geq0} and \form{leq0} imply that
for every $\cX_U\subset\cX_{U'}$ the restriction functor $\cD(\cX_{U'})\to\cD(\cX_{U})$
is $t$-exact. Therefore it follows from \rl{tstr}(b) that there exists a unique $t$-structure on $\cD(\cX)\simeq\lim_{U}\cD(\cX_U)$, satisfying
\form{geq0} and \form{leq0} for all $\al\in\cI$.

Finally, since every functor $\eta_{\al}^!$ preserves small colimits, the last assertion follows from  \form{geq0} and the
corresponding assertion for placid $\infty$-stacks (see \rp{tstrpl}).

(b)  First of all, the uniqueness assertion follows from \form{geq0} and \rl{remtstr}(e). Next, by assumption, $\cX$ can be written as a filtered colimit $\cX\simeq\colim_{\la\in \La}\cX_{\la}$, where each  $\cX_{\la}$ has a bounded stratification by $\{\cX_{\al}\}_{\al\in \cI_{\cX_{\la}}}$, and all transition maps $i_{\la,\mu}:\cX_{\la}\to\cX_{\mu}$ are topologically
fp-closed embeddings.

Then, by (a),  each $\cD(\cX_{\la})$ has a unique perverse $t$-structure $({}^{p_{\nu}}\cD^{\leq 0}(\cX_{\la}),{}^{p_{\nu}}\cD^{\geq 0}(\cX_{\la}))$ satisfying \form{geq0} for $\al\in\cI_{\cX_{\la}}$, and the subcategory $\mathstrut^{p_{\nu}}\cD^{\geq 0}(\cX_{\al})\subset \cD(\cX_{\al})$ is closed under filtered colimits. Moreover, each pushforward $(i_{\la,\mu})_!:\cD(\cX_{\la})\to\cD(\cX_{\mu})$ is $t$-exact (see \rl{midext}(a) below), and has a
continuous right adjoint $i_{\la,\mu}^!$.

Therefore all the assumptions of \rt{lu1} and \rp{limtstr} are satisfied, hence the limit category
$\cD(\cX)=\lim_{\la}\cD(\cX_{\la})$ is equipped with a canonical $t$-structure $({}^{p_{\nu}}\cD^{\leq 0}(\cX),{}^{p_{\nu}}\cD^{\geq 0}(\cX))$. Let $i_{\la}:\cX_{\la}\hra\cX$ be the inclusion. Then the formula \form{geq0'} says in our case that
\begin{equation} \label{Eq:geq0''}
{}^{p_{\nu}}\cD^{\geq 0}(\cX)=\{K\in \cD^{\leq 0}(\cX)\,|\, i_{\la}^!K\in{}^{p_{\nu}}\cD^{\geq 0}(\cX_{\la}) \text{ for all }\la\in \La\}.
\end{equation}
Combining \form{geq0''} and equality \form{geq0} for ${}^{p_{\nu}}\cD^{\geq 0}(\cX_{\la})$ for each $\la\in\La$, we conclude that equality \form{geq0} holds for ${}^{p_{\nu}}\cD^{\geq 0}(\cX)$.
\end{proof}	

\begin{Emp} \label{E:stperv}
{\bf The ``canonical'' perversity.}
(a) Let $\cX$ be a placid $\infty$-stack, and let
$\{\cX_{\al}\}_{\al\in\cI}$ be a bounded constructible stratification. Then every $\cX_{\al}$ is placid (by \rco{red-fp} and \rl{glplacid}(c)), therefore $\cX$ is a placidly stratified $\infty$-stack.

(b) Assume now that each $\cX_{\al}\subset\cX$ is of pure codimension $\nu_{\al}$. We denote by $p_{\can}$ the {\em canonical} perversity
$p_{\can}:=\{\nu_{\al}\}_{\al}$ on $\cX$.
\end{Emp}

The following lemma explains why we call this perversity {\em canonical}.

\begin{Lem} \label{L:sttstr}
In the situation of \re{stperv}, the canonical $t$-structure ${}^{p_{\can}}\cD(\cX)$ on $\cD(\cX)$, defined by the perversity $p_{\can}$, coincides with the $!$-adapted perverse $t$-structure ${}^p\cD(\cX)$.
\end{Lem}

\begin{proof}
By \rl{exact2}(c), our assumption on $\eta_{\al}:\cX_{\al}\to\cX$
implies that for every $K\in {}^{p}\cD^{\leq 0}(\cX)$ (resp. $K\in {}^{p}\cD^{\geq 0}(\cX)$), we have $\eta_{\al}^*K\in {}^{p}\cD^{\leq -\nu_{\al}}(\cX_{\al})$
(resp. $\eta_{\al}^!K\in {}^{p}\cD^{\geq -\nu_{\al}}(\cX_{\al})$). Therefore by formulas \form{geq0} and \form{leq0},
we have inclusions ${}^{p}\cD^{\leq 0}(\cX)\subset{}^{p_{\can}}\cD^{\leq 0}(\cX)$ and  ${}^{p}\cD^{\geq 0}(\cX)\subset{}^{p_{\can}}\cD^{\geq 0}(\cX)$. But then both inclusions have to be equalities (see \rl{remtstr}(c)), and the assertion follows.
\end{proof}


Next we show that many of the properties of the classical perverse $t$-structure extend to our setting almost word-by-word.

\begin{Lem} \label{L:openemb}
Let $(\cX,\{\cX_{\al}\}_{\al\in\cI})$ be a placidly stratified $\infty$-stack, admitting gluing of sheaves, and let $j:\cU\rightarrow \cX$ an fp-open embedding. Then $(\cU,\{j^{-1}(\cX_{\al})\}_{\al\in\cI})$ is a placidly stratified $\infty$-stack, admitting gluing of sheaves as well.

Moreover, if $p_{\nu}=\{\nu_{\al}\}_{\al}$ is a perversity on $\cX$, and $p'_{\nu}:=\{\nu_{\al}\}_{\al}$ is the corresponding perversity on $\cU$, then the functor $j^{!}$ is $t$-exact, $j_{!}$ is right $t$-exact and $j_{*}$ is left $t$-exact.
\end{Lem}

\begin{proof}
Since $\cU$ admits gluing of sheaves by \rl{openglue}(a), the first assertion follows from the fact every fp-open $\infty$-substack of a placid $\infty$-stack is placid (see \rl{glplacid}(c)).

When $|\cI|=1$, the $\infty$-stack $\C{X}$ is placid, and $j$ is smooth. In this case, the $t$-exactness of $j^!$ follows for example from  \rp{tstrpl}(c), while the $t$-exactness assertions for $j_!$ and $j_*$ follow by adjunction.

In the general case, it suffices to show that $j^!$ and $j_*$ are left $t$-exact (by adjunction). Using \form{geq0} together with the fact that functor $j_*$ satisfies the base change (see \rl{glue1}(a)), we reduce to the case of $|\cI|=1$, shown above.
\end{proof}

\begin{Emp}
{\bf Remark.} When the stratification is bounded, then the argument of \rl{openemb} can be simplified. Namely, $t$-exactness of $j^!$ follows from \form{geq0} and \form{leq0} and the $|\cI|=1$ case, while the $t$-exactness properties of $j_!$ and $j_*$ follow by adjunction.
\end{Emp}

\begin{Cor} \label{C:closedemb}
Let $(\cX,\{\cX_{\al}\}_{\al\in\cI})$ be a placidly stratified $\infty$-stack, admitting gluing of sheaves, and let $i:\cZ\rightarrow \cX$ be a topologically  fp-closed embedding. Then the composition $i_!i^*$ is right $t$-exact, while $i_!i^!$ is left $t$-exact.
\end{Cor}

\begin{proof}
Consider fp-open embedding $j:\cU:=\cX\sm\cZ\hra \cX$. Since both $j_!$ and $j^!$ are right $t$-exact (by \rl{openemb}), the assertion for $i_!i^*$ follows from the fiber sequence $K\to i_!i^*K\to j_!j^!K[1]$ (see \rl{glue}(c)). The second assertion follows by adjunction.
\end{proof}

\begin{Emp} \label{E:intext}
{\bf The intermediate extension.}
Let $\cX$ be a placidly stratified $\infty$-stack, admitting gluing of sheaves and equipped with perversity $p$,
let $j:\cU\rightarrow\cX$ an fp-open immersion, and let $p'$ be the induced perversity on $\cU$ (see \rl{openemb}).

(a) For every $K\in\cD(U)$, we have a canonical map $\theta:j_!K\to j_*K$, adjoint
to the isomorphism $K\isom j^!j_*K$, hence a map ${}^pH^0(\theta):\mathstrut^{p}H^{0}(j_{!}K)\rightarrow \mathstrut^{p}H^{0}(j_{*}K)$.

(b) For every $K\in\Perv^{p'}(U)$, we define
\[
j_{!*}K:=\Ima(\mathstrut^{p}H^{0}(j_{!}K)\to \mathstrut^{p}H^{0}(j_{*}K))
\]
be the image of ${}^pH^0(\theta)$ and call it {\em the intermediate extension} of $K$.
In particular, we have a canonical surjection $\theta_1:\mathstrut^{p}H^{0}(j_{!}K)\to j_{!*}K$ and
injection $\theta_2:j_{!*}K\to \mathstrut^{p}H^{0}(j_{*}K)$.

(c) As in \re{support}(a), we say that $K\in\cD(\cX)$ is {\em supported on} $\cX\sm\cU$, if $K\in \cD_{\cX\sm\cU}(\cX)$, that is, $j^!K\simeq 0$
(see \re{compl}(c)).
\end{Emp}

\begin{Cor} \label{C:supr}
In the situation of \re{intext}, let $K\in\Perv^{p'}(\cU)$. Then

(a) The kernel of $\theta_1:\mathstrut^{p}H^{0}(j_{!}K)\to j_{!*}K$ and cokernel of $\theta_2:j_{!*}K\to \mathstrut^{p}H^{0}(j_{*}K)$
are supported on $\cX\sm\cU$.

(b) The perverse sheaf $\mathstrut^{p}H^{0}(j_! K)$ (resp. $\mathstrut^{p}H^{0}(j_{*}K)$) has no non-zero quotients (resp. subobjects)
supported on $\cX\sm\cU$.

(c) The intermediate extension $j_{!*}K\in\Perv^p(\cX)$ is the unique $\wt{K}\in \Perv^p(\cX)$ such that
$j^!\wt{K}\simeq K$, and $\wt{K}$ has no non-zero subobjects and quotients, supported on $\cX\sm\cU$.
\end{Cor}

\begin{proof}
All assertions formally follow from \rl{openemb} and adjunctions.

(a) Follows from the fact that $j^!$ is $t$-exact and $j^!(\theta)$ is an isomorphism.

(b) Assume that $L\in\Perv^p(\cX)$ is supported on $\cX\sm\cU$, that is, $j^!L\simeq 0$.  As $j_{!}K\in\mathstrut^{p}\cD^{\leq 0}(\cX)$ and $j_{*}K\in\mathstrut^{p}\cD^{\geq 0}(\cX)$ (by \rl{openemb}), we have isomorphisms
\[
\Hom(\mathstrut^{p}H^{0}(j_{!}K),L)\simeq\Hom(j_{!}K,L)\simeq\Hom(K,j^!L)\simeq 0
\]
and
\[
\Hom(L, \mathstrut^{p}H^{0}(j_* K))\simeq\Hom(L, j_*K)\simeq\Hom(j^!L,K)\simeq 0.
\]

(c) Since $j^!$ is $t$-exact, we have $j^!j_{!*}K\simeq j^!\Ima({}^pH^0(\theta))\simeq K$. Next, if $L$ is a subobject (resp. quotient) of $j_{!*}K$,
supported on $\cX\sm\cU$, then $L$ is a subobject (resp. quotient) of
$\mathstrut^{p}H^{0}(j_{*}K)$ (resp. $\mathstrut^{p}H^{0}(j_{!}K)$). So $L\simeq 0$ by (b).

Conversely, let $\wt{K}\in \Perv^p(\cX)$ be such that
$j^!\wt{K}\simeq K$ and $\wt{K}$ has no non-zero subobjects and quotients supported on $\cX\sm\cU$.
By adjunction, the isomorphism $j^!\wt{K}\simeq K$ induces morphisms
$j_!K\to\wt{K}\to j_*K$, hence morphisms
\[
\mathstrut^{p}H^{0}(j_!K)\overset{a}{\to}\wt{K}\overset{b}{\to}\mathstrut^{p}H^{0}(j_*K).
\]
We want to show that $a$ is surjective, while $b$ is injective. Since $j^!$ is $t$-exact, we deduce that
$\Coker a$ and $\Ker b$ are supported on $\cX\sm\cU$. Hence both of them are zero by the assumption on $\wt{K}$.
\end{proof}

\begin{Cor} \label{C:contperv}
In the situation of \re{intext}, for every pair of perverse sheaves $A,B\in\Perv^{p'}(\cU)$, the pullback map
$j^!: \Hom(j_{!*}A,j_{!*}B)\to \Hom(A,B)$ is an isomorphism.
\end{Cor}

\begin{proof}
As $j_{!}A\in\mathstrut^{p}\cD^{\leq 0}(\cX)$ and $j_{*}B\in\mathstrut^{p}\cD^{\geq 0}(\cX)$, we obtain natural isomorphisms
\[
\Hom(j_{!}A,j_{*}B)\isom\Hom(\mathstrut^{p}H^{0}(j_{!}A),\mathstrut^{p}H^{0}(j_{*}B))\overset{\sim}{\lar}
\Hom(j_{!*}A,j_{!*}B),
\]
where the isomorphism on the right follows from \rco{supr}(a),(b). Since the map $A\to j^{!}j_{!}A$ is an isomorphism, $j^!$ induces an isomorphism
\[
\Hom(j_{!}A,j_{*}B)\simeq \Hom(j^{!}j_{!}A,B)\simeq \Hom(A,B),
\]
thus the assertion follows.
\end{proof}


\begin{Lem}\label{L:midext}
Let $(\cX,\{\cX_{\al}\}_{\al\in\cI})$  be a placidly stratified $\infty$-stack, equipped with perversity $p_{\nu}$. Let $j:\cU\hra \cX$ is an fp-open inclusion of an $\{\cX_{\al}\}_{\al}$-adapted $\infty$-substack, and let $i:\cZ:=\cX\sm\cU\hra\cX$ be the complementary topologically fp-closed embedding. Equip $\cU$ and $\cZ$ with the induced perversities, and let $K\in\Perv^{p_{\nu}}(\cU)$.

(a) The functor $i_!$ is $t$-exact, $i^!$ is left $t$-exact, while $i^*$ is right $t$-exact.

(b) Every $L\in\Perv^{p_{\nu}}(\cX)$ supported on $\cZ$ is of the form $L\simeq i_!M$ for $M\in\Perv^{p_{\nu}}(\cZ)$.

(c) The intermediate extension $j_{!*}K\in\Perv^p(\cX)$ is the unique perverse extension $\wt{K}$ of $K$ such that $i^*\wt{K}\in {}^{p_{\nu}}\cD^{\leq -1}(\cZ)$ and
$i^!\wt{K}\in {}^{p_{\nu}}\cD^{\geq 1}(\cZ)$.

(d) Assume that the stratification is bounded. Then  $j_{!*}K\in\Perv^{p_{\nu}}(\cX)$ is the unique perverse extension $\wt{K}$ of $K$ such that for all $\al\in\cI\sm\cI_{\cU}$, we have
\[
\eta_{\al}^{*}\wt{K}\in \mathstrut^{p}\cD^{\leq -\nu_{\al}-1}(\cX_{\al}) \text { and } \eta_{\al}^{!}\wt{K}\in \mathstrut^{p}\cD^{\geq -\nu_{\al}+1}(\cX_{\al}).
\]
\end{Lem}

\begin{proof}
(a) By adjunction, it suffices to show that $i^!$ and $i_!$ are left $t$-exact. Both assertions immediately follow from formula
\form{geq0} and identities $i^!i_!\simeq\Id$ and $j^!i_!\simeq 0$.

(b) Since $j^!L\simeq 0$, we have  $L\simeq i_!M$  with $M:=i^!L$ (by \rl{glue1}(d)). Then $M\in {}^p\cD^{\geq 0}(\cZ)$ by (a). On the other hand, we have $i^*L\simeq i^*i_!M\simeq M$ (because $i_!$ is fully faithful), therefore $M\in {}^p\cD^{\leq 0}(\cZ)$.

(c) By \rco{supr}(c), it suffices to show that $\wt{K}\in\Perv^{p_{\nu}}(\cX)$ has no non-zero subobjects (resp. quotients) supported on $\cZ$ if and only if $i^!\wt{K}\in {}^{p_{\nu}}\cD^{\geq 1}(\cZ)$ (resp.  $i^*\wt{K}\in {}^{p_{\nu}}\cD^{\leq -1}(\cZ)$).

Since $i^!\wt{K}\in {}^{p_{\nu}}\cD^{\geq 0}(\cZ)$,  by (a), we have  $i^!\wt{K}\in {}^{p_{\nu}}\cD^{\geq 1}(\cZ)$ if and only if
${}^{p_{\nu}}H^0(i^!\wt{K})\simeq 0$. This happens if and only if
for every $M\in\Perv^{p_{\nu}}(\cZ)$, we have an isomorphism
\[
\Hom(i_!M,\wt{K})\simeq\Hom(M,i^!\wt{K})\simeq \Hom(M,{}^{p_{\nu}}H^0(i^!\wt{K}))\simeq 0.
\]
Thus, by (b), this happens if and only if $\wt{K}$ has no non-zero subobjects supported on $\cZ$.
The proof of the second assertion is similar.

(d) Note that when the stratification is bounded we have $i^*\wt{K}\in {}^{p_{\nu}}\cD^{\leq -1}(\cZ)$ if and only if $\eta_{\al}^{*}\wt{K}\in \mathstrut^{p}\cD^{\leq -\nu_{\al}-1}(\cX_{\al})$ for every $\al\in\cI_{\cZ}=\cI\sm\cI_{\cU}$ (by \form{leq0}) and similarly for $i^!\wt{K}\in {}^{p_{\nu}}\cD^{\geq 1}(\cZ)$. Now the assertion follows from (c).
\end{proof}

\subsection{Semi-small morphisms}
In this subsection we extend classical (finite dimensional) results to our setting.

\begin{Emp} \label{E:pervf}
{\bf Perversity, induced by $f$.}
Let $f:\cX\to\cY$ be a morphism, where $\cX$ is a placid $\infty$-stack, and $(\cY,\{\cY_{\al}\}_{\al})$ is a placidly stratified $\infty$-stack, satisfying assumptions of \re{semismall}(a),(b). Consider perversity $p_{f}:=\{\nu_{\al}\}_{\al\in\cI}$, \label{N:pf} defined by
$\nu_{\al}:=b_{\al}+\dt_{\al}$ for all $\al$. Then $f$ is semi-small (see \re{semismall}(c)) if and only if we have
\begin{equation} \label{Eq:semismall}
2\dt_{\al}\leq\nu_{\al}\leq 2b_{\al} \text{ for every } \al\in \cI.
\end{equation}
Moreover, a semi-small morphism $f$ is $\cU$-small (see \re{semismall}(d)) if and only if we have
\begin{equation} \label{Eq:small}
2\dt_{\al}<\nu_{\al}< 2b_{\al} \text{ for every } \al\in \cI\sm \cI_{\cU}.
\end{equation}
\end{Emp}

\begin{Emp}
{\bf Remark.} Our definition of the perversity $p_f$ is motivated by the observation that in ``good'' cases, e.g. when
$f:X\to Y$ is a dominant generically finite morphism between irreducible schemes of finite type over $k$, the perversity $p_f$ coincides with the canonical perversity from \re{stperv} (see \form{codimss}), thus the corresponding $t$-structure is the $!$-adapted perverse $t$-structure
(see \rl{sttstr}).
\end{Emp}

\begin{Thm} \label{T:small}
(a) Let $f:\cX\to\cY$ be an ind-fp-proper semi-small morphism of $\infty$-stacks, where $\cX$ is smooth, while $\cY$ admits gluing of sheaves. Then the pushforward $K:=f_!(\om_{\cX})$ is $p_f$-perverse.

(b) Moreover, assume that $f$ is $\cU$-small, and let $j:\cU\hra\cY$ be the open embedding.
Then we have an isomorphism $K\simeq j_{!*}j^!(K)$.
\end{Thm}

\begin{proof}
By \rp{tstrglue} and \rl{midext}, we have to show that we have
\begin{equation*} \label{Eq:incl1}
\eta_{\al}^*K\in {}^{p}\cD^{\leq -\nu_{\al}}(\cY_{\al}) \text{ and }
\eta_{\al}^! K\in {}^{p}\cD^{\geq -\nu_{\al}}(\cY_{\al})
\end{equation*}
for every $\al\in\cI$, and stronger inclusions
\begin{equation*} \label{Eq:incl2}
\eta_{\al}^*K\in {}^p\cD^{\leq -\nu_{\al}-1}(\cY_{\al}) \text{ and }
\eta_{\al}^! K\in {}^p\cD^{\geq -\nu_{\al}+1}(\cY_{\al})
\end{equation*}
for every $\al\in\cI\sm\cI_{\cU}$. Using \form{semismall} and \form{small}, it thus suffices to show that for every $\al\in\cI$ we have
\begin{equation} \label{Eq:incl3}
\eta_{\al}^*K\in {}^p\cD^{\leq -2b_{\al}}(\cY_{\al}) \text{ and }
\eta_{\al}^! K\in {}^p\cD^{\geq -2\dt_{\al}}(\cY_{\al}).
\end{equation}

Since $f$ is ind-fp-proper, its restriction $f_{\al}:\cX_{\al}\to\cY_{\al}$ (see \re{semismall}(a)) is ind-fp-proper as well.
Therefore pullbacks $f^!$ and $f_{\al}^!$ have left adjoints (by \rp{locindpr}). Moreover, applying \rco{invtop}(a), we see that the commutative  diagram
\begin{equation} \label{Eq:basic2}
\begin{CD}
\cX_{\al}=f^{-1}(\cY_{\al})_{\red} @>\wt{\eta}_{\al}>>\cX \\
 @Vf_{\al} VV @VV f V\\
\cY_{\al} @>\eta_{\al}>>\cY
\end{CD}
\end{equation}
 gives rise to isomorphisms of functors $\eta_{\al}^*f_!\simeq (f_{\al})_!\wt{\eta}_{\al}^*$ (see \rco{glue}) and  $\eta_{\al}^!f_!\simeq (f_{\al})_!\wt{\eta}_{\al}^!$ (see \rp{locindpr}).
Therefore we get isomorphisms $\eta_{\al}^! K\simeq (f_{\al})_!(\om_{\cX_{\al}})$ and
$\eta_{\al}^*K\simeq (f_{\al})_!\wt{\eta}_{\al}^*(\om_{\cX})$.

Since $\cX$ is a smooth placid $\infty$-stack, and $\wt{\eta}_{\al}:\cX_{\al}\to\cX$ is fp-locally closed, weakly equidimensional
of relative dimension $-b_{\al}$, we conclude from \rl{exact2}(d) that $\wt{\eta}_{\al}^*(\om_{\cX})\in {}^p\cD_c^{\leq -2b_{\al}}(\cX_{\al})$.
Moreover, since $f_{\al}$ is equidimensional, the pullback $f_{\al}^!$ is left $t$-exact (by \rl{exact2}(b)).
Therefore by adjunction, we conclude that $(f_{\al})_!$ is right $t$-exact, thus
\[
\eta_{\al}^* K\simeq (f_{\al})_!(\wt{\eta}_{\al}^*(\om_{\cX}))\in {}^p\cD^{\leq -2b_{\al}}(\cY_{\al}),
\]
proving the first inclusion in \form{incl3}.

Similarly, since $\om_{\cX_{\al}}\in {}^p\cD_c^{\geq 0}(\cX_{\al})$ (by \rl{exact2}(a)), and the functor $(f_{\al})_![-2\dt_{\al}]$ is right $t$-exact (by \rl{exact3}), we deduce that
\[
\eta_{\al}^! K\simeq (f_{\al})_!(\om_{\cX_{\al}})\in {}^p\cD^{\geq -2\dt_{\al}}(\cY_{\al}),
\]
proving the second inclusion in \form{incl3}.
\end{proof}





\part{The affine Springer theory}

\section{Application to the affine Springer theory}

\subsection{Main theorem}

\begin{Emp} \label{E:grsprsh}
{\bf The affine Grothendieck--Springer sheaf.}

(a) Since the projection $\ov{\frak{p}}:[\wt{\fC}/\cL G]\rightarrow[\fC/\cL G]$ is ind-fp-proper (see \re{not}),
the pullback $\ov{\frak{p}}^!:\cD([\fC/\cL G])\to \cD([\wt{\fC}/\cL G])$ has a left adjoint $\ov{\frak{p}}_{!}:\cD([\wt{\fC}/\cL G])\rightarrow\cD([\fC/\cL G])$ and satisfies the base change (by \rp{locindpr}).

(b) We set
\[
\cS:=\ov{\frak{p}}_{!}(\om_{[\wt{\fC}/\cL G]})\in \cD([\fC/\cL G]) \label{N:cS}
\]
and call it the {\em affine Grothendieck--Springer sheaf}.

(c) Using notation \re{consstr2}, we denote by $\cS_{\bullet}\in \cD([\fC_{\bullet}/\cL G])$ \label{N:cSbullet} and $\cS_{\leq 0}\in \cD([\fC_{\leq 0}/\cL G])$ \label{N:cSleq0} the restrictions of $\cS$. Since $\ov{\frak{p}}_!$ satisfies the base change, we have
\[
\cS_{\bullet}\simeq (\ov{\frak{p}}_{\bullet})_!(\om_{[\wt{\fC}_{\bullet}/\cL G]})\text{ and }
\cS_{\leq 0}\simeq (\ov{\frak{p}}_{\leq 0})_!(\om_{[\wt{\fC}_{\leq 0}/\cL G]}),
\]
where $\ov{\frak{p}}_{\bullet}:[\wt{\fC}_{\bullet}/\cL G]\rightarrow [\fC_{\bullet}/\cL G]$ and
$\ov{\frak{p}}_{\leq 0}:[\wt{\fC}_{\leq 0}/\cL G]\rightarrow [\fC_{\leq 0}/\cL G]$ are the restrictions of $\ov{\frak{p}}$.

(d) Let $j:[\fC_{\leq 0}/\cL G]\hra [\fC_{\bullet}/\cL G]$ be the inclusion of the open stratum. By definition, we have
$j^!(\cS_{\bullet})\simeq \cS_{\leq 0}$.
\end{Emp}

\begin{Lem} \label{L:admglsh}
The $\infty$-stack $[\fC_{\bullet}/\cL G]$ admits gluing of sheaves,
\end{Lem}

\begin{proof}
Since each  $\fC_{\leq m}$ is an ind-placid ind-scheme (see \re{consstr2}(a)), while $\cL G$ is an ind-placid group
(see \re{afffl}(c)), the quotient $[\fC_{\leq m}/\cL G]$ admits gluing of sheaves by \rp{glue2}. Moreover, since
$[\fC_{\bullet}/\cL G]\simeq\colim_m[\fC_{\leq m}/\cL G]$, and all the transition maps are fp-open embeddings (see \re{consstr2}(c)),
it admits gluing of sheaves by \rl{openglue}(b).
\end{proof}

\begin{Emp} \label{E:basictstr}
{\bf The perverse $t$-structure.}
(a) By \rl{sprsmall}(a),  $([\fC_{\bullet}/\cL G],\{[\fC_{w,\br}/\cL G]_{\red}\}_{w,\br})$ is a placidly stratified $\infty$-stack, which we equip with perversity
$p_{\nu}=\{\nu_{w,\br}\}$, defined by $\nu_{w,\br}:=d_{\br}+a_{w,\br}$  (see \re{codim}). By \rp{tstrglue}, this perversity $p_{\nu}$ gives rise to the $t$-structure on $\cD([\fC_{\bullet}/\cL G])$, which we call the {\em perverse $t$-structure.}

(b) The perversity $p_{\nu}$ from (a) coincides with the perversity $p_{\ov{\frak{p}}_{\bullet}}$ (see \re{pervf}), corresponding to the Grothendieck--Springer fibration $\ov{\frak{p}}_{\bullet}$.
Indeed, the perversity $p_{\ov{\frak{p}}_{\bullet}}$  is defined to be  $p_{\ov{\frak{p}}_{\bullet}}(w,\br)=b_{w,\br}+\dt_{w,\br}$
(by \rl{sprsmall}(b) and \rco{topproper}(b)), and this expression is equal to  $2\dt_{w,\br}+c_w+a_{w,\br}$ (by \rp{codim}), hence to  $d_{\br}+a_{w,\br}$  (see \re{codim}(b)).
\end{Emp}


Now we are ready to prove the main result of this work.

\begin{Thm}\label{T:main}
(a) The affine Grothendieck--Springer sheaf $\cS_{\bullet}\in \cD([\fC_{\bullet}/\cL G])$ is $p_{\ov{\frak{p}}_{\bullet}}$-perverse and satisfies $\cS_{\bullet}\simeq j_{!*}(\cS_{\leq 0})$.

(b) There are natural algebra isomorphisms $\End(\cS_{\bullet})\simeq\End(\cS_{\leq 0})\simeq\bql[\widetilde{W}]$.
\end{Thm}

\begin{proof}
(a) is an immediate consequence of the combination of \rl{sprsmall}(c) and \rt{small}.

(b) The first isomorphism follows from (a) and \rco{contperv}, while the second one is shown in \rp{endrs} below.
\end{proof}

\begin{Prop}\label{P:endrs}
We have a natural algebra isomorphism $\End(\cS_{\leq 0})\simeq\bql[\widetilde{W}]$.
\end{Prop}

\begin{proof}
Since  $[\wt{\fC}_{\leq 0}/\cL G]\simeq[\clp(\kt^{\rs})/\clp(T)]$ (by \rco{regstr2}), the affine scheme $\clp(\kt^{\rs})$ is connected, while the projection $\clp(\kt^{\rs})\to [\clp(\kt^{\rs})/\clp(T)]$ is surjective, we conclude that the $\infty$-stack $[\wt{\fC}_{\leq 0}/\cL G]$ is connected (see \re{pi0}(f)).

By \re{topeq}(b), the topological equivalence $[\wt{\fC}_{\leq 0}/\wt{W}]\to\fC_{\leq 0}$ from \rco{openstr} induces a topological equivalence
$[\wt{\fC}_{\leq 0}/(\cL G\times \wt{W})]\to [\fC_{\leq 0}/\cL G]$. Thus the projection $\ov{\frak{p}}_{\leq 0}:[\wt{\fC}_{\leq 0}/\cL G]\to [\fC_{\leq 0}/\cL G]$ satisfies the assumption of \rco{end} with $\Gm:=\wt{W}$. Since $\cS_{\leq 0}\simeq (\ov{\frak{p}}_{\leq 0})_!(\om_{[\wt{\fC}_{\leq 0}/\cL G]})$ (see \re{grsprsh}(c)) and $[\wt{\fC}_{\leq 0}/\cL G]$ is connected, the assertion now follows from a combination of \rco{end} and \rl{endomor}(a).
\end{proof}

\begin{Emp} \label{E:waffaction}
{\bf Remark.} It follows from \rt{main}(b) that $\cS_{\bullet}$ is equipped with a natural action of $\widetilde{W}$. Namely, it follows from the proof (see \rp{endrs}) that the action of $\widetilde{W}$ on $\cS_{\leq 0}$ is induced by the geometric action of $\widetilde{W}$ on $\wt{\fC}_{\leq 0}$ over $\fC_{\leq 0}$, and this action uniquely extends to the action on $\cS_{\bullet}$.
\end{Emp}

\begin{Emp} \label{E:affsprsh}
{\bf The affine Springer sheaf.} In the notation of \re{affsprfib}, we denote by $\cS_{\tn}\in\cD([\fC_{\tn}/\cL G])$ \label{N:cStn} the $!$-pullback of $\cS$ and call it the {\em affine Springer sheaf}. We also let $\cS_{\tn,\bullet}\in \cD([\fC_{\tn,\bullet}/\cL G])$ \label{N:cStnbullet} be the $!$-pullback of $\cS_{\tn}$.
\end{Emp}


\begin{Thm} \label{T:perv3}
The affine Springer sheaf $\cS_{\tn,\bullet}$ is $p_{\ov{\frak{p}}_{\tn,\bullet}}$-perverse and satisfies
\[
\cS_{\tn,\bullet}\simeq (\ov{\frak{p}}_{\tn,\bullet})_!(\om_{[\Lie(I)_{\tn,\bullet}/I]}).
\]
\end{Thm}

\begin{proof}
Since $\ov{\frak{p}}$ is ind-fp-proper, \rp{locindpr} applies. Therefore the base change morphism
$(\ov{\frak{p}}_{\tn,\bullet})_!\wt{i}_{\tn}^!\to i_{\tn}^!(\ov{\frak{p}}_{\bullet})_!$, corresponding to
the Cartesian diagram
\begin{equation*} \label{Eq:basicun}
\begin{CD}
[\Lie(I)_{\tn,\bullet}/I] @>\wt{i}_{\tn}>>[\Lie(I)_{\bullet}/I] \\
@V{\ov{\frak{p}}_{\tn,\bullet}}VV @VV{\ov{\frak{p}}_{\bullet}}V\\
[\fC_{\tn,\bullet}/\cL G] @>i_{\tn}>>[\fC_{\bullet}/\cL G]
\end{CD}
\end{equation*}
is an isomorphism. Therefore we get an isomorphism $\cS_{\tn,\bullet}\simeq (\ov{\frak{p}}_{\tn,\bullet})_!(\om_{[\Lie(I)_{\tn,\bullet}/I]})$.

Next, using \rco{topproper}(b) and \rl{sprsm}(b), we see that the corresponding perversity $p_{\ov{\frak{p}}_{\tn,\bullet}}$ satisfies
$p_{\ov{\frak{p}}_{\tn,\bullet}}(w,\br)=(b_{w,\br}-r)+\dt_{w,\br}$ (compare \re{basictstr}(a)). The assertion now follows from a combination of \rl{sprsm}(c) and \rt{small}.
\end{proof}

One can ask whether $\cS_{\tn,\bullet}$ is an intermediate extension of its restriction to a suitable fp-open substack, and what is the minimal substack satisfying this property.

\begin{Emp} \label{E:essstr2}
{\bf Remark.} Assume that Conjecture \re{conjecture}(a) holds, and let $\fC_{\tn,+}\subset \fC_{\tn,\bullet}$ be as in \re{conjecture}(b).
Then $[\fC_{\tn,+}/\cL G]\subset [\fC_{\tn,\bullet}/\cL G]$ is an open union of strata. Now \rt{small} and \rl{sprsm} would imply that $\cS_{\tn,\bullet}$ is the intermediate extension of its restriction to $[\fC_{\tn,+}/\cL G]$, and this is the smallest open union of strata, satisfying this property.
\end{Emp}

\subsection{Perverse $t$-structure on $[(\cL\kg)_{\bullet}/\cL G]$}

\begin{Emp} \label{E:gmaction}
{\bf The $\gm$-action.} (a) Recall that the natural $\gm$-action $(a,x)\mapsto ax$ on $\kg$ commutes with the adjoint action of $G$.
Thus it induces the $\gm$-action on the GIT quotient $\fc=\kg//G$ such that the projection $\chi:\kg\to\fc$ is $\gm$-equivariant. In particular, the induced map
$\pi:\kt\to\fc$ is $\gm$-equivariant.

(b) Furthermore, there exists a (noncanonical) isomorphism $\fc\isom\B{A}^r$ under which the $\gm$-action on $\fc$ corresponds to a $\gm$-action on $\B{A}^r$, given by $a(x_1,\ldots, x_r):=(a^{d_1}x_1,\ldots,a^{d_r}x_r)$ for certain positive integers $d_1,\ldots,d_r$.

(c) The $\gm$-actions on $\kt$ and $\fc$ induce $\cL(\gm)$-actions on $\cL\kg$, $\cL(\kt_w)$ and $\cL\fc$ such that the induced maps
$\chi:\cL\kg\to\cL\fc$ and $\pi:\cL(\kt_w)\to\cL\fc$ are $\cL(\gm)$-equivariant.
\end{Emp}

\begin{Emp} \label{E:consstrfc}
{\bf Constructible stratification of $(\cL\fc)_{\bullet}$.}
(a) By definition, for every $n\geq 0$ and every GKM-stratum $\kt_{w,\br}\subset \clp(\kt_w)$ (see \re{twisted}(b)) the action of element $t^n\in \cL(\gm)$ on $\cL(\kt_w)$ from \re{gmaction}(c) induces an isomorphism $\kt_{w,\br}\isom \kt_{w,\br+n}$, that is, $\kt_{w,\br+n}=t^n\kt_{w,\br}$.

(b) Since the GKM stratum $\fc_{w,\br}\subset\clp(\fc)$ (see \re{strofc}(c)) is defined to be the image $\pi(\kt_{w,\br})$, and $\pi$ is $\cL(\gm)$-equivariant, we conclude that the action of element $t^n\in \cL(\gm)$ induces an isomorphism $\fc_{w,\br}\isom \fc_{w,\br+n}$, that is, $\fc_{w,\br+n}=t^n\fc_{w,\br}$.

(c) For every pair $(w,\br)$, where $w\in W$ and $\br:R\to\B{Q}$, we choose $n\geq 0$ such that $\br+n\geq 0$. To this data we can associate an
fp-locally closed subscheme $\fc_{w,\br+n}\subset \clp(\fc)\subset\cL\fc$ (see \re{strofc}), so we can consider another fp-locally closed subscheme  $\fc_{w,\br}:=t^{-n}\fc_{w,\br+n}\subset \cL\fc$. Moreover, using the observation of (b) one sees that $\fc_{w,\br}$ is independent of the choice of $n$ (hence coincides with that of \re{strofc} when $\br\geq 0$).

(d) We claim that the collection $\{\fc_{w,\br}\}_{w,\br}$ forms a constructible stratification of the regular part $(\cL\fc)_{\bullet}:=(\cL\fc)_{\frak{D}\neq 0}$ \label{N:lcbullet} of $\cL\fc$.

First of all, since $\clp(\fc)\subset \cL\fc$ is an fp-closed subscheme, the same is true for each $t^{-n}\clp(\fc)\subset \cL\fc$.
Moreover, using  isomorphism $\cL(\B{A}^1)\simeq\colim_n t^{-n}\clp(\B{A}^1)$ (see \re{loop}(b)) and observation \re{gmaction}(b), we conclude that
we have a presentation $\cL\fc\simeq\colim_n t^{-n}\clp(\fc)$ as a filtered colimit of its fp-closed subschemes,
hence a similar presentation  $(\cL\fc)_{\bullet}\simeq\colim_n t^{-n}\clp(\fc)_{\bullet}$.

Next we notice that we have $\fc_{w,\br}\subset t^{-n}\clp(\fc)$ if and only if $\fc_{w,\br+n}=t^n\fc_{w,\br}\subset\clp(\fc)$.
Since $\{\fc_{w,\br}\}_{w,\br\geq 0}$ forms a bounded constructible stratification of $\clp(\fc)_{\bullet}$ (see \re{consstr2}(d)), we thus conclude that
$\{\fc_{w,\br}\}_{w,\br\geq -n}$ forms a bounded constructible stratification of $t^{-n}\clp(\fc)_{\bullet}$, hence
$\{\fc_{w,\br}\}_{w,\br}$ forms
a constructible stratification of $(\cL\fc)_{\bullet}$.
\end{Emp}

\begin{Emp} \label{E:conslg}
{\bf Constructible stratification of $(\cL\kg)_{\bullet}$.}
Set $(\cL\kg)_{\bullet}:=\chi^{-1}((\cL\fc)_{\bullet})\subset\cL\kg$.\label{N:lgbullet}

(a) For every $(w,\br)$ as in \re{consstrfc}(c), the preimage $\kg_{w,\br}:=\chi^{-1}(\fc_{w,\br})\subset\cL\kg$ is an fp-locally closed ind-subscheme, and $\{\kg_{w,\br,\red}\}_{w,\br}$ forms a constructible stratification of $(\cL\kg)_{\bullet}$ (by \re{consstrfc}(d) and \rl{propcons}(a)). Therefore $\{[\kg_{w,\br}/\cL G]_{\red}\}_{w,\br}$ forms a constructible stratification of the quotient $\infty$-stack $[(\cL\kg)_{\bullet}/\cL G]$.

(b) Since the map $\chi$ is $\gm$-equivariant, the action of $t^n\in \cL(\gm)$ on $\cL\kg$ induces an isomorphism
$\kg_{w,\br}\isom\kg_{w,\br+n}$, hence $[\kg_{w,\br}/\cL G]\isom [\kg_{w,\br+n}/\cL G]$. Using equality $\kg_{w,\br}=\fC_{w,\br}$ for all $\br\geq 0$, we thus conclude from \rco{stratum2} that each $[\kg_{w,\br}/\cL G]_{\red}$ is placid.
\end{Emp}

\begin{Emp}
{\bf The perverse $t$-structure.}
(a) By \re{conslg},  $[(\cL\kg)_{\bullet}/\cL G]$ is a placidly stratified $\infty$-stack.

(b)  Since the map $\chi$ is $\gm$-equivariant, the presentation $(\cL\fc)_{\bullet}\simeq\colim_n t^{-n}\clp(\fc)_{\bullet}$ from
\re{consstrfc}(d) induces presentations $(\cL\kg)_{\bullet}\simeq\colim_n t^{-n}\fC_{\bullet}$ and
$[(\cL\kg)_{\bullet}/\cL G]\simeq\colim_n t^{-n}[\fC_{\bullet}/\cL G]$, where all the transition maps are fp-closed embeddings.
Therefore it follows from \rl{admglsh} and \rl{openglue}(b) that the quotient stack $[(\cL\kg)_{\bullet}/\cL G]$ admits gluing of sheaves.

(c) Notice that for every GKM stratum $(w,\br)$ with $\br\geq 0$ and every $n\geq 0$, the perversity $\nu_{w,\br}$ from \re{basictstr} satisfies
the identity $\nu_{w,\br+n}=\nu_{w,\br}+n\dim G$.

(d) For an arbitrary $(w,\br)$, we choose $n\geq 0$ such that $\br+n\geq 0$. In this case, $\nu_{w,\br+n}$ was defined in \re{basictstr},
 and we set $\nu_{w,\br}:=\nu_{w,\br+n}-n\dim G$.
By (c), $\nu_{w,\br}$ is independent of $n$ and coincides with that of \re{basictstr} when $\br\geq 0$.

(e) By \rp{tstrglue}, the perversity $p_{\nu}:=\{\nu_{w,\br}\}_{w,\br}$ from (d) gives rise to the $t$-structure on $\cD([(\cL\kg)_{\bullet}/\cL G])$, which we call the {\em perverse $t$-structure.}
\end{Emp}

\part{Appendices}
\appendix
\section{Categorical framework} \label{S:category}
In this section we develop a categorical framework which is needed for the construction of placid $\infty$-stacks.

\subsection{A variant of Simpson's construction} \label{S:simpson}

In this subsection we recall a general categorical construction, which is essentially due to Simpson \cite{Si}.

\begin{Emp} \label{E:setup1}
{\bf Set-up.} (a) Let $\cC$ be an $\infty$-category, admitting all fiber products. Assume that we are given

$\bullet$ a class $\Cov$ of morphisms in $\cC$, called {\em coverings}, containing isomorphisms and closed under pullbacks and compositions;

$\bullet$ a class $\Ob_0(\cC)\subset\Ob(\cC)$ of objects in $\cC$, and

$\bullet$  a class $\Mor_0(\cC)\subset\Mor(\cC)$ of morphisms in $\cC$ such that

\indent\indent(i) class $\Mor_0(\cC)$ is closed under compositions and pullbacks and contains all isomorphisms;

\indent\indent(ii) if $f:x\to y$ is in $\Mor_0(\cC)$ and $y\in\Ob_0(\cC)$, then $x\in\Ob_0(\cC)$.
\end{Emp}

\begin{Emp} \label{E:cons}
{\bf Construction.}
Assume that we are in the situation of \re{setup1}. By recursion, for every $n\in\B{N}_{>0}$ we assume that classes $\Ob_{n-1}(\cC)$ and $\Mor_{n-1}(\cC)$ are constructed, and we are going to construct classes $\Ob_n(\cC)\subset\Ob(\cC)$ and $\Mor_n(\cC)\subset\Mor(\cC)$.

(a) Denote by $\Ob_{n}(\cC)$ the class of objects $x\in\Ob(\cC)$ for which there exists a covering $g:z\to x$ in $\Mor_{n-1}(\cC)$ with $z\in\Ob_0(\cC)$.

(b) Denote by $\Mor_{n,0}(\cC)$ the class of all morphisms $f:x\to y$ in $\Mor(\cC)$ with $x\in \Ob_n(\cC)$ and $y\in\Ob_0(\cC)$ such that
there exists a covering $z\to x$ in $\Mor_{n-1}(\cC)$, where $z\in\Ob_0(\cC)$ and the composition $z\to x\to y$ is in $\Mor_{0}(\cC)$.

(c) Denote by $\Mor_{n}(\cC)$ the class of all morphisms $f:x\to y$  in $\Mor(\cC)$ such that for every morphism $y'\to y$ in $\Mor(\cC)$ with  $y'\in \Ob_0(\cC)$ the pullback $x\times_y y'\to y'$ belongs to $\Mor_{n,0}(\cC)$.
\end{Emp}

\begin{Emp} \label{E:rempb}
{\bf Remark.}
By construction, $\Mor_n(\cC)$ contains all isomorphisms and is closed under all pullbacks.
\end{Emp}

The following rather straightforward lemma summarizes basic properties of this construction.

\begin{Lem} \label{L:simpson}
For every $n\in\B{N}$, we have the following assertions:

(a) If $(f:x\to y)\in\Mor_n(\cC)$ and $y\in\Ob_n(\cC)$, then $x\in\Ob_n(\cC)$.

(b) The class $\Mor_n(\cC)$ is closed under compositions.

(c) If $x\to y$ is a covering in $\Mor_n(\cC)$ and $x\in\Ob_{n+1}(\cC)$, then $y\in \Ob_{n+1}(\cC)$.

(d) We have $\Ob_n(\cC)\subset\Ob_{n+1}(\cC)$ and $\Mor_n(\cC)\subset \Mor_{n+1}(\cC)$.

(e) For every morphism $f:x\to y$ in $\cC$ such that $x\in\Ob_{n+1}(\cC)$ and $y\in\Ob_{0}(\cC)$, we have $f\in\Mor_{n+1}(\cC)$ if and only if
$f\in\Mor_{n+1,0}(\cC)$.
\end{Lem}

\begin{proof}
(b) The proof goes by induction on $n$. If $n=0$, the assertion follows from assumption \re{setup1}. Assume now that $n>0$.

Let $f:x\to y$ and $g:y\to z$ be in $\Mor_{n}(\cC)$, and we want to show that $g\circ f\in\Mor_{n}(\cC)$. Taking pullback with respect to $z'\to z$ with $z'\in\Ob_0(\cC)$, we can assume that $g\in\Mor_{n,0}(\cC)$ and $f\in\Mor_n(\cC)$ (by \re{rempb}),
and we want to show that $g\circ f\in\Mor_{n,0}(\cC)$.

Since $g:y\to z$ is in $\Mor_{n,0}(\cC)$ there exists a covering $y'\to y$ in $\Mor_{n-1}(\cC)$ with $y'\in\Ob_0(\cC)$ such that
$y'\to y\to z$ is in $\Mor_{0}(\cC)$. Then the pullback  $y'\times_y x\to x$ is a covering in $\Mor_{n-1}(\cC)$, while
$y'\times_y x\to y'$ is in $\Mor_{n,0}(\cC)$. Thus there exists a covering
$x'\to  y'\times_y x$ in $\Mor_{n-1}(\cC)$ with $x'\in\Ob_0(\cC)$ such that the composition $x'\to  y'\times_y x\to y'$ is in $\Mor_{0}(\cC)$.

Hence the composition $x'\to  y'\times_y x\to x$ is a covering in $\Mor_{n-1}(\cC)$ (by the induction hypothesis),
while the composition $x'\to x\to z$, or, what is the same, $x'\to y'\to z$ is in $\Mor_{0}(\cC)$ (by assumption \re{setup1}).

(c) Choose a covering $z\to x$ in $\Mor_{n}(\cC)$ with $z\in\Ob_0(\cC)$. Then the composition
$z\to x\to y$ is a covering in $\Mor_{n}(\cC)$ (by (b)), thus $y\in\Ob_{n+1}(\cC)$ by definition.

(a) The assertion for $n=0$ follows from assumption \re{setup1}. Assume now that $n>0$.
Since $y\in \Ob_{n}(\cC)$ there exists  a covering $y'\to y$ in $\Mor_{n-1}(\cC)$ such that $y'\in \Ob_0(\cC)$. Since coverings
in $\Mor_{n-1}(\cC)$ are closed under pullbacks, $x\times_y y'\to x$ is a covering, which belongs to $\Mor_{n-1}(\cC)$, and
$x\times_y y'\to y'$ belongs to $\Mor_{n,0}(\cC)$. Then $x\times_y y'\in \Ob_{n}(\cC)$, hence  $x\in\Ob_{n}(\cC)$, by (c).

(d) The assertion easily follows by induction and is left to the reader.

(e) Clearly, if $f:x\to y\in\Mor_n(\cC)$ with $y\in\Ob_0(\cC)$, then  $f\in\Mor_{n,0}(\cC)$. Conversely, we have to show that if
$f\in \Mor_{n,0}(\cC)$, then for every morphism $y'\to y$ in $\Mor_0(\cC)$ with $y'\in \Ob_0(\cC)$, the pullback
$x\times_y y'\to y'$ is in $\Mor_{n,0}(\cC)$. By definition, there exists a covering $z\to x$ in $\Mor_{n-1}(\cC)$ with $z\in\Ob_0(\cC)$ such that the composition $z\to x\to y$ is in $\Mor_0(\cC)$. Hence the pullback $z\times_y y'\to x\times_y y'$ is a covering in $\Mor_{n-1}(\cC)$, while
$z\times_y y'\to y'$ is in $\Mor_0(\cC)$. Therefore $z\times_y y'\in\Ob_0(\cC)$, thus $x\times_y y'\in \Ob_n(\cC)$, and $x\times_y y'\to y'$ is in $\Mor_{n,0}(\cC)$.
%
%
\end{proof}

\begin{Emp} \label{E:geom}
{\bf Notation.} We set $\Ob_{\infty}(\cC):=\cup_n\Ob_n(\cC)\subset \Ob(\cC)$ and $\Mor_{\infty}(\cC):=\cup_n\Mor_n(\cC)\subset \Mor(\cC)$ (compare \rl{simpson}(d)).
\end{Emp}


%

\begin{Cor} \label{C:mor0}
(a) Let $f:x\to y$ be a morphism in $\Mor_{\infty}(\cC)$ with $x,y\in\Ob_{0}(\cC)$.
Then $f\in\Mor_1(\cC)$, that is, there exists a covering $g:z\to x$ in $\Mor_0(\cC)$ with $x\in\Ob_0(\cC)$ such that
$f\circ g:z\to y$ is in $\Mor_0(\cC)$.

(b) For every pair of morphisms $f:x\to y$ and $g:z\to y$ in $\Mor_{\infty}(\cC)$ with $y\in\Ob_{\infty}(\cC)$, and $x,z\in\Ob_0(\cC)$
there exists a covering $t\to x\times_y z$ such that $t\in\Ob_0(\cC)$ and both compositions $t\to x\times_y z\to x$ and
$t\to x\times_y z\to z$ are in $\Mor_0(\cC)$.
\end{Cor}

\begin{proof}
(a) Assume that $f\in\Mor_m(\cC)$ with $m>0$, and we want to show the assertion by induction on $m$. By definition, there exists a covering
$g_1:z_1\to x$ in $\Mor_{m-1}(\cC)$ with $z_1\in\Ob_0(\cC)$ such that $f\circ g_1:z_1\to y$ is in $\Mor_{0}(\cC)$. By induction hypothesis,
there exists a covering $g_2:z\to z_1$ in $\Mor_{0}(\cC)$ such that $g:=g_1\circ g_2:z\to x$ is in $\Mor_{0}(\cC)$. Then
$g:z\to x$ is a covering, and $f\circ g=(f\circ g_1)\circ g_2\in\Mor_0(\cC)$, because $f\circ g_1,g_2\in\Mor_0(\cC)$.

(b) By assumption, $x\times_y z\to z$ and $x\times_y z\to x$ are in $\Mor_{\infty}(\cC)$, because $f$ and $g$ are, hence
$x\times_y z\in\Ob_{\infty}(\cC)$, by \rl{simpson}(a). Therefore there exists a covering $p:t\to x\times_y z$ in  $\Mor_{\infty}(\cC)$ with $t\in\Ob_0(\cC)$. By \rl{simpson}(b), both compositions $t\to x\times_y z\to z$ and $t\to x\times_y z\to x$ are in $\Mor_{\infty}(\cC)$. Moreover, using (a), precomposing $p$ with a covering from $\Mor_0(\cC)$, we can assume that the composition $t\to x\times_y z\to z$ is in $\Mor_{0}(\cC)$.
Precomposing $p$ once again, we can get that the composition $t\to x\times_y z\to x$ is in $\Mor_{0}(\cC)$ as well.
\end{proof}

\subsection{The case of $\infty$-categories of sheaves} \label{S:topos}
In this subsection we will specify the construction of \ref{S:simpson}
to the case where $\cC$ is an $\infty$-topos, that is, an $\infty$-category of sheaves on some  $\infty$-category $\cA$ equipped with a
Grothendieck topology.

\begin{Emp} \label{E:shv}
{\bf Notation.} Let $\frak{S}$ be the $\infty$-category of spaces, which are often referred as $\infty$-groupoids. For every $\infty$-category $\cA$, we denote by $\PShv(\C{A})$ the $\infty$-category of functors  $\cA^{\op}\rightarrow\frak{S}$.
Moreover, when $\cA$ is equipped with a Grothendieck topology $\cT$, we denote by  $\Shv(\C{A})\subset\PShv(\C{A})$ be the $\infty$-subcategory of sheaves in the $\cT$-topology.
\end{Emp}

\begin{Emp} \label{E:setup2}
{\bf Set-up.}
(i) Let $\cA$ be an $\infty$-category, and let $\Ob_0(\C{A})\subset\Ob(\C{A})$
and $\Mor_0(\C{A})\subset\Mor(\C{A})$ be classes of objects and morphisms such that:

(a) the class $\Mor_0(\C{A})$ contains all isomorphisms, and is closed under compositions
and pullbacks;

(b) for every $f:x\to y$ in $\Mor_0(\C{A})$ with $y\in\Ob_0(\cA)$ we have $x\in\Ob_0(\cA)$.

\noindent(ii) Moreover, assume that $\cA$ is equipped with a Grothendieck topology $\cT$ such that:

(c) the topology $\cT$ is {\em subcanonical}, that is, every representable presheaf is a sheaf;

(d) every $x\in\Ob(\cA)$ has basis of $\cT$-coverings of the form  $\{f_{\al}:x_{\al}\to x\}$ with $f_{\al}\in\Mor_0(\cA)$ for all $\al$;

(e) the class $\Ob(\cA)\subset\Ob(\Shv(\cA))$ is closed under {\em direct summands}, by which we mean that if $a\in\Ob(\cA)$ decomposes in $\Shv(\cA)$ as a coproduct $a\simeq b\sqcup c$, then $b,c\in \Ob(\cA)$, and, moreover, inclusions $b\hra a$ and $c\hra a$ belong to $\Mor_0(\cA)$.
\end{Emp}

\begin{Emp}
{\bf Remark.} In our applications, $\cA$ will be an ordinary category.
\end{Emp}

\begin{Emp} \label{E:cons2}
{\bf Construction.}
(a) To the data of \re{setup2} we  associate the data of \re{setup1} as follows:

(i) Set $\cC:=\Shv(\C{A})$, and let $\Cov$ be the class of all epimorphisms in $\cC$, that is, morphisms $f:x\to y$ such that $f(a):x(a)\to y(a)$ locally has a section for every $a\in\Ob(\cA)$.

(ii) Let $\Ob_0(\cC)$ be the class of all objects of the form $\sqcup_{\al}a_{\al}$, with $a_{\al}\in\Ob_0(\C{A})$ for all $\al$.

(iii) Let $\Mor_0(\cC)$ be the class of all morphisms $x\to y$ in $\cC$ such that for every morphism $a\to y$ in $\cC$ with $a\in\Ob_0(\cA)$ there exists a decomposition $x\times_y a\simeq\sqcup_{\al}b_{\al}$, where each $b_{\al}\in\Ob(\cA)$ and each composition $b_{\al}\hra x\times_y a\to a$ is in $\Mor_0(\cA)$.

(b) Since pullbacks commute with coproducts, assumptions \re{setup2}(i) imply that the assumptions of \re{setup1} are satisfied. Thus the construction \re{cons} applies, so we can talk about classes of objects
$\Ob_n(\cC)\subset\Ob(\cC)$ and morphisms $\Mor_n(\cC)\subset \Mor(\cC)$.

(c) By \re{setup2}(e),  for every collection $\{x_{\al}\}_{\al}$ of objects of $\cC$, the inclusion $x_{\al_0}\hra\sqcup_{\al}x_{\al}$ is in $\Mor_0(\cC)$.

(d) We say that a collection $\{x_{\al}\to y\}_{\al}$  of morphisms in $\cC$ is a {\em covering}, if the induced map $\sqcup_{\al}x_{\al}\to y$ is a covering.
\end{Emp}

\begin{Emp} \label{E:chech}
{\bf \v{C}ech nerve.}
(a) Recall that to every morphism $f:x\to y$ in $\Mor(\cC)$ one can associate its \v{C}ech complex $\check{C}(f)=\{x^{[m]}\}_{[m]\in\Dt^{\op}_s}$, parameterized by the semi-simplicial category $\Dt_s$, where each $x^{[m]}$ is defined to be the $(m+1)$-times fiber product $x\times_y\times\ldots\times_y x$ of $x$ over $y$, and morphisms are projections $x^{[m']}\to x^{[m'']}$ corresponding to injective maps $[m'']\to [m']$.

(b) It follows from \rl{simpson}, that if $x\in\Ob_n(\cC)$ and $f\in\Mor_n(\cC)$, then all terms in the \v{C}ech complex $\check{C}(f)$ are
in $\Ob_n(\cC)$ and all maps are in $\Mor_n(\cC)$. In particular, we are going to apply this when $y\in\Ob_{n+1}(\cC)$ and $x\in\Ob_0(\cC)$.

(c) If $f:x\to y$ is a covering, then the canonical morphism $\colim_{[m]\in \Dt^{\op}_s}x^{[m]}\to y$ an isomorphism (use, for example, \cite[Proposition 7.2.1.14]{Lu}). Therefore, by the observation (b), every $y\in\Ob_{n+1}(\cC)$ can be written as a colimit of objects from $\Ob_n(\cC)$ with respect to morphisms from $\Mor_n(\cC)$. Similarly, every morphism $y\to z$ in $\Mor_{n+1,0}(\cC)$ can be written as colimit of morphisms $x^{[m]}\to z$ from $\Mor_{n,0}(\cC)$.
\end{Emp}

\begin{Lem} \label{L:covering}
Let $f:x\to y$ be a morphism in $\cC$, and $z\to y$ a covering in $\cC$. For every $n>0$, we have $f$ is in $\Mor_n(\cC)$ if and only if its pullback $x\times_y z\to z$ is in $\Mor_n(\cC)$.
\end{Lem}

\begin{proof}
Since $\Mor_n(\cC)$ is stable under pullbacks, the ``only if'' assertion follows.
Conversely, assume that $x\times_y z\to z$ is $\Mor_n(\cC)$. We need to show that
$f:x\times_y a\to a$ is $\Mor_{n,0}(\cC)$ for every morphism $a\to y$ with $a\in\Ob_0(\cA)$.

Since $z\to y$ is a covering, there exists a $\cT$-covering $\{a_{\al}\to a\}_{\al}$ such that every composition $a_{\al}\to a\to y$ has a lifting to $a_{\al}\to z$. By our assumption \re{setup2}(d), we can assume that every $a_{\al}\to a$ belongs $\Mor_0(\cA)$.
Set $t:=\sqcup_{\al}a_{\al}\in\cC$. Then $t\in\Ob_0(\cC)$, the covering $t\to a$ belongs to $\Mor_0(\cC)$, and the composition
$t\to a\to y$ has a lifting to $t\to z$.

Since  $x\times_y z\to z$ is in $\Mor_n(\cC)$ while $t\in\Ob_0(\cC)$, the map $x\times_y t\to t$ is in $\Mor_{n,0}(\cC)$. Thus there exists a covering $t'\to x\times_y t$ in $\Mor_{n-1}(\cC)$
such that the composition $t'\to t$ is in $\Mor_{0}(\cC)$. Since $t\to a$ is a covering from $\Mor_0(\cC)$, we get that the composition $t'\to x\times_y t \to x\times_y a$ is a covering from $\Mor_{n-1}(\cC)$ and the composition $t'\to a$ is in $\Mor_{0}(\cC)$.
Thus $x\times_y a\to a$ belongs to $\Mor_{n,0}(\cC)$, and the proof is complete.
\end{proof}

The following assertion will not be used in the sequel.

\begin{Cor} \label{C:covering}
Let $f:x\to y$ be a morphism in $\Mor_{\infty}(\cC)$ with $x,y\in\Ob_n(\cC)$ and $n>0$.
Then $f\in\Mor_n(\cC)$.
\end{Cor}

\begin{proof}
Choose a covering $z\to y$ in $\Mor_{n-1}(\cC)$ with $z\in\Ob_0(\cC)$. Since $x\times_y z\to x$ is $\Mor_{n-1}(\cC)$ and  $x\in\Ob_n(\cC)$, we conclude that $x\times_y z\in \Ob_n(\cC)$ (by \rl{simpson}(c)). By \rl{covering}, it suffices to show that
the projection $x\times_y z\to z$ is $\Mor_{n}(\cC)$. In other words, it suffices to show the assertion assuming that $y\in \Ob_0(\cC)$.

Choose a covering $t\to x$ in $\Mor_{n-1}(\cC)$ with $t\in\Ob_0(\cC)$. If the composition $t\to x\to y$ is in $\Mor_0(\cC)\subset \Mor_{n-1}(\cC)$, we are done. On the other hand, using \rco{mor0}(a), one can always get to this situation replacing $t\to y$ by
$t'\to t\to x$ for some  covering $t'\to t$ in $\Mor_0(\cC)$ with $t'\in\Ob_0(\cC)$.
\end{proof}

The proof of the following important result will be given in Section \ref{S:pfcolim}.

\begin{Thm} \label{T:colim}
For every $x\in\Ob_{\infty}(\cC)$ we consider an $\infty$-category $J_x$, whose objects are morphisms $y\to x$ in $\Mor_{\infty}(\cC)$ with $y\in\Ob_0(\cA)$, and morphisms are morphisms $y\to y'$ in $\Mor_{0}(\cA)$ over $x$. Then the canonical morphism $(\colim_{(y\to x\in J_x)}y)\to x$ is an isomorphism.
\end{Thm}

The following nice characterization of classes $(\Ob_{\infty}(\cC),\Mor_{\infty}(\cC))$ is not used in the sequel.

\begin{Lem} \label{L:charact}
The classes $\Ob_{\infty}(\cC)$ and $\Mor_{\infty}(\cC)$ are the smallest classes  $\Ob'\subset\Ob(\cC)$ and $\Mor'\subset\Mor(\cC)$,
satisfying the following properties:

$\quad(i)$ For every collection $\{x_{\al}\}_{\al}\in\Ob_0(\cA)$, we have $\sqcup_{\al}x_{\al}\in\Ob'$.

$\quad(ii)$ For every collection $\{x_{\al}\to y\}_{\al}\in\Mor_0(\cA)$, we have $(\sqcup_{\al}x_{\al}\to y)\in\Mor'$.

$\quad(iii)$  The class $\Mor'$ is closed under compositions and pullbacks.

$\quad(iv)$ An object $y\in\Ob(\cC)$ is in $\Ob'$, if there exists a covering
$x\to y$ in $\Mor'$ with $x\in\Ob'$.

$\quad(v)$ A morphism $f:x\to y$ is in $\Mor'$, if for every morphism $z\to y$ in $\Mor(\cC)$ with $z\in\Ob_0(\cA)$ the pullback
$x\times_y z\to z$ is in $\Mor'$.

$\quad(vi)$ A morphism $f:x\to y$ is in $\Mor'$, if there exists a covering $z\to x$ from $\Mor'$ such that composition $z\to x\to y$ is in $\Mor'$.
\end{Lem}

\begin{proof}
First we claim that classes $\Ob':=\Ob_{\infty}(\cC)$ and $\Mor':=\Mor_{\infty}(\cC)$ satisfy properties (i)-(vi). Indeed, (i) and (ii) follow from definition (see \re{cons2}(a)); (iii) follows from \re{rempb} and \rl{simpson}(b); (iv) follows from \rl{simpson}(c); while (v) and (vi)
follow essentially from definitions (see \re{cons}(b),(c)).

Conversely, by induction on $n$, we claim that any pair of classes $(\Ob',\Mor')$ satisfying (i)-(vi) contains classes
$(\Ob_n(\cC),\Mor_n(\cC))$ for all $n$. For $n=0$, this follows from (i),(ii) and (v). Assume now $n>0$. By definition, for every $y\in\Ob_n(\cC)$ there exists a covering $x\to y$ in $\Mor_{n-1}(\cC)$ with $x\in\Ob_0(\cC)$. Thus $y$ belongs to $\Ob'$ by (iv) and induction. Finally, let $f:x\to y$ be in $\Mor_n(\cC)$, and let $z\to y$ be a morphism with $z\in\Ob_0(\cA)$.
Using (v), it suffices to show that the pullback $x\times_y z\to z$ belongs to $\Mor'$. But this follows immediately from (vi) and induction.
\end{proof}

\begin{Emp} \label{E:subcategory}
{\bf Restriction to a subcategory.}
(a) Let $\cA'\subset \C{A}$ be a full subcategory, {\em compatible with $\cT$}, by which we mean that every $x\in\cA'$ has a basis of covering of the form $\{x_{\al}\to x\}_{\al}$ with $x_{\al}\in \cA'$.

(b) Let $\iota:\C{A}'\to\C{A}$ be the inclusion. Then the restriction functor $\iota^*:\PShv(\C{A})\to\PShv(\C{A}')$ induces the functor  $\iota^*:\Shv(\C{A})\to\Shv(\C{A}')$, whose left adjoint we denote by $\iota_!:\Shv(\C{A}')\to\Shv(\C{A})$.
\end{Emp}

\begin{Lem} \label{L:ff}
In the situation of \re{subcategory}, the functor $\iota_!:\Shv(\C{A}')\to\Shv(\C{A})$ is fully faithful, and its essential image consists of all
$y\in\Shv(\cA)$ such that the counit $\iota_!\iota^*y\to y$ is an isomorphism.
\end{Lem}

\begin{proof}
We have to show that the unit morphism $x\to \iota^*\iota_!x$ is an isomorphism for every $x\in\Shv(\C{A}')$. Since $\iota^*$ and $\iota_!$ commute with (homotopy) colimits and every $x$ is a colimit of representable objects $a\in\cA'$, it suffices to show that each map
$a\to \iota^*\iota_!a$ is an isomorphism. By the Yoneda lemma and our assumption \re{setup2}(c), $\iota_!a$ is the representable presheaf $\iota(a)$, so the assertion follows from the fact that $\iota:\C{A}'\to\C{A}$ is fully-faithful. The second assertion is standard.
\end{proof}


\subsection{Passing to pro-categories} In most of our applications the $\infty$-category $\cA$ from \re{setup2} will be of the form $\cA\simeq\Pro\cB$ for
some $\infty$-category $\cB$. In this subsection, we will describe what kind of data on $\cB$ gives rise to the data of  \re{setup2}.

\begin{Emp} \label{E:general}
{\bf Construction.}
(a) Let $\C{B}$ be an $\infty$-category, and let $\Mor_0(\cB)$ be a class of morphisms in $\cB$ which contains all isomorphisms, and closed under compositions and pullbacks. In particular, for every $x\to y$ in $\Mor_0(\cB)$ and $z\to y$ in $\Mor(\C{B})$, the fiber product $x\times_y z$ exists in $\cB$ and the projection $x\times_y z\to z$ is in $\Mor_0(\cB)$.

(b) Let $\C{A}:=\Pro(\C{B})$ be the pro-category of $\cB$, and let $\Ob_0(\C{A})$ be the class of objects $x\in\Ob(\cA)$ which have presentations as  filtered limits $x\simeq\lim_{\al}x_{\al}$, where all transition maps $x_{\al}\to x_{\beta}$ belong to $\Mor_0(\cB)$. Such presentations will be later referred to as $\Mor_0(\cB)$-presentations.

(c) Notice that assumption (a) implies that if $f:x\to y$ is in $\Mor_0(\cB)\subset\Mor(\cB)\subset\Mor(\cA)$, then for every morphism $y'\to y$ in $\cA$, the fiber product $x\times_{y} y'$ exists in $\cA$. Explicitly, if $y'\simeq\lim_{\al}y'_{\al}$ is a $\Mor_0(\cB)$-presentation of $y'$, then the projection $y'\to y$ factors through a morphism $y'_{\al}\to y$ for some $\al$, and $\lim_{\al'>\al}(y'_{\al'}\times_{y}x)$ is a $\Mor_0(\cB)$-presentation of $x\times_{y} y'$.

(d) We denote by $\Mor_0(\cB)_{\cA}$ the class of all morphisms $f':x'\to y'$ in $\cA$ of the form $f'\simeq y'\times_{y} f$ for some
morphism $f:x\to y$ in $\Mor_0(\cB)$ and morphism $y'\to y$ in $\Mor(\cA)$. By construction, the class $\Mor_0(\cB)_{\cA}$ contains all  isomorphisms and is closed under compositions and pullbacks. Moreover, $\Mor_0(\cB)$ is nothing else but the class of all morphisms $f':x'\to y'$ in $\Mor_0(\cB)_{\cA}$ such that $y'\in \cB$.

(e) We denote by $\Mor_0(\cA)$ the class of all morphisms $f:x\to y$ in $\cA$ such that $x$ has a presentation as a filtered limit
$x\simeq\lim_{\al} x_{\al}$ over $y$ such that all projection maps $x_{\al}\to y$ and all transition maps $x_{\al}\to x_{\beta}$ are in $\Mor_0(\cB)_{\cA}$. This class contains all isomorphisms and is closed under pullbacks.

(f) Notice that $x\in\Ob(\cA)$ belongs to $\Ob_0(\cA)$ if and only if there exists a morphism $x\to y$ in $\Mor_0(\cA)$ with $y\in\Ob(\cB)$.
\end{Emp}

\begin{Emp} \label{E:summary}
{\bf Summary.}
(a) By \rl{pro} below, the pair $(\Ob_0(\cA),\Mor_0(\cA))$, constructed in \re{general}, satisfies all the assumptions of \re{setup2}(i).

(b) Note that every Grothendieck topology $\cT_{\cB}$ on $\cB$ induces a Grothendieck topology $\cT$ on $\cA$. Namely, for every presentation $x\simeq\lim_{\al}x_{\al}$, coverings of $x$ are generated by coverings of the form $\{x\times_{x_{\al}}x_{\al,i}\}_i$, where $\{x_{\al,i}\to x_{\al}\}_i$ is a covering of $x_{\al}$. In particular, if the Grothendieck topology $\cT_{\cB}$ is generated by morphisms belonging to $\Mor_0(\cB)$, then $\cT$ satisfies the assumption \re{setup2}(d).
\end{Emp}

\begin{Emp}
{\bf Remark.} In our applications, $\cB$ will be an ordinary category, in which case, $\cA$ will be an ordinary category as well.
\end{Emp}

\begin{Lem} \label{L:pro}
In the situation of \re{general},

(a) The class $\Mor_0(\cA)$ is closed under compositions.

(b) For every $f:x\to y$ in $\Mor_0(\cA)$ with $y\in\Ob_0(\C{A})$, we have $x\in \Ob_0(\C{A})$.
\end{Lem}

\begin{proof}
Notice first that it follows from the observation \re{general}(f) that assertion (b) follows from (a). Thus, it remains to show that for every $f:x\to y$ and $g:y\to z$ in $\Mor_0(\cA)$ we have $g\circ f\in\Mor_0(\cA)$.

Though the assertion is straightforward, we sketch the argument for completeness. For every pair of $\Mor_0(\cB)_{\cA}$-presentations $x\simeq\lim_{\al\in I}x_{\al}$ over $y$ and $y\simeq\lim_{\beta\in J}y_{\beta}$ over $z$ as in \re{general}(e), we are going to present an explicit $\Mor_0(\cB)_{\cA}$-presentation $x\simeq\lim_{\g\in K}x_{\g}$ over $z$.

Namely, consider category $K$, whose

$\bullet$ objects are triples $\g=(\al,\beta,D)$, where $\al\in I,\beta\in J$, and $D$ is a Cartesian diagram
\begin{equation} \label{Eq:filtered}
\begin{CD}
x_{\al} @>p_{\al}>> y\\
@VVV         @VV\pr_{\beta}V\\
x_{\g} @>p_{\g,\beta}>> y_{\beta}
\end{CD}
\end{equation}
in $\cA$, where $p_{\al}$ and $\pr_{\beta}$ are the projections, and  $\pr_{\g,\beta}\in \Mor_0(\cB)_{\cA}$;

$\bullet$ morphisms $\g'=(\al',\beta',D')\to \g=(\al,\beta,D)$ are triples, consisting of a morphism $\al'\to \al$ in $\Mor(I)$, a morphism $\beta'\to\beta$ in $ \Mor(J)$ and a morphism $D'\to D$ of commutative diagrams \form{filtered} such that the induced morphisms $x_{\al'}\to x_{\al}, y_{\beta'}\to y_{\beta}$ are projections, the map $y\to y$ is the identity, and the induced morphism $x_{\g'}\to x_{\g}\times_{y_{\beta}}y_{\beta'}$
is in $\Mor_0(\cB)_{\cA}$.

By construction, we have a functor $K\to\C{A}:\g\mapsto x_{\g}$, which sends a morphism $\g'\to\g$ in $K$ to the composition
$x_{\g'}\to x_{\g}\times_{y_{\beta}}y_{\beta'}\to x_{\g}$. Moreover, we have a natural morphism
$x\to\lim_{\g\in K}x_{\g}$, where all the compositions $x_{\g}\to y_{\beta}\to z$ and all transition maps $x_{\g'}\to x_{\g}$ are in
$\Mor_0(\cB)_{\cA}$. It suffices to show that the category $K$ is filtered, and the morphism $x\to \lim_{\g\in K}x_{\g}$ is an isomorphism.

This is standard and is left to the reader.
\end{proof}

\subsection{Extending of classes of morphisms} In this subsection we assume that we are in the situation of \ref{S:topos} and will outline a general procedure on how to construct more general classes of morphisms between objects in $\Ob_{\infty}(\cC)$. In the main part of the paper it is used to study equidimensional morphisms between placid $\infty$-stacks
and their variants.

\begin{Emp} \label{E:cons1ext}
{\bf Extension of morphisms in $\cA$.}
(a) In the situation of \re{setup2}, denote by $\Mor_{0,0}(\cA)$ the class of morphisms $f:x\to y$ in $\Mor_0(\cA)$ with $x,y\in\Ob_0(\cA)$.

(b) Let $\Mor^+_0(\cA)\subset \Mor(\cA)$ be a class of morphisms $f:x\to y$ with $x,y\in\Ob_0(\cA)$, which contains $\Mor_{0,0}(\cA)$, closed under compositions and $\Mor_{0,0}(\cA)$-pullbacks, that is, pullbacks with respect to morphisms from $\Mor_{0,0}(\cA)$.

(c) We denote by  $\Mor^+_{\infty}(\cC)\subset \Mor(\cC)$ the class of morphisms $f:x\to y$ with $x,y\in\Ob_{\infty}(\cC)$ satisfying the following property: for every $f:a\to y$ in $\Mor_{\infty}(\cC)$ with $a\in\Ob_0(\cA)$ there exists a covering $t\to x\times_y a$ in $\Mor_{\infty}(\cC)$ and a decomposition $t\simeq\sqcup_{\al}b_{\al}$ with $b_{\al}\in\Ob_0(\cA)$ such that each composition $b_{\al}\hra t\to x\times_y a\to a$ is in $\Mor^+_0(\cA)$.

(d) We denote by $\Mor_{\infty,\infty}(\cC)$ the class of morphisms $f:x\to y$ in $\Mor_{\infty}(\cC)$ with $x,y\in\Ob_{\infty}(\cC)$.
\end{Emp}


\begin{Lem} \label{L:cM}
(a) The class $\Mor^+_{\infty}(\cC)$ contains $\Mor^+_0(\cA)$;

(b) The class $\Mor^+_{\infty}(\cC)$ contains $\Mor_{\infty,\infty}(\cC)$ and is closed under compositions and $\Mor_{\infty,\infty}(\cC)$-pullbacks;

(c) The class $\Mor^+_{\infty}(\cC)$ is {\em $\Mor_{\infty,\infty}(\cC)$-local}, that is, if $x\overset{f}{\to}y\overset{g}{\to}z$ is a composition of morphisms in $\cC$ such that $f$ is a covering in $\Mor_{\infty,\infty}(\cC)$ and $g\circ f$ is in $\Mor^+_{\infty}(\cC)$, then $g$ is in $\Mor^+_{\infty}(\cC)$.

(d) Let $f:x\to y$ be a morphism with $x,y\in \Ob_{\infty}(\cC)$ such that pullback $x\times_{y}z\to z$ is in $\Mor^+_{\infty}(\cC)$
for some covering $z\to y$ in $\Mor_{\infty,\infty}(\cC)$. Then $f$ is in $\Mor^+_{\infty}(\cC)$.
\end{Lem}

\begin{proof}
(a) Let $f:x\to y$ be in $\Mor^+_0(\cA)$, and let $a\to y$ be in $\Mor_{\infty}(\cC)$ with $a\in\Ob_0(\cA)$.
By \rco{mor0}(a), there exists a covering $t\to a$ in $\Mor_0(\cC)$ such that the composition $t\to a\to y$ is in
$\Mor_0(\cC)$. Therefore there exists a decomposition  $t\simeq\sqcup_{\al}b_{\al}$ such that $b_{\al}\in\Ob_0(\cA)$ and both compositions
$b_{\al}\hra t\to a$ and $b_{\al}\hra t\to a\to y$ is in $\Mor_{0,0}(\cA)$.
Since $\Mor^+_0(\cA)$ is closed under $\Mor_{0,0}(\cA)$-pullbacks, each projection $x\times_y b_{\al}\to b_{\al}$ is in
$\Mor^+_0(\cA)$ (use \re{setup2}(e)). Since $\Mor^+_{0}(\cA)$ is closed under compositions and contains $\Mor_{0,0}(\cA)$, each composition
$x\times_y b_{\al}\to b_{\al}\to a$, or, what is the same, $x\times_y b_{\al}\to x\times_y a\to a$ is in $\Mor^+_0(\cA)$.
But the induced map $x\times_y t\simeq \sqcup_{\al}(x\times_y b_{\al})\to x\times_y a$ is a covering in $\Mor_0(\cC)$, this means that
$f\in\Mor^+_0(\cC)$, as claimed.

(b),(c) The first assertion of (b) follows from the fact that $\Mor^+_0(\cA)$ contains $\Mor_{0,0}(\cA)$ and \rco{mor0}(a).
The rest follow from the fact that coverings, $\Mor_{\infty}(\cC)$ and $\Mor^+_{0}(\cC)$ are closed under compositions.

(d) By (b), the composition $x\times_y z\to z\to y$ and hence $x\times_y z\to x\to y$ is in $\Mor^+_{0}(\cC)$. Since
$x\times_y z\to x$ is a covering, the assertion follows from (c).
\end{proof}

\begin{Cor} \label{C:cM}
(a) A morphism $\sqcup_{\al}f_{\al}:\sqcup_{\al}x_{\al}\to y$ is in $\Mor^+_{\infty}(\cC)$ if and only if every $f_{\al}$ is in $\Mor^+_{\infty}(\cC)$.

(b) A morphism $\sqcup_{\al}f_{\al}:\sqcup_{\al}x_{\al}\to \sqcup_{\al} y_{\al}$ is in $\Mor^+_{\infty}(\cC)$ if and only if every $f_{\al}$ is in $\Mor^+_{\infty}(\cC)$.
\end{Cor}
\begin{proof}
(a) While the ``if'' assertion follows from definition, the converse follows from the fact that each inclusion $x_{\al_0}\to\sqcup_{\al}x_{\al}$ is in $\Mor_0(\cC)$ (see \re{cons2}(c)), and \rl{cM}(b).

(b) By (a),  $\sqcup_{\al}f_{\al}$ is in $\Mor^+_{\infty}(\cC)$ if and only if each
$x_{\al_0}\overset{f_{\al}}{\lra} y_{\al_0}\hra\sqcup_{\al}y_{\al}$ is in $\Mor^+_{\infty}(\cC)$, so again the assertion follows from
\rl{cM}(b) and \re{cons2}(c).
\end{proof}

\begin{Emp} \label{E:proQ}
{\bf Extension of morphisms in $\cB$.} (a) In the situation of \re{general}(a), let $\Mor^+_0(\cB)\subset \Mor(\cB)$ be a class of morphisms,  containing $\Mor_0(\cB)$, closed under compositions and under $\Mor_0(\cB)$-pullbacks.

(b) We denote by  $\Mor^+_0(\cA)\subset \Mor(\cA)$ the class of morphisms $f:x\to y$ with $x,y\in\Ob_0(\cA)$ satisfying the following property:
for every $\Mor_0(\cB)$-presentation $y\simeq\lim_{\al}y_{\al}$  there exists a $\Mor_0(\cB)$-presentation $x\simeq\lim_{\beta}x_{\beta}$ such that:

\indent\indent$(\star)$ for every $\al$ there exists $\beta$ and a morphism $f_{\beta,\al}:x_{\beta}\to y_{\al}$ belonging to $\Mor^+_0(\cB)$ \\
\indent\indent\indent\indent such that $\pr_{\al}\circ f:x\to y\to y_{\al}$ factors as $f_{\beta,\al}\circ \pr_{\beta}:x\to x_{\beta}\to y_{\al}$.
\end{Emp}

\begin{Emp} \label{E:proQprop}
{\bf Simple properties.}
(a) By definition, the class $\Mor^+_0(\cA)$ is closed under compositions. Also, a morphism $f:x\to y$ with $x,y\in\Ob_0(\cA)$ is in $\Mor^+_0(\cA)$ if and only if for every $\Mor_0(\cB)$-presentation $y\simeq\lim_{\al}y_{\al}$, each composition $\pr_{\al}\circ f:x\to y_{\al}$ is in $\Mor^+_0(\cA)$.

(b) By definition, a morphism $f:x\to y$ with $x\in\Ob_0(\cA),y\in\Ob(\cB)$ is in $\Mor^+_0(\cA)$ if and only if it decomposes as a composition $x\to z\to y$ with $x\to z$ in $\Mor_0(\cA)$ and $z\to y$ in $\Mor^+_0(\cB)$.

(c) The class $\Mor^+_0(\cA)$ contains $\Mor_{0,0}(\cA)$. Indeed, choose a morphism $x\to y$ in $\Mor_{0,0}(\cA)$, and let $y\simeq\lim_{\al}y_{\al}$ be a  $\Mor_0(\cB)$-presentation. By (a), it suffices to show that each composition $x\to y\to y_{\al}$ is in $\Mor^+_0(\cA)$. Since this composition is in $\Mor_0(\cA)$ (by \rl{pro}), the assertion follows from (b).
\end{Emp}

\begin{Emp} \label{E:IPP}
{\bf Independence of presentation property (IPP).} (a) We say that the class $\Mor^+_0(\cB)$ from \re{proQ}
satisfies the {\em independence of presentation property} (relative to $\Mor_0(\cB))$), if for every two $\Mor_0(\cB)$-presentations
$x\simeq\lim_{\al}y_{\al}$ and $x\simeq\lim_{\beta}x_{\beta}$, condition $(\star)$ holds for $\id:x\to x$.

(b) Notice that if $\Mor^+_0(\cB)$ satisfies IPP-property, then a morphism $f:x\to y$ between objects $x,y\in\Ob_0(\cA)$ belongs to $\Mor^+_0(\cA)$ if and only if condition $(\star)$ holds for {\em some} pair of $\Mor_0(\cB)$-presentations $y\simeq\lim_{\al}y_{\al}$ and $x\simeq\lim_{\beta}x_{\beta}$.

Indeed, this follows from the fact that $\Mor^+_0(\cB)$ is closed under compositions and contains $\Mor_{0}(\cB)$. Therefore $f$ is in $\Mor^+_0(\cA)$ if and only if each composition $\pr_{\al}\circ f:x\to y_{\al}$ is in $\Mor^+_0(\cA)$ for some $\Mor_0(\cB)$-presentation $y\simeq\lim_{\al}y_{\al}$.

(c) Combining \re{proQprop} and \rl{proQ} below, we see that in the situation of (a) the class $\Mor^+_0(\cA)$ satisfies all the properties of \re{cons1ext}(b), thus construction \re{cons1ext}(c) applies.
\end{Emp}


%

\begin{Emp} \label{E:proQrem}
{\bf Particular case.} (a) Since $\Mor_0(\cB)$ contains isomorphisms and is closed under compositions, we get that $\cB_0:=(\Ob(\cB),\Mor_0(\cB))$ is a ($1$-full) subcategory of $\cB$. Thus we have a natural functor $\iota:\Pro(\cB_0)\to \Pro(\cB)=\cA$, induced by the embedding $\cB_0\hra\cB$,
and the class $\Mor_0(\cB)$ satisfies the IPP-property if and only if $\iota$ induces an equivalence between $\Pro(\cB_0)$ and a $1$-full subcategory of $\cA$.

(b) In the situation of (a), every $x\in\Ob_0(\cA)$ has a canonical $\Mod_0(\cB)$-presentation $x\simeq\lim_{x\to y}y$, where the limit runs over  category $\C{I}$ such that

$\bullet$ objects in $\C{I}$ are morphisms $x\to y$ in $\Mor_0(\cA)$ with $y\in\Ob(\cB)$, and

$\bullet$ morphisms in $\C{I}$ are morphisms $y\to y'$ in $\Mor_0(\cB)$ under $x$.

Indeed, for every $\Mod_0(\cB)$-presentation $x\simeq\lim_{\al}x_{\al}$,
each projection $x\to x_{\al}$ is in $\Mor_0(\cA)$. Hence if
$\Mor_0(\cB)$ satisfies the IPP-property, then the category
$\cI$ is filtered, and the family of projections $\{x\to x_{\al}\}_{\al}$ is cofinal in $\cI$.

(c) By definition, if $\Mor_0(\cB)$ satisfies the IPP-property, then every class $\Mor_0^+(\cB)$ satisfying \re{proQ}(a)
also satisfies it.

\end{Emp}

\begin{Lem} \label{L:proQ}
If  $\Mor^+_0(\cB)$ satisfies the IPP-property, then $\Mor^+_0(\cA)$ is closed under $\Mor_{0,0}(\cA)$-pullbacks.
\end{Lem}

\begin{proof}
Let $g:z\to y$ be in $\Mor_{0,0}(\cA)$, and $f:x\to y$ in $\Mor^+_0(\cA)$. We want to show that the fiber product $x\times_y z$ in $\cA$ exists, belongs to $\Ob_0(\cA)$, and the projection $x\times_y z\to z$ is in $\Mor^+_0(\cA)$.

Choose a $\Mor_0(\cB)_{\cA}$-presentation $z\simeq\lim_{\al} z_{\al}$ over $y$ (see \re{general}(e)). Then the fiber product  $x\times_y z_{\al}$ exists
(see \re{general}(c)). Moreover, since $\cA$ admits all small filtered limits the fiber product $x\times_y z\simeq\lim_{\al}(x\times_y z_{\al})$ exists as well. Next, since $z\in \Ob_0(\cA)$ and $x\times_y z\to z$ is in $\Mor_0(\cA)$, we conclude that $x\times_y z\in\Ob_0(\cA)$ by \rl{pro}.

It remains to show that the projection $x\times_y z\to z$ is in $\Mor^+_0(\cA)$. Arguing as in the proof of \rl{pro}, there exists a $\Mor_0(\cB)$ presentation $z\simeq\lim_{\g}z_{\g}$ such that for each $\g$ there exists a morphism $z_{\al}\to z_{\g}$ in
$\Mor_0(\cA)$.  Since $\Mor^+_0(\cB)$ satisfies the IPP-property, it suffices to show that each composition $x\times_y z\to z\to z_{\al}\to z_{\g}$ is in $\Mor_0^+(\cA)$ (by \re{IPP}(b)). Thus, by \re{proQprop}(a), it suffices to show that each composition $x\times_y z\to z\to z_{\al}$ is in $\Mor_0^+(\cA)$.

Since the latter composition decomposes as $x\times_y z\to x\times_y z_{\al}\to z_{\al}$, and the first map is in $\Mor_0(\cA)$, it suffices to show that the map $x\times_y z_{\al}\to z_{\al}$ is in $\Mor^+_0(\cA)$. Replacing $z\to y$ by $z_{\al}\to y$, we can assume that $z\to y$ in $\Mor_0(\cB)_{\cA}$.

Then $z\to y$ is a pullback of a morphism $z'\to y'$ in $\Mor_0(\cB)$ with respect to some morphism $y\to y'$ in $\Mor_0(\cA)$.
Thus, replacing $x\to y$ by the composition $x\to y\to y'$ (and using \re{IPP}(b)), we can assume that $z\to y$ in $\Mor_0(\cB)$.
By, \re{proQprop}(a),(c), morphism $x\to y$ decomposes as $x\to x'\to y$, where $x\to x'$ is in $\Mor_0(\cA)$ and $x'\to y$ is in $\Mor^+_0(\cB)$.  Thus the assertion follows from the fact that classes $\Mor_0(\cA)$ and $\Mor^+_0(\cB)$ are stable under $\Mor_0(\cB)$-pullbacks.
\end{proof}



\begin{Emp} \label{E:fp}
{\bf $\cB$-presented morphisms.} We say that a morphism $f:x\to y$ in $\cA$ is {\em $\cB$-presented}, if $f\simeq f'\times_{y'}y$ for some morphism $f':z'\to y'$ in $\cB$.

\end{Emp}

\begin{Emp} \label{E:local}
{\bf Assumptions.}
(i) The Grothendieck topology $\cT$ on $\cA$ is generated by the Grothendieck topology $\cT_{\cB}$ on $\cB$ (see \re{summary}).

(ii) For every collection
$f_{\al}:x_{\al}\to y$ of morphisms in $\Mor_0(\cB)$ we have:

$\bullet$ the sheaf image of $\sqcup_{\al}f_{\al}:\sqcup_{\al}x_{\al}\to y$ is representable by a subobject $y'$ of $y$;

$\bullet$ for every morphism $g:y\to z$ in $\cB$, the restriction $g|_{y'}:y'\to z$ is in $\Mor^+_0(\cB)$ if and only if
$g\circ f_{\al}\in\Mor^+_0(\cB)$ for all $\al$.
\end{Emp}

\begin{Prop} \label{P:proQ}
Assume that $\Mor^+_0(\cB)$ satisfies the IPP-property (see \re{IPP}) and assumptions of \re{local}.

(a) For a morphism $f:x\to y$ with $x,y\in\Ob_{\infty}(\cC)$ the following are equivalent:

\quad (i) $f\in\Mor_{\infty}(\cC)$;

\quad (ii) there exist coverings $\{z_{\al}\to y\}_{\al}$ and $\{t_{\al\beta}\to x\times_y z_{\al}\}_{\beta}$ in $\Mor_{\infty}(\cC)$ with $z_{\al},t_{\al\beta}\in\Ob_0(\cA)$ such that each composition $t_{\al\beta}\to x\times_y z_{\al}\to z_{\al}$ is in $\Mor^+_0(\cA)$;

\quad (iii) for every pair of morphisms $a\to y$ and $b\to x\times_y z$ in $\Mor_{\infty}(\cC)$ with $a,b\in\Ob_0(\cA)$, the composition
$b\to x\times_y a\to a$ is in $\Mor^+_0(\cA)$.

(b) For a morphism $f:x\to y$ in $\Mor_{0,0}(\cA)$, we have $f\in\Mor^+_{0}(\cA)$ if and only if
$f\in\Mor^+_{\infty}(\cC)$.

(c) For every $\cB$-presentable morphism $f:x\to y$ in $\Mor_{0,0}(\cA)$, we have $f\in\Mor^+_{0}(\cA)$ if and only if there exists a $\Mor_0(\cB)$-presentation $y\simeq\lim_{\al}y_{\al}$ an index $\al$ and a morphism $f_{\al}:x_{\al}\to y_{\al}$ in $\Mor^+_0(\cB)$ such that $f\simeq f_{\al}\times_{y_{\al}}y$.
\end{Prop}

\begin{proof}
(a) Clearly, (iii)$\implies$(i)$\implies$(ii), so it remains to show that (ii)$\implies$(iii).

Let  $\{z_{\al}\to y\}_{\al}$ and $\{t_{\al\beta}\to x\times_y z_{\al}\}_{\beta}$ be as in (ii). We want to show that for every $a\to y$ and $b\to x\times_y a$ as in (iii), the composition $b\to x\times_y a\to a$ is in $\Mor^+_0(\cA)$. The argument consists of two steps.

\vskip 8truept

{\bf Step 1.} There exists a covering $u=\sqcup_{\g}u_{\g}\to b$ in $\Mor_{\infty}(\cC)$ with
$u_{\g}\in \Ob_0(\cA)$ such that each composition $u_{\g}\to b\to x\times_y a\to a$ is in $\Mor^+_0(\cA)$.
\begin{proof}

By \rco{mor0}(b), there exists a covering $\wt{a}_{\al}\to z_{\al}\times_y a$ in $\Mor_{\infty}(\cC)$ such that both compositions
$\wt{a}_{\al}\to a$  and $\wt{a}_{\al}\to z_{\al}$ are in $\Mor_0(\cC)$. By definition (and \re{setup2}(e)), there exists a decomposition
$\wt{a}_{\al}\simeq\sqcup_{\dt}\wt{a}_{\al,\dt}$ such that all compositions $\wt{a}_{\al,\dt}\hra \wt{a}_{\al}\to z_{\al}$ and $\wt{a}_{\al,\dt}\hra \wt{a}_{\al}\to a$ are in $\Mor_0(\cA)$.

Since $\{z_{\al}\to y\}_{\al}$ and  $\{\wt{a}_{\al,\dt}\to z_{\al}\times_y a\}_{\dt}$ are covering, we conclude that
$\{\wt{a}_{\al,\dt}\to a\}_{\al,\dt}$ is a covering as well. Since each $\wt{a}_{\al,\dt}\to a$ are in $\Mor_0(\cA)$, in order to construct a covering $u\to b$, we can replace $a\to y$ and $b\to x\times_y a$ with $\wt{a}_{\al,\dt}\to a\to y$ and $b\times_{a}\wt{a}_{\al,\dt}\to x\times_y \wt{a}_{\al,\dt}$, respectively. In other words, we can assume that $a\to y$ factors as $a\to z_{\al}\to y$, and $a\to z_{\al}$ is in $\Mor_0(\cA)$.

Set $t_{\al}:=\sqcup_{\beta}t_{\al\beta}$. Applying \rco{mor0}(b) to a pair of morphisms $b\to x\times_y a$ and
\[
t_{\al}\times_{z_{\al}}a\to (x\times_y z_{\al})\times_{z_{\al}}a=x\times_y a,
\]
there exists a covering $u\to b\times_{x\times_y a}(t_{\al}\times_{z_{\al}}a)$ in $\Mor_{\infty}(\cC)$ such that the composition $u\to t_{\al}\times_{z_{\al}}a$ is in $\Mor_0(\cC)$. Then $u\to b$ is the covering in $\Mor_{\infty}(\cC)$, because $t_{\al}\to x\times_y z_{\al}$ is.

Moreover, since $u\to \sqcup_{\beta}(t_{\al\beta}\times_{z_{\al}} a)$ is in $\Mor_0(\cC)$, there exists a decomposition $u\times_{t_{\al}}t_{\al\beta}\simeq\sqcup_{\g}u_{\beta\g}$ such that each $u_{\beta\g}\in\Ob_0(\cA)$ and each induced map $u_{\beta\g}\to t_{\al\beta}\times_{z_{\al}}a$ is in $\Mor_0(\cA)$. Then $u\simeq\sqcup_{\beta,\g}u_{\beta\g}$ and each $u_{\beta\g}\to t_{\al\beta}\times_{z_{\al}}a\to a$ or, what is the same, $u_{\beta\g}\to b\to a$ is in $\Mor^+_0(\cA)$, because $t_{\al\beta}\to z_{\al}$ is in $\Mor^+_0(\cA)$, while $a\to z_{\al}$ is in $\Mor_0(\cA)$.
\end{proof}

{\bf Step 2.} Now we are ready to show that $b\to a$ is in $\Mor_0^+(\cA)$. Choose $\Mor_0(\cB)$-presentations $a\simeq\lim_{i}a_{i}$ and
$b\simeq\lim_j b_j$. We have to show that each composition $b\to a\to a_{i}$ is in $\Mor^+_0(\cA)$. Note that the projection $b\to a_i$ factors through some $b_j$.

Since both maps $u_{\g}\to b$ and $b\to b_j$ are in $\Mor_0(\cA)$, their composition $u_{\g}\to b\to b_j$  is in $\Mor_0(\cA)$ as well, so there exists a $\Mod_0(\cB)$-presentation $u_{\g}\simeq\lim_{r}u_{\g,r}$ over $b_j$. Moreover, since $u_{\g}\to b_j\to a_i$ is in $\Mor^+_0(\cA)$, there exists $r_{\g}$ such that the composition $u_{\g,r_{\g}}\overset{p_{\g}}{\lra} b_j\to a_i$ is in $\Mor^+_0(\cB)$.
Let $b'_j\subset b_j$ be the image of $\sqcup_{\g}p_{\g}:\sqcup_{\g}u_{\g,r_{\g}}\to b_j$ (use \re{local}).
Then the composition $b'_j\hra b_j\to a_i$ is in $\Mor^+_0(\cB)$ (by \re{local}).

By construction, the composition $u=\sqcup_{\g}u_{\g}\to \sqcup_{\g}u_{\g,r_{\g}}\to b_j$ or, what is the same, the composition
$u\to b\to b_j$ factors through $b'_j$. Since $u\to b$ is a covering, the assumption \re{local}(i) implies that
the projection $b\to b_j$ factors through $b'_j$. Then the induced map $b\to b'_j$ is in $\Mor_0(\cA)$, and the composition
$b\to b'_j\to a_i$ is in $\Mor^+_0(\cA)$, as claimed.

(b) Follows immediately from (a). More directly, the ``only if'' assertion was shown in \rl{cM}(a), while Step 2 of the argument of (a)
shows the ``if'' assertion.

(c) Choose a $\Mor_0(\cB)$-presentation $y\simeq\lim_{\al}y_{\al}$. Since $f$ is $\cB$-presentable, there exists an index $\al$ and a morphism $f_{\al}:x_{\al}\to y_{\al}$ in $\Mor(\cB)$ such that $f\simeq f_{\al}\times_{y_{\al}}y$. Then $x$ has a $\Mor_0(\cB)$-presentation $x\simeq\lim_{\al'>\al}x_{\al'}$, with $x_{\al'}:=x_{\al}\times_{y_{\al}}y_{\al'}$.

First we claim that if $f_{\al}\in\Mor^+_0(\cB)$, then $f\in\Mor^+_0(\cA)$.  Since $\Mor^+_0(\cB)$ satisfies the IPP-property, it suffices to show that each composition $x\to y\to y_{\al'}$ is in $\Mor^+_{0}(\cA)$. But this composition decomposes as $x\to x_{\al}\times_{y_{\al}}y_{\al'}\to y_{\al'}$, the first of which in $\Mor_0(\cA)$ and the second one is in $\Mor^+_0(\cB)$, because $\Mor^+_0(\cB)$ is closed under $\Mor_0(\cB)$-pullbacks. Thus the assertion follows by \re{proQprop}(b).

Conversely, assume that $f\in\Mor^+_0(\cA)$. Then there exists $\al'>\al$ such that the composition
$x_{\al'}\overset{\pr_{\al',\al}}{\lra} x_{\al}\overset{f_{\al}}{\lra}y_{\al}$ is in
$\Mor^+_0(\cB)$. By the assumptions \re{local}(ii), the projection $\pr_{\al',\al}$ has an image
$x'_{\al}\subset x_{\al}$, and the restriction $f'_{\al}:=f_{\al}|_{x'_{\al}}:x'_{\al}\to y_{\al}$ is in $\Mor^+_0(\cB)$. Therefore
the pullback $f'_{\al'}:=f'_{\al}\times_{y_{\al}}y_{\al'}$ is in $\Mor^+_0(\cB)$.
As $\pr_{\al',\al}$, factors through $x'_{\al}\subset x'_{\al}$, the pullback $f'_{\al'}$ is nothing but
$f_{\al}\times_{y_{\al}}y_{\al'}:x_{\al'}\to y_{\al'}$, and the proof is complete.
\end{proof}

\subsection{Examples and complements}

\begin{Emp} \label{E:catplinfst}
{\bf Placid $\infty$-stacks.} (a) Let  $\cB$ be the category, whose objects are affine schemes (resp. schemes, resp. algebraic spaces) of finite type over $k$ and morphisms are affine maps, and let $\Mor_0(\cB)\subset\Mor(\cB)$ be the class of smooth maps. Then $\Mor_0(\cB)$ contains isomorphisms and closed under compositions and pullbacks. In other words, $\Mor_0(\cB)$ satisfies all the assumptions of \re{general}(a), thus the
construction of \re{general} applies. In this case,

$\bullet$ $\cA=\Pro(\cB)$ is the category of affine schemes (resp. qcqs schemes, resp. qcqs algebraic
spaces) with affine morphisms between them. Namely, this is obvious in the case of affine schemes, which suffices for this work,
and it follows from \cite[Theorem D and Proposition B.1]{Ry2} in the case of qcqs algebraic spaces.

$\bullet$ $\Mor_0(\cB)_{\cA}\subset\Mor(\cA)$ is the class of fp-smooth affine morphisms (see \cite[Proposition B.3(xiii)]{Ry2}).

$\bullet$ $\Ob_0(\cA)$ is the class of  affine schemes (resp. qcqs scheme, resp. qcqs algebraic spaces), admitting placid presentations
(in the sense of \re{plpres}(b)).

$\bullet$ $\Mor_0(\cA)$ is the class of strongly pro-smooth morphisms (in the sense of \re{plpres}(a)).


(b) Consider category $\cA$ together classes $\Ob_0(\cA)$ and $\Mor_0(\cA)$ constructed in (a), equipped with \'etale topology.
Then in all three cases the corresponding categories $\cC=\Shv(\cA)$ are canonically equivalent (see \re{rem infst}).
Furthermore, all assumptions  of \re{setup2}(ii) are satisfied. Namely, properties (c) and (d) are standard, while to see property (e) notice that if $X$ is an affine scheme admitting a placid presentation, and has a decomposition $X\simeq X_1\sqcup X_2$ as a coproduct in $\Shv(\Aff_k)$, then each $X_i$  is an affine scheme, and each embedding $j_i:X_i\to X$ is an fp-(open and closed) embedding, corresponding to an idempotent in $k[X]$. In particular, every $j_i$ is belongs to $\Mor_0(\cA)$. The two other cases are similar.

By the above observation, the construction \re{cons2} applies. Explicitly, in this case

$\bullet$ $\cC$ is the $\infty$-category of $\infty$-stacks (see \re{infst});

$\bullet$ $\Ob_{\infty}(\cC)$ is the class of placid $\infty$-stacks (see \re{plinfst}(c));

$\bullet$ $\Mor_{\infty}(\cC)$ is the class of smooth morphisms (see \re{plinfst}(c)).
\end{Emp}

\begin{Emp} \label{E:catequid}
{\bf Equidimensional and universally open morphisms.}
Let $\cB$ be the category $\Affft_k$ of affine schemes of finite type over $k$, and let $\Mor^+_0(\cB)\subset\Mor(\cB)$ be one of the following classes:

$\bullet$ weakly equidimensional morphisms (see \re{locdim}(b));

$\bullet$ equidimensional morphisms (see \re{locdim}(c));

$\bullet$ universally open morphisms;

(a) All these classes contain smooth morphisms and are closed under compositions and smooth pullbacks (see \rl{2outof3}(c), \rco{pullback} and \rco{equidim2}(a)). Therefore construction \re{proQ} applies, so each class $\Mor^+_0(\cB)$ gives rise to the corresponding class $\Mor^+_0(\cA)$
of morphisms between $0$-placid affine schemes.

(b) Moreover, since $\Mor_0(\cB)$ satisfy the IPP-property (by \rco{indep}), each of the above classes $\Mor_0^+(\cB)$ also does. Therefore classes $\Mor^+_0(\cA)$ contain smooth morphisms and are closed under compositions and smooth pullbacks (see \re{IPP}(c)), thus construction \re{cons1ext} applies. In particular, each $\Mor^+_0(\cA)$ gives rise to the corresponding class $\Mor^+_{\infty}(\cC)$ of morphisms between placid $\infty$-stacks.

(c) Finally, each class $\Mor^+_0(\cB)$ satisfies all the assumption of \re{local}. Namely, assumption (i) is clear, the first property of assumption (ii) follows from the fact if $f:X\to Y$ is a smooth morphism of schemes of finite type, then the image $f(X)\subset Y$ is open, and the induced map $X\to f(X)$ \'etale locally has a section. Finally, second property of (ii) follows from \rl{2outof3}(a),(d) and \rco{equidim2}(b) applied to the surjective map $X\to f(X)$. By the proven above, \rp{proQ} applies. In particular, the classes $\Mor^+_{\infty}(\cC)$ coincide with the classes of weakly equidimensional (resp. equidimensional, resp. pro-universally open) morphisms respectively in the sense of \re{clM2} (compare \rl{clM}).
\end{Emp}

\begin{Emp} \label{E:variants}
{\bf Other variants.} The categorical framework we developed above allows to define other classes of objects and morphisms as well. For example:

(i) In construction \re{catplinfst}(b), one could take $\Ob_0(\cA)$ to be all of $\Ob(\cA)$, thus getting a more general class of $\infty$-stacks.

(ii) Alternatively, one could take $\Mor_0(\cA)$ be the class of all pro-\'etale morphisms instead of strongly pro-smooth, thus getting
a Deligne--Mumford version of placid $\infty$-stacks. Alternatively, one could take  $\Mor_0(\cA)$ to be the class of fp-smooth (or fp-\'etale)  morphisms.

(iii) One could develop a version of the above theory, where $\cA$ is replaced by perfect affine schemes, and classes $\Ob_0(\cA)$ and
$\Mor_0(\cA)$ by perfectizations (see \re{remsttopeq}) of the classes of $0$-placid affine schemes and strongly pro-smooth morphisms, respectively.  This version is needed if one wants to extend the results of this work to the mixed characteristic case.
\end{Emp}

Though the following lemma is not used in this work, we include it for completeness and for further applications.

\begin{Lem} \label{L:indep}
Let $\cQ$ be the class of open (resp. universally open, resp. uo-equidimensional) affine morphisms of algebraic spaces of finite type over $k$,
and let $X\simeq\lim_{\al}X_{\al}$ and $X\simeq\lim_{\beta}X'_{\beta}$ be two presentations of an algebraic space $X$ with all transition maps
in $\cQ$. Then for every $\beta$ and every sufficiently large $\al$ the projection $\pr_{\al}:X\to X_{\al}$ factors as a composition $X\overset{\pr_{\beta}}{\lra} X'_{\beta}\overset{f_{\beta,\al}}{\lra}X_{\al}$ with $f_{\beta,\al}\in\cQ$.
\end{Lem}

\begin{proof}
Since $X_{\al}$ is of finite type over $k$, there exists $\beta$ such that  $\pr_{\al}:X\simeq\lim_{\beta}X'_{\beta}\to X_{\al}$ factors through $f:X'_{\beta}\to X_{\al}$. We claim that there exists $\dt>\beta$ such that the composition
$X'_{\dt}\overset{\pr'_{\dt,\beta}}{\lra}X'_{\beta}\overset{f}{\to} X_{\al}$ belongs to $\cQ$.

Note that the projection $\pr_{\beta}:X\to X'_{\beta}$ factors through $g:X_{\g}\to X'_{\beta}$. Moreover, increasing $\g$ we can
further assume that $\g>\al$ and the composition $X_{\g}\overset{g}{\to}X'_{\beta}\overset{f}{\to}X_{\al}$ is the transition map.
In particular, $f\circ g\in\cQ$.

Similarly, there exists $\dt>\beta$ such that $\pr_{\g}:X\to X_{\g}$ factors through $h:X'_{\dt}\to X_{\g}$ and such that
$g\circ h:X'_{\dt}\to X'_{\beta}$ is in $\cQ$.

First we claim that if $\cQ$ is the class of (universally) open morphisms, then the composition $f\circ g\circ h:X'_{\dt}\to X_{\al}$ belongs to $\cQ$. Let $U\subset X'_{\beta}$ be the image of $g\circ h$. Since $g\circ h$ is open, we conclude that $U$ is open. Since $f\circ g\circ h= (f|_U)\circ (g\circ h)$, it remains to show that $f|_U:U\to X_{\al}$ belongs to $\cQ$. Set $V:=g^{-1}(U)\subset X_{\g}$. It is an open subset, because $U$ is. Note that the map $g|_V:V\to U$ is surjective, because $U=\Ima(g\circ f)\subset\Ima g$, and $(f|_U)\circ (g|_V)=(f\circ g)|_V$ belongs to $\cQ$, because $f\circ g$ is. Therefore $f|_U$ belongs to $\cQ$ by \rl{2outof3}(a).

Now assume that $\cQ$ is the class of uo-equidimensional morphisms. By the proven above, we can increase $\beta,\g$ and $\dt$ if necessary, so that $f$ and $g$ are open. In this case, we claim that the composition $f\circ g\circ h$ is weakly equidimensional. As before, it suffices to show that $f|_U$ is such. By our assumptions, $g|_V$ is open surjective, $f|_U$ is open and $(f|_U)\circ (g|_V)=(f\circ g)|_V$ is weakly equidimensional. Therefore $f|_U$ is weakly equidimensional by  \rl{2outof3}(e), and the proof is complete.
\end{proof}

\subsection{Proof of \rt{colim}} \label{S:pfcolim}

Our argument is very similar to that of \cite[6.4.3]{Ga}. The following observations will be used several times.

\begin{Emp} \label{E:obscolim}
{\bf Morphisms of colimits}. (a) Note that for every functor $\al:I\to J$ of small $\infty$-categories induces a functor $\al^*:\cC^J\to\cC^I$.
Next, every morphism $\phi:X\to \al^*(Y)$ in $\cC^I$ with $X\in \cC^I$ and $Y\in\cC^J$ induces a morphism
\[
\phi_*:\colim_I X\overset{\phi}{\to}\colim_I \al^*(Y)\simeq\colim_J \al_!\al^*(Y)\overset{\on{counit}}{\lra}\colim_J Y
\]
in $\cC$, where $\al_!:\cC^I\to\cC^J$ is the left adjoint of $\al^*$.

Explicitly, $\phi$ is a collection of morphisms $\phi_i:X_i\to Y_{\al(i)}$ for $i\in I$, and $\phi_*$ is the composition
$\colim_{i\in I} X_i\to \colim_{i\in I} Y_{\al(i)}\to\colim_{j\in J} Y_j$.

(b) The construction of (a) is {\em compatible with compositions}. In other words, if $\beta:J\to K$ is another functor of small $\infty$-categories, and $\psi:Y\to\beta^*(Z)$ is a morphism in $\cC^J$ with $Z\in \cC^K$, then the composition
$\psi_*\circ\phi_*:\colim_I X\to \colim_J Y\to\colim_K Z$ is naturally isomorphic to the morphism $(\al^*(\psi)\circ\phi)_*$, corresponding to the
composition $\al^*(\psi)\circ\phi:X\to \al^*(Y)\to \al^*\beta^*(Z)$.

(c) The construction of (a) is {\em functorial in $\al$}. In other words, if $\eta:\al\to\al'$ is a morphism of functors $I\to J$, and
$\phi'$ is the composition $\eta^*\circ\phi: X\to\al^*(Y)\to\al'^*(Y)$, then the induced morphism
$\phi'_*:\colim_I X\to\colim_J Y$ is naturally homotopic to $\phi$. Explicitly, each $\phi'_i$ is the composition
$X_i\overset{\phi_i}{\to} Y_{\al(i)}\overset{\eta_i}{\to} Y_{\al'(i)}$, and the assertion follows from the fact that each composition
\[
Y_{\al(i)}\overset{\eta_i}{\lra}Y_{\al'(i)}\overset{\ins_{\al'(i)}}{\lra}\colim_j Y_j
\]
is isomorphic to $\ins_{\al(i)}$.

(d) By (b) and (c), the morphism $(\id_{\al^*(Y)})_*:\colim_I \al^*(Y)\to \colim_J Y$ is automatically an isomorphism,
if $\al$ has an adjoint.
\end{Emp}

The following lemma is the central point of the argument.

\begin{Lem} \label{L:trcolim}
Let $\cC$ be an $\infty$-topos, $x\in\Ob(\cC)$, let $J\subset\cC/x$ be a small 1-full subcategory, and let $Y:J\subset\cC/x\to\cC$ be the forgetful functor.

Let $\al:I\to J$ be a functor of small $\infty$-categories such that

\noindent(i) the functor $\mu:I\times J\overset{\al\times \Id}{\lra}J\times J\subset \cC/x\times \cC/x\overset{\times}{\lra}\cC/x$
factors through $J\subset \cC/x$;

\noindent(ii) the natural morphism $\colim_I \al^*(Y)\to x$ is an isomorphism.

Then the natural morphism $\colim_J Y\to x$ is an isomorphism.
\end{Lem}

\begin{Emp}
{\bf Remark.} Actually, we only use the fact that $\cC$ has all (small) colimits and fiber products, and that colimits commute with pullbacks.
\end{Emp}

\begin{proof}
By assumption (i), we are given three functors $\al\circ\pr_I,\pr_J,\mu:I\times J\to J$, which send $(i,j)$ to $\al(i)$, $j$ and $\al(i)\times_x j$, respectively, and two morphisms of functors $p_1:\mu\to\al\circ\pr_I$ and $p_2:\mu\to\pr_J$, corresponding to the projections
$\al(i)\times_x j\to\al(i)$ and $\al(i)\times_x j\to j$.

It now follows from observation \re{obscolim}(c) applied to
$X:=\mu^*(Y)\in\cC^{I\times J}$ that three morphisms
$\colim_{I\times J}X\to \colim_J Y$, corresponding to morphisms
\[
\Id_X:X\to \mu^*(Y), p_1:X\to (\al\circ\pr_I)^*(Y)=\pr_I^*(\al^*(Y))\text{ and }p_2:X\to\pr_J^*(Y)
\]
are homotopic.

By definition and assumption (ii), the morphism $(p_1)_*:\colim_{I\times J}X\to \colim_J Y$ is a morphism in $\cC/x$, which factors through $\colim_I \al^*(Y)\simeq x$. Therefore in order to show that the morphism $\colim_J Y\to x$ is an isomorphism, it suffices to show that $(p_1)_*$ is an isomorphism.
Thus, it suffices to show that the morphism $(p_2)_*$, which is homotopic to $(p_1)_*$, is an isomorphism. Unwinding definitions, one checks that $(p_2)_*$ is the morphism
\[
\colim_{I\times J}(\pr_I^*(\al^*(Y))\times_x \pr_J^*(Y))\simeq (\colim_{I}\al^*(Y))\times_x(\colim_{J}Y)\to \colim_{J}Y,
\]
induced by the projection  $\colim_{I}\al^*(Y)\to x$. Therefore it is an isomorphism by assumption (ii).
\end{proof}

\begin{Cor} \label{C:colim}
For every $n>0$ and $x\in\Ob_n(\cC)$, let $J_{n-1}\subset\cC/x$ be the $\infty$-category, whose objects are  morphisms $z\to x$ in $\Mor_{\infty}(\cC)$ with $z\in\Ob_{n-1}(\cC)$ and morphisms are morphisms $z\to z'$ in $\Mor_{n-1}(\cC)$ over $x$. Then the natural morphism $(\colim_{(z\to x\in J_{n-1})}z)\to x$ is an isomorphism.
\end{Cor}

\begin{proof}
Choose a covering $f:y\to x$ in $\Mor_{n-1}(\cC)$ with $y\in\Ob_{n-1}(\cC)$, and let $y^{[m]}$ be its \v{C}ech nerve (see \re{chech}).
We let $J:=J_{n-1}$, $I:=\Dt_s$, and $\al:I\to J$ be the functor $\al([m]):=(y^{[m]}\to x)$.
We claim that $\al$ satisfies the assumptions of \rl{trcolim}.

To check (i), notice that for every $z\to x$ in $J$ and $[m]\in\Dt_s$, the projection $y^{[m]}\times_x z\to z$ belongs to $\Mor_{n-1}(\cC)$,
because $y^{[m]}\to z$ is. Hence $y^{[m]}\times_x z\in\Ob_{n-1}(\cC)$, because $z\in\Ob_{n-1}(\cC)$. Next, for every morphism
$[m]\to[m']$ in $\Dt_s$ and $z\to z'$ in $J$, the fiber product   $y^{[m]}\times_x z\to y^{[m']}\times_x z'$ belongs to $\Mor_{n-1}(\cC)$
because $y^{[m]}\to y^{[m']}$ and $z\to z'$ are.  Finally, assumption (ii) follows from the fact that $y\to x$ is a covering (see \re{chech}(c)).
\end{proof}

Now we are ready to show \rt{colim}.

\begin{Emp}
\begin{proof}[Proof of \rt{colim}] We carry out the proof in three steps.

\vskip 8truept

{\bf Step 1.} Assume that $x\in\Ob_0(\cC)$, and fix a decomposition $x\simeq \sqcup_{i\in I}x_{i}$ with $x_{i}\in \Ob_0(\cA)$.
We set $J:=J_x$, view $I$ as a category with no non-identity morphisms, and let $\al:I\to J_x$ be the functor $\al(i):=(x_i\hra x)$. We claim that $\al$ satisfies all the assumptions of \rl{trcolim}, which implies the assertion in this case.

To check (i), notice that it follows from \re{setup2}(e) that for every $z\to x$ in $J$ and $i\in I$,
we have $x_i\times_x z\in\Ob_0(\cA)$ and the projection $x_i\times_x z\to z$ belongs to $\Mor_{0}(\cA)$, thus
$x_i\times_x z\to x$ is in $\Mor_{\infty}(\cC)$. Similarly, the transition map $x_i\times_x z\to x_i\times_x z'$ belongs to $\Mor_0(\cA)$,
if $z\to z'$ does. Finally, assumption (ii) follows from the fact that the natural map $\sqcup_{i\in I}x_{i}\to x$ is an  isomorphism by assumption.

\vskip 8truept

{\bf Step 2.} Assume now that $x\in\Ob_n(\cC)$ with $n>0$ and that the assertion is proven for every $z\in\Ob_{n-1}(\cC)$.
Let $J_{n-1}\subset\cC/x$ be the $\infty$-category defined in \rco{colim}, and let $J'_{n-1}$ be the $\infty$-category whose objects are composable morphisms $y\to z\to x$ with $(z\to x)\in J_{n-1}$, $(y\to z)\in\Mor_{\infty}(\cC)$ and $y\in\Ob_0(\cA)$ and morphisms are commutative diagrams such that $y\to y'$ is in $\Mor_0(\cA)$ and $z\to z'$ is in $\Mor_{n-1}(\cC)$. We denote by $\al:J'_{n-1}\to J_{n-1}$ the forgetful functor $(y\to z\to x)\mapsto (z\to x)$, and by
$\beta:J'_{n-1}\to J_x$ the composition functor $(y\to z\to x)\mapsto (z\to x)$.
Functors $\al$ and $\beta$ induce maps
\begin{equation} \label{Eq:colimits}
(\colim_{(y\to x\in J_x)}y)\overset{\beta_*}{\lla}(\colim_{(y\to z\to x\in J'_{n-1})}y)\overset{\al_*}{\lra}(\colim_{(z\to x\in J_{n-1})}z)\overset{\g}{\to} x,
\end{equation}
and it suffices to show that all these maps are isomorphisms.
\vskip 8truept

{\bf Step 3.} First of all, the map $\g$ is an isomorphism by \rco{colim}. Next, functor $\beta$ has a left adjoint
$(y\to x)\mapsto (y\overset{\id_y}{\to}y\to x)$. Therefore the map $\beta_*$ is an isomorphism by \re{obscolim}(d).

To show that $\al_*$ is an isomorphism, consider forgetful functors $Y:J'_{n-1}\to\cC$ and $Z:J_{n-1}\to\cC$ given by $Y(y\to z\to x):=y$ and $Z(z\to x):=z$. Then we have a natural morphism
$Y\to\al^*(Z)$ of functors $J'_{n-1}\to \cC$, which by adjunction induces a morphism $\al_!(Y)\to Z$. Then the map
\[
\al_*:\colim_{J'_{n-1}}Y\simeq \colim_{J_{n-1}}\al_!(Y)\to \colim_{J_{n-1}}Z
\] is induced by the morphism $\al_!(Y)\to Z$,
so it remains to show that the map $\al_!(Y)\to Z$ is an isomorphism, that is, the map $\al_!(Y)(z\to x)\to z$ is an isomorphism for all
$(z\to x)\in J_{n-1}$.

Note that the functor $\al:J'_{n-1}\to J_{n-1}$ is coCartesian. Indeed, for every object $y\to z\to x$ of $J'_{n-1}$, every morphism
$(z\to x)\to (z'\to x)$ in $J_x$ has a coCartesian lift
\[
(y\to z\to x)\to (y\to z'\to x).
\]
Therefore the value of $\al_!(Y)$ at $(z\to x)\in J_{n-1}$ is naturally isomorphic to the colimit $\colim_{\al^{-1}(z\to x)}y$ over the fiber $\al^{-1}(z\to x)\subset J'_{n-1}$ (see \cite[Proposition 4.1.2.15]{Lu}). But $\al^{-1}(z\to x)$ is naturally isomorphic to the category $J_z$. Thus the morphism $(\colim_{\al^{-1}(z\to x)}y)\to z$ is an isomorphism, because $(\colim_{(y\to z\in J_z)}y)\to z$ is an isomorphism by the induction hypothesis.
\end{proof}
\end{Emp}

\section{Completion of proofs.}

\subsection{Quotients of ind-schemes} \label{S:quotient}



\begin{Emp} \label{E:discaction}
Let $\Dt$ be a group acting on an ind-scheme $X$ over an ind-scheme $Y$.

(a) We say that $\Dt$ acts {\em discretely}, if for every fp-closed qcqs subscheme
$X'\subset X$, the set of $\dt\in \Dt$ such that $\dt(X')\cap X'\neq \emptyset$ is finite.

(b) We say that $\Delta$ acts {\em freely}, if the action map $a:\Dt\times X\to X\times X:(\dt,x)\mapsto(\dt(x),x)$ is a
monomorphism.
\end{Emp}

\begin{Prop} \label{P:quot}
Let $f:X\to Y$ be an ind-fp-proper morphism between ind-schemes, and let $\Dt$ be a group acting on $X$ over $Y$.

(a) Assume that $\Dt$ acts freely and discretely. Then the quotient $\ov{X}:=[X/\Delta]$ is an ind-algebraic space, ind-fp-proper over $Y$.

(b) Assume in addition that for every fp-closed qcqs subscheme $Y'\subset Y$ there exists an fp-closed qcqs subscheme $X'\subset X$ such that
$f^{-1}(Y'(K))\subset\bigcup_{\dt\in\Dt}\dt(X'(K))$ for every algebraically closed field $K$. Then

\quad(i) $f$ is topologically locally fp-schematic, and the induced morphism $\ov{f}:\ov{X}\to Y$ is topologically fp-proper representable.

\quad(ii) for every morphism $Y'\to Y$ from a qcqs scheme $Y'$, there exists a subgroup $\Dt_0\subset\Dt$ of finite index such that the induced map  $[X\times_Y Y'/\Delta_0]\to Y'$ is topologically fp-proper schematic.
\end{Prop}

\begin{proof}
Every presentation $Y\simeq\colim_i Y_i$ of an ind-scheme $Y$ induces a presentation $X\simeq\colim_i X_i$ with $X_i:=X\times_Y Y_i$. Since all assertions for $f:X\to Y$ formally follow from the corresponding assertions for $f_i:X_i\to Y_i$, we can replace $f$ by $f_i$, thus assuming that $Y$ is a qcqs scheme.

Let $X'\subset X$ be an fp-closed qcqs subscheme. For every finite subset $D\subset\Delta$, we denote by $X'_D:=\cup_{\dt\in D}\dt(X')\subset X$
the smallest closed subscheme of $X$, containing all the $\dt(X')$'s. Then for every subset $\Dt'\subset \Dt$, we consider the colimit $X'_{\Dt'}:=\colim_{D\subset\Dt'}X'_{D}$, taken over all finite subsets of $\Dt'$. Since $X'\to Y$ is fp-proper, we conclude that $X'_D\to Y$ is fp-proper for every $D$, thus  $X'_{\Delta'}\to Y$ is ind-fp-proper.

We claim that the inclusion $X'_{\Dt'}\hra X$ is an fp-closed embedding. For this we have to show that for every  fp-closed qcqs subscheme $X''\subset X$, the intersection $X'_{\Dt'}\cap X''\subset X''$ is an fp-closed subscheme.
Since the action of $\Dt$ on $X$ is discrete, the set $D:=\{\dt\in\Dt\,|\,\dt(X')\cap X''\neq\emptyset\}$ is finite, thus
$X'_{\Dt'}\cap X''=X'_{D}\cap X''$ is an fp-closed subscheme of $X''$.

Note that $X'_{\Delta}:=\colim_{D}X'_{D}$  is the smallest $\Dt$-invariant closed ind-subscheme of $X$, containing $X'$.
We form the quotient $\overbar{X}':=[X'_{\Delta}/\Delta]$.

\begin{Cl}\label{C:projk}
(a) The ind-scheme $X'_{\Delta}$ is a scheme, locally fp over $Y$.

(b) The quotient $\overbar{X}'$ is an algebraic space, fp-proper over $Y$, and the projection $X'\rightarrow \overbar{X}'$ is surjective.

(c)	For every pair of fp-closed qcqs subschemes $X'\subset X''$ of $X$, the induced map $\overbar{X}'\rightarrow\overbar{X}''$ is an fp-closed embedding.

(d) There exists a subgroup $\Dt_0\subset \Dt$ of finite index such that the quotient $\wt{X}':=[X'_{\Delta}/\Delta_0]$ is a scheme, fp-proper over
$Y$, and the projection $X'_{\Delta}\to \wt{X}'$ is a Zariski local isomorphism.
\end{Cl}

We now complete the  proof of \rp{quot}, assuming the claim.

(a) Choose a presentation $X\simeq\colim_{\al} X_{\al}$ of the ind-scheme $X$. Since $X_{\al}\subset X_{\al,\Dt}\subset X$, we get an isomorphism
$X\simeq\colim_{\al} X_{\al,\Dt}$. Taking the quotient by $\Dt$, we get an isomorphism $\ov{X}\simeq\colim_{\al} \ov{X}_{\al}$. Since every $\ov{X}_{\al}$ is an algebraic space, fp-proper over $Y$ (by \rcl{projk}(b)), and each $\ov{X}_{\al}\to \ov{X}_{\beta}$ is  an fp-closed embedding (by \rcl{projk}(c)), the assertion follows.

(b) By our assumption, there exists $\al$ such that $X_{\Dt}(K)=X_{\Dt,\al}(K)$ and $\ov{X}(K)=\ov{X}_{\al}(K)$ for all algebraically closed fields $K$.
In particular, for all $X_{\beta}\supseteq X_{\al}$, the fp-closed embeddings $X_{\Dt,\al}\to X_{\Dt,\beta}$ and $\ov{X}_{\al}\to\ov{X}_{\beta}$ induce bijections  on $K$-points. Hence the induced maps  $(X_{\Dt,\al})_{\red}\to(X_{\Dt,\beta})_{\red}$ and $(\ov{X}_{\al})_{\red}\to(\ov{X}_{\beta})_{\red}$ are isomorphisms.
Therefore the maps $(X_{\Dt,\al})_{\red}\to X_{\red}$ and $(\ov{X}_{\al})_{\red}\to \ov{X}_{\red}$ are isomorphisms as well. Since  $X_{\Dt,\al}\to Y$ is schematic, locally fp
(by \rcl{projk}(a)), while $\ov{X}_{\al}\to Y$ is fp-proper (by \rcl{projk}(a)), assertion (i) follows.

Finally, assertion (ii) follows immediately from \rcl{projk}(d).
\end{proof}

It remains to show \rcl{projk}.

\begin{Emp}
\begin{proof}[Proof of \rcl{projk}]
(a) We have to show that every point $x\in X'_{\Delta}$ has an open neighborhood, which is a scheme finitely presented over $Y$. Since every point of $X'_{\Delta}$ is a $\Dt$-translate of a point of $X'$, it is sufficient to prove it for $x\in X'$. We claim that the whole $X'$ has such a neighborhood. Set $\Sigma:=\{\delta\in\Delta~\vert~\delta(X')\cap X'\neq\emptyset\}$. By assumption, it is a finite set.

Then $X'_{\Dt\sm\Si}\subset X$ is a closed subfunctor, hence
$U:=X'_{\Delta}\sm X'_{\Dt\sm\Si}\subset X$ is an open subfunctor. Since $X'\cap X'_{\Dt\sm\Si}=\emptyset$ by the definition of $\Si$, we have $X'\subset U$, and clearly $U\subset X'_{\Si}$.
Thus $U$ is an open subscheme of $X'_{\Si}$. Hence it is a scheme, locally fp over $Y$, as claimed.

(b) As $\Delta$ acts freely, it defines an \'etale equivalence relation on $X'_{\Delta}$. Thus $\overbar{X}'=[X'_{\Delta}/\Delta]$  is an algebraic space (see \cite[Tag 0264]{Sta}), locally fp over $Y$.

Moreover, since  $X'_{\Delta}$ is a filtered colimit $\colim_D X'_{D}$ with $X'_{D}$ proper (thus separated) over $Y$, we conclude that $X'_{\Delta}$ is separated over $Y$. Next we claim that the map $\overbar{X}'\rightarrow Y$ is separated.

We have to show that the map $a:\Delta\times X'_{\Delta}\rightarrow X'_{\Delta}\times_{Y}X'_{\Delta}$
is proper. It suffices to show the properness of the restriction of $a$ to the inverse image of $X'_{D}\times_{Y}X'_{D}$ for every finite subset $D\subset\Delta$.
But this inverse image is the disjoint union of maps $a_{D,\delta}:X'_{D}\cap\delta^{-1}(X'_{D})\rightarrow X'_{D}\times_{Y}X'_{D}$.
As the action is discrete, this union is finite. So one has to prove that each $a_{D,\delta}$ is a closed embedding. But each $a_{D,\dt}$ is the restriction of the graph of $\delta:X'_{\Delta}\to X'_{\Delta}$, which is a closed embedding, as $X'_{\Delta}\rightarrow Y$ is separated.

Finally, we claim that  $\overbar{X}'\rightarrow Y$ is fp-proper.
Indeed, since $X'\rightarrow\overbar{X}'$ is surjective, $X'$ is fp-proper over $Y$ and $\overbar{X}'$ separated over $Y$, we conclude that $\overbar{X}'\rightarrow Y$ is proper by \cite[Tag 08AJ]{Sta}. As it is both locally fp and proper, it is fp-proper.

(c) Since $X'\subset X''$ is a closed subfunctor, we conclude that  $X'_{\Dt}\subset X''_{\Dt}$ and hence also $\ov{X}'\subset\ov{X}''$ is a closed subfunctor.
It is finitely presented by  \cite[Tag 02FV]{Sta}, because both $\ov{X}'$ and $\ov{X}''$ are fp-proper over $Y$.

(d) Let $U\subset X'_{\Si}\subset X'_{\Dt}$ be the open neighborhood of $X'$, constructed in (a). As the action of $\Dt$ is discrete,
the set $\Si':=\{\dt\in\Dt\,|\,\dt(U)\cap U\neq\emptyset\}$ is finite. Therefore there exists a normal subgroup $\Dt_0\subset\Dt$
of finite index such that $\Dt_0\cap \Si'=\emptyset$. We claim that this $\Dt_0$ satisfies the required property.
We have to show that for every $x\in X'_{\Dt}$, there is an open neighborhood $U_x$ such that
$\dt(U_x)\cap U_x=\emptyset$ for every $\dt\in\Dt_0$. Indeed, when $x\in X'$, the open set $U_x:=U$ does the job. The general case follows from it, because $\Dt_0\subset\Dt$ is a normal subgroup.
\end{proof}
\end{Emp}

\begin{Cor}\label{C:quot}
Let $f:X\to Y$ be an ind-fp-proper morphism between ind-schemes, and let $\Dt$ be a group acting on $X$ over $Y$ such that for every algebraically closed field $K$, the induced map $f:X(K)\to Y(K)$ is a $\Dt$-torsor.

Then the assumptions of \rp{quot}(a),(b) are satisfied,
and the induced morphism $\ov{f}:[X/\Dt]\to Y$ is a topological equivalence.
\end{Cor}

The proof of \rco{quot} is based on the following assertion.

\begin{Lem} \label{L:constop}
Let $Y$ be a qcqs scheme over $k$, let $X$ be an ind-scheme with a presentation $X\simeq \colim_{\al}X_{\al}$, and let $f:X\to Y$ be a morphism such that each restriction $f_{\al}:=f|_{X_{\al}}:X_{\al}\to Y$ is an fp-closed embedding, and $f(K):X(K)\to Y(K)$ is a surjection (hence a bijection) for every  algebraically closed field $K$.
Then the closed embedding $f_{\al}$ is surjective for some $\al$.
\end{Lem}

\begin{proof}
Denote by $Y_{\on{cons}}$ the scheme $Y$ equipped with the constructible topology. The assumptions imply that $\{f_{\al}(X_{\al})\}_{\al}$ forms an open covering of $Y_{\on{cons}}$. Since the topological space $Y_{\on{cons}}$ is compact (see \cite[Tags 08YF, 094L]{Sta}), while the collection $\{X_{\al}\}_{\al}$ is filtered, there exists $\al$ such that the closed embedding $f_{\al}:X_{\al}\hra Y$ is surjective.
\end{proof}


\begin{Emp}
\begin{proof}[Proof of \rco{quot}]
As in the proof of \rp{quot}, we can assume that $Y$ is a reduced qcqs scheme. By assumption, the action map $a:\Dt\times X\to X\times_Y X$ induces a bijection on $K$-points for every algebraically closed field $K$. Hence for every $1\neq\dt\in\Dt$, the automorphism $a_{\dt}:X\to X$ of $X$ is different from the identity, so the action of $\Dt$ on $X$ is free.

To show the assumption of \rp{quot}(b), we choose a presentation $X\simeq\colim_{\al}X_{\al}$. Since each $f_{\al}:X_{\al}\to Y$ is fp-proper, there exists an fp-closed subscheme $Y_{\al}\subset Y$ such that
$f(X_{\al}(K))=Y_{\al}(K)$ for all $K$. Since $f(X(K))=\cup_{\al}Y_{\al}(K)$ equals $Y(K)$ by assumption, we conclude from \rl{constop} that there exists $\al$ such that the map
$f_{\al}:X_{\al}\to Y$ and hence the map $X_{\al}(K)\to Y(K)$ is surjective. Since $X(K)\to Y(K)$ is a $\Dt$-torsor, this implies that $X(K)=\cup_{\dt\in\Dt}\dt(X_{\al}(K))$, as claimed.

Finally, to show that the $\Dt$-action is discrete, we fix an fp-closed qcqs subscheme $X'\subset X$, and we want to show that the set $\Si:=\{\dt\in\Dt\,|\,\dt(X')\cap X'\neq\emptyset\}$ is finite.
Consider the preimage $Z:=a^{-1}(X'\times_Y X')\subset \Dt\times X'\subset\Dt\times X$. Then $Z$ is an ind-scheme equipped with a presentation
$Z\simeq\colim_D Z_D$, induced by the presentation $\Dt\times X'\simeq\colim_D(D\times X')$, where $D\subset\Dt$ runs over the set of all finite subsets.
Now the finiteness follows from \rl{constop}, applied to the morphism $a:\colim_D Z_D\to X'\times_Y X'$.

By \rp{quot}(b), the morphism $\ov{f}:[X/\Dt]\to Y$ is topologically fp-proper. Hence, replacing $X$ by $X_{\red}$ we can assume that $\ov{f}:[X/\Dt]\to Y$ is a proper morphism of algebraic spaces.  Next, since $f(K):X(K)\to Y(K)$ is a $\Dt$-torsor for every algebraically closed field $K$,
the diagonal map $\Dt_{X/Y}:[X/\Dt]\to [X/\Dt]\times_Y [X/\Dt]$ is bijective on $K$-points  for all $K$, thus
$\ov{f}:[X/\Dt]\to Y$ is a proper surjective monomorphism.

Then $\ov{f}$ is proper and quasi-finite, so it follows for example from the Zariski Main theorem (see \cite[Tag 082K]{Sta}) that $\ov{f}$ is finite. Moreover, since $f$ is a monomorphism, it is a closed embedding (by \cite[Tag 04XV]{Sta}), hence it is a topological equivalence because it is surjective on points.
\end{proof}
\end{Emp}

\subsection{Proof of \rt{orbit} and \rt{topproper}}

\begin{Emp} \label{E:pfrtorbit}
\begin{proof}[Proof of \rt{orbit}]
Since the projection $[\cL G/\cL(S)_{\red}]_{\red}\to [\cL G/\cL(S)]_{\red}$ is an isomorphism, it suffices to show that the morphism $[\cL G/\cL(S)_{\red}]\to \cL(G_F/S)$ is a topological equivalence. By \re{ostr}(b), we have a natural isomorphism $\La_S\simeq[(\cL S)_{\red}/\clp(S)]$. In particular, the group $\La_S$ acts naturally on $[\cL G/\clp(S)]$ such that $[\cL G/(\cL S)_{\red}]\simeq [[\cL G/\clp(S)]/\La_S]$. Therefore \rt{orbit} follows from a combination of \rco{quot} and \rcl{orbit} below.
\end{proof}
\end{Emp}

\begin{Cl} \label{C:orbit}
(a) The projection $\cL G/\clp(S)\to \cL(G_F/S)$ is ind-fp-proper.

(b) The induced morphism $(\cL G/\clp(S))(K)\to \cL(G_F/S)(K)$ is a $\La_S$-torsor for every algebraically closed field $K$.
\end{Cl}

\begin{proof}
(a) Replacing $S$ by its $G_F$-conjugate, we can assume that the embedding $S\hra G_F$ is defined over $\cO$ (see \re{maxtori} and \re{ostr}(c)).
Then our projection decomposes as
\[
\cL G/\clp(S)\simeq \cL G\times^{\clp(G)}(\clp G/\clp(S))\simeq \cL G\times^{\clp G}\clp(G_{\cO}/S_{\cO})\overset{\iota}{\hra}\]
\[
\cL G\times^{\clp(G)}\cL(G_F/S)\overset{\phi}{\lra} (\cL G/\clp G)\times\cL(G_F/S)\overset{\pr}{\lra}\cL(G_F/S),
\]
where the second map is the isomorphism from \re{proparc}(f), $\iota$ is induced by the fp-closed embedding $\clp(G_{\cO}/S_{\cO})\hra \cL(G_F/S)$, while
$\phi$ is induced by the isomorphism
\[
\cL G\times\cL(G_F/S)\isom\cL G\times\cL(G_F/S):(g,x)\mapsto (g,gx).
\]
Now the assertion follows from the fact that the quotient $\cL G/\clp(G)$ is ind-fp-proper.

(b) Notice that
\[
[\cL G/\clp(S)](K)\simeq  [\cL G(K)/\clp(S)(K)]=[G(K((t))/S(K[[t]])],
\]
because $K$ is algebraically closed, and
\[
\cL(G_F/S)(K)= (G_F/S)(K((t)))= G(K((t)))/S(K((t))),
\]
because $H^1(K((t)),S)=1$. Since $S(K((t))/S(K[[t]])\simeq\La_S$, the assertion follows.
\end{proof}


Recall that the ind-scheme $\cL G$ has a presentation $\cL G\simeq\colim_i \wt{Y}_i$ (see \re{afffl}(c)).

\begin{Cor} \label{C:open}
For every fp-closed qcqs subscheme $Z\subset\cL(G_F/S)$, there exists an index $i$ such that for every algebraically closed field $K$ the projection $\pr_S:\cL G\to \cL(G_F/S)$ satisfies $\pr_S^{-1}(Z(K))\subset\wt{Y}_i(K)\cdot\La_S$, thus $\pr_S(\wt{Y}_i(K))\supseteq Z(K)$.
\end{Cor}

\begin{proof}
Note that $\pr_S$ decomposes as  $\cL G\overset{\al}{\to}\cL G/\clp(S)\overset{\beta}\to \cL(G_F/S)$.
By \rcl{orbit} and \rco{quot}, there exists an fp-closed qcqs subscheme $Z'\subset\cL G/\clp(S)$ such that $\beta^{-1}(Z(K))\subset \La_S\cdot Z'(K)$
for every $K$. Set $Z'':=\al^{-1}(Z')\subset\cL G$. Then $Z''$ is an fp-closed qcqs subscheme, thus $Z''\subset\wt{Y}_i$ for some $i$.
By construction, we have
\[
\pr_S^{-1}(Z(K))=\al^{-1}\beta^{-1}(Z(K))\subset\al^{-1}(Z'(K)\cdot\La_S)=Z''(K)\cdot\La_S\subset\wt{Y}_i(K)\cdot\La_S.
\]
\end{proof}

\begin{Emp} \label{E:pfrttopproper}
\begin{proof}[Proof of \rt{topproper}]
Recall that the projection $\frak{p}:\wt{\fC}\to \fC$ is an ind-fp-proper morphism between ind-schemes (by \rl{affspr}). Taking pullback to ${\fC}_{\ft,w,\br}$, we conclude that the map $\frak{p}_{\ft,w,\br}:\wt{\fC}_{\ft,w,\br}\to{\fC}_{\ft,w,\br}$ is an ind-fp-proper morphism between
ind-schemes as well. Moreover, the group $\La_w$ acts on $\wt{\fC}_{\ft,w,\br}$ over ${\fC}_{\ft,w,\br}$ (see \re{laaction}(d)).
We claim that all assumptions of \rp{quot}(a),(b) are satisfied.

To show that the action of $\La_w$ on $\wt{\fC}_{\ft,w,\br}$ is discrete, note that the presentation $\Fl=\colim_i Y_i$ from \re{afffl}(b) gives rise to a presentation $\wt{\fC}_{\ft,w,\br}\simeq\colim_i \wt{\fC}_{\ft,w,\br,i}$, where $\wt{\fC}_{\ft,w,\br,i}\subset \wt{\fC}_{\ft,w,\br}$
consists of triples $(g,h,x)\in\Fl\times\cL(G_F/T_w)\times\ft_{w,\br}$ such that $g\in Y_i$. It suffices to show that the set of all
$\la\in\La_w$ such that $\la(\wt{\fC}_{\ft,w,\br,i})\cap \wt{\fC}_{\ft,w,\br,i}\neq\emptyset$ is finite.

By definition, for every such $\la$ there exist $g\in Y_i$ and $h\in \cL(G_F/T_w)$ such that $g':=(h\la h^{-1})g$ is in $Y_i$.
Choose representatives $\wt{g}\in\wt{Y}_i$ of $g$ and $\wt{h}\in \cL G$ of $h$. Then
$\wt{g}':=(\wt{h}\la \wt{h}^{-1})\wt{g}\in \wt{Y}_i$, thus
$\wt{h}\la \wt{h}^{-1}=\wt{g}'\wt{g}^{-1}\in \wt{Y}_i\wt{Y}_i^{-1}$.
Then the conjugacy class of such $\la$'s in $\cL G$ is bounded, thus the set of such $\la$'s is finite.

To show that $\La_w$ acts freely, notice that $\la$ has a fixed point if and only if  $h\la h^{-1}g\in gI$, which means that $\la\in (h^{-1}g)I(h^{-1}g)^{-1}$, that is, $\la\in\La_w\cap (h^{-1}g)I(h^{-1}g)^{-1}$. But the latter intersection is torsion free, discrete and bounded, thus trivial.

Thus, the conditions of \rp{quot} are satisfied, hence $\ov{\fC}_{w,\br}$ is an ind-algebraic space, ind-fp-proper over ${\fC}_{\ft,w,\br}$.

To show that it is topologically proper, we have to check that the condition of \rp{quot}(b) is satisfied as well, that is, for every fp-closed subscheme $Z$ of ${\fC}_{\ft,w,\br}\simeq\cL(G_F/T_w)\times \ft_{w,\br}$, there exists $i$ such that
\begin{equation} \label{Eq:topproper}
\text{ for all }(g,h,x)\in\frak{p}_{\ft,w,\br}^{-1}(Z(K)) \text{ there exists }\la\in\La_w \text{ such that }(h\la h^{-1})g\in Y_i.
\end{equation}
Recall (see \rl{isom}(a)) that the action map $(g,x)\mapsto \Ad_g(x)$ induces a finite morphism $a:\cL(G_F/T_w)\times \ft_{w,\br}\simeq{\fC}_{\ft,w,\br}\to \fC_{w,\br}$.
Therefore $a^{-1}(\Lie(I))\subset \cL(G_F/T_w)\times \ft_{w,\br}$ is an fp-closed qcqs subscheme.

Choose fp-closed qcqs subschemes $Z_1,Z_2\subset \cL(G_F/T_w)$ such that we have $Z\subset Z_1\times  \ft_{w,\br}$ and $a^{-1}(\Lie(I))\subset Z_2\times \ft_{w,\br}$. Then, by \rco{open}, there exist indices $i_1,i_2$ such that we have $\pr_{T_w}(\wt{Y}_{i_1}(K))\supseteq Z_1(K)$ and $\pr_{T_w}^{-1}(Z_2(K))\subset\wt{Y}_{i_2}(K)\cdot\La_w$.
We claim that every index $i$ such that $\wt{Y}_{i_1}\cdot\wt{Y}_{i_2}^{-1}\subset\wt{Y}_i$ satisfies \form{topproper}.

Indeed, choose a representative $\wt{g}\in \cL G(K)$ of $g$. Since $(h,x)\in Z\subset Z_1\times\kt_{w,\br}$, there exists a
representative $\wt{h}\in \wt{Y}_{i_1}(K)$ of $h$. Since  $(g,h,x)\in\wt{\fC}$, we have $\Ad_{\wt{g}^{-1}\wt{h}}(x)\in\Lie(I)$. Hence
$\pr_{T_w}(\wt{g}^{-1}\wt{h})\in Z_2$, therefore there exists $\la\in\La_w$ such that $\wt{g}^{-1}\wt{h}\la^{-1}\in \wt{Y}_{i_2}$.
Then $(\wt{h}\la \wt{h}^{-1})\wt{g}=\wt{h}(\wt{g}^{-1}\wt{h}\la^{-1})^{-1}$ belongs to $\wt{Y}_{i_1}\wt{Y}_{i_2}^{-1}\subset\wt{Y}_i$,
hence $(h\la h^{-1})g\in Y_i$, as claimed.

Now assertions (a) and (b) of \rt{topproper} follow from assertions (i) and (ii) of \rp{quot}(b), respectively.
\end{proof}
\end{Emp}

\subsection{Proof of \rp{BC}} \label{S:pfrlBC}

\begin{Emp}
\begin{proof}[Proof of \rp{BC}]
Arguing as in the proof of \rl{twist}, we see that the action map
$ W\times\cL'(X)\to \cL'(X)\times_{\cL'(Y)}\cL'(X)$ and hence also
$W\times  \cL'(X)_{\cL Y}\to \cL'(X)_{\cL Y}\times_{\cL Y}\cL'(X)_{\cL Y}$ is an isomorphism.

It remains to show that the map $\cL'(X)_{\cL Y}\to \cL(Y)$ is an epimorphism. In other words, for every $k$-algebra $A$, and an $A$-point $y\in \cL(Y)(A)$, there exists an \'etale covering $\Spec B\to \Spec A$ and a point $x\in\cL'(X)(B)$ such that $\pi(x)=y\in \cL'(Y)(B)$.

By definition, $y$ corresponds to a morphism $y:\Spec A((t))\to Y$. Since $X\to Y$ is a $W$-torsor, the pullback $\pi_y:X_y\to \Spec A((t))$ is a $W$-torsor. Now we have to show that there exists an \'etale covering $\Spec B\to\Spec A$ such that
the pullback $X_{y,B((t^{1/h}))}\to \Spec B((t^{1/h}))$ of $\pi_y$ has a section. In other words, the assertion follows from \rl{BC} below.
\end{proof}
\end{Emp}

\begin{Lem} \label{L:BC}
Let $A$ be a $k$-algebra,  $W$ a finite group of exponent $h$ prime to the characteristic of $k$, and $\pi:X\to \Spec A((t))$ a $W$-torsor.
Then there exists an \'etale covering $\Spec B\to\Spec A$ such that pullback $X_{B((t^{1/h}))}\to \Spec B((t^{1/h}))$ of $\pi$ is a trivial $W$-torsor.
\end{Lem}

Our argument is based on results of \cite{BC}. First we show the following claim.

\begin{Cl} \label{C:colimBC}
Let $B\simeq\colim_{\al}B_{\al}$ be a filtered colimit of $A$-algebras, and let $\pi:X\to \Spec A((t))$ be a $W$-torsor such that
the pullback $X_{B}\to \Spec B((t))$ of $\pi$ is a trivial $W$-torsor. Then for every sufficiently large $\al$, the pullback
$X_{B_{\al}}\to \Spec B_{\al}((t))$ of $\pi$ is a trivial $W$-torsor.
\end{Cl}

\begin{proof}
Let $A\{\{t\}\}$ be the Henselization of $(A[t],(t))$. By \cite[Corollary 2.1.21]{BC}, the homomorphism $A\{\{t\}\}[\frac{1}{t}]\to A((t))$ induces
an equivalence of categories between $W$-torsors on $\Spec A((t))$ and $W$-torsors on $\Spec A\{\{t\}\}[\frac{1}{t}]$. In particular, $\pi$ is induced from
a $W$-torsor $\pi':X'\to \Spec A\{\{t\}\}[\frac{1}{t}]$, and the pullback $X'_B\to \Spec B\{\{t\}\}[\frac{1}{t}]$ of $\pi'$ is a trivial $W$-torsor.

Since Henselization commute with filtered colimits (see \cite[Tag 0A04]{Sta}), we conclude that  $B\{\{t\}\}[\frac{1}{t}]\simeq\colim_{\al}B_{\al}\{\{t\}\}[\frac{1}{t}]$. Therefore the pullback
$X'_{B_{\al}}\to \Spec B_{\al}\{\{t\}\}[\frac{1}{t}]$ of $\pi'$ is a trivial $W$-torsor for all sufficiently large $\al$, as claimed.
\end{proof}

\begin{Emp}
\begin{proof}[Proof of \rl{BC}]

Assume first that $A$ is a local strictly Henselian ring. Then it follows from \cite[Lemma 3.2.1]{BC} that there exists $d\in\B{N}$, prime to the characteristic of $k$, such that the pullback  $\pi_d:X_{A((t^{1/d}))}\to \Spec A((t^{1/d}))$ of $\pi$ is a trivial $W$-torsor.

Thus $\pi_d$ has a section, hence the finite \'etale covering  $\phi:\Spec A((t^{1/d}))\to \Spec A((t))$ factors as  $\Spec A((t^{1/d}))\to X'\to \Spec A((t))$ for some
clopen subscheme $X'\subset X$. Since $\phi$ is a finite Galois covering with Galois group $\mu_d$, we have $X'\simeq  \Spec A((t^{1/d'}))$ for some $d'|d$. Moreover, since $\pi$ is a $W$-torsor, where $W$ is of exponent $h$, we conclude that $d'|h$.
Hence the pullback $\pi_m:X_{A((t^{1/h}))}\to \Spec A((t^{1/h}))$ of $\pi$ has a section, as claimed.

Now let $A$ be a general $k$-algebra. Since $\Spec A$ is quasi-compact, it suffices to show that for every point
$y$ of $\Spec A$ there exists an \'etale covering $\Spec B_{\al}\to \Spec A$ of $y$ such that
the pullback $X_{B_{\al}((t^{1/h}))}\to\Spec B_{\al}((t^{1/h}))$ of $\pi$ is trivial.

Choose a geometric point $\ov{y}$ of $\Spec A$ supported at $y$, and let $A_{\ov{y}}^{sh}$ be the strict Henselization of $A$ at $\ov{y}$. Then, by the already shown particular case, the pullback $X_{A_{\ov{y}}^{sh}((t^{1/h}))}\to  \Spec A_{\ov{y}}^{sh}((t^{1/h}))$ of $\pi$ is trivial.
By definition, the strict Henselisation $A_{\ov{y}}^{sh}$ is a filtered colimit $\colim_{\al} B_{\al}$ of $A$-algebras such that  each $\Spec B_{\al}\to\Spec A$ is an \'etale covering of $y$. Thus, by \rcl{colimBC}, there exists $\al$ such that the pullback
$X_{B_{\al}((t^{1/h}))}\to  \Spec B_{\al}((t^{1/h}))$ of $\pi$ is trivial. This completes the proof.
\end{proof}
\end{Emp}

\subsection{Proof of \rt{drinfeld}} \label{S:pfrtdrinfeld}

\begin{Emp}
{\bf Remarks.} (a) The strategy of proof was communicated to us by V. Drinfeld.

(b) In the case when the characteristic of $k$ is zero, the assertion for $\chi_n$ was proven by Mustata-Einsenbud-Frenkel \cite[Theorem A.4]{Mus}.
\end{Emp}

We prove both assertions at the same time. It suffices to show that there exist faithfully flat morphisms
$Z_n\to \clp_{n}(\kg)$ and $Z_{I,n}\to \Lie_n(I)$ such that both compositions $Z_n\to \clp_{n}(\kg)\to \clp_{n}(\kc)$ and $Z_{I,n}\to \Lie_n(I)\to \clp_{n}(\kc)$ are flat. We will use a global argument that involves flatness of the Hitchin fibration and its parabolic variant.
For convenience of the reader, we will divide our argument into steps.

\begin{Emp} \label{E:step1}
Consider two distinct points $x,\infty\in \bP^{1}(k)$ and an effective divisor $D$ on $\bP^1$, supported on $\bP^{1}\sm\{x,\infty\}$.
We have a $\bG_m$-action on $\kg$ by homothety that commutes with adjoint action, thus inducing an $\bG_m$-action on $\kc$. Hence we can form the
twisted version $\kc_D:=\kc\times^{\bG_{m}} \C{Z}^{\times}(D)$, where $\C{Z}^{\times}(D)$ is the $\bG_m$-torsor, corresponding to the line bundle $\C{O}(D)$ and similarly $\kg_D:=\kg\otimes\C{O}(D)$. Both are vector bundles over $\bP^1$, trivialized on $\bP^1\sm D$.
For every $G$-torsor $E$ on $\bP^1$, let $\ad(E)$ be the corresponding vector bundle on $\bP^1$. Then the map $\chi:\kg\to\kc$ induces a morphism \[H^{0}(\bP^1,\ad(E)\otimes\C{O}(D))\to H^{0}(\bP^1,\kc_D).\]

We choose $D$ sufficiently big so that the maps
\begin{equation}
(\ev^{(n)}_{x},\ev_{\infty}):H^{0}(\bP^{1},\kg_{D})\rightarrow H^{0}(nx\cup\infty,\kg_{D})\simeq\clp_{n}(\kg)\oplus\kg,
\label{surj1}
\end{equation}
\begin{equation}
(\ev^{(n)}_{x},\ev_{\infty}):H^{0}(\bP^{1},\kc_{D})\rightarrow H^{0}(nx\cup\infty,\kc_{D})\simeq\clp_{n}(\kc)\oplus\kc,
\label{surj2}
\end{equation}
where $\ev_x^{(n)}$ (resp. $\ev_{\infty}$) is the restriction to the $n$-th formal neighborhood at $x$ (resp. to $\infty$), are both surjective.
\end{Emp}

\begin{Emp} \label{E:step2}
Set $\cA_{D,\infty}:=\{a\in H^{0}(\bP^{1},\kc_{D})~\vert~\ev_{\infty}(a)\in\kc^{\rs}\}$, and let
 $\cM_{D,\infty}$  be the corresponding Hitchin total space. More precisely, $\cM_{D,\infty}$ classifies pairs $(E,\phi)$, where $E$ is a $G$-torsor on $\bP^1$ and $\phi\in H^{0}(\bP^1,\ad(E)\otimes \C{O}(D))$ such that $\chi(\phi)\in\cA_{D,\infty}$.

From the surjectivity of \eqref{surj2} we get that
\begin{equation}
\text{ the map } \ev_x^{(n)}:\cA_{D,\infty}\rightarrow\clp_{n}(\kc) \text{ is smooth and surjective.}
\label{surj3}
\end{equation}

\end{Emp}

\begin{Emp} \label{E:step3.}
Following Yun (see \cite{Yun}), we consider the parabolic Hitchin space $\cM^{par}_{D,\infty}$, which classifies triples $(E,\phi,E_{ B})$ such that
\begin{itemize}
	\item
	$(E,\phi)\in\cM_{D,\infty}$,
	\item
	$E_{ B}$ is a $ B$-reduction of the restriction $E|_x$ such that $\ev_x(\phi)\in \ad(E|_x)$ belongs to $\ad(E_{ B})$.
\end{itemize}

By \cite[4.16.4]{N} and \cite[2.5.2]{Yun}, the fibrations
\begin{center}
$\cM_{D,\infty}\rightarrow \cA_{D,\infty}$ and $\cM^{par}_{D,\infty}\rightarrow \cA_{D,\infty}$
\end{center}
are faithfully flat, so by \eqref{surj3},  the compositions
\begin{equation} \label{ffs}
\cM_{D,\infty}\rightarrow \cA_{D,\infty}\stackrel{\ev_x^{(n)}}{\lra} \clp_{n}(\kc) \text { and }
\cM^{par}_{D,\infty}\rightarrow \cA_{D,\infty}\stackrel{\ev_x^{(n)}}{\lra} \clp_{n}(\kc)
\end{equation}
also are.

\end{Emp}

\begin{Emp} \label{E:step4}
Let $\cM_{D,\infty}^{nx}\rightarrow\cM_{D,\infty}$ be the $\clp_n(G)$-torsor, classifying trivializations $\iota$ of $E$ at the $n$-th formal neighborhood at $x$.
Then one has a map
\begin{center}
$\res_{n}:\cM^{nx}_{D,\infty}\rightarrow\clp_n(\kg)$,
\end{center}
which sends a triple $(E,\phi,\iota)$ to the image of $\phi$ under the natural map
\[
H^{0}(\bP^1,\ad(E)\otimes \C{O}(D))\overset{\ev_x^{(n)}}{\lra}H^{0}(nx,\ad(E))\overset{\iota}{\lra}\clp_n(\kg).
\]

Let $Z_{n}\subset \cM^{nx}_{D,\infty}$ be the largest open substack of $\cM^{nx}_{D,\infty}$, where $\res_{n}$ is smooth.

\end{Emp}

\begin{Emp} \label{E:step5}
We claim that the restriction $r_{n}:Z_n\to \clp_n(\kg)$ of $\res_n$ is faithfully flat, and its composition with $\chi_n$ is flat.

First of all, $r_n$ is smooth, by assumption, hence flat.
Next,  since the first map of \eqref{ffs} and the projection $\cM^{nx}_{D,\infty}\to \cM_{D,\infty}$ are flat, we conclude from the commutative  diagram
\begin{equation} \label{Eq:zn}
\xymatrix{Z_{n}\ar[r]&\cM_{D,\infty}^{nx}\ar[r]\ar[d]&\clp_n(\kg)\ar[d]^{\chi_n}\\&\cA_{D,\infty}\ar[r]&\clp_n(\kc)}
\end{equation}
that the composition $Z_n\to\clp_n(\kg)\to\clp_n(\kc)$ is flat.

We claim that $r_n$ is surjective. More precisely, we claim that the locus of those triples $(E,\phi,\iota)$, where $E$ is trivial, is contained in
$Z_n$ and the restriction of $\res_n$ to such points is surjective. Let $\wt{Z}_n$ be the moduli space of quadruples $(E,\phi,\eta,\iota)$, where $(E,\eta,\iota)\in \cM^{nx}_{D,\infty}$ and $\eta$ is a trivialization of $E$, and let $\om:\wt{Z}_n\to \cM_{D,\infty}^{nx}$ be the forgetful morphism.

Then the image of $\om$ consists of all triples $(E,\phi,\iota)$, where $E$ is trivial, and it suffices to show that the composition $\res_n\circ\om$ is smooth and surjective. Indeed, since the morphism $\res_n\circ\om$ is a smooth morphism between smooth stacks, the differential $d(\res_n\circ\om)$ is surjective, therefore the differential $d\res_n$ is surjective. Since $\res_n$ is a morphism between smooth algebraic stacks
(by \cite[Theorem 4.14.1]{N}), this implies that $\res_n$ is smooth at each point in the image of $\om$, and we are done.

Note that $\wt{Z}_n$ decomposes as $\clp_n(G)\times H^{0}(\bP^1,\kg_D)_{\infty-rs}$, where $H^{0}(\bP^1,\kg_D)_{\infty-rs}\subset H^{0}(\bP^1,\kg_D)$ consists of $\phi\in H^{0}(\bP^1,\kg_D)$ such that $\ev_{\infty}(\phi)\in\kg^{\rs}$.  Moreover, under this identification, $\res_n\circ\om$ is nothing else but composition of the evaluation map
\[
\ev^{(n)}_x:\clp_n(G)\times H^{0}(\bP^1,\kg_D)_{x-rs}\to \clp_n(G)\times \clp_n(\kg)
\]
and the action map $\clp_n(G)\times\clp_n(\kg)\to\clp_n(\kg)$. Therefore the smoothness and the surjectivity of $\res_n\circ\om:\wt{Z}_n\to\clp_n(\kg)$ follows from the surjectivity of \eqref{surj1}.
\end{Emp}

\begin{Emp} \label{E:step6}
Similarly, consider the moduli space $\cM^{par,nx}_{D,\infty}$, which classifies quadruples $(E,\phi,E_{ B},\iota)$ such that
\begin{itemize}
	\item
	$(E,\phi,E_{ B})\in\cM^{par}_{D,\infty}$,
	\item
	$\iota$ is a trivialization of $E$ at the $n$-th formal neighborhood at $x$, which induces a trivialization of $E_{ B}$.
\end{itemize}
Then we have a Cartesian diagram
\begin{equation}
\xymatrix{\cM^{par,nx}_{D,\infty}\ar[d]\ar[r]&\Lie_n(I)\ar[d]\\\cM^{nx}_{D,\infty}\ar[r]^{\res_n}&\clp_n(\kg)},
\label{cartB}
\end{equation}
so the pullback $r_{I,n}:Z_{I,n}\to \Lie_n(I)$ of $r_{n}:Z_{n}\to \clp_n(\kg)$ is smooth and surjective. It remains to show that the
composition  $Z_{I,n}\to\Lie_n(I)\to\clp_n(\kc)$ is flat. But the last composition decomposes as a composition of three flat maps
\[
Z_{I,n}\to\cM^{par,nx}_{D,\infty}\to \cM^{par}_{D,\infty}\to\clp_n(\kc),
\]
the first of which is an open embedding, the second one is smooth, and the third one the second map of \eqref{ffs}.
Therefore the composition is flat, and the proof is complete.
\end{Emp}

\section*{List of main terms and symbols}


\begin{multicols}{3}

\e admits gluing of sheaves, \pg{I:gluing}

\e affine morphism, \pg{I:affine morpism}

\e arc space, \pg{I:arc space}

\e connected $\infty$-stack, \pg{I:connected}

\e constructible stratification, \pg{I:constructible stratification}

\e covering, \pg{I:covering}

\e closed embedding, \pg{I:closed embedding}

\e equidimensional, \pg{I:equidimensional morphism}, \pg{I:eq}, \pg{I:eq1}

\e fp, \pg{I:fp}

\e fp-open embedding, \pg{I:fp-open embedding}

\e fp-closed embedding, \pg{I:fp-closed embedding}

\e fp-locally closed embedding, \pg{I:fp-locally closed embedding}

\e fp-representable morphism, \pg{I:fp representable morpism}

\e fp-schematic morphism, \pg{I:fp schematic morpism}

\e ind-algebraic space, \pg{I:ind-algebraic space}

\e ind-fp-proper, \pg{I:ind-fp-proper}

\e ind-scheme, \pg{I:ind-scheme}

\e  ind-placid, \pg{I:ind-placid}


\e $\infty$-stack  , \pg{I:infty stacks}

\e locally closed embedding, \pg{I:locally closed embedding}

\e locally fp-representable, \pg{I:loc fp representable morpism}

\e locally fp-schematic, \pg{I:loc fp schematic morpism}

\e locally fp-affine morphism, \pg{I:loc fp affine morpism}


\e loop space, \pg{I:loop space}

\e $n$-placid, \pg{I:nplacid}

\e $n$-smooth, \pg{I:nsmooth}

\e open embedding, \pg{I:open embedding}

\e open morphism, \pg{I:open morphism}

\e $(P)$-representable, \pg{I:prepresentable}

\e $(P_{\red})$-representable, \pg{I:pred representable}

\e perversity, \pg{I:perversity}

\e placid algebraic space, \pg{I:placid scheme}

\e placid $\infty$-stack, \pg{I:placid infty stack}

\e placid presentation, \pg{I:placid presentation}

\e placid scheme, \pg{I:placid scheme}

\e placidly stratified, \pg{I:placidly stratified infty stack}

\e pure codimension, \pg{I:pure codimension}, \pg{I:pure2}

\e reduced $\infty$-stack, \pg{I:reduced infty stack}

\e representable morphism, \pg{I:representable morpism}

\e schematic morphism, \pg{I:schematic morpism}

\e semi-small morphism, \pg{I:semismall morphism}

\e smooth morphism,  \pg{I:smooth morphism}

\e space over $k$, \pg{I:infty space}

\e stratified $\infty$-stack, \pg{I:stratified infty stack}

\e strongly pro-smooth, \pg{I:strongly pro-smooth}

\e support, \pg{I:support}

\e topological equivalence, \pg{I:topological equivalence}

\e topologically fp-proper, \pg{I:topologically fp-proper}

\e topologically locally fp, \pg{I:topologically locally fp}


\e topologically fp-(locally) closed embedding, \pg{I:topologically (fp)-locally closed embedding}

\e torsor, \pg{I:torsor}

\e universally open, \pg{I:universally open morphism}

\e uo-equidimensional, \pg{I:uo-equidimensional}, \pg{I:uoeq}, \pg{I:uoeq1}

\e $\cU$-small morphism, \pg{I:small morphism}

\e weakly equidimensional morphism, \pg{I:weakly equidimensional morphism}, \pg{I:we}, \pg{I:we1}

\e $\{\cX_{\al}\}_{\al}$-adapted $\infty$-substack, \pg{I:adapted infty substack}

\e

\e

\e

\e

\e

\e

\e

\e

\e

\e

\e

\e

\e

\e

\e

\e

\e

\e

\e

\end{multicols}

\begin{multicols}{5}
\e

\e

\e

\e $\Ad_g$, \pg{N:adg}

\e $\Aff_k$, \pg{N:affk}

\e $\Affft_k$, \pg{N:affft}

\e $\Alg_k$, \pg{N:algk}

\e $\Algft_k$, \pg{N:algft}

\e $\Alg_k^{\qcqs}$, \pg{N:qcqsalgsp}

\e  $a_{w,\br}$, \pg{N:awbr}

\e $a^+_{w,\br}$, \pg{N:a+wbr}

\e $B$, $\kb$, \pg{N:b}

\e $b_{w,\br}$, \pg{N:bwbr}

\e $b^+_{w,\br}$, \pg{N:b+wbr}

\e  $\Cat_{\ell}$ \pg{N:catell}

\e $\Catst$, \pg{N:catst}

\e $\fC$, \pg{N:fC}

\e $\fC\bu$, \pg{N:fCbullet}

\e $\fC_{\leq m}$, \pg{N:fCleqm}

\e $\fC_{\tn,\bullet}$, \pg{N:fCtnbullet}

\e  ${\fC}_{\ft,w,\br}$, \pg{N:fCftwbr}

\e $\fC_{w,\br}$, \pg{N:fCwbr}

\e $\wt{\fC}$, \pg{N:wtfC}

\e $\wt{\fC}\bu$, \pg{N:wtfCbullet}

\e $\wt{\fC}_{\leq m}$, \pg{N:wtfCleqm}

\e $\wt{\fC}_{\tn,\bullet}$, \pg{N:wtfCtnbullet}

\e  $\wt{\fC}_{\ft,w,\br}$, \pg{N:wtfCftwbr}

\e $\wt{\fC}_{w,\br}$, \pg{N:wtfCwbr}

\e $\ov{\fC}_{\ft,w,\br}$, \pg{N:ovfCftwbr}

\e $c_w$, \pg{N:cw}

\e $\kc$, \pg{N:c}

\e $\kc^{\rs}$, \pg{N:crs}

\e $\fc_{w,\br}$, \pg{N:fcwbr}

\e $\mathfrak{D}$, \pg{N:mathfrakD}

\e $\cD_{c}, \cD$, \pg{N:Dc}, \pg{N:D2}

\e  $\cD_{\cdot}$, \pg{N:Ddot}

\e ${}^p\cD_{c}^{\leq 0},{}^p\cD_{c}^{\geq 0}$, \pg{N:pDc}

\e $\cD_{c}(X)$, \pg{N:DcX}

\e $\cD_{\cY}(\cX)$, \pg{N:DYX}

\e $\un{\dim}_{f}$, \pg{N:dimf}, \pg{N:dimf1}, \pg{N:dimf2}

\e $\ev_{i}$, \pg{N:evi}

\e $\ev_{\cX}$, \pg{N:evx}

\e $F$, \pg{N:F}

\e  $\Fl$, \pg{N:Fl}

\e $f^{*,\ren}$, \pg{N:fren}

\e $G$, $\kg$, \pg{N:g}

\e $\kg^{\rs}$, \pg{N:grs}

\e  $h$, \pg{N:m}

\e $I$, \pg{N:I}

\e  $\Ind$, \pg{N:ind}

\e $\cI_{\cX'}$, \pg{N:IX'}

\e $\ins_{i}$, \pg{N:insi}

\e $\Lie(I)\bu$, \pg{N:lieibullet}

\e $\Lie(I)_{\leq m}$, \pg{N:lieleqm}

\e $\Lie(I)_{\tn}$, \pg{N:lieitnbullet}

\e $\Lie(I)_{w,\br}$, \pg{N:lieIwbr}

\e $k$, \pg{N:k}

\e $(\cL\fc)_{\bullet}$, \pg{N:lcbullet}

\e  $(\cL\kg)_{\bullet}$, \pg{N:lgbullet}

\e $\cL\cX$, \pg{N:lx}

\e $\clp(\B{A}^1)_{\bullet}$, \pg{N:a1bullet}

\e $\cL^+(\kc)_{\bullet}$, \pg{N:clpcbullet}

\e $\cL^+(\kc)_{\tn}$, \pg{N:clpkctn}

\e $\cL^+(\fc)_{\tn,\bullet}$, \pg{N:clpfcbullet}

\e $\clp(\ft)_\bullet$, \pg{N:clpftbullet}

\e $\cL^+(\kt_{w})_{\tn}$, \pg{N:clpktwtn}

\e $\clp(\cX)$, \pg{N:clpx}

\e $\clp(X)_{(f;n)}$, \pg{N:clpxfn}

\e $\clp(X)_{f\neq 0}$, \pg{N:clpxfneq0}

\e  $\cL^+_n(\cX)$, \pg{N:clpnx}

\e $\cL'(X)$, \pg{N:l'x}

\e $\cL'^+(X)$, \pg{N:clp'x}

\e $\ell$, \pg{N:ell}

\e $\co$, \pg{N:O}

\e $(P_{\red})$, \pg{N:pred}

\e $\PrCat$, \pg{N:prcat}

\e $\Prest_k$, \pg{N:prestk}

\e $p_{f}$, \pg{N:pf}

\e $p_{\nu}$, \pg{N:pnu}

\e $\frak{p}$, $\frak{p}_{w,\br}$, \pg{N:frakp}

\e $\ov{\frak{p}}$, $\ov{\frak{p}}_{\bullet}$, \pg{N:ovfrakp}

\e $\ov{\frak{p}}_{\tn,\bullet}$, \pg{N:ovfraktnbullet}

\e $\ov{\frak{p}}_{w,\br}$, \pg{N:ovfrakpwbr}.

\e $R$, \pg{N:R}

\e $r$, \pg{N:r}

\e  $\Sch_k$, \pg{N:schk}

\e $\Sch_k^{\qcqs}$, \pg{N:qcqssch}

\e $\St_k$, \pg{N:stk}

\e $\cS$, $\cS_{\bullet}$, $\cS_{\leq 0}$, \pg{N:cS}

\e $\cS_{\tn}$, $\cS_{\tn,\bullet}$, \pg{N:cStnbullet}

\e $T$, \pg{N:T}

\e $T_w$, \pg{N:Tw}

\e $\kt$, \pg{N:t}

\e $\kt^{\rs}$, \pg{N:trs}

\e $\ft_\br$, \pg{N:ftbr}

\e  $\ft_{w,\br}$, \pg{N:ftwbr}

\e $v_{w,\br}$, \pg{N:vwbr}

\e $W$, $\widetilde{W}$, \pg{N:Waff}

\e $W_{w,\br}$, \pg{N:Wwbr}

\e $[\C{X}]$, \pg{N:[x]}

\e $\cX_{\red}$, \pg{N:xred}

\e $\cX\sm\cY$, \pg{N:x-y}

\e  $X_w$, \pg{N:xw}

\e $X_{*}(T)$, \pg{N:X*T}

\e $\delta_{w,\br}$, \pg{N:dtwbr}

\e $\eta_*$, \pg{N:eta*}

\e  $\La_S$, \pg{N:laS}

\e $\La_w$, \pg{N:Law}

\e $\nu_{\al}$, \pg{N:nual}

\e $\pi$, \pg{N:pi}

\e $\un{\pi}_0$, \pg{N:pi0}

\e $\chi$, \pg{N:chi}

\e $\psi_S$, \pg{N:psiS}

\e $\psi_{w,\br}$, \pg{N:psiwbr}

\e $\ov{\psi}_{w,\br}$, \pg{N:ovpsiwbr}

\e $[\psi_{w,\br}]$, \pg{N:[psiwbr]}

\e $\om_{\cX}$, \pg{N:omX}

\e $\cdot\lan d\ran$, $\cdot\lan \un{d}\ran$, \pg{N:lanran}





\e

\e

\e
\e

\e

\e

\end{multicols}

\Addresses
\end{document}